\documentclass[11pt,a4paper]{article}
\usepackage[top=1in, bottom=1in, left=1in, right=1in]{geometry}
\usepackage[utf8]{inputenc}
\usepackage[english]{babel}
\usepackage{amssymb}
\usepackage{graphicx}
\graphicspath{ {./images/} }
\usepackage[colorlinks=true, allcolors=blue]{hyperref}
\usepackage{tikz}
\usepackage{amsthm}
\usepackage{amsmath}
\usepackage{cancel}
\usepackage{tikz-cd}
\usepackage{mathtools}
\usepackage{tikz}
\usepackage{pgf}
\usepackage{blindtext}
\usepackage{hyperref}
\usetikzlibrary{matrix,arrows,decorations.pathmorphing}
\tikzset{commutative diagrams/.cd}
\usetikzlibrary{cd}
\usepackage{cancel}
\usepackage{nicefrac}
\usepackage{MnSymbol}
\usepackage{centernot}
\usepackage{mathtools}
\usepackage{ stmaryrd }
\usepackage{wrapfig}
\usepackage{float}
\usepackage{hyperref}
\usepackage{soul}
\hypersetup{
    colorlinks=true,       
    linkcolor=blue,          
    citecolor=magenta,       
    filecolor=magenta,      
    urlcolor=blue          
}
\usepackage{aliascnt}
\usepackage{soul}
\usepackage{color}

\theoremstyle{plain}% default
\newtheorem{thm}{Theorem}[section]
\newtheorem{lem}[thm]{Lemma}
\newtheorem{prop}[thm]{Proposition}
\newtheorem{cor}[thm]{Corollary}

\newtheorem*{thm*}{Theorem}

\newtheorem{rcor}{Corollary}
\newtheorem{rthm}{Theorem}

\newenvironment{repcorollary}[1]
  {\renewcommand\thercor{\ref*{#1}}\rcor}
  {\endrcor}
\newenvironment{reptheorem}[1]
  {\renewcommand\therthm{\ref*{#1}}\rthm}
  {\endrthm}

\theoremstyle{definition}
\newtheorem{defn}[thm]{Definition}
\newtheorem{ex}[thm]{Example}

\theoremstyle{remark}
\newtheorem{rmk}[thm]{Remark}
\theoremstyle{remark}
\newtheorem{obs}[thm]{Observation}
\newtheorem{nt}[thm]{Notation}

\makeatletter
\DeclareFontFamily{OMX}{MnSymbolE}{}
\DeclareSymbolFont{MnLargeSymbols}{OMX}{MnSymbolE}{m}{n}
\SetSymbolFont{MnLargeSymbols}{bold}{OMX}{MnSymbolE}{b}{n}
\DeclareFontShape{OMX}{MnSymbolE}{m}{n}{
    <-6>  MnSymbolE5
   <6-7>  MnSymbolE6
   <7-8>  MnSymbolE7
   <8-9>  MnSymbolE8
   <9-10> MnSymbolE9
  <10-12> MnSymbolE10
  <12->   MnSymbolE12
}{}
\DeclareFontShape{OMX}{MnSymbolE}{b}{n}{
    <-6>  MnSymbolE-Bold5
   <6-7>  MnSymbolE-Bold6
   <7-8>  MnSymbolE-Bold7
   <8-9>  MnSymbolE-Bold8
   <9-10> MnSymbolE-Bold9
  <10-12> MnSymbolE-Bold10
  <12->   MnSymbolE-Bold12
}{}

\let\llangle\@undefined
\let\rrangle\@undefined
\DeclareMathDelimiter{\llangle}{\mathopen}%
                     {MnLargeSymbols}{'164}{MnLargeSymbols}{'164}
\DeclareMathDelimiter{\rrangle}{\mathclose}%
                     {MnLargeSymbols}{'171}{MnLargeSymbols}{'171}
\makeatother

\makeatletter
\newcommand\ackname{Acknowledgements}
\if@titlepage
  \newenvironment{acknowledgements}{%
      \titlepage
      \null\vfil
      \@beginparpenalty\@lowpenalty
      \begin{center}%
        \bfseries \ackname
        \@endparpenalty\@M
      \end{center}}%
     {\par\vfil\null\endtitlepage}
\else
  \newenvironment{acknowledgements}{%
      \if@twocolumn
        \section*{\abstractname}%
      \else
        \small
        \begin{center}%
          {\bfseries \ackname\vspace{-.5em}\vspace{\z@}}%
        \end{center}%
        \quotation
      \fi}
      {\if@twocolumn\else\endquotation\fi}
\fi
\makeatother

\newcommand{\Rf}{\mathcal{R}^{\textit{f}}}
\newcommand{\Sf}{\mathcal{S}^{\textit{f}}}
\newcommand{\widehatSf}{\widehat{\mathcal{S}}^{\textit{f}}}
\newcommand{\salbar}{\overline{Sal}(\Gamma)}
\newcommand{\sal}{Sal(\Gamma)}

\newcommand{\salbargam}{\overline{Sal}(\widehat{\Gamma})}
\newcommand{\bigslant}[2]{{\raisebox{.2em}{$#1$}\left/\raisebox{-.2em}{$#2$}\right.}}
\newcommand{\ggr}{\langle R \rangle}
\newcommand{\cX}{\mathcal{X}}
\newcommand{\cY}{\mathcal{Y}}
\newcommand{\cZ}{\mathcal{Z}}
\newcommand{\cT}{\mathcal{T}}
\newcommand{\Sn}{\mathfrak{S}_n}
\newcommand{\A}{\mathrm{A}}
\newcommand{\W}{\mathrm{W}}
\newcommand{\PVB}{\mathrm{PVB}}

\newcommand{\VB}{\mathrm{VB}}
\newcommand{\VA}{\mathrm{VA}}
\newcommand{\KVA}{\mathrm{KVA}}
\newcommand{\PVA}{\mathrm{PVA}}
\newcommand{\LVA}{\mathrm{LVA}}
\newcommand{\Prod}{\mathrm{Prod}}
\newcommand{\M}{\mathrm{M}}
\newcommand{\Stab}{\mathrm{Stab}}
\newcommand{\id}{\mathrm{id}}
\newcommand{\Dd}{\mathcal{D}}
\begin{document}
\title{Spaces Related to Virtual Artin Groups}
\author{Federica Gavazzi}

\maketitle

\begin{abstract}\noindent
This work explores the topological properties of virtual Artin groups, a recent extension of the ``virtual" concept\textemdash  initially developed for braids\textemdash  to all Artin groups, as introduced by Bellingeri, Paris, and Thiel. For any given Coxeter graph $\Gamma$, we define a CW-complex $\Omega(\Gamma)$ whose fundamental group is isomorphic to the pure virtual Artin group $\PVA[\Gamma]$, which coincides with the pure virtual braid group when $\Gamma$ is $A_{n-1}$. This construction generalizes the previously studied BEER complex, originally defined for pure virtual braids, to all Coxeter graphs. We investigate the asphericity of $\Omega(\Gamma)$ and demonstrate that it holds when $\Gamma$ is of spherical type or of affine type, thereby characterizing
$\Omega(\Gamma)$ as a classifying space for 
$\PVA[\Gamma]$. To achieve this, we establish a connection between $\Omega(\Gamma)$ and the Salvetti complex associated with a specific Coxeter graph $\widehat{\Gamma}$ related to $\Gamma$, showing that they share a common covering space. This finding links the asphericity of $\Omega(\Gamma)$ to the $K(\pi, 1)$-conjecture for Artin groups associated with $\widehat{\Gamma}$. Additionally, the paper introduces and studies almost parabolic (AP) reflection subgroups, which play a crucial role in constructing these complexes.
\medskip

{\footnotesize

}

\end{abstract}

\section{Introduction}
Virtual Artin groups were defined by Bellingeri, Paris and Thiel in \cite{BellParThiel}, extending to all Artin groups the already largely studied structure of virtual braids. Indeed, given $S$ a countable set and $\Gamma$ a Coxeter graph on $S$, a virtual Artin group (denoted by $\VA[\Gamma]$) has two distinct kinds of generators: a first set $\{\sigma_s\,\mid\,s\in S\}$ playing the role of the classical Artin generators, and a second one $\{\tau_s\mid s\in S\}$, representing the virtual generators. The relations appearing in the presentation
are of three types: Artin relations involving the classical generators, Coxeter relations involving the $\tau_s$'s, and, lastly, mixed relations, mimicking the action of the Coxeter group $\W[\Gamma]$ on its root system.
As for virtual braid groups, two natural homomorphisms are defined: $\iota_{\A}:\A[\Gamma]\longrightarrow \VA[\Gamma]$ and $\iota_{\W}:\W[\Gamma]\longrightarrow \VA[\Gamma]$ that are proven to be injective, and two surjective homomorphisms
$\pi_P:\VA[\Gamma]\longrightarrow \W[\Gamma]$ and $\pi_{K}:\VA[\Gamma]\longrightarrow \W[\Gamma]$, which give rise to two emblematic kernel subgroups of $\VA[\Gamma]$: the pure virtual Artin group $\PVA[\Gamma]$ and $\KVA[\Gamma]$, respectively. 
The latter two groups are studied by the same authors in \cite{BellParThiel}, in which they exhibit a presentation with generators and relations for both and prove that $\KVA[\Gamma]$ is an Artin group with respect to some Coxeter graph $\widehat{\Gamma}$, closely related to the root system of $\W[\Gamma]$.

\medskip
\noindent A connected CW-complex $X$ is called a \emph{classifying space} for a discrete group $G$ if its fundamental group is $G$ and all the higher homotopy groups vanish.
When the Coxeter graph $\Gamma$ is $A_{n-1}$, then the virtual Artin group $\VA[A_{n-1}]$ is the virtual braid group on $n$ strands, and $\PVA[A_{n-1}]$ is the pure virtual braid group, whose presentation was given by Bardakov in \cite{Bard04}.
This last group was further studied by Bartholdi, Enriquez, Etingof and Rains in \cite{BEER} in relation with the Yang\textendash Baxter equations. In particular, in the cited work, the authors build a space $\Omega_n$, referred to as the \textit{BEER complex}, that has fundamental group isomorphic to $\PVB_n$. Moreover, they claim that this space is locally CAT(0), which would imply asphericity and hence that $\Omega_n$ is a classifying space for the pure virtual braid group on $n$ strands. However, this last result is incorrect. In Subsection \ref{originalBEERspace}, we prove the following result:
\begin{reptheorem}{FalseCat0}
The complex $\Omega_3$ is not locally CAT(0).
\end{reptheorem}
\noindent Thus, the local CAT(0) property can no longer be used to imply that the universal covering of $\Omega_n$ is contractible. 

\medskip
\noindent The purpose of this work is to generalize the construction of the BEER space $\Omega_n$, originally done for braids, to all Coxeter graphs $\Gamma$. We then obtain a CW-complex $\Omega(\Gamma)$ (see Subsection \ref{generalBeer}) whose fundamental group is isomorphic to the pure virtual Artin group $\PVA[\Gamma]$, and which coincides with $\Omega_n$ for $\Gamma=A_{n-1}$ (Theorem \ref{omegan=omegaxan}). We then investigate whether this generalized space is aspherical. In Section \ref{commoncovering}, we answer this question affirmatively if $\Gamma$ is a Coxeter graph of spherical or affine type.
We have then the following results:

\begin{repcorollary}{beeraspht}
 If $\Gamma$ is a Coxeter graph of spherical type, then  the BEER complex $\Omega(\Gamma)$ is aspherical. Consequently, it is a classifying space for the pure virtual Artin group $\PVA[\Gamma]$. 
\end{repcorollary}

\noindent
Thus, the space originally built in \cite{BEER} turns out to be aspherical, even though it is not locally CAT(0).

\begin{repcorollary}{originalbeeraspht}
    The Bartholdi\textendash Enriquez\textendash Etingof\textendash Rains complex $\Omega_n$ is aspherical.
\end{repcorollary}
\noindent
For the affine case, we obtain an analogous corollary after rephrasing some results in \cite{BellParThiel} that we used to show the spherical-type case:
\begin{repcorollary}{beerasphtAFF}
 If $\Gamma$ is a Coxeter graph of affine type, then  the BEER complex $\Omega(\Gamma)$ is aspherical. Consequently, it is a classifying space for the pure virtual Artin group $\PVA[\Gamma]$.
\end{repcorollary}
\noindent
The tools that we use to prove the asphericity of $\Omega(\Gamma)$ rely on comparing one of its covering spaces $\overline{\Omega}(\Gamma)$ with another cell complex $\overline{\Sigma}(\Gamma)$, that is known to be aspherical when $\Gamma$ is of spherical type or of affine type.

\medskip
\noindent
For any Artin group $\A[\Gamma]$ with  $\Gamma$ a finite Coxeter graph, we can build the Salvetti complex $\salbar$, introduced by Salvetti (see \cite{Sal87},\cite{Sal94}), and widely used to study and discuss the long-standing \emph{$K(\pi,1)$-conjecture}. Given $(\W[\Gamma],S)$ a Coxeter system with $S$ a finite set and $\W[\Gamma]$ acting properly and faithfully on a nonempty open convex cone $I$ in a real vector space $V$, we call reflections the elements in $\mathcal{R}=\{wsw^{-1}\,|\,w\in \W[\Gamma], s\in S\}$, and each $r\in \mathcal{R}$ has an associated reflection hyperplane $H_r$. The complement of the complexification of the union of all these reflection hyperplanes in
the complexification of $I$ gives a space on which $\W[\Gamma]$ acts freely and properly discontinuously. We can then consider the quotient under the action of the Coxeter group, obtaining a space denoted by $N(\Gamma)$, which is conjectured to be a classifying space for $\A[\Gamma]$. Thanks to Van der Lek (\cite{van1983homotopy}), we know that its fundamental group is indeed the Artin group $\A[\Gamma]$, so the $K(\pi,1)$-conjecture is true if the described space is aspherical. The complex defined by Salvetti has the same homotopy type as such a complement of the complexification of the reflection arrangement, and it has been proven to be aspherical for many types of Coxeter graphs so far. We recall here the case of our interest, that is, for $\Gamma$ of affine type, proven by Paolini and Salvetti in \cite{PaoSal}.

\medskip \noindent
In this work, we recall the definition of the Salvetti complex for a Coxeter graph $\Gamma$, and we point out that we do not necessarily need a finite set of vertices (see Subsection \ref{salgam}). In particular, we define this CW-complex inductively on its skeleta, and we describe explicitly its cells and attaching maps. In \cite{BellParThiel} the authors prove that $\KVA[\Gamma]$ is the Artin group $\A[\widehat{\Gamma}]$ (that in general is not finitely generated), thus, we can build the complex $\salbargam$ such that $\pi_1(\salbargam)=\A[\widehat{\Gamma}]$. For this space we find an explicit description (in Subsection \ref{salbargam}) that is independent of the definition of the graph $\widehat{\Gamma}$. We call this cell complex $\Sigma(\Gamma)$, and its fundamental group is isomorphic to $\KVA[\Gamma]$. \\
In Section \ref{commoncovering} we describe an isomorphism between two covering spaces $\overline{\Sigma}(\Gamma)$ and $\overline{\Omega}(\Gamma)$ of $\Sigma(\Gamma)$ and $\Omega(\Gamma)$, respectively. In particular, for any $\Gamma$ (finite or not), we show the following:
\begin{reptheorem}{maintheorem}
The Salvetti complex $\Sigma(\Gamma)$ and the BEER complex $\Omega(\Gamma)$ have a common covering space. 
\end{reptheorem}
\noindent
Therefore, the two spaces share the same universal covering, and $\Omega(\Gamma)$ is aspherical if and only if $\Sigma(\Gamma)$ is aspherical. Now we see that in the two cases in which $\Gamma$ is either of spherical type or of affine type, we can conclude that $\Sigma(\Gamma)$ (and, consequently, $\Omega(\Gamma)$) is aspherical.
\medskip 

\noindent When $\Gamma$ is of spherical type, the Coxeter graph $\widehat{\Gamma}$ has a finite number of vertices, so the open convex cone $I$ is well-defined and we can state the $K(\pi,1)$-conjecture for $\A[\widehat{\Gamma}]$.
Thanks to \cite[Theorem 6.3]{BellParThiel}, we know that in this case $\A[\widehat{\Gamma}]$ satisfies the $K(\pi,1)$-conjecture, and so $\salbargam=\Sigma(\Gamma)$ is aspherical. This implies (Corollary \ref{beeraspht}) the result of asphericity for the generalized BEER complex $\Omega(\Gamma)$ in the spherical-type case.
\medskip

\noindent If $\Gamma$ is not of spherical type, then $\widehat{\Gamma}$ has an infinite set of vertices. The classical $K(\pi,1)$-conjecture does not make sense in the case in which the Coxeter graph is infinite, as remarked in \cite[Section 6]{BellParThiel}. However, we can always define the Salvetti complex, and we know that it has the same homotopy type as that of $N(\Gamma)$ when the Coxeter graph is finite. For these cases, the conjecture is equivalent to asking whether $\salbar$ is aspherical. It is therefore reasonable to ``extend'' the conjecture also to the infinite cases, even if $N(\Gamma)$ and $I$ are not well-defined, by asking whether the Salvetti complex is aspherical. With this intent, we can rephrase Theorem \ref{maintheorem} saying that $\Omega(\Gamma)$ is aspherical if and only if the $K(\pi,1)$-conjecture is true for $\widehat{\Gamma}$.
\medskip

\noindent In Theorem 4.1 of \cite{BellParThiel}, the authors prove that for $\Gamma$, a Coxeter graph of spherical type or of affine type, all free of infinity subsets of the generators of $\widehat{\Gamma}$ are finite, and that they are of spherical or of affine type. Furthermore, in \cite{GodPar12} and \cite{EllSk10}, the authors prove that given a Coxeter graph $\Gamma$ on a set $S$, if for all free of infinity subsets $X\subset S$ the Salvetti complex $\overline{Sal}(\Gamma_X)$ is aspherical, then $\salbar$ is aspherical. Combining these results, we find the following.

\begin{reptheorem}{gammahatkpiAFF} Let $\Gamma$ be a Coxeter graph of spherical type or of affine type, then $\A[\widehat{\Gamma}]$ satisfies the $K(\pi,1)$-conjecture. Namely, $\salbargam=\Sigma(\Gamma)$ is aspherical.
\end{reptheorem}

\medskip

\noindent The Salvetti complex $\salbar$ of a Coxeter graph $\Gamma$ has a cell for each standard parabolic subgroup of $\W[\Gamma]$ of spherical type. The two CW-complexes $\Sigma(\Gamma)$ and $\Omega(\Gamma)$ will have the cells in bijection with a more general class of subgroups of $\W[\Gamma]$: the \textit{almost parabolic} (abbreviated as AP) reflection subgroups of spherical type. The AP reflection subgroups are defined and studied in Subsection \ref{APrefsubgrp}.

\medskip
\noindent
The structure of this article is as follows. In Section \ref{prem} we include all the preliminaries for the construction of the two spaces. We first introduce classical notions about Coxeter groups and Artin groups (Subsection \ref{CoxandArt}), then describe the fundamental cells that will be used to build $\Sigma(\Gamma)$ and $\Omega(\Gamma)$, namely, the Coxeter polytopes (\ref{coxeterpolytopes}). Here, we also establish the conventions for the orientations of the edges that will be followed throughout the work. At the end of the section, we recall the basic results about virtual Artin groups (\ref{virtualartingroups}) taken from \cite{BellParThiel}. In Section \ref{salberrap} we give the definition of the Salvetti complex $\salbar$ with attaching maps and skeleton structure (\ref{salgam}), then in Subsection \ref{APrefsubgrp} we treat the reflection subgroups and we define the AP reflection subgroups. We then use this concept in Subsection \ref{salbargam} to build the CW-complex $\Sigma(\Gamma)$ with fundamental group $\KVA[\Gamma]$. In Section \ref{BEERspace}, we first explain the construction of the space originally defined in \cite{BEER} (see \ref{originalBEERspace}), include the proof that its fundamental group coincides with $\PVB_n$, and later show that the local CAT(0) property already fails for $n=3$. In Subsection \ref{generalBeer}, we propose the definition of a CW-complex $\Omega(\Gamma)$ for any $\Gamma$ Coxeter graph, using again the structure of the AP reflection subgroups. In Subsection \ref{eqialityofspaces}, we verify that the space $\Omega(\Gamma)$ that we introduce in the previous section, when $\Gamma=A_{n-1}$, coincides with $\Omega_n$.
 Finally, in Section \ref{commoncovering}, we prove the main result (Theorem \ref{maintheorem}), showing that the two spaces $\Sigma(\Gamma)$ and $\Omega(\Gamma)$ have a common covering space, and then that $\Sigma(\Gamma)$ is aspherical if and only if $\Omega(\Gamma)$ is aspherical. From this, we have as a consequence the asphericity of $\Omega(\Gamma)$ for $\Gamma$ of spherical type (Corollary \ref{beeraspht}), and for $\Gamma$ of affine type (Corollary \ref{beerasphtAFF}).

\begin{acknowledgements}\noindent
The author thanks her Ph.D. advisor, Prof. Luis Paris, for his valuable insights and guidance throughout their many discussions. She is also grateful to the anonymous referee of AGT for their precise, helpful, and thorough corrections, which have significantly enhanced both the quality and the clarity of exposition of this work. In addition, she wishes to thank the reviewers of her Ph.D. thesis (of which this article forms a part), Professors Thomas Gobet and Eddy Godelle, for their careful reading and insightful suggestions. 
\end{acknowledgements}

\section{Preliminaries}\label{prem}
We summarize here all the definitions, theorems and tools that will be needed in Sections \ref{salberrap}, \ref{BEERspace} and \ref{commoncovering}. Furthermore, we fix the notation for Coxeter and Artin groups, Coxeter polytopes, and virtual Artin groups.
\subsection{Coxeter groups and Artin groups}\label{CoxandArt}

Given $S$ a countable set, we consider $\mathrm{M}:=(m_{s,t})_{s,t\in S}$ a symmetric square matrix indexed by $S$. $\mathrm{M}$ is a \textit{Coxeter matrix} on $S$ if its entries are in $\mathbb{N}\cup\{\infty\}$, $m_{s,s}=1$ for all $s \in S$ and $m_{s,t}\geq 2$ for all $s,t\in S$ with $s\neq t$. This information is typically encoded by a graph $\Gamma$, referred to as the \textit{Coxeter graph}. The set of vertices in $\Gamma$ corresponds to $S$, and there is an edge between two vertices $s$ and $t$ if and only if $m_{s,t}\geq 3$. Additionally, the edge is labeled with $m_{s,t}$ if $m_{s,t}\geq4$. We often write $\mathrm{M}=\mathrm{M}[\Gamma]$.
We say that $\Gamma$ or $\mathrm{M}$ is \textit{finite} if $S$ is a finite set.

\begin{nt}
Take two letters $a,b$, and a positive integer $m$. We denote by $ \Prod_L(a,b;m)=aba\,\cdots $ the alternating word in $a$ and $b$ starting with an $a$ on the left, of length $m$. Similarly, we denote by $\Prod_R(b,a;m)=\cdots \,aba$ the alternating word in $a$ and $b$ starting with an $a$ on the right, of length $m$.
\end{nt}
    \noindent Let $S$ be a countable set and $\Gamma$ be a Coxeter graph on $S$. The \textit{Coxeter group} $\W[\Gamma]$ is the group defined by the presentation with generators and relations
    \begin{equation}\label{prescoxgroup}
       \W[\Gamma]:=\langle S\, |\, s^2=1 \; \forall s\in S, \; (st)^{m_{s,t}}=1 \, \forall s,t\in S, \, s\neq t\rangle, 
    \end{equation}
   \noindent where for $m_{s,t}= \infty$ we just interpret it as no relation between $s$ and $t$.
   The pair $(\W[\Gamma],S)$ is called a \textit{Coxeter system}. The \textit{Artin group} $\A[\Gamma]$ is the group defined by the presentation with generators and relations
    \[\A[\Gamma]:=\langle S\, |\,\Prod_R(s,t;m_{s,t})=\Prod_R(t,s;m_{s,t}) \, \forall s,t\in S, \, s\neq t\rangle,\]
    where for $m_{s,t}= \infty$ we just interpret it as no relation between $s$ and $t$.
    The pair $(\A[\Gamma],S)$ is called an \textit{Artin system}. \\

\noindent The objects $\Gamma, \A[\Gamma], \W[\Gamma]$, and $\mathrm{M}[\Gamma]$ are said to be \textit{of spherical type} when $\W[\Gamma]$ is a finite group. Observe that for $\W[\Gamma]$ to be finite, we necessarily have that $\Gamma$ is a finite graph. When $\Gamma$ is a connected graph, we say that the Coxeter system (or the Coxeter group) is \textit{irreducible}. The irreducible Coxeter systems of spherical type have been completely classified by Coxeter (see, for example, \cite[Chapter 2]{Hump}).\\
\\
\noindent
The majority of the general notions about Coxeter groups that we will recall in this section are usually stated for finite Coxeter graphs. However, it is quite straightforward to generalize them to the infinite case, looking at the finite supports of the elements. Thus, we will assume that $(\W[\Gamma],S)$ is a Coxeter system, with a set of generators $S$ not necessarily finite. The main references for this matter are \cite{Hump}, \cite{Bourbaki}, and \cite{Davis2008}.\\
\\
\noindent There is a quotient map  $\theta: \A[\Gamma]\longrightarrow \W[\Gamma]$ sending $s\longmapsto s$ and defining a short exact sequence $1\longrightarrow \mathrm{CA}[\Gamma]\longrightarrow \A[\Gamma]\longrightarrow \W[\Gamma]\longrightarrow 1$. The kernel $\ker(\theta):=\mathrm{CA}[\Gamma]$ is called in the literature a \textit{colored Artin group} or a \textit{pure Artin group}. In a Coxeter group, there is a natural \textit{length function} $l:\W[\Gamma]\longrightarrow \mathbb{N}$ with respect to the generating set $S$. For any element $w=s_1\,\cdots\, s_r\in \W[\Gamma]$ with $s_1,\ldots, s_r\in S$ and $r$ as small as possible, we write $l(w)=r$ and we call any expression for $w$ of length $r$ a \textit{reduced expression}.
Given a Coxeter system $(\W[\Gamma],S)$, we take a formal real vector space $V$ with a basis $\Pi=\{\alpha_s\,\mid \, s\in S\}$ in one-to-one correspondence with $S$. Thus $V:=\bigoplus_{s\in S}\mathbb{R}\cdot \alpha_s$, and the dimension of $V$ is finite (and equal to $|S|$) if and only if $S$ is finite.
We denote by $\langle\,\cdot,\cdot\rangle$ the symmetric bilinear form  $V\times V \longrightarrow \mathbb{R}$ given by
\[
\langle\alpha_s,\alpha_t\rangle=\begin{cases}
    -\cos{\left(\frac{\pi}{m_{s,t}}\right)} &\mbox{  if $m_{s,t}\neq \infty$}\\
    -1 &\mbox{  if $m_{s,t}=\infty$.}
\end{cases}
\]
Define a group homomorphism $\rho:\W[\Gamma]\longrightarrow GL(V)$ as $s\longmapsto \rho(s)$, where for all $v\in V$ $\rho(s)(v):=v-2\langle\alpha_s,v \rangle \alpha_s$.
\begin{thm}(\cite[Chapter V, Section 4]{Bourbaki})
   The homomorphism $\rho$ is a faithful linear representation.
\end{thm}
\noindent
The map $\rho$ is called the \textit{canonical linear representation}, and the associated action of an element $w\in \W[\Gamma]$ on $v\in V$ is denoted here by $\rho(w)(v)=w (v)$. This action preserves the bilinear form, so $\langle u,v\rangle=\langle w (u), w (v)\rangle$ for all $u,v\in V$ and $w\in \W[\Gamma]$. A result that we will frequently use is the following:
\begin{thm}[\cite{Cox34},\cite{Cox35}]\label{finiteandscalproduct}
    Let $(\W[\Gamma],S)$ be a finite Coxeter system and $\langle\cdot,\cdot\rangle$ be its associated symmetric bilinear form. Then $\langle\cdot,\cdot\rangle$ is a scalar product if and only if $\W[\Gamma]$ is of spherical type.
\end{thm}
\noindent
If the symmetric bilinear form $\langle\cdot,\cdot\rangle$ is positive semidefinite but not positive definite, then we say that $\mathrm{M},\Gamma, \W[\Gamma]$ or $\A[\Gamma]$ are of \textit{affine type}. The irreducible Coxeter systems of affine type have also been classified (see for example \cite[Section 2.7]{Hump}).

\begin{defn}
    Given a Coxeter system $(\W[\Gamma],S)$, the \textit{root system} $\Phi[\Gamma]$ is the set $\Phi[\Gamma]=\{w(\alpha_s)\;|\;w\in \W[\Gamma], s\in S\}$.
\end{defn}
\noindent
Observe that $\Phi[\Gamma]=-\Phi[\Gamma]$, and that the root system is finite if and only if $\Gamma$ is of spherical type.
We can write every root $\beta\in \Phi[\Gamma]$ as a linear combination of the vectors in the basis $\Pi$ (called the \textit{simple roots} or the \textit{base of $\Phi[\Gamma]$}) $\beta=\sum_{s\in S}\lambda_s\alpha_s$. We say that $\beta$ is a \textit{positive root} if $\lambda_s\geq 0$ for all $s\in S$, and we call $\beta$ a \textit{negative root} in the case where the coefficients $\lambda_s\leq 0$ for all $s\in S$. We denote the two disjoint sets of positive and negative roots, respectively, by $\Phi^+[\Gamma]$ and $\Phi^-[\Gamma]$, and by \cite{deodh82} it holds that $\Phi[\Gamma]=\Phi^+[\Gamma]\sqcup \Phi^-[\Gamma]$.
By \cite[Section 5.4]{Hump}, we have that for $w\in \W[\Gamma]$, $s\in S$,
if $l(ws)<l(w)$ then $w(\alpha_s)\in \Phi^-[\Gamma]$, while if $l(ws)>l(w)$ then $w(\alpha_s)\in \Phi^+[\Gamma]$.\\
\\
\noindent
Consider now $X\subset S$. If $X$ is such that $m_{s,t}\neq \infty$ for all $s,t\in X$, then we say that $X$ is \textit{free of infinity}. 

\begin{defn}\label{defparsub}
   Let $\Gamma$ be a Coxeter graph on a vertex set $S$. Take now any $X\subset S$. The subgroup of $\W[\Gamma]$ generated by $X$ is called a \textit{standard parabolic subgroup} and it is denoted by $\W_{X}[\Gamma]$. For any $w\in \W[\Gamma]$, the conjugate $w\W_{X}[\Gamma]w^{-1}$ is called a \textit{parabolic subgroup} of $\W[\Gamma]$. 
\end{defn} 
\noindent When there is no ambiguity, we will omit the Coxeter graph $\Gamma$ and we will just write $\mathrm{W}$ and $\mathrm{W}_X$. Consider now $\Gamma_X$ as the full subgraph of $\Gamma$ spanned by the vertices in $X$. According to \cite[Section 5.5]{Hump}, the natural map $s\longmapsto s$ for all $s\in X$ induces an isomorphism between $\W[\Gamma_X]$ and $\W_{X}[\Gamma]$. Thus, $\W[\Gamma_X]$ can be viewed as a subgroup of $\W[\Gamma]$. Additionally, we have that the length function in $\W[\Gamma_X]$ coincides with the restriction of the length function of $\W[\Gamma]$ to $\W_{X}[\Gamma]$. Now let us consider the cosets and the double cosets with respect to the standard parabolic subgroups of $\W[\Gamma]$.

\begin{defn}
    Let $\Gamma,\mathrm{W},S$ be as before, and $X,Y\subset S$. We say that \textit{$w$ is $(Y,X)$-minimal} if it is of minimal length in the double coset $\W_{Y} w \W_{X}=\{uwv\;|\; u\in \W_{Y},v\in \W_{X}\}$.
\end{defn}
\noindent When $X\subset S$ is a singleton $X=\{s\}$, we often write $\W_{s}$ instead of $\W_{\{s\}}$ and $(Y,s)$-minimal (resp. $(s,Y)$-minimal) instead of $(Y,\{s\})$-minimal (resp. $(\{s\},Y)$-minimal). 
\begin{lem}(\cite[Chapter IV, Section 1, Exercise 3]{Bourbaki} or \cite[Lemmas 4.3.1 and 4.3.3]{Davis2008})\label{w0mindoublecoset}
Let $(\W[\Gamma],S)$ be a Coxeter system, $w\in \W[\Gamma]$ and $X,Y\subset S$. Then there exists a unique $w_0\in \W$ that is $(Y,X)$-minimal in the double coset $\W_{Y}w\W_{X}$. Moreover, there exist $u\in \W_{Y}$ and $v\in \W_{X}$ such that $w=uw_0v$ and $l(w)=l(u)+l(w_0)+l(v)$.
\end{lem}
\begin{cor}
Let $(\W,S)$ be a Coxeter system, $w\in \W[\Gamma]$ and $X\subset S$. Then:
\begin{enumerate}
    \item The element $w$ is $(\emptyset,X)$-minimal if and only if $l(ws)>l(w)$ for all $s\in X$, and  $l(ws)>l(w)$ for all $s\in X$ if and only if $l(wu)=l(w)+l(u)$ for all $u\in \W_{X}$;
    \item The element $w$ is $(X,\emptyset)$-minimal if and only if $l(sw)>l(w)$ for all $s\in X$, and  $l(sw)>l(w)$ for all $s\in X$ if and only if $l(uw)=l(u)+l(w)$ for all $u\in \W_{X}$.
\end{enumerate}
\end{cor}

\subsection{Coxeter polytopes}\label{coxeterpolytopes}
The subject treated here can be found in detail in \cite[Chapter 7.3]{Davis2008}.\\
In this subsection, let $\Gamma$ be a Coxeter graph of spherical type, and let $(\W[\Gamma],S)$ be its Coxeter system, set $\Pi=\{\alpha_s\;|\; s\in S\}$ and $V=\bigoplus_{s\in S}\mathbb{R}\cdot \alpha_s$, and consider the canonical faithful linear representation $\rho:\W[\Gamma]\longrightarrow GL(V)$. We will denote $\rho(w)(v)$ by $w(v)$, for all $w\in \W[\Gamma]$ and $ v\in V$. We will summarize here all the definitions and properties that we need concerning Coxeter polytopes.\\
Since $\W[\Gamma]$ is a finite group, we know that $V$ has finite dimension and that the bilinear symmetric form induced on $V$ is a scalar product.
Let $\Pi^*=\{\alpha_s^*\;|\;s\in S\}$ be the dual basis of $\Pi$ with respect to $\langle\cdot,\cdot\rangle$, namely $\langle\alpha^*_s,\alpha_t\rangle=0$ if $s\neq t$ and $\langle\alpha^*_s,\alpha_t\rangle=1$ if $s=t$. Set $o=\sum_{s\in S}\alpha_s^*$ and note that it is not fixed by any nontrivial element in $\W[\Gamma]$.

\begin{defn}
The \textit{Coxeter polytope} of $(\W[\Gamma],S)$ is the convex hull of the points in $\{w(o)\;|\;w\in \W[\Gamma]\}$, denoted by $C[S]$, and contained in $V$.     
\end{defn}
\noindent Given a finite set of points $B\subset V$ in a vector space $V$, we denote by $conv\{b\,|\,b\in B\}$ their convex hull in $V$.

\begin{rmk}
    Note that the polytope $C[S]$ is only defined when $\W[\Gamma]$ is of spherical type; otherwise, the vertices $w(o)$ would not form a finite set of points. Observe that the vertices span $V$, viewed as an affine space, so the interior of $C[S]$ is nonempty. 
\end{rmk}

\noindent The Coxeter group $\W[\Gamma]$ acts on its Coxeter polytope $C[S]$ in the obvious way.

\begin{ex}
    Let $\Gamma$ be the Coxeter graph of type $A_{n-1}$, with $S$ given by the simple transpositions $S=\{(12), (23),\ldots,(n-1 \; n)\}$. The associated Coxeter group $\W[A_{n-1}]$ is the symmetric group $\mathfrak{S}_n$ and the Coxeter polytope $C[S]$ is the well-known \textbf{permutohedron} of order $n$.
\end{ex}
\noindent
A polytope can equivalently be defined as the nonempty and bounded intersection of a finite number of half-spaces in an affine space. We say that an (affine) hyperplane $H$ \textit{supports} a polytope $C$ if the intersection $C\cap H$ is nonempty and if the polytope lies in one of the two closed half-spaces bounded by $H$. A \textit{face} of a polytope is the intersection of the polytope itself with a supporting hyperplane. The set of faces of $C[S]$ with the inclusion relation forms a poset $\mathcal{F}(C[S])$. \\
\\
\noindent
Given $X\subset S$ and $u\in \W[\Gamma]$, denote now by $F(u,X,S)$ the subset of $C[S]$ given by the convex hull of $\left\{uw(o)\;|\;w\in \W_{X} \right\}$. Of course, $u(F(\id,X,S))=F(u,X,S)$ for every $u$ in $\W[\Gamma]$.

\begin{lem}(\cite[Lemma 7.3.3]{Davis2008})
\label{facescoxpol}
The set of faces of $C\left[S\right]$ is exactly \\        $\left\{F(u,X,S)\;|\; u\in \W, X \subset S, X\neq S \right\}$, and the map that associates $F(u,X,S)$ to $u\W_{X}$ is a poset isomorphism between $\mathcal{F}(C[S])$ and the poset of cosets with respect to the standard parabolic subgroups of $\W$, ordered by inclusion.
\end{lem}
\noindent
Thanks to this result, we know that all the faces of $C[S]$ are of the form $F(u,X,S)$, and that $F(u,X,S)\subset F(u',X',S)$ if and only if $u\W_{X}\subset u'\W_{X'}$.\\
\\
\noindent Given $X \subset S$, set 
$\Pi_X=\left\{\alpha_s\,|\, s\in X\right\}$
and let $V_X$ be the vector subspace of $V$ whose basis is $\Pi_X$. Then, we call $\Pi_X^*=\left\{\alpha_{s,X}^*\,|\, s \in X\right\}$ the dual basis of $\Pi_X$ in $V_X$ (which does not coincide with the set $\left\{\alpha_s^*\,|\, s \in X\right\}$). Since the form $\langle\cdot,\cdot\rangle$ is positive-definite, its restriction to $V_X$ is non-degenerate. This allows us to identify $\Pi_X^*$ with a subset of $V_X\subset V$.\\
\\
 \noindent Now set $o_{X}=\sum_{s\in X}\alpha_{s,X}^*$. Denote by $C\left[X\right]$ the Coxeter polytope associated to the Coxeter system $(\W_{X},X)$ with vector space $V_X$, namely the convex hull of $\left\{w( o_{X}) \;|\; w \in \W_{X}\right\}\subset V_X$.
\begin{nt}
     If $X$ is a singleton $\{s\}$, we conventionally write $C[s]$ and $o_s$ instead of $C[\{s\}]$ and $o_{\{s\}}$. In the same way, in the faces, we write $F(u,s,S)$ instead of $F(u,\{s\},S)$. Moreover, consider the face $F(u,X,S)=conv\left\{uw (o)\;|\; w \in \W_{X}\right\}$ of $C[S]$, with $X\subset S$ and $u\in \W$. When we draw a Coxeter polytope (like in Figure \ref{fig:faceisom}), we label the vertices of the form $F(u,\emptyset,S)$ simply by the element $u$.
\end{nt}
\noindent
We will show that there is a natural isometry between $F(\id,X,S)$ and $C[X]$.

\begin{lem}\label{transl}
The translation of $V$ by the vector $o-o_{X}$ commutes with all $w\in \W_{X}$. In particular, it sends $w\left(o_{X}\right)$ to $w\left(o\right)$ for every $w \in \W_{X}$ and it sends isometrically $C[X]$ onto $F(\id,X,S)$.
\end{lem}
\begin{proof}
    First, observe that $\langle \alpha_s, o-o_X\rangle=1-1=0$ for all $s\in X$. Now take any $v$ in $V$ and let $\tau$ denote the translation by vector $o-o_X$. We want to show that $\tau\circ s(v)=s\circ \tau (v)$ for any $s\in X$. Computing the latter we obtain $s(v+(o-o_X))=v+(o-o_X)-2\langle \alpha_s, v+(o-o_X)\rangle \alpha_s=v-2\langle \alpha_s, v\rangle \alpha_s+(o-o_X)$, that is, $\tau\circ s(v)$. Hence $\tau$ commutes with every generator of $\W_{X}$, and then with any $w\in \W_{X}$. Therefore, $\tau(w(o_X))=w(o_X+(o-o_X))=w(o)$, and $\tau$ sends isometrically $C[X]$ onto $F(\id,X,S)$.
\end{proof}

\begin{ex}
\noindent An example of a Coxeter polytope $C[S]$ where $S=\{s,t\}$ and $m_{s,t}=4$ is illustrated in Figure \ref{fig:faceisom}. The corresponding Coxeter group $\W$ is the dihedral group of order 8. In Figure \ref{fig:rootsystem} we draw the root system of $\W$, with base $\Pi=\{\alpha_s,\alpha_t\}$.  We can see in Figure \ref{fig:faceisom} the isometry between $C[t]=C[\{t\}]$ and the face $F(\id,t,S)$ of $C[S]$.

\begin{figure}[h!]
    \begin{minipage}{7cm}
\includegraphics[width=6cm]{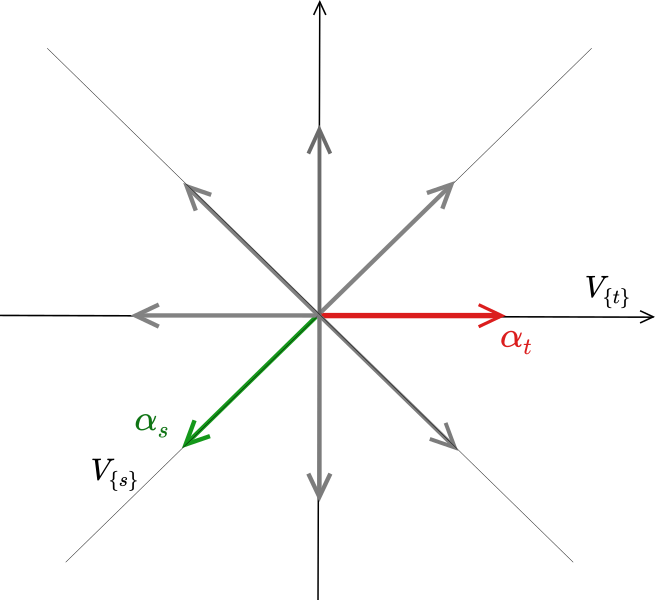}
    \caption{The root system of the dihedral group of order 8.}
    \label{fig:rootsystem}
    \end{minipage} \hfill
    \begin{minipage}{7cm}
\includegraphics[width=6.7cm]{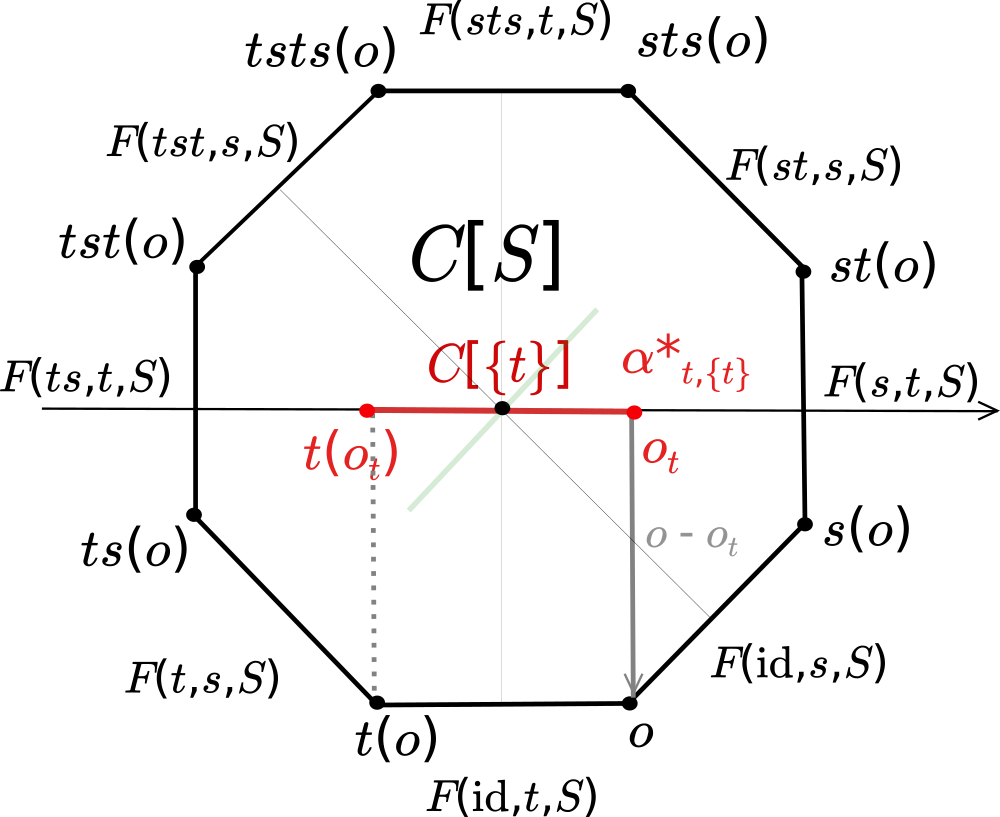}
    \caption{The isometry between the face $F(\id,t,S)$ of $C[S]$ and $C[t]$.}
    \label{fig:faceisom}
    \end{minipage} 
\end{figure}
\end{ex}

\noindent
When working in a vector space $V$ viewed as an affine space, we will denote the translation of $V$ by vector $v\in V$ by $\tau(v)$. Thanks to Lemma \ref{transl}, we get that any face $F(u,X,S)=u(F(\id,X,S))$ of $C[S]$ is naturally isometric to the Coxeter polytope $C[X]$ via the translation $\tau(o-o_X)$, composed with the isometry $u\in \W[\Gamma]\subset GL(V)$.
\noindent
Since the faces of $C[S]$ are in bijection with the cosets $u\W_{X}$ with $u\in \W[\Gamma]$ and $X\subset S$, from now on we will choose $u$ to be the $(\emptyset,X)$-minimal representative.

\begin{nt}\label{orientedges}\textbf{Orientation on the edges of the Coxeter polytope.}
In this work we will need the edges of a Coxeter polytope to be oriented. Generally, for any  pair of points $P,Q$ in the vector space $V$, the segment joining them will be denoted by $[P,Q]=\overline{PQ}$. If we assign an orientation to it, we define two functions $\mathfrak{s},\mathfrak{t}:\overline{PQ}\longrightarrow \{P,Q\}$, called the \textit{source} and the \textit{target} of the segment, respectively. For example, if $\mathfrak{s}(\overline{PQ})=Q$ and $\mathfrak{t}(\overline{PQ})=P$, the segment is oriented from the vertex $Q$ to the vertex $P$. If $f$ is an affine isometry of $V$, the orientation of $f(\overline{PQ})=\overline{f(P)f(Q)}$ is such that $\mathfrak{s}(f(\overline{PQ})):=f(\mathfrak{s}(\overline{PQ}))$ and $\mathfrak{t}(f(\overline{PQ})):=f(\mathfrak{t}(\overline{PQ}))$.\\
\\
\noindent Take now $s\in S$, and consider $\W_{\{s\}}=\{\id,s\}$. The fundamental point to build the Coxeter polytope $C[s]$ is $o_{s}=\alpha_{s,\{s\}}^*=\alpha_s$, so we get that $C[s]=conv\{\alpha_s,s(\alpha_s)=-\alpha_s\}$ is the segment $[-\alpha_s,\alpha_s]$. From now on we will choose an orientation from the point $\alpha_s$ to the point $-\alpha_s$, i.e. $\mathfrak{s}(C[s])=\alpha_s$ and $\mathfrak{t}(C[s])=-\alpha_s$. If we consider the face $F(\id,s,S)=[o,s(o)]$ of $C[S]$, we know by Lemma \ref{transl} that it is the translation of $[\alpha_s,-\alpha_s]$ under $\tau(o-\alpha_s)$, so, for what we have just said, the orientation will be from $\tau(o-\alpha_s)(\mathfrak{s}(C[s]))=o$ to $\tau(o-\alpha_s)(-\alpha_s)=s(o)$. Similarly, in the face $F(u,s,S)=[u(o),us(o)]$ with $u$ being $(\emptyset,\{s\})$-minimal, the orientation is from $u(o)=u(\mathfrak{s}(F(\id,s,S)))$ to $us(o)=u(\mathfrak{t}(F(\id,s,S)))$.
\end{nt}

\subsection{Virtual Artin groups}\label{virtualartingroups}
All the definitions and notation in this subsection are taken from \cite{BellParThiel}, in which virtual Artin groups were introduced.
\begin{defn}
Given $\Gamma$, a Coxeter graph on the countable set $S$, and $\mathrm{M}=(m_{s,t})_{s,t \in S}$ its associated Coxeter matrix, consider the following two sets:
\begin{align*}
&\mathcal{S}=\{\sigma_s\,|\,s\in S\}  &\mathcal{T}=\{\tau_s\,|\,s\in S\},    
\end{align*}
both in one-to-one correspondence with $S$. 
We define the \textit{virtual Artin group} associated with $\Gamma$ to be the abstract group $\VA[\Gamma]$, defined by the presentation with generating set $\mathcal{S} \cup \mathcal{T}$ and the following relations:
\begin{enumerate}
\item[(V1)]\label{v1} $\Prod_R(\sigma_s,\sigma_t;m_{s,t})=\Prod_R(\sigma_t,\sigma_s;m_{s,t})$ for all $s,t \in S$ such that $s\neq t$ and $m_{s,t}\neq \infty$;
\item[(V2)]\label{v2} $\Prod_R(\tau_s,\tau_t;m_{s,t})=\Prod_R(\tau_t,\tau_s;m_{s,t})$ for all $s,t \in S$ such that $s\neq t$ and $m_{s,t}\neq \infty$; and $\tau_s^2=1$ for all $s \in S$;
\item [(V3)]\label{v3} $\Prod_R(\tau_s,\tau_t;m_{s,t}-1) \sigma_s=\sigma_r \Prod_R(\tau_s,\tau_t;m_{s,t}-1)$ $\forall s,t \in S$ such that $s\neq t$ and $m_{s,t}\neq \infty$, where $r=s$ if $m_{s,t}$ is even and $r=t$ if $m_{s,t}$ is odd.
\end{enumerate}
\end{defn}

\begin{ex}
If $\Gamma$ is the Coxeter graph $A_{n-1}$, then $\W[A_{n-1}]$ is isomorphic to the symmetric group $\Sn$ on $n$ elements, $\A[A_{n-1}]$ is the braid group on $n$ strands $\mathrm{B}_n$, and $\VA[A_{n-1}]$ is the virtual braid group on $n$ strands, denoted by $\VB_n$.
\end{ex}
\noindent
We can define two natural group homomorphisms:
\begin{align}
\iota_{\A}: \A[\Gamma] &\longrightarrow \VA[\Gamma]  &\iota_{\W}: \W[\Gamma] &\longrightarrow \VA[\Gamma] \nonumber\\
s & \longmapsto \sigma_s,  &s & \longmapsto \tau_s.
\nonumber\end{align}
It is easy to check that these maps respect the relations in $\A[\Gamma]$ and $\W[\Gamma]$. It will follow from Theorem \ref{preskva} and Remark \ref{waction} that $\iota_{\W}$ is injective. It is shown in \cite[Corollary 2.3]{BellParThiel} that $\iota_{\A}$ is also injective. Consider as well the two maps
\begin{align}
\pi_P: \VA[\Gamma] &\longrightarrow \W[\Gamma]  &\pi_K: \VA[\Gamma] &\longrightarrow \W[\Gamma] \nonumber\\
\sigma_s & \longmapsto s,  &\sigma_s & \longmapsto 1,\nonumber\\
\tau_s & \longmapsto s,  &\tau_s & \longmapsto s.
\nonumber\end{align}
Observe that both $\pi_K$ and $\pi_P$ are group homomorphisms. In this work we will devote our attention to the two kernels of these maps.

\begin{defn}
The \textit{pure virtual Artin group} $\PVA[\Gamma]$ is the kernel of the homomorphism $\pi_P$. The kernel of $\pi_K$ is instead denoted by $\KVA[\Gamma]$ (sometimes called \textit{kure virtual Artin group}).
\end{defn}

\begin{rmk}\label{waction}
 Observe that $\pi_P \circ \iota_{\W}=id_{\W[\Gamma]}$, thus we get that $\pi_P$ is surjective, $\iota_{\W}$ is injective (so we can interpret $\W[\Gamma]$ as a subgroup of $\VA[\Gamma]$), and by the short exact split sequence
    \[
\begin{tikzpicture}[node distance=2cm, auto]
  \node (uno) {$1$};
  \node (pvb) [right of=uno] {$\PVA[\Gamma]$};
  \node (vb) [right of=pvb] {$\VA[\Gamma]$};
	\node (sn) [right of=vb] {$\W[\Gamma]$};
	\node (unodue) [right of=sn] {$1$};
 \draw[->] (uno) to node {}  (pvb);
 \draw[->] (pvb) to node {$i$} (vb);
\draw[transform canvas={yshift=0.5ex},->] (vb) - -(sn) node[above,midway] {$\pi_P$};
\draw[transform canvas={yshift=-0.5ex},->](sn) -- (vb) node[below,midway] {$\iota_{\W}$}; 
\draw[->] (sn) to node {} (unodue);
\end{tikzpicture} \]
we deduce that the virtual Artin group has a semidirect product structure: $\VA[\Gamma] \cong \PVA[\Gamma] \rtimes \W[\Gamma]$.
This determines a left action of $\W[\Gamma]$ on $\PVA[\Gamma]$ defined by $w \cdot g := \iota_{\W}(w) \,g\, \iota_{\W}(w)^{-1}$ for all $w \in \W[\Gamma]$ and $g \in \PVA[\Gamma]$.\\
\\
\noindent
Similarly, observe that $\pi_K \circ \iota_{\W}=id_{\W[\Gamma]}$. Using the same reasoning, $\pi_K$ is surjective, and the virtual Artin group has a semidirect product structure: $\VA[\Gamma] \cong \KVA[\Gamma] \rtimes \W[\Gamma]$.
This determines a left action of $\W[\Gamma]$ on $\KVA[\Gamma]$ defined by $w \cdot g := \iota_{\W}(w)\, g
\,\iota_{\W}(w)^{-1}$ for all $w \in \W[\Gamma]$ and $g \in \KVA[\Gamma]$.
\end{rmk}
\noindent
Let $\beta$ be a root in $ \Phi[\Gamma]$, and choose $w\in \W[\Gamma]$, $s\in S$ such that $w(\alpha_s)=\beta$. Set $\delta_{\beta}=w\cdot \sigma_s=\iota_{\W}(w)\,\sigma_s\, \iota_{\W}(w)^{-1}\in \KVA[\Gamma]$ and $\zeta_{\beta}=w\cdot(\tau_s\sigma_s)=\iota_{\W}(w)\,\tau_s\sigma_s \, \iota_{\W}(w)^{-1}\in \PVA[\Gamma]$.  By \cite[Lemma 2.2]{BellParThiel}, we know that $\delta_{\beta}$ and $\zeta_{\beta}$ do not depend on the choice of $w$ and $s$.\\
\\
\noindent
Given $\Gamma$ a Coxeter graph, we now recall how to define the Coxeter graph $\widehat{\Gamma}$. The latter has the vertices in bijection with the roots $\Phi[\Gamma]$, and for all $\beta,\gamma\in \Phi[\Gamma]$ the labels $\widehat{m}_{\beta,\gamma}$ are defined as follows:
\begin{itemize}
\item For each $\beta\in \Phi[\Gamma]$, set $\widehat{m}_{\beta,\beta}=1$;
    \item If $\beta\neq \gamma$ and there exist $w\in \W[\Gamma]$, $s,t\in S$ such that $\beta=w(\alpha_s)$, $\gamma=w(\alpha_t)$ and $m_{s,t}\neq \infty$, then set $\widehat{m}_{\beta,\gamma}=m_{s,t}$;
    \item Otherwise, set $\widehat{m}_{\beta,\gamma}=\infty$.
\end{itemize}
\noindent
Again, the definition of $\widehat{m}_{\beta,\gamma}$ does not depend on the choice of $w$ and $s,t$ (see Section 2 in \cite{BellParThiel}).\\
\\
\noindent Let $\cX\subset \Phi[\Gamma]$ be a subset of roots. We say that $\cX$ is \textit{free of infinity} if for all $\beta,\gamma\in \cX$, the label $\widehat{m}_{\beta,\gamma}\neq \infty$.\\
\\
\noindent The next two results are important for the following discussion, as they give presentations for the kernels $\PVA[\Gamma]$ and $\KVA[\Gamma]$. Specifically, $\KVA[\Gamma]$ turns out to be an Artin group with respect to the just defined Coxeter graph $\widehat{\Gamma}$, while $\PVA[\Gamma]$ is generated by the elements $\zeta_{\beta}$.
Denote the standard generating set of $\A[\widehat{\Gamma}]$ by $\{\widehat{\delta}_{\beta}\;|\; \beta\in \Phi[\Gamma]\}$.

\begin{thm}(\cite[Theorem 2.3]{BellParThiel})\label{preskva}
Let $\Gamma$ be a Coxeter graph. Then the map $\{\widehat{\delta}_{\beta}\;|\; \beta\in \Phi[\Gamma]\}\longrightarrow \{\delta_{\beta}\;|\; \beta\in \Phi[\Gamma]\}$, $\widehat{\delta}_{\beta}\longmapsto \delta_{\beta}$, induces an isomorphism $\varphi:\A[\widehat{\Gamma}]\longrightarrow \KVA[\Gamma]$.
\end{thm}
\noindent
From now on we will often identify $\KVA[\Gamma]$ with $\A[\widehat{\Gamma}]$.\\ \noindent
Now, consider an abstract set $\{\widehat{\zeta}_{\beta}\;|\;\beta \in \Phi[\Gamma]\}$ in one-to-one correspondence with $\Phi[\Gamma]$. As we will extensively discuss in Section \ref{APrefsubgrp}, 
for every $\beta=w(\alpha_s)\in \Phi[\Gamma]$ with $w\in \W[\Gamma]$ and $s\in S$, we can consider the element $r_{\beta}:=wsw^{-1}$, which behaves like a reflection acting on the vector space $V$. Namely, if we consider the linear map $\rho(r_{\beta})$ acting on $V$, it sends the nontrivial vector $\beta$ to its opposite $-\beta$, and it fixes pointwise the hyperplane orthogonal to $\beta$ with respect to the bilinear form $\langle\cdot,\cdot\rangle$.\\

\noindent Let $\beta,\gamma\in \Phi[\Gamma]$ be such that $\widehat{m}_{\beta,\gamma}\neq \infty$, and let $m=\widehat{m}_{\beta,\gamma}$. We define the roots $\beta_1,\ldots,\beta_m, \gamma_1,\ldots,\gamma_m \in \Phi[\Gamma]$ by $\beta_1=\beta$, $\gamma_1=\gamma$ and, for $k\geq 2$,
\begin{equation}\label{betakappa}
    \beta_k=\begin{cases}
    \Prod_R(r_{\gamma},r_{\beta};k-1)(\gamma) \mbox{  if $k$ is even,}\\
    \Prod_R(r_{\beta},r_{\gamma};k-1)(\beta) \mbox{  if $k$ is odd.}
\end{cases}\\
\gamma_k=\begin{cases}
    \Prod_R(r_{\beta},r_{\gamma};k-1)(\beta) \mbox{  if $k$ is even,}\\
    \Prod_R(r_{\gamma},r_{\beta};k-1)(\gamma) \mbox{  if $k$ is odd.}
\end{cases}\\
\end{equation}
\noindent
We set $Z(\gamma,\beta,m)=\widehat{\zeta}_{\beta_m}\cdots\,\widehat{\zeta}_{\beta_2}\,\widehat{\zeta}_{\beta_1}$, which is a word over the set $\{\widehat{\zeta}_{\beta}\;|\;\beta \in \Phi[\Gamma]\}$. By \cite[Lemma 2.5]{BellParThiel}, the word $Z(\beta,\gamma,m)$ is the reverse word of $Z(\gamma,\beta,m)$, meaning, $\widehat{\zeta}_{\beta_1}\,\widehat{\zeta}_{\beta_2}\cdots\,\widehat{\zeta}_{\beta_m}$. We denote by $\widehat{\PVA}[\Gamma]$ the group defined by the presentation with generating set $\{\widehat{\zeta}_{\beta}\;|\;\beta \in \Phi[\Gamma]\}$ and relations \[
Z(\gamma,\beta,\widehat{m}_{\beta, \gamma})=Z(\beta,\gamma,\widehat{m}_{\beta, \gamma}),
\]
 for $\beta, \gamma \in \Phi [\Gamma]$ such that
$\widehat{m}_{\beta, \gamma} \neq \infty$.
\begin{thm}(\cite[Theorem 2.6]{BellParThiel})\label{prespva}
   Let $\Gamma$ be a Coxeter graph. Then the map $\{\widehat{\zeta}_{\beta}\;|\;\beta \in \Phi[\Gamma]\}\longrightarrow \{{\zeta}_{\beta}\;|\;\beta \in \Phi[\Gamma]\}$, $\widehat{\zeta}_{\beta}\longmapsto {\zeta}_{\beta}$, induces an isomorphism $\varphi:\widehat{\PVA}[\Gamma]\longrightarrow \PVA[\Gamma]$.
\end{thm}
\noindent
The purpose of this work is to construct a space whose fundamental group is $\PVA[\Gamma]$, and then investigate its asphericity by comparing it with a space that serves as the classifying space of $\KVA[\Gamma]$, provided $\Gamma$ satisfies certain conditions. The entire study will be made possible through the presentations of the two groups given in Theorems \ref{preskva} and \ref{prespva}.

\section{The Salvetti complex and the AP reflection subgroups}\label{salberrap}
\subsection{The Salvetti complex for \texorpdfstring{$\Gamma$}{TEXT} }\label{salgam}
In Subsection \ref{coxeterpolytopes} we have seen how to build the Coxeter polytope for a Coxeter system $(\W[\Gamma],S)$, with $\Gamma$ of spherical type. If the Coxeter group is not finite, we can still define a Coxeter polytope for each standard parabolic subgroup of spherical type. In this subsection, we will use these Coxeter polytopes as the fundamental cells to define the Salvetti complex.
\begin{defn}
   Let $(\W[\Gamma],S)$ be the Coxeter system associated to the Coxeter graph $\Gamma$. Write $\Sf:=\{X\subset S\,|\, \W_{X} \mbox{ is of spherical type}\}$. The \textit{spherical dimension of $\Gamma$}, denoted by $n_{sph}(\Gamma)$, is the integer $n_{sph}(\Gamma):=max\{\,|X| \, | \,X\in \Sf\}$ if $\{|X| \mid X \in \mathcal S^f\}$ is bounded, and $n_{sph}(\Gamma) = \infty$ if $\{|X| \mid X \in \mathcal S^f\}$ is unbounded.
If $1\leq k\leq n_{sph}(\Gamma)$,
denote also by $\Sf_k=\{X\in \Sf \,|\, |X|=k\}$.
\end{defn}
\noindent 
If $\Gamma$ is of spherical type, then the spherical dimension of $\Gamma$ coincides with the rank $|S|$.
\\
\noindent 
For any $X$ in $\Sf$, we can then consider the Coxeter polytope $C[X]=conv\{w(o_X)\,|\, w\in \W_{X}\}\subset V_X$. Applying Lemma \ref{facescoxpol} to $C[X]$, we obtain that the faces of $C[X]$ are all of the form $F(u,Y,X)=conv\{uw(o_X)\,|\, w\in \W_{Y}\}$ with $Y\subset X\in \Sf$ and $u\in \W_{X}$ $(\emptyset,Y)$-minimal. In the same way, we can consider the Coxeter polytope $C[Y]=conv\{w(o_Y)\,|\, w\in \W_{Y}\}\subset V_Y$. Thanks to Lemma \ref{transl}, the composition $\tau(o_Y-o_X)\circ u^{-1}$ sends isometrically the face $F(u,Y,X)$ of $C[X]$ onto the Coxeter polytope $C[Y]$.
\begin{nt}
Typically, an \textit{$n$-cell} in a CW-complex is a space that is homeomorphic to a closed Euclidean ball in $\mathbb{R}^n$. In this context, by a slight abuse of terminology, we will also refer to the images of these balls under attaching maps as cells. These maps identify the boundary of the ball with a subset of the $n-1$-skeleton of the CW-complex.
\end{nt}

\noindent We are now ready to describe the cellular structure of the Salvetti complex.

\begin{defn}\label{defsalcel}
The \textit{Salvetti complex of $\Gamma$}, denoted by $\salbar$, is the CW-complex whose skeleta are inductively defined as follows.\\
\\
\textbf{Description of the 0-skeleton.} The 0-skeleton $\overline{Sal}^{0}(\Gamma)$ consists of just one point, which we will denote by $x^0$.\\
    \\
\textbf{Description of the 1-skeleton.} For each generator $s \in S$, take a copy $\mathbb{A}(s)$ of the Coxeter polytope $C[s]=[-\alpha_s,\alpha_s]$, which we identify with $[-\alpha_s,\alpha_s]$ via an isometry $f_s: [-\alpha_s,\alpha_s]\longrightarrow \mathbb{A}(s)$. As established in Notation \ref{orientedges}, we set the orientation of $C[s]$ from the point $\alpha_s$ to the point $-\alpha_s$, and $\mathbb{A}(s)$ inherits the orientation through $f_s$. 
    We then define a map $\varphi^1_{s}:\partial\mathbb{A}(s)\longrightarrow \overline{Sal}^{0}(\Gamma)$ by setting $\varphi^1_{s}(f_s(\alpha_s))=\varphi^1_{s}(f_s(-\alpha_s))=x^0$. Then we set 
   \[
    \overline{Sal}^{1}(\Gamma):=\bigslant{\{x^0\}\bigsqcup \left( \bigsqcup_{s\in S}\mathbb{A}(s)\right)}{\sim},
    \]
    where, for all $s\in S$ and $x\in \partial \mathbb{A}(s)$, we have $x \sim \varphi^1_s(x)$. The embedding of $\mathbb{A}(s)$ into $\{x^0\}\bigsqcup$ $\left(\bigsqcup_{s'\in S}\mathbb{A}(s')\right)$ induces a characteristic map $\phi^1_s: \mathbb{A}(s)\longrightarrow \overline{Sal}^1(\Gamma)$ whose image is a 1-cell denoted by $\Delta^1(s)$. Notice that $\phi^1_s$ extends the attaching map $\varphi^1_s$ and it is a homeomorphism from the interior of the segment $\mathbb{A}(s)$ onto $\Delta^1(s)\,\backslash\, \{x^0\}$. Thus, as a set, $\overline{Sal}^1(\Gamma)=\{x^0\}\bigcup\left(\bigcup_{s\in S}\Delta^1(s)\right)$ where each $\Delta^1(s)$ is a segment with both ends attached to the only vertex $x^0$. The arcs $\Delta^1(s)$ are naturally oriented with the orientation induced by that of $\mathbb{A}(s)$.\\
    \\\textbf{Description of the $k$-skeleton.} Suppose that we have defined the $(k-1)$-skeleton $\overline{Sal}^{k-1}(\Gamma)$, for $2\leq k \leq n_{sph}(\Gamma)$. Namely, for all $Y\in \Sf$ with $|Y|=h\leq k-1$, we have a copy $\mathbb{D}(Y)$ of the Coxeter polytope $C[Y]$ which we identify with $C[Y]$ via an isometry $f_{Y}:C[Y]\longrightarrow \mathbb{D}(Y)$, and we have an $h$-cell $\Delta^h(Y)\subset \overline{Sal}^{k-1}(\Gamma)$, which is the image of a continuous characteristic map $\phi^h_Y: \mathbb{D}(Y)\longrightarrow \Delta^h(Y)$ establishing a homeomorphism between the interior of $\mathbb{D}(Y)$ and $\Delta^h(Y)\,\backslash\, \partial\Delta^h(Y)$. If $|Y|=1$, then there exists $s\in S$ such that $Y=\{s\}$, and we assume $\mathbb{D}(Y)=\mathbb{A}(s)$. If $Y=\emptyset$, then we adopt the following conventions: 
    \[
    V_{\emptyset}=\{0\},\qquad \W_{\emptyset}=\{\id\},\qquad o_{\emptyset}=0,\qquad C[\emptyset]=\{0\}.
    \]
    \noindent We also set $\mathbb{D}(\emptyset)=\{0\}$. Take $f_{\emptyset}:C[\emptyset]\longrightarrow \mathbb{D}(\emptyset)$ to be the identity map, and define $\varphi_{\emptyset}^0:\mathbb{D}(\emptyset)\longrightarrow\overline{Sal}^0(\Gamma)$ to be the map sending $0$ to the unique vertex $x^0$. Thus $\Delta^0(\emptyset)=\{x^0\}$.
    \\
    \\
    \noindent Now, take for each $X\in \Sf_k$ a copy $\mathbb{D}(X)$ of the Coxeter polytope $C[X]$, identified with $C[X]$ via an isometry $f_{X}:C[X]\longrightarrow \mathbb{D}(X)$. Let $Y\subset X$ with $|Y|=h\leq k-1$ and $u\in \W_{X}$ be an $(\emptyset,Y)$-minimal element.
    Define now $\xi^k_{u,Y,X}:F(u,Y,X)\longrightarrow \overline{Sal}^{k-1}(\Gamma)$ to be the following composition of maps:

    \begin{equation}
        F(u,Y,X)\xrightarrow[]{u^{-1}}F(\id,Y,X)\xrightarrow[]{\tau(o_Y-o_X)}C[Y]\xrightarrow[]{f_{Y}}\mathbb{D}(Y)\xrightarrow[]{\phi^h_Y}\Delta^{h}(Y)\subset \overline{Sal}^{k-1}(\Gamma).
    \end{equation}
In other words,
\[\xi^k_{u,Y,X}=\phi^h_Y\circ f_Y \circ \tau(o_Y-o_X)\circ u^{-1}.\]
Recall that if $F(u,Y,X)$ is a face of $C[X]$ with $Y$ and $u$ as described before, the faces of $F(u,Y,X)$ are of the form $F(uv,Z,X)=conv\{uvw(o_X)\,|\,w\in W_Z\}$ with $Z\subset Y$ and $v\in \W_{Y}$, $(\emptyset,Z)$-minimal. We now state a result to be proven later.
   \begin{lem}\label{varphiiscontinuous}
    Let $X,Y,Z \in \Sf$ be such that $Z\subset Y\subset X$, and $|X|=k,|Y|=h,|Z|=l$. Take $u\in \W_{X}$, $u$ $(\emptyset,Y)$-minimal and $v\in \W_{Y}$, $v$ $(\emptyset,Z)$-minimal. We have that $F(uv,Z,X)\subset F(u,Y,X)\subset C[X]$. Then, the restriction of $\xi^k_{u,Y,X}$ to $F(uv,Z,X)$ coincides with $\xi^k_{uv,Z,X}$. 
\end{lem}
\noindent
Therefore, we can define a continuous map $\xi^k_X:\partial C[X]\longrightarrow \overline{Sal}^{k-1}(\Gamma)$ as follows. Let $x\in \partial C[X]$. We choose $Y\subset X$ and $u\in \W_X$ an $(\emptyset, Y)$-minimal element such that $x\in F(u,Y,X)$ and we set $\xi^k_X(x)=\xi^k_{u,Y,X}(x)$. Thanks to Lemma \ref{varphiiscontinuous}, this definition does not depend on the choice of $F(u,Y,X)$, and the map $\xi^k_X$ is continuous. Then we define $\varphi^k_X:\partial \mathbb{D}(X)\longrightarrow \overline{Sal}^{k-1}(\Gamma)$ by setting
\[
\varphi^k_X=\xi^k_{X}\circ (f^{-1}_{X})|_{\partial \mathbb{D}(X)}.
\]

\noindent
The $k$-skeleton of the cell complex will be defined as  
    \[
    \overline{Sal}^{k}(\Gamma):=\bigslant{\left(\overline{Sal}^{k-1}(\Gamma)\bigsqcup \left( \bigsqcup_{X\in \Sf_k}\mathbb{D}(X)\right)\right)}{\sim},
    \]
    where, for all $X\in \Sf_k$ and all $x\in \partial \mathbb{D}(X)$, we have $x\sim \varphi^k_{X}(x)$.\\ The embedding of $\mathbb{D}(X)$ into $\overline{Sal}^{k-1}(\Gamma)\bigsqcup \left( \bigsqcup_{X'\in \Sf_k}\mathbb{D}(X')\right)$ induces a characteristic map $\phi^k_X:\mathbb{D}(X)\longrightarrow \overline{Sal}^{k}(\Gamma)$ whose image is denoted by $\Delta^k(X)$.\\
    \\
\textbf{Description of the complex.} We set $\salbar=\bigcup_{k=0}^{\infty}\overline{Sal}^{k}(\Gamma)$ and we endow it with the weak topology.

\end{defn}
\noindent
Observe that the highest dimension of the cells in the Salvetti complex is the spherical dimension of $\Gamma$.
We now prove the result that was stated in the definition of the Salvetti complex.
\begin{proof}[Proof of Lemma \ref{varphiiscontinuous}]
We denote by $\tau(a)$, as in Subsection \ref{coxeterpolytopes}, the translation by vector $a$.
The map $\xi^k_{uv,Z,X}$ in details is 
\begin{align*}
    F(uv,Z,X)\xrightarrow[]{v^{-1}u^{-1}}F(\id,Z,X) \xrightarrow[]{\tau(o_{Z}-o_{X})}C[Z]\xrightarrow[]{f_Z}\mathbb{D}(Z)\xrightarrow[]{\phi^l_Z} \Delta^l(Z).
\end{align*}
Now, we study the restriction of $\xi^k_{u,Y,X}$ to $F(uv,Z,X)$, which is described as:
\begin{align*}
    F(uv,Z,X)\xrightarrow[]{u^{-1}}F(v,Z,X) \xrightarrow[]{\tau(o_{Y}-o_{X})}F(v,Z,Y)\subset  C[Y]\xrightarrow[]{f_Y}\mathbb{D}(Y)\xrightarrow[]{\phi^h_Y} \overline{Sal}^h(\Gamma).
\end{align*}
Since $F(v,Z,Y)\subset \partial C[Y]$, we have that $f_Y(F(v,Z,Y))\subset \partial \mathbb{D}(Y)$, so by induction, the restriction of $\phi^h_Y$ to the boundary of the cell is $\varphi^h_{Y}=\xi^h_Y\circ (f^{-1}_Y)|_{\partial{\mathbb{D}(Y)}}$. Thus, $\phi_Y^h \circ f_Y|_{F(v,Z,Y)} = \xi_{v,Z,Y}^h|_{F(v,Z,Y)}$, which is the composition of maps
\begin{align*}
F(v,Z,Y)\xrightarrow[]{v^{-1}}F(\id,Z,Y)\xrightarrow[]{\tau(o_{Z}-o_{Y})}C[Z]\xrightarrow[]{f_Z}\mathbb{D}(Z)\xrightarrow[]{\phi^l_Z}\Delta^l(Z).
\end{align*}
Gathering the previous equations, we obtain that the restriction of  $\xi^k_{u,Y,X}$ to $F(uv,Z,X)$ is the composition $\phi^l_Z\circ f_Z\circ \tau(o_{Z}-o_{Y})\circ v^{-1}\circ \tau(o_{Y}-o_{X})\circ u^{-1}$.
By Lemma \ref{transl} we know that the translation $\tau(o_{Y}-o_{X})$ commutes with all the elements in $\W_{Y}$, and, in particular, with $v^{-1}$. Therefore, we can rewrite the restriction of  $\xi^k_{u,Y,X}$ as $\phi^l_Z\circ f_Z\circ \tau(o_{Z}-o_{Y})\circ \tau(o_{Y}-o_{X})\circ v^{-1}u^{-1}$, which simplifies to $\phi^l_Z\circ f_Z \circ \tau(o_{Z}-o_{X})\circ v^{-1}u^{-1}$, and this coincides with $\xi^k_{uv,Z,X}$.
\end{proof}
\noindent
We now have a definition of the Salvetti complex $\salbar$ with an explicit description of its skeleton structure. 

\begin{ex}
Take $(\W[\Gamma],S)$ a Coxeter system, and let $X\in \Sf_2$, where $X=\{s,t\}\subset S$.
 Figures \ref{coxpolfig} and \ref{2cellsalfig} illustrate an example of the fundamental cell $C[X]$ and a 2-cell $\Delta^2(X)$ of the Salvetti complex $\salbar$.   

   \begin{figure}[ht]
\centering
\begin{minipage}{8.5cm}
  \centering
  \includegraphics[width=8cm]{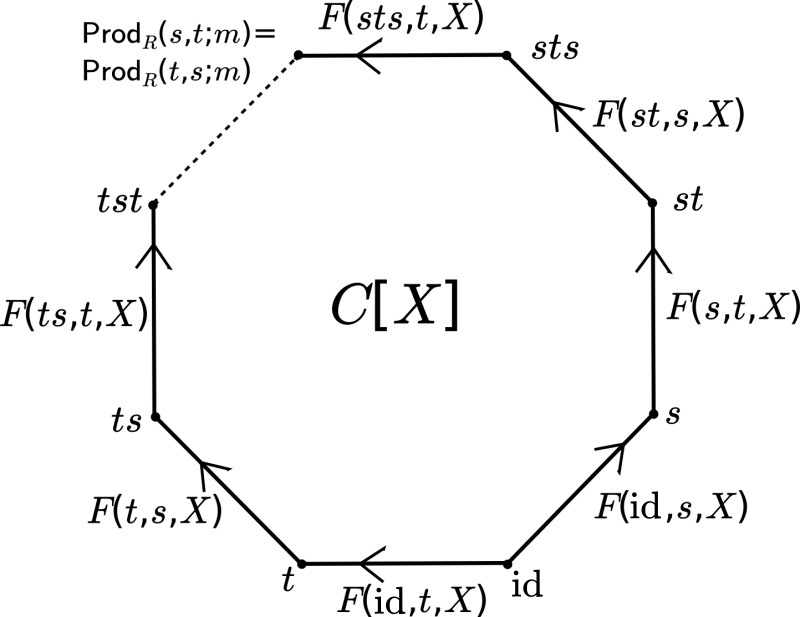}
  \caption{The Coxeter polytope for $X={\{s,t\}}$ and $m_{s,t}=4$.}
  \label{coxpolfig}
\end{minipage}
\qquad
\begin{minipage}{7cm}
  \centering
\includegraphics[width=7cm]{ 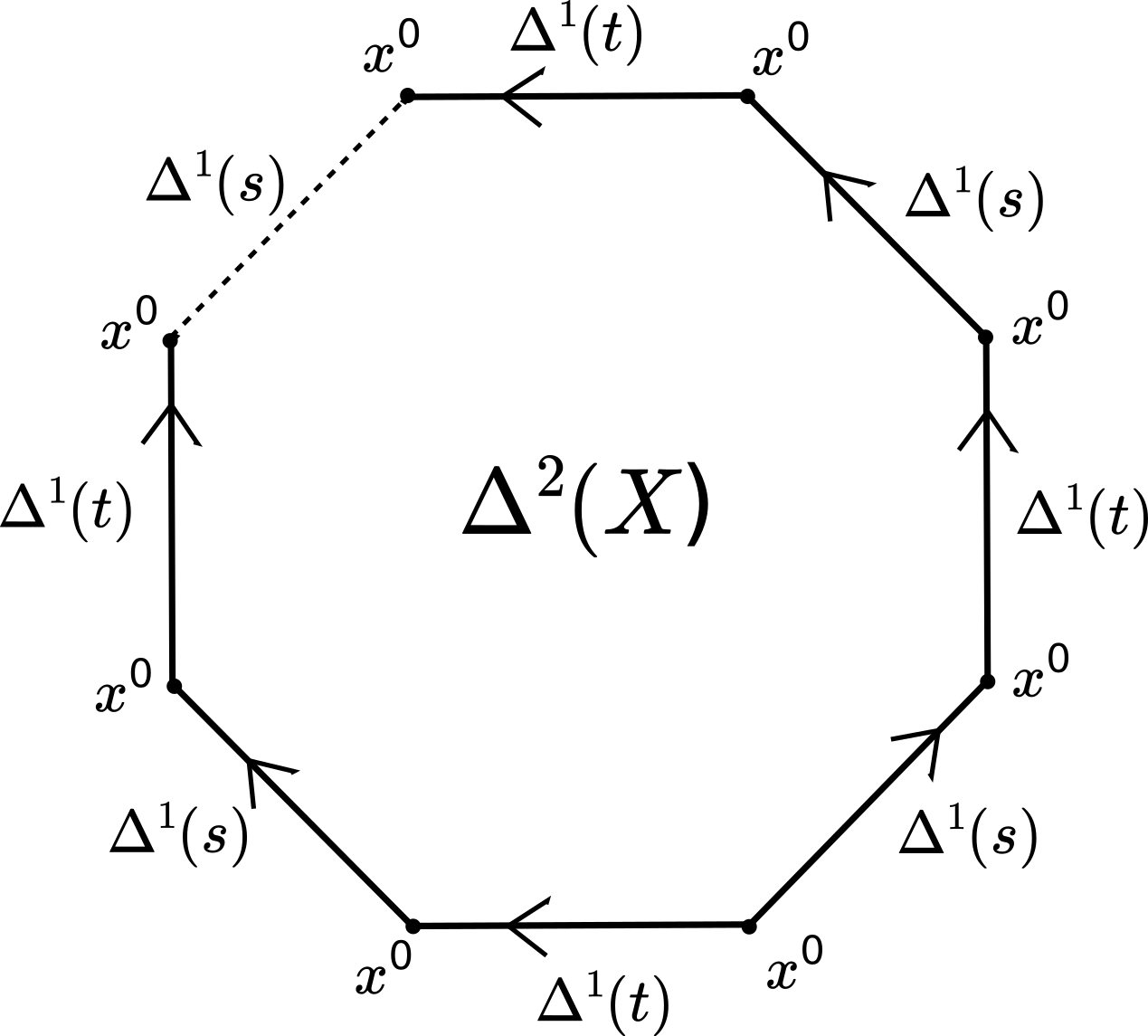}
  \caption{The 2-cell $\Delta^2(X)$ of the Salvetti complex $\salbar$.}
  \label{2cellsalfig}
\end{minipage}
\end{figure}
\end{ex}

\begin{obs}
    To be more precise, since the fundamental cells attached in the complex are polytopes, $\salbar$ is indeed a polyhedral complex, with a piecewise Euclidean structure.
\end{obs}

\begin{obs}
    The maps $\xi^k_{u,Y,X}$ preserve the orientation on the edges. Take $X\in \Sf_k$. For any $\{s\}=Y\subset X$ and $u\in \W[\Gamma]$, $u$ $(\emptyset,Y)-$minimal, the edge $F(u,s,X)$, oriented from $u(o_{X})$ to $us(o_{X})$, is sent by $\xi^k_{u,Y,X}=\phi^1_Y \circ f_Y\circ \tau(o_{Y}-o_{X})\circ u^{-1}$ to $\phi^1_{s}(f_s(C[s]))$, which is the segment $\mathbb{A}(s)$ oriented from $f_s(\alpha_s)$ to $f_s(-\alpha_s)$ with the vertices collapsed at $x^0$.
\end{obs}
\noindent
The Salvetti complex $\salbar$ was originally defined by Salvetti (see \cite{Sal87},\cite{Sal94}), who proved that it has the same homotopy type as the complement of the complexified reflection hyperplane arrangement $\W[\Gamma]$, under the action of $\W[\Gamma]$. Here, we use the notation  $\salbar$ instead of $\sal$ to indicate that the space is already taken under the action of $\W[\Gamma]$. The famous long-standing \textbf{$K(\pi,1)$-conjecture} can be reformulated to state that the Salvetti complex is a $K(\A[\Gamma],1)$-space, meaning that its fundamental group is $\A[\Gamma]$ and all higher homotopy groups vanish. The first part of the statement is known to be true thanks to the results in \cite{van1983homotopy}, so the open question is whether the universal covering of $\salbar$ is contractible. This conjecture has been solved for many Artin groups (see \cite{Deligne1972} for the spherical type case, or the more recent \cite{PaoSal} for the affine type), but the general case for any $\Gamma$ is still open.
\noindent
A widely used definition of $\salbar$ as a quotient of a simplicial complex is due to Paris (see Sections 3.2 and 3.3 in \cite{Par14}), where a similar cellular decomposition of the space also appears. From this description, the next result follows.

\begin{thm}(\cite{Sal87})\label{fundsalvetti}
    Let $\Gamma$ be a finite Coxeter graph, let $(\W[\Gamma],S)$ be its Coxeter system, $\A[\Gamma]$ be the associated Artin group and $\salbar$ be its Salvetti complex. Then\[ \pi_1(\salbar)\cong \A[\Gamma].\]
\end{thm}

\noindent Observe that the Salvetti complex of $(\W[\Gamma],S)$ has the cells in bijection with the standard parabolic subgroups of spherical type $\W_{X}$ of $\W[\Gamma]$, where $X\subset S$.

\subsection{Almost parabolic reflection subgroups}\label{APrefsubgrp}
The two CW-complexes that we are about to introduce will have their cells in bijection with a specific type of reflection subgroups of $\W[\Gamma]$. Before delving into the details, we will first present some general results about reflection subgroups of a Coxeter group, and then proceed to analyze the case of our interest.

\begin{defn}
    Let $(\W[\Gamma],S)$ be a Coxeter system. Define $\mathcal{R}=\{wsw^{-1}|\; s\in S, w\in \W\}$ as the set of all possible conjugates of the generators $S$. These are called \textit{reflections of $\W$}, and any subgroup of $\W$ generated by a subset $R\subset \mathcal{R}$ is called a \textit{reflection subgroup} of $\W$.
\end{defn}
\noindent
In particular, standard parabolic subgroups are reflection subgroups. Moreover, if $X\subset S$ and $w\in \W[\Gamma]$, then the non-standard parabolic subgroup $w\W_{X}w^{-1}$ is the reflection subgroup generated by the set $wXw^{-1}\subset \mathcal{R}$. Parabolic subgroups are a special case of the groups that we will introduce. \\
\\
\noindent
The reason why an element $wsw^{-1}$ of $\mathcal{R}$ is called a reflection is because it behaves as a reflection acting on the vector space $V$. Specifically, it fixes pointwise a hyperplane of $V$ and sends a nonzero vector (orthogonal to the hyperplane) to its opposite. For each $\beta=w(\alpha_s)$ with $w\in \W[\Gamma]$ and $s\in S$, recall that we defined $r_{\beta}=wsw^{-1}$. It can be observed that for every $v$ in $V$, $r_{\beta}(v)=v-2\langle v,\beta\rangle \beta $. In particular $r_{\beta}(\beta)=-\beta$ and $r_{\beta}(v)=v$ for all $v$ in $ H_{\beta}=\{v\in V\,|\, \langle \beta,v\rangle =0\}$. Notice that $r_{\beta}=r_{-\beta}$, creating a bijection between the set of positive (or negative) roots $\Phi^+[\Gamma]$ and the set of reflections $\mathcal{R}$. \\
\\
\noindent
Now consider a reflection subgroup $\langle R \rangle$ of $\W[\Gamma]$, with $R\subset\mathcal{R}$. If $\Gamma$ is of spherical type, then $\W[\Gamma]$ is finite, hence every reflection subgroup $\langle R\rangle$ is
a finite reflection group. By the classical theorem that asserts that finite real reflection groups
are finite Coxeter groups (see, for instance, \cite[Section 1.9]{Hump}), it follows that
$\langle R\rangle$ is a Coxeter group. More precisely, the fact that finite reflection groups admit a presentation of the form as in Equation \ref{prescoxgroup} is a classical result in the field, which can be found for instance in \cite[Section 1.9]{Hump}.\\
In the case where the group $\W[\Gamma]$ is not finite, it is less evident that the reflection subgroup admits a Coxeter presentation with respect to some set of generators, but it is true and it has independently been proven by Deodhar in \cite{Deodhar1989} and Dyer in \cite{DYER199057}. This leads to the following result:
\begin{thm}(\cite{Deodhar1989},\cite{DYER199057})\label{deodref}
    A reflection subgroup of a Coxeter group is itself a Coxeter group.
\end{thm}
\noindent
Thus, given a reflection subgroup $\langle R\rangle$ of $\W[\Gamma]$ with $R\subset \mathcal{R}$, it is possible to find a set $S'\subset \mathcal{R}$ such that $(\langle R\rangle,S')$ is a Coxeter system.
In general the set of Coxeter generators $S'$ does not coincide with the set of reflections $R$. In particular, the rank of $\langle R\rangle$ can be smaller than $|R|$. The reflection subgroups that we introduce in the following, will have the property that the set of Coxeter generators $S'$ is equal to $R$.
\begin{defn}
Given a set of roots $\cX \subset \Phi[\Gamma]$, we can consider the reflection subgroup generated by $R_{\cX}=\{r_{\beta}\,|\,\beta\in \cX\}\subset\mathcal{R}$. From now on, we will define any reflection subgroup $\ggr$ as $\langle R_{\cX}\rangle$ for some $\cX\subset \Phi[\Gamma]$. Observe that the choice made the other way around is not univocal, i.e. given a set of reflections $R$, there are many different choices for the associated subset of roots. Indeed, for any $r=wsw^{-1}\in R$, we can choose either $w(\alpha_s)$ or $-w(\alpha_s)$. Thus, if $|R|=n$, there are $2^n$ different subsets of roots $\cX\subset\Phi[\Gamma]$ that realize the same set of reflections $ R_{\cX}$.
\end{defn}

\begin{rmk}\label{parabolic}
A parabolic subgroup is associated to a very special subset of roots $\cX$. Specifically, if $P=w\W_{X}w^{-1}=\langle wXw^{-1}\rangle$ with $w\in \W[\Gamma]$ and $X\subset S$, assuming without loss of generality that $w$ is $(\emptyset,X)$-minimal, then $\cX=w(\Pi_X)\subset \Phi^+[\Gamma]$ is going to be the associated set of roots. Being the image under the isometry $w$ of the subset of the basis $\Pi_X$, $\cX=\{w(\alpha_s)|\,s\in X\}$ consists of linearly independent vectors, all of which can be written as $\beta=w(\alpha_{s})$ with the same $w\in \W[\Gamma]$, for $s\in X$. 
\end{rmk}
\begin{defn}
Let $\cX$ be a subset of $\Phi[\Gamma]$. If $\cX=w(\Pi_X)$ for some $X\subset S$ and $w\in \W[\Gamma]$, then we say that $\cX$ is \textit{parabolic}. If the reflection subgroup associated to $\cX$ is of spherical type, then we say that $\cX$ is of \textit{spherical type}.
\end{defn}
\begin{rmk}
    In the previous definition we did not require for $w$ to be $(\emptyset,X)$-minimal. The choices of $w\in \W[\Gamma]$ and $X\subset S$ for which we can express $\cX$ as $w(\Pi_{X})$ may not be unique.
\end{rmk}
\noindent
We are now interested in subsets $\cX\subset\Phi[\Gamma]$ that resemble parabolic ones, but weakening some of the properties. We will require $\cX$ to be composed of linearly independent roots, but we will drop the condition that all the roots $\beta$ in $\cX$ must be expressible as $\beta=w(\alpha_s)$ for the same $w\in \W[\Gamma]$ and for some $s\in S$.
\begin{defn}\label{defapsetroots}
    Let $\cX\subset \Phi[\Gamma]$. We say that $\cX$ is \textit{almost parabolic} (and we write $\cX$ is AP) if
    \begin{enumerate}
        \item [(AP1)] The vectors in $\cX$ are linearly independent;
        \item [(AP2)] For all $ \beta,\gamma \in \cX$, there exist $w\in \W[\Gamma]$ and $s,t\in S$ such that $\beta=w(\alpha_s)$, $\gamma=w(\alpha_t)$.
    \end{enumerate}
    \noindent Given a reflection subgroup $\ggr$ for $R\subset \mathcal{R}$, if there exists a choice of $\cX\subset \Phi[\Gamma]$ such that $\langle R\rangle=\langle R_{\cX}\rangle$ and $\cX$ is almost parabolic, then we also call the reflection subgroup $\langle R_{\cX}\rangle$ an \textit{almost parabolic reflection subgroup}, or an \textit{AP reflection subgroup}, for short. We note that if $\cX$ is parabolic, then it is almost parabolic.

\end{defn}

\begin{rmk}\label{card2parabolic}
    If $\cX\subset \Phi[\Gamma]$ is AP and $|\cX|=2$, then, by Property (AP2), $\cX$ is parabolic. Indeed, let $\cX=\{\beta,\gamma\}$ with $\beta=w(\alpha_s)$,$\gamma=w(\alpha_t)$ for some $s,t\in S$ and $w\in \W[\Gamma]$. Then, by setting $X=\{s,t\}$, we obtain $\cX=w(\Pi_X)$.
\end{rmk}

\begin{rmk}
   Given $\langle R\rangle$ a reflection subgroup, in Definition \ref{defapsetroots} we established that $\ggr$ is AP if there exists $\cX\subset \Phi[\Gamma]$ an AP subset of roots such that $\ggr=\langle R_{\cX}\rangle$. Note that such a choice of $\cX$ may not yield a generating set of reflections $R_{\cX}$ that coincides with $R$ (see Example \ref{exparabolicsofSn}). Moreover, even if two sets of roots give rise to the same set of reflections, the fact that one of them is AP does not imply that the other is AP as well. For instance, if $\cX=\{\beta,\gamma\}$ is AP, then $\cX'=\{-\beta,\gamma\}$ and $-\cX'$ are not necessarily AP, even though $R_{\cX}=R_{\cX'}=R_{-\cX'}$.
\end{rmk}
\noindent
We will now see that if a reflection subgroup given by $\cX\subset\Phi[\Gamma]$ is AP, then a set of Coxeter generators is given exactly by the reflections associated with the roots in $\cX$.

\begin{thm}\label{refAP}
 Let $(\W[\Gamma],S)$ be a Coxeter system and let $\cX \subset \Phi[\Gamma]$ be almost parabolic. Then its associated reflection subgroup $\langle R_{\cX} \rangle$ with $R_{\cX}=\{r_{\beta}\;|\;\beta\in \cX\}$ has a Coxeter presentation with the canonical set of generators $R_{\cX}$. Namely, $(\langle R_{\cX}\rangle,R_{\cX})$ is a Coxeter system.
\end{thm}
\begin{proof}
    By Theorem \ref{deodref}, we know that $\langle R_{\cX} \rangle$ is a Coxeter group. We now aim to show that $R_{\cX}$ is a Coxeter generating set for $\langle R_{\cX} \rangle$ and we will exhibit the associated Coxeter matrix. By property (AP2), for every pair of reflections $r_{\beta},r_{\gamma}\in R_{\cX}$, there exist $ w$ in $\W[\Gamma]$ and $ s,t$ in $S$ such that $r_{\beta}=wsw^{-1}$ and $r_{\gamma}=wtw^{-1}$. Thus, the product $(r_{\beta}r_{\gamma})=w(st)w^{-1}$ has order $m_{s,t}$. If there exists $w'\in \W[\Gamma]$ and $s',t'\in S$ such that $\beta=w'(\alpha_{s'})$ and $\gamma=w'(\alpha_{t'})$, the previous equality implies that $m_{s,t}=m_{s',t'}$. 
    %\textit{Claim 1.} The order of the product $r_{\beta}r_{\gamma}$ does not depend on the choice of $s,t$ and $w$.\\
    %\textit{Proof of Claim 1.} Take $\beta\neq\gamma\in \cX$, and suppose that there exist $w,v\in \W[\Gamma]$ and $s_1,s_2,t_1,t_2\in S$ such that $\beta=w(\alpha_{s_1})=v(\alpha_{s_2})$ and $\gamma=w(\alpha_{t_1})=v(\alpha_{t_2})$. The action of $\W[\Gamma]$ on $V$ preserves the bilinear symmetric form, so suppose that $m_{s_1,t_1},m_{s_2,t_2}\neq \infty$ and compute
    %\[-\cos\left({\frac{\pi}{m_{s_2,t_{2}}}}\right)=\langle v(\alpha_{s_2}),v(\alpha_{t_2})\rangle=\langle \beta,\gamma\rangle=\langle w(\alpha_{s_1}),w(\alpha_{t_1})\rangle=-\cos\left({\frac{\pi}{m_{s_1,t_{1}}}}\right),
    %\]
    %which implies $m_{s_1,t_1}=m_{s_2,t_2}$. If one of the two is equal to $\infty$, then necessarily they are both infinite for the same reason, and this concludes the proof of Claim 1.\\
    We denote by $m_{\beta,\gamma}$ the order of $r_{\beta}r_{\gamma}$
    and we can then provide a Coxeter matrix $\mathrm{M}_{\cX}:=(m_{\beta,\gamma})_{\beta,\gamma\in \cX}$ on $\cX$. Let us define an abstract set $Q_{\cX}=\{q_{\beta}\,|\, \beta\in \cX\}$ and let $\W_{\cX}$ be the abstract Coxeter group given by the following presentation:
    \begin{equation}
        \W_{\cX}=\langle q_{\beta},\,\beta\in \cX\;|\; q_{\beta}^{2}=1\; \forall \beta\in \cX, \; (q_{\beta}q_{\gamma})^{m_{\beta,\gamma}}=1 \; \forall \beta,\gamma\in \cX,\; \beta\neq \gamma\;, {m}_{\beta,\gamma}\neq \infty\rangle.\nonumber 
    \end{equation}
    We wish to prove that the groups $\W_{\cX}$ and $\langle R_{\cX}\rangle$ are isomorphic. Define $f_{\cX}:\W_{\cX}\longrightarrow \langle R_{\cX}\rangle$ by sending $q_{\beta}\longmapsto r_{\beta}$ for all $\beta$ in $\cX$. From what has been shown, this is a well-defined group homomorphism. Since $\langle R_{\cX}\rangle$ is generated by $R_{\cX}$, $f_{\cX}$ is surjective.\\
    \\
    \noindent Now, define as usual $V=\bigoplus_{s\in S}\mathbb{R}\cdot \alpha_s$ and $V_{\cX}=\bigoplus_{\beta\in \cX}\mathbb{R}\cdot \beta\subset V$. Recall that for any $\beta\in \Phi[\Gamma]$ and $v\in V$, the reflection $r_{\beta}$ acts on $V$ through $r_{\beta}(v)=v-2\langle v, \beta\rangle \beta$, where $\langle\cdot,\cdot\rangle$ is the symmetric bilinear form defined by the Coxeter matrix on $S$. Since all the reflections $r_{\beta}$ stabilize $V_{\cX}$, we can consider the restriction of the canonical representation of $\W[\Gamma]$ to $\rho|_{\langle R_{\cX}\rangle}:\langle R_{\cX}\rangle\longrightarrow 
   GL(V_{\cX})$.\\
   \\
   \noindent Next, define a set $\Pi_{\cX}=\{\eta_{\beta}\,|\, \beta\in \cX\}$ in one-to-one correspondence with the set of generators $Q_{\cX}$ of $\W_{\cX}$, and a formal real vector space $U_{\cX}=\bigoplus_{\beta\in \cX}\mathbb{R}\cdot \eta_{\beta}$ whose basis is $\Pi_{\cX}$. There is a symmetric bilinear form $\llangle \cdot, \cdot \rrangle_{\cX}$ on $U_{\cX}$ defined by $\llangle \eta_{\beta}, \eta_{\gamma} \rrangle_{\cX}=-\cos(\frac{\pi}{m_{\beta,\gamma}})$ if $m_{\beta,\gamma}\neq \infty$ and $\llangle \eta_{\beta}, \eta_{\gamma} \rrangle_{\cX}=-1$ otherwise.
    The canonical linear representation $\rho_{\cX}:\W_{\cX}\longrightarrow GL(U_{\cX})$ is given by $q_{\beta}\longmapsto (u\longmapsto u-2\llangle \eta_{\beta}, u \rrangle_{\cX}\eta_{\beta})$ for all $u$ in $U_{\cX}$, and we know that it is faithful.\\
    \\
    \noindent The abstract Coxeter group $\W_{\cX}$ naturally acts on $V_{\cX}$ through the composition
    \begin{align*}
        \W_{\cX}\xrightarrow[]{f_{\cX}} \langle R_{\cX} \rangle &\xrightarrow[]{\rho|_{ \langle R_{\cX} \rangle }} GL(V_{\cX})\\
        q_{\beta}\longmapsto r_{\beta}&\longmapsto (v\longmapsto v-2\langle v,\beta \rangle \beta),
    \end{align*}
    for all $v\in V_{\cX}$, giving a representation that we call $\rho'_{\cX}$.\\
    \\
    \noindent We now show that the two representations $\rho_{\cX}$ and $\rho'_{\cX}$ of $\W_{\cX}$ are isomorphic. It is sufficient to show that there is an isomorphism $a:U_{\cX}\longrightarrow V_{\cX}$ such that for all $q_{\beta}\in Q_{\cX}$ we have $a\circ \rho_{\cX}(q_{\beta})=\rho'_{\cX}(q_{\beta})\circ a$. Define the isomorphism $a$ by $a(\eta_{\beta})=\beta\in V_{\cX}$ and observe that for all $\beta,\gamma\in \cX$, $\llangle \eta_{\beta},\eta_{\gamma}\rrangle_{\cX}=\langle \beta,\gamma\rangle $. Now, compute for any $\eta_{\gamma}\in U_{\cX}$
    the value of $a(\rho_{\cX}(q_{\beta})(\eta_{\gamma}))=a(\eta_{\gamma}-2\llangle \eta_{\beta}, \eta_{\gamma} \rrangle_{\cX}\eta_{\beta})=\gamma-2\langle \beta,\gamma\rangle \beta$, which is exactly $\rho'_{\cX}(q_{\beta})(a(\eta_{\gamma}))$. Hence, the equality holds for all the vectors $u$ in $U_{\cX}$ and the two representations are isomorphic.\\
    \\
    \noindent Thus the representation $\rho'_{\cX}=\rho|_{\langle R_{\cX}\rangle}\circ f_{\cX}$ is faithful, and the group homomorphism $f_{\cX}$ is injective. Therefore, the two groups $\W_{\cX}$ and $\langle R_{\cX}\rangle$ are isomorphic and $\langle R_{\cX}\rangle$ admits a presentation of the form
    \begin{equation}\label{presrefsub}
        \langle R_{\cX}\rangle\cong \W_{\cX}=\langle r_{\beta},\,\beta\in \cX\;|\; r_{\beta}^{2}=1\; \forall \beta\in \cX, \; (r_{\beta}r_{\gamma})^{m_{\beta,\gamma}}=1 \; \forall \beta,\gamma\in \cX,\; \beta\neq \gamma,\,{m}_{\beta,\gamma}\neq \infty\rangle. 
    \end{equation}
    Namely, $(\langle R_{\cX}\rangle,R_{\cX})$ is a Coxeter system.
\end{proof}

\begin{nt}
    By the previous theorem, we know that an AP subset of roots $\cX$ is associated to a reflection subgroup that admits a Coxeter presentation with set of generators given by the reflections relative to the roots in $\cX$, like in Equation \ref{presrefsub}. For this reason, from now on, whenever we have $\cX$ AP, we will denote the reflection subgroup by $\W_{\cX}=\langle R_{\cX}\rangle$ and the Coxeter system by $(\W_{\cX}, R_{\cX})$ to mean that we consider the presentation for $\W_{\cX}$ with Coxeter matrix $\mathrm{M}_{\cX}$.
\end{nt}

\begin{rmk}
    If $\ggr$ is an AP reflection subgroup, the Coxeter presentation obtained in Equation \ref{presrefsub} does not depend on the choice of the AP root set $\cX$ such that $R=R_{\cX}$.
    Suppose that $\cX,\cX'$ are two AP sets of roots such that $\cX'\neq \cX$ and $R=R_{\cX}=R_{\cX'}$.
 Then $\cX'$ is obtained from $\cX$ by changing the sign of a certain number of roots. For each $\beta\in \cX$, we have that $\beta'\in \cX'$, where either $\beta'=\beta$ or $\beta'=-\beta$. Since the two sets of roots $\cX$ and $\cX'$ are AP, we apply Theorem \ref{refAP} and we obtain two Coxeter presentations as in Equation \ref{presrefsub}, one for $\cX$ and one for $\cX'$. These two presentations for the same group have obviously the same generating set $R_{\cX}=R_{\cX'}$, and the two Coxeter matrices $\mathrm{M}_{\cX}$ and $\mathrm{M}_{\cX'}$ coincide as well. Indeed, if $\beta,\gamma\in \cX$, the entry $m_{\beta',\gamma'}$ of $M_{\cX'}$ is the order $m_{\beta,\gamma}$ of $r_{\beta'}r_{\gamma'}=r_{\beta}r_{\gamma}$.
\end{rmk}

\begin{obs}\label{obschangingcXap}
    Note also that the base of the root system associated with the Coxeter group $\W_{\cX}$ with Presentation \ref{presrefsub} is precisely the set $\cX$. Changing the associated AP set of roots corresponds to changing the choice of the base for the
    root system of $\W_{\cX}$.
\end{obs}
\noindent
In the next subsection we will focus on reflection subgroups of spherical type that are
also AP. Spherical-type parabolic subgroups provide examples of such subgroups,
but not all such subgroups are parabolic.
\begin{ex}
    It is not straightforward to find an example of an AP reflection subgroup of spherical type that is not parabolic. In \cite{DPR13}, the authors give a classification of reflection subgroups of finite Coxeter groups, and thanks to this we know that for $\Gamma=E_7$ the Coxeter group $\W[E_7]$ contains a reflection subgroup isomorphic to the direct product of 7 copies of $\W[A_1]=\mathbb{Z}_2$ (see the last row of Table 4 in the reference). Therefore, in $\W[E_7]$ we can find $r_1,\ldots, r_7$ reflections such that $r_ir_j=r_jr_i$ for all $i\neq j$, $i,j\in \{1,\ldots,7\}$. The reflection subgroup that they generate cannot be parabolic, since, being of rank 7, it should be the whole group $\W[E_7]$ whose generators do not commute. But it is almost parabolic. To show that, we need the following lemma.

\begin{lem}\label{JeCorrige}
Let $(\W[\Gamma],S)$ be a Coxeter system, with $\Gamma$ such that all labels $m_{s,t}$ belong to $\{2,3\}$. Let $\beta, \beta' \in \Phi[\Gamma]$ be roots such that the reflections $r_{\beta}$
and $r_{\beta'}$ are distinct and commute. Then there exist $w\in \W[\Gamma]$ and $s,t\in S$
such that $\beta=w(\alpha_s)$, $\beta'=w(\alpha_{t})$. In particular, $m_{s,t}=2$. 
\end{lem}

\noindent 
The proof of this lemma is based on the study of the dual representation $\rho^*:\W[\Gamma]\longrightarrow GL(V^*)$ and on the action of $\W[\Gamma]$ on the Tits cone. We recall the fundamental notions and we refer to \cite{Bourbaki} and \cite{krammer} for the details. \medskip\\
 \noindent As in Section \ref{CoxandArt}, given a Coxeter group $\W[\Gamma]$ with generating set $S$, we denote by $V$ the real formal vector space with basis $\Pi=\{\alpha_s\mid s\in S\}$. The canonical linear representation of $\W$ is the map $\rho:\W[\Gamma]\longrightarrow GL(V)$. In what follows, let $V^*$ denote the dual of the real vector space $V$. The \textit{dual representation} of $\rho$, denoted as $\rho^*$, is defined by:
\begin{align}
\rho^* \;:\; \W[\Gamma] &\longrightarrow GL\left(V^*\right) \nonumber  \\
w &\longrightarrow \rho^{*}\left(w\right), \nonumber
\end{align}
where, for all $\xi \in V^*$, we set $\rho^{*}\left(w\right)\left(\xi\right):=\xi\circ \rho\left(w^{-1}\right)$. As we saw previously in this section, each reflection $r=wsw^{-1}\in \mathcal{R}$ acts on $V$ through $\rho$ as a reflection with respect to the hyperplane $H_{\beta}=H_{w(\alpha_s)}$, where $H_{\beta}=\{v\in V\,|\, \langle\beta,v\rangle=0\}$. We can see that $r^*=\rho^*(r)$ acts on $V^*$ as a reflection with fixed hyperplane $H^*_{w(\alpha_s)}=\left\{\xi \in V^*\,|\,\xi(w(\alpha_s))=\xi(\beta)=0\right\}$.\\
The \textit{fundamental chamber} of the action of $\W[\Gamma]$ on $V^*$ through $\rho^*$ is
\[
\overline{C}=\{\xi\in V^*\mid  \xi(\alpha_s)\geq 0\; \text{ for all $s\in S$}\},
\]
\noindent and we consider the following:
\[
I=\bigcup_{w\in \W[\Gamma]}\rho^*(w)(\overline{C}).
\]
\noindent The interior of $I$, denoted by $I^{\circ}$, is an open, nonempty convex cone called \textit{Tits cone}. The following result can be found in \cite[Section 2]{krammer}.
\begin{enumerate}
    \item [(1)] For all $\beta\in \Phi[\Gamma]$, $\rho^*(r_{\beta})$ is a reflection that fixes $H^*_{\beta}$.
    \item[(2)] For all $\beta\in \Phi[\Gamma]$, we have $H^*_{\beta}\cap I^{\circ}\neq \emptyset$.
    \item[(3)] The set $\{H^*_{\beta}\mid\beta \in \Phi[\Gamma]\}$ is locally finite in $I^{\circ}$. Namely, for any $\xi \in I^{\circ}$, there exists an open neighborhood $O\subseteq I^{\circ}$ of $\xi$ such that the set of reflection hyperplanes $H^*_{\beta}$ intersecting $O$ is finite. 
    \item[(4)] For all $\xi\in I^{\circ }$, the subgroup $\W_{\xi}=\{w\in \W[\Gamma]\mid \rho^*(w)(\xi)=\xi\}$ of $\W[\Gamma]$ is a finite parabolic subgroup.
\end{enumerate}
\begin{proof}[Proof of Lemma \ref{JeCorrige}]
Since $r_{\beta}$ and $r_{\beta'}$ commute, we have $r_{\beta}r_{\beta'}r_{\beta}= r_{\beta'}$. On the other hand,
$r_{\beta}r_{\beta'}r_{\beta}=r_{r_{\beta}(\beta')}$. Thus $r_{r_{\beta}(\beta')}= r_{\beta'}$ , so $r_{\beta}(\beta')=\pm \beta'$. The negative sign
would imply that $\beta'$ belongs to the (-1)-eigen-space of $r_{\beta}$ , hence that $r_{\beta'}=r_{\beta}$, contrary to the hypothesis. Therefore $r_{\beta}(\beta')=\beta'$, and the formula
$r_{\beta}(\beta') = \beta' - 2\langle\beta' , \beta\rangle\beta$ gives $\langle \beta,\beta'\rangle=0$.\\\\
\noindent Set now $L=H^*_{\beta}\cap H^*_{\beta'}$ the intersection of the reflecting hyperplanes in the dual vector space $V^*$ associated respectively with $\beta$ and $\beta'$. Since $\langle \beta,\beta'\rangle=0$, the subgroup $G$ generated by $r_{\beta}$ and $r_{\beta'}$ is finite and isomorphic to $\mathbb{Z}/2\mathbb{Z}\times\mathbb{Z}/2\mathbb{Z}$. We choose now a point $\xi$ in the interior of $\overline{C}$ and we set $\mu_0=\sum_{g\in G}\rho^*(g)(\xi)$. Since $I^{\circ}$ is convex, then $\mu_0$ must belong to $I^{\circ}$. Moreover, observe that $\mu_0$ is fixed by all elements of $G$, and in particular by $r_{\beta}$ and $r_{\beta'}$. Therefore, $\mu_0\in L$, and thus $L\cap I^{\circ}\neq \emptyset$. Now, since $I^{\circ}$ is open and the set $\{H^*_{\gamma}\mid \gamma \in \Phi[\Gamma]\}$ is locally finite, we can choose a ``generic'' point $\mu\in L\cap I^{\circ}$, i.e. a point such that, for all $\gamma \in \Phi[\Gamma]$, if $\mu \in H^*_{\gamma}$, then $L\subseteq H^*_{\gamma}$.\\
\\
\noindent By the previously listed properties, we know that the group $\W_{\mu}=\{w\in \W[\Gamma]\mid \rho^*(w)(\mu)=\mu\}$ is a finite parabolic subgroup of $\W[\Gamma]$. Furthermore, since any element of $\W_{\mu}$ fixes $L$, this group is necessarily of rank 2. Indeed, write $\W_{\mu}=w\W_Yw^{-1}$ for some $Y\subseteq S$. For each $y\in Y$, the reflecting hyperplane $H^*_{w(\alpha_y)}$ contains $\mu$, hence contains $L$. Thus $L\subseteq \bigcap _{y\in Y}H^*_{w(\alpha_y)}$. The roots $w(\alpha_y)$ for $y\in Y$ are linearly independent, so the intersection on the right has codimension $|Y|$. Since $L$ has codimension 2, it follows that $|Y|\leq 2$. Since $\W_{\mu}$ contains the two distinct reflections $r_{\beta}$ and $r_{\beta'}$, we also have $|Y|\geq 2$. Hence $|Y|=2$. Then, there exist $s,t\in S$ and $w\in \W[\Gamma]$ such that $Y=\{s,t\}$, $m_{s,t}\neq \infty$ and $\W_{\mu}=w\W_{\{s,t\}}w^{-1}$. In particular, $\beta, \beta'\in w(\Phi[\Gamma_{\{s,t\}}])$.\\
\\
\noindent
Recall that the value of $m_{s,t}$ is either 2 or 3. If $m_{s,t}=3$, then two roots $\beta, \beta'$ in $w(\Phi[\Gamma_{\{s,t\}}])$ such that $\langle \beta,\beta' \rangle=0$ cannot exist. Then it must necessarily be $m_{s,t}=2$. In this case
\[
w(\Phi[\Gamma_{\{s,t\}}])=\{w(\alpha_s),-w(\alpha_s), w(\alpha_t),-w(\alpha_t)\}.
\]
\noindent
Without loss of generality, we can suppose that $\beta \in \{w(\alpha_s),-w(\alpha_s)\}$ and $\beta' \in \{w(\alpha_t),-w(\alpha_t)\}$. If $\beta=w(\alpha_s)$ and $\beta'=w(\alpha_t)$, then the result is trivially true. If  $\beta=-w(\alpha_s)$ and $\beta'=w(\alpha_t)$, then $\beta=ws(\alpha_s)$ and $\beta'=ws(\alpha_t)$, and the result is true. Similarly, if  $\beta=w(\alpha_s)$ and $\beta'=-w(\alpha_t)$, then $\beta=wt(\alpha_s)$ and $\beta'=wt(\alpha_t)$, and the result is true. Finally, if $\beta=-w(\alpha_s)$ and $\beta'=-w(\alpha_t)$, then $\beta=wst(\alpha_s)$ and $\beta'=wst(\alpha_t)$, and the result is true.
    \end{proof}
    \noindent
Now, coming back to $\W[E_7]$, for each $i \in \{1, \dots, 7\}$ we fix a root $\beta_i \in \Phi [\Gamma]$ such that $r_i = r_{\beta_i}$.
By Lemma \ref{JeCorrige}, for each pair $\beta_i, \beta_j$ with $i \neq j$ and $i,j \in \{1, \dots, 7\}$, we can find $w \in W [\Gamma]$ and $s_i, s_j \in S$ such that $\beta_i = w (\alpha_{s_i})$ and $\beta_j = w (\alpha_{s_j})$.
Therefore, $\mathcal{X} = \{\beta_1, \dots, \beta_7\}$ is AP and $\langle R_{\mathcal X} \rangle$ is an AP reflection subgroup which is not parabolic.
%Thus, thanks to this lemma, for each pair $r_i,r_j$ with $i\neq j$ and $i,j\in \{1,\cdots,7\}$, we can find a $w\in \W[\Gamma]$ and $s_i,s_j\in S$ such that $r_i=ws_iw^{-1}$ and $r_j=ws_jw^{-1}$.
%Call now $\beta_i$ the root $w(\alpha_{s_i})$  for all $i=1,\cdots,7$, which satisfies $r_{\beta_i}=r_i$, and call $\cX$ the set of roots $\{\beta_1,\cdots,\beta_7\}$. These roots are orthogonal, so clearly linear independent. Moreover, for what we have just said, they satisfy property AP2, thus $\cX$ is AP. Therefore, the reflection subgroup $\{r_1,\cdots, r_7\}$ is almost parabolic but not parabolic. 
    \end{ex}

\begin{ex}\label{exd4}
Another example of an almost parabolic (AP) reflection subgroup that is not parabolic can be found in the Coxeter group of type $D_4$, illustrated in Figure \ref{fig:d4}. 
    \begin{figure}[h!]
        \centering
        \includegraphics[width=3cm]{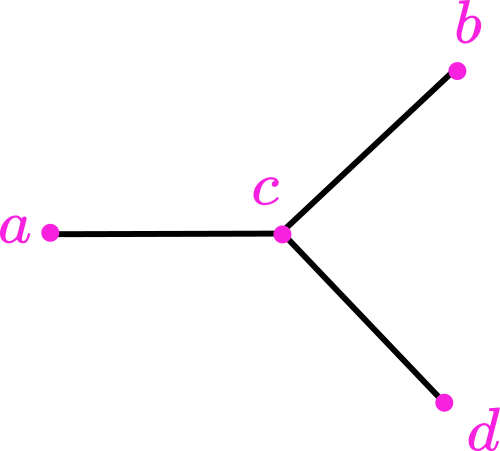}
        \caption{The Coxeter graph of type $D_4$.}
        \label{fig:d4}
    \end{figure}
  \noindent 
  Compatibly with Subsection \ref{CoxandArt}, we can construct the root system of $\W[D_4]$ as follows. Let $e_1,e_2,e_3,e_4$ be the canonical basis of $\mathbb{R}^4$, and set
  \[
  \alpha_a=\frac{e_1-e_2}{\sqrt{2}},\qquad\alpha_c=\frac{e_2-e_3}{\sqrt{2}},\qquad\alpha_b=\frac{e_3-e_4}{\sqrt{2}},\qquad\alpha_d=\frac{e_3+e_4}{\sqrt{2}}.
  \]
  \noindent Then $\langle \alpha_s,\alpha_s\rangle=1$ for all $s\in \{a,b,c,d\}$, and the roots $\alpha_s,\alpha_b,\alpha_c,\alpha_d$ are the simple roots of the root system of type $D_4$ (see \cite[Ch.VI, $\S4$, n.8]{Bourbaki}). Now consider the longest root
  \[
  \eta=\alpha_a+2\alpha_c+\alpha_b+\alpha_d=\frac{e_1+e_2}{\sqrt{2}}.
  \]
  \noindent This root is orthogonal to the three simple roots $\alpha_a$, $\alpha_b$ and $\alpha_d$:
  \[
  \langle\eta,\alpha_a \rangle=\langle \alpha_a,\alpha_a\rangle+2\langle\alpha_c,\alpha_a\rangle=1+2\left(-\frac{1}{2}\right)=0,
  \]
  \noindent and similarly 
  \[
  \langle \alpha_b,\eta \rangle=0, \qquad\qquad\langle\alpha_d,\eta\rangle=0.
  \]
  \noindent Since also
  \[
  \langle\alpha_a ,\alpha_b \rangle=\langle\alpha_a,\alpha_d\rangle=\langle\alpha_b,\alpha_d\rangle=0,
  \]
  \noindent the four roots
  \[
  \cX=\{\alpha_a,\alpha_b,\alpha_d, \eta\}
  \]
\noindent are pairwise orthogonal, hence linearly independent. Therefore the corresponding reflections
\[
a=r_{\alpha_a},\qquad b=r_{\alpha_b},\qquad d=r_{\alpha_d},\qquad r:=r_{\eta},\qquad
\]
\noindent commute pairwise and generate a reflection subgroup of type of type $\W[A_1]\times \W[A_1]\times \W[A_1]\times \W[A_1]$.\medskip\\
\noindent If one wishes to get an expression of $r$ in the standard generators of $\W[D_4]$, one can express the highest root step-by-step:
\begin{align*}
    c(\alpha_a)=\alpha_a+\alpha_c,\\
    bc(\alpha_a)=\alpha_a+\alpha_b+\alpha_c,\\
    dbc(\alpha_a)=\alpha_a+\alpha_b+\alpha_c+\alpha_d,
\end{align*}
\noindent and therefore
\[ 
\eta=(cdbc)(\alpha_a)=\alpha_a+\alpha_b+2\alpha_c+\alpha_d.
\]
\noindent Thus, $r=r_{\eta}=(cdbc)a(cdbc)^{-1}$. By Lemma \ref{JeCorrige}, since all edge labels in the Coxeter graph of $D_4$ are either 2 or 3, it follows that $\langle R_{\cX} \rangle$ is an AP reflection subgroup. However, since its rank is four, if $\langle R_{\cX}\rangle$ were parabolic, it would necessarily be the entire group $\W[D_4]$, which is a contradiction. Therefore, $\langle R_{\cX}\rangle$ is an almost parabolic reflection subgroup, but not a parabolic one.
\end{ex}    
\begin{rmk}
From Example \ref{exd4}, we see that the property of being an almost parabolic reflection subgroup is not preserved under intersection. Indeed, we can construct another AP reflection subgroup by conjugating $\langle R_{\cX} \rangle$ by the generator $c$. This yields a new reflection subgroup  $\langle R'\rangle$ of $\W[D_4]$ generated by $\{cac,cbc,cdc,crc\}$. Again, by Lemma \ref{JeCorrige}, $\langle R'\rangle$ is an AP reflection subgroup, which is not parabolic.
\\
\\
\noindent It can be checked computationally (for instance in GAP/CHEVIE) that the intersection of these two AP reflection subgroups is $\{\id, w_0\}$, where $w_0=abdr=cabdrc$ is the longest element of $\W[D_4]$. This element acts on the vector space $V$ as a rotation of $\pi$ radiant (i.e., it is the central symmetry with respect to the origin), and hence cannot be a reflection. Therefore, the intersection \[\langle R\rangle \cap \langle R'\rangle=Z(\W[D_4])=\{\id,w_0\},\] 
\noindent is not a reflection subgroup, and thus it is not an AP reflection subgroup.\\
\\
\noindent In general, the intersection of two reflection subgroups is not necessarily a reflection subgroup. There are exceptions: most notably, the intersection of two parabolic subgroups in a Coxeter group is again a parabolic subgroup, and hence a reflection subgroup.  Other families of reflection subgroups, such as \textit{geodesic} reflection subgroups, are also known to be stable under intersection. A subgroup $\W'$ of a Coxeter group $\W$ is said to be \textit{geodesic} if, for every $w\in \W'$, whenever $w=t_1t_2\cdots t_k$
is a reduced expression of $w$ (as a product of reflections $t_i\in \mathcal{R}$, not necessarily simple ones) in $\W$, then each $t_i$ belongs to $\W'$.
It follows immediately from the definition that any such subgroup is a reflection subgroup. Moreover, the intersection of geodesic subgroups is again geodesic. Parabolic subgroups are geodesic (see Theorem 1.4 in \cite{BaumeisterDyerStumpWegener2014}), but generally there are geodesic subgroups that are not parabolic.\\
\end{rmk}
\bigskip
\noindent In the following, we relate the AP reflection subgroups of $\W[\Gamma]$ with the standard parabolic subgroups of $\W[\widehat{\Gamma}]$.\\
\\
\noindent
We have seen in Subsection \ref{virtualartingroups} that given a Coxeter graph $\Gamma$, we can build another Coxeter graph $\widehat{\Gamma}$ whose set of vertices is $\Phi[\Gamma]$ and the Coxeter matrix $\widehat{\mathrm{M}}=(\widehat{m}_{\beta,\gamma})_{\beta,\gamma \in \Phi}$ is defined as $\widehat{m}_{\beta,\gamma}=m_{s,t}$ if there exist $w\in \W, s,t\in S$ such that $m_{s,t}\neq\infty$ and $\beta=w(\alpha_s)$,$\gamma=w(\alpha_t)$, while we set $\widehat{m}_{\beta,\gamma}=\infty$ otherwise. Denote by $\widehat{S}$ a set $\{ \widehat{s}_{\beta}\;|\; \beta \in \Phi[\Gamma]\}$ in one-to-one correspondence with the root system $\Phi[\Gamma]$. Then, consistent with the definition of $\widehat{\W}=\W[\widehat{\Gamma}]$ recalled in Subsection \ref{virtualartingroups}, we have that the Coxeter presentation of $\widehat{\W}$ is 
    \[
    \widehat{\W}=\langle \widehat{S}\;|\; \widehat{s}_{\beta}^2=1 \; \forall \widehat{s}_{\beta}\in \widehat{S},\; (\widehat{s}_{\beta}\widehat{s}_{\gamma})^{\widehat{m}_{\beta,\gamma}}=1 \; \forall \widehat{s}_{\beta},\widehat{s}_{\gamma}\in \widehat{S}, \quad \widehat{s}_{\beta}\neq\widehat{s}_{\gamma},\;\widehat{m}_{\beta,\gamma}\neq \infty \rangle.
    \]
    \noindent
    Thus, if we take a subset $\cX$ of the root system $\Phi[\Gamma]$, the presentation as a Coxeter group of the standard parabolic subgroup $\widehat{\W}_{\cX}\subset \widehat{\W}$ will be 
    \begin{equation}
        \label{wchi}
    \widehat{\W}_{\cX}=\langle \widehat{S}_{\cX}\;|\; \widehat{s}_{\beta}^2=1 \; \forall \widehat{s}_{\beta}\in \widehat{S}_{\cX},\; (\widehat{s}_{\beta}\widehat{s}_{\gamma})^{\widehat{m}_{\beta,\gamma}}=1 \; \forall \widehat{s}_{\beta},\widehat{s}_{\gamma}\in \widehat{S}_{\cX},\quad \widehat{s}_{\beta}\neq\widehat{s}_{\gamma}, \;\widehat{m}_{\beta,\gamma}\neq \infty   \rangle,
    \end{equation}
    where $\widehat{S}_{\cX}=\{\widehat{s}_{\beta}\;|\;\beta\in \cX\}$.\\
    \\
\noindent
We now relate the AP reflection subgroups of $\W[\Gamma]$ with the standard parabolic subgroups of $\W[\widehat{\Gamma}]$.
\begin{prop}\label{reflecrionstandard}
    Let $\cX\subset\Phi[\Gamma]$ be an AP set of roots, and let $(\W_{\cX},R_{\cX})$ be the associated Coxeter system as in Theorem \ref{refAP}. Then there is a group isomorphism
    \begin{align*}
        \mu: \widehat{\W}_{\cX}&\longrightarrow \W_{\cX}\\
              \widehat{s}_{\beta}&\longmapsto r_{\beta}.
    \end{align*}
\end{prop}
\noindent
The proof is simply the observation that the two groups have identical presentations with the same Coxeter matrix ($\widehat{m}_{\beta,\gamma}=m_{\beta,\gamma}$), as shown in Equations \ref{presrefsub} and \ref{wchi}. This also implies that the map $\mu$ preserves the length function in the Coxeter groups, that is a property that will be needed later. For all $\widehat{u}\in \widehat{\W}_{\cX}$, let $\mu(\widehat{u})=u\in \W_{\cX}$.
\begin{cor}\label{minimal}
    Let $\cY\subset\cX\subset \Phi[\Gamma]$, with $\cX$ AP. Suppose $\widehat{u}\in \widehat{\W}_{\cX}$ is $(\emptyset,\widehat{S}_{\cY})$-minimal. Then $(\W_{\cY},R_{\cY})$ is a standard parabolic subgroup of $\W_{\cX}$ isomorphic to $\widehat{\W}_{\cY}$ and $\mu(\widehat{u})=u\in \W_{\cX}$ is $(\emptyset,R_{\cY})$-minimal. 
\end{cor}
\begin{proof}
Denote by $l_{\cX}$ the length function in $\W_{\cX}$ with respect to $R_{\cX}$.
If $\cY\subset \cX$, then $\cY$ is also AP, and its associated reflection subgroup $\W_{\cY}$ has a presentation like in Equation \ref{presrefsub}, relative to the Coxeter system $(\W_{\cY},R_{\cY})$, with $R_{\cY}\subset R_{\cX}$. Thus, it is a standard parabolic subgroup of $\W_{\cX}$. It is clear that $\W_{\cY}\cong \widehat{\W}_{\cY}$ through $\mu$. Moreover, if $\widehat{u}\in \widehat{\W}_{\cX}$ is $(\emptyset,\widehat{S}_{\cY})$-minimal, it means that for all $\beta \in\cY$ we have $\widehat{l}(\widehat{u}\widehat{s}_{\beta})=\widehat{l}(\widehat{u})+1$, where $\widehat{l}$ denotes the length function in $\widehat{\W}_{\cX}$. By definition of $\mu$, for all $\beta \in \mathcal Y$, we have $\widehat{l}(\widehat{u})+1= \widehat{l} (\widehat{u} \widehat{s_\beta}) =  l_{\cX}(\mu(\widehat{u}\widehat{s}_{\beta}))=l_{\cX}(ur_{\beta})=l_{\cX}(u)+1$. This implies that $u\in \W_{\cX}$ is $(\emptyset, R_{\cY})$-minimal. 
\end{proof}

\begin{rmk}\label{warningreflectionsub}
    We have already noticed that if $\cX$ and $\cX'$ are two different subsets of $\Phi[\Gamma]$, they could still generate the same reflection subgroup $\langle R_{\cX}\rangle=\langle R_{\cX'}\rangle$. It suffices that one of the subsets of roots associated to $R_{\cX}$ is AP to have a presentation as in Equation \ref{presrefsub} for the reflection subgroup, that we call $\W_{\cX}$ and that is isomorphic to the standard parabolic subgroup $\widehat{\W}_{\cX}$ of $\widehat{\W}$ by Proposition \ref{reflecrionstandard}. But if $\cX\neq \cX'$, the two groups $\widehat{\W}_{\cX}$ and $\widehat{\W}_{\cX'}$ are always different, even if they are isomorphic.
\end{rmk}

\subsection{The Salvetti complex of \texorpdfstring{$\widehat{\Gamma}$}{TEXT} }\label{salbargam}

We recall that thanks to \cite{BellParThiel} the Artin group $\A[\widehat{\Gamma}]$ is isomorphic to the subgroup $\KVA[\Gamma]$ of $\VA[\Gamma]$ (Theorem \ref{preskva}). As we have seen in Section \ref{salgam}, for any Coxeter graph we can build the Salvetti complex, which is a CW-complex whose fundamental group is the associated Artin group. We will now construct the Salvetti complex for $\widehat{\Gamma}$, and to do so we will need to study the spherical-type standard parabolic subgroups of $\widehat{\W}:=\W[\widehat{\Gamma}]$. \\

\noindent
Set $\widehatSf=:\{\cX\subset \Phi[\Gamma]\;|\; \widehat{\W}_{\cX} \mbox{ is of spherical type}\}$ and $\widehatSf_k=\{\cX\in \widehatSf\,|\,|\cX|=k\}$, for $k\in \mathbb{N}$. According to Definition \ref{defsalcel}, we have that the space $\salbargam$ has a cell for each $\cX\in \widehatSf$. The following lemma gives a characterization of the sets of roots $\cX$ for which $\widehat{\W}_{\cX}$ is of
spherical type. 

\begin{lem}\label{chiapsphtyp}
    Let $\cX\subset \Phi[\Gamma]$ be a set of roots. Then $\widehat{\W}_{\cX}$ is of spherical type if and only if $\cX$ is AP and of spherical type.
\end{lem}
\begin{proof}
    Suppose first $\widehat{\W}_{\cX}$ is of spherical type. Then, its set of generators $\cX$ must be free of infinity, so for all $\beta,\gamma \in \cX$, it must be $\widehat{m}_{\beta,\gamma}\neq \infty$. As in Theorem 4.1 of \cite{BellParThiel}, we can build a formal set $\widehat{\Pi}:=\{\widehat{E}_{\beta}\;|\; \beta\in \Phi[\Gamma]\}$ and a vector space $\widehat{V}:=\bigoplus_{\beta\in \Phi[\Gamma]}\mathbb{R}\cdot \widehat{E}_{\beta}$ with basis $\widehat{\Pi}$. We can also introduce a bilinear symmetric form $\llangle\cdot,\cdot\rrangle:\widehat{V}\times \widehat{V}\longrightarrow \mathbb{R}$ by $\llangle\widehat{E}_{\beta},\widehat{E}_{\gamma}\rrangle:=-\cos{\left(\frac{\pi}{\widehat{m}_{\beta,\gamma}}\right)}$ if $\widehat{m}_{\beta,\gamma}\neq\infty$ and $\llangle\widehat{E}_{\beta},\widehat{E}_{\gamma}\rrangle:=-1$ otherwise. By Claim 2 in the same reference, we have that $\llangle\widehat{E}_{\beta},\widehat{E}_{\gamma}\rrangle=\langle \beta,\gamma\rangle$ if $\widehat{m}_{\beta,\gamma}\neq \infty$, that is our case for all $\beta,\gamma$ in $\cX$. Let $\widehat{V}_{\cX}:=\bigoplus_{\beta\in \cX}\mathbb{R}\cdot \widehat{E}_{\beta}$, and note that if $\widehat{\W}_{\cX}$ is of spherical type, then by Theorem \ref{finiteandscalproduct} we have that $\llangle\cdot,\cdot\rrangle|_{\widehat{V}_{\cX}}$ is positive definite.\\\\
    \noindent Suppose that the set $\cX$ is linearly dependent, i.e., there exist some coefficients $\lambda_{\beta}$ not all zero such that $\sum_{\beta\in \cX}\lambda_{\beta}\,\beta=0$. Set $\widehat{u}=\sum_{\beta\in \cX}\lambda_{\beta}\widehat{E}_{\beta}\in \widehat{V}_{\cX}$. Since the vectors $\widehat{E}_{\beta}, \beta \in \cX$ form a basis of $\widehat{V}_{\cX}$, we have that $\widehat{u}\neq 0$. Since the form $\llangle\cdot,\cdot\rrangle$ is positive definite on $\widehat{V}_{\cX}$, we get $\llangle\widehat{u},\widehat{u}\rrangle|_{\widehat{V}_{\cX}}>0$. On the other hand, using $\llangle \widehat{E}_{\beta},\widehat{E}_{\gamma}\rrangle=\langle\beta,\gamma\rangle$ for $\beta,\gamma\in \cX$, we obtain
    \[
    \llangle\widehat{u},\widehat{u}\rrangle|_{\widehat{V}_{\cX}}=\llangle \sum_{\beta\in \cX}\lambda_{\beta}\widehat{E}_{\beta},\sum_{\gamma\in \cX}\lambda_{\gamma}\widehat{E}_{\gamma}\rrangle|_{\widehat{V}_{\cX}}=\left \langle\sum_{\beta\in \cX}\lambda_{\beta}\,\beta, \sum_{\gamma\in \cX}\lambda_{\gamma}\,\gamma\right\rangle=\langle 0,\sum_{\gamma\in \cX}\lambda_{\gamma}\,\gamma\rangle=0;  \]
    which is a contradiction. Thus $\cX$ is a set of linearly independent vectors, which proves property (AP1). Furthermore, since $\widehat{\Gamma}_{\cX}$ is free of infinity, we have that for all $\beta,\gamma\in \cX$, $\widehat{m}_{\beta,\gamma}\neq \infty$, and by the definition of $\widehat{m}_{\beta,\gamma}$, we obtain property (AP2).
    Thus, $\cX$ is AP. Now, by Proposition \ref{reflecrionstandard}, the reflection subgroup associated to $\cX$ is isomorphic to $\widehat{\W}_{\cX}=\W[\widehat{\Gamma}_{\cX}]$, which is finite by hypothesis. Thus, $\cX$ must be AP and of spherical type.\\
    Now suppose that $\cX$ is AP. By Proposition \ref{reflecrionstandard} we have that the reflection subgroup associated to $\cX$ is isomorphic to $\widehat{\W}_{\cX}$. If $\cX$ is also of spherical type, then $\widehat{\W}_{\cX}$ is of spherical type.
\end{proof}

\noindent
We now know that the elements in $\widehatSf$ are in bijection with the subsets of roots $\cX\subset\Phi[\Gamma]$ that are AP and of spherical type. For each $\cX\in \widehatSf_k$, in $\salbargam$ there is a cell $\Delta^{k}(\cX)$ with a characteristic map $\phi^k_{\cX}: \mathbb{D}(\cX) \longrightarrow \Delta^{k}(\cX)$, where $\mathbb{D}(\cX)$ is a copy of the Coxeter polytope $\widehat{C}[\cX]$, identified via an isometry $f_{\cX}:\widehat{C}[\cX]\longrightarrow \mathbb{D}(\cX)$. We write $\widehat{C}[\cX]$ instead of $C[\cX]$ to emphasize that we are working with the Coxeter polytope of a spherical type standard parabolic subgroup of $\widehat{\W}=\W[\widehat{\Gamma}]$. This polytope $\widehat{C}[\cX]=conv\{\widehat{z}(\widehat{o}_{\cX})\;|\; \widehat{z}\in \widehat{\W}_{\cX}\}$ lives in the vector space $\widehat{V}_{\cX}=\bigoplus_{\beta\in \cX}\mathbb{R}\cdot \widehat{E}_{\beta}$, but we can transfer all the information in $V_{\cX}:=\bigoplus_{\beta\in \cX}\mathbb{R}\cdot \beta\subset V$. Recall that $\widehat{o}_{\cX}=\sum_{\beta\in \cX}\widehat{E}_{\beta,\cX}^*$, where $\widehat{\Pi}_{\cX}^*=\{\widehat{E}_{\beta,\cX}^*\,|\, \beta\in \cX\}$ is the dual basis of $\widehat{\Pi}_{\cX}=\{\widehat{E}_{\beta}\,|\,\beta\in \cX\}$ in $\widehat{V}_{\cX}$ with respect to the bilinear symmetric form $\llangle\cdot,\cdot\rrangle|_{\widehat{V}_{\cX}}$, which is a scalar product since $\widehat{\W}_{\cX}$ is finite. The faces of $\widehat{C}[\cX]$ are of the form $\widehat{F}(\widehat{u},\cY,\cX)$, with $\cY\subset \cX$, $\widehat{u}\in \widehat{\W}_{\cX}, \; (\emptyset, \widehat{S}_{\cY})$-minimal, and $\widehat{F}(\widehat{u},\cY,\cX)=conv\{\widehat{u}\widehat{z}(\widehat{o}_{\cX})\;|\; \widehat{z}\in \widehat{\W}_{\cY}\}$.\\
\\
As was shown in Subsection \ref{coxeterpolytopes}, we can construct the Coxeter polytope for any spherical type Coxeter group. In particular, if $\cX\subset \Phi[\Gamma]$ is of spherical type and AP, then the associated reflection subgroup $\langle R_{\cX}\rangle:=\W_{\cX}$ is of spherical type, and its simple root system is precisely $\cX$. Recall that we set $V_{\cX}=\bigoplus_{\beta\in \cX}\mathbb{R}\cdot \beta\subset V$. The group $\W_{\cX}$ acts on $V_{\cX}$ through the canonical faithful linear representation, that is a restriction of the representation $\rho:\W\longrightarrow GL(V)$. Since the group $\W_{\cX}$ is finite, the standard bilinear form in $V_{\cX}$ is a scalar product, and we can then identify $V_{\cX}$ with its dual $V_{\cX}^*$. We then define $o_{\cX}:=\sum_{\beta\in \cX}\beta_{\cX}^*$, where $\cX^*=\{\beta_{\cX}^*\,|\, \beta\in \cX\}$ is the dual basis of $\cX$ in the subspace $V_{\cX}$ with respect to the scalar product. Then, define
\begin{equation}
\label{coxpolrefsub}
C[\cX]=conv\{z(o_{\cX})\;|\;z\in \W_{\cX}\}\subset V_{\cX},
\end{equation}
that is the Coxeter polytope of the spherical type AP reflection subgroup $\W_{\cX}$. Consistent with the previous notation, the faces of $C[\cX]$ are $F(u,\cY,\cX)=conv\{uz(o_{\cX})\,|\,z\in \W_{\cY}\}$ with $\cY\subset \cX$, $u\in \W_{\cX}$, $u \; (\emptyset,R_{\cY})$-minimal, as in Theorem \ref{facescoxpol}. Additionally, $F(u,\cY,\cX)\subset F(u',\cY',\cX)$ holds if and only if $u\W_{\cY}\subset u'\W_{\cY'}$, when considered as cosets of the Coxeter group $\W_{\cX
}$.\\
\noindent In the following result we highlight that it is possible to have $C[\cX]=C[\cX']$, even if $\cX\neq \cX'$. 

\begin{prop}\label{diffXeqCoxpol}
Suppose that $\cX,\cX'$ are two sets of roots which are AP and of spherical type, such that $\cX'\neq \cX$ and $R_{\cX}=R_{\cX'}$. Then, as sets, $C[\cX]=C[\cX']$.
\end{prop}
\begin{proof}
    We have already noted that if  $R_{\cX}=R_{\cX'}$, then $\mathrm{M}_{\cX}=\mathrm{M}_{\cX'}$ as Coxeter matrices. Clearly, this also implies that the two vector spaces $V_{\cX}$ and $V_{\cX'}$ coincide, although, in general, $o_{\cX}\neq o_{\cX'}$. Consider now the two Coxeter polytopes in $V_{\cX}$ defined as follows:
    \begin{align*}
    &C[\cX]=conv\{z(o_{\cX})\;|\;z\in \W_{\cX}\},  &C[\cX']=conv\{z'(o_{\cX'})\;|\;z'\in \W_{\cX}\}.
\end{align*}
To prove that they are the same object, it suffices to show that the two sets $\{z(o_{\cX})\;|\;z\in \W_{\cX}\}$ and $\{z'(o_{\cX'})\;|\;z'\in \W_{\cX}\}$ coincide.
To this end, we show that there exists an element $w\in \W_{\cX}$ such that $w(o_{\cX})=o_{\cX'}$. Let $\mathcal{T}=\cX\setminus \cX'$. For all $\beta\in \mathcal{T}$, $-\beta\in \cX'$. We can then write $\cX=\mathcal{T}\sqcup\mathcal{V}$ and $\cX'=-\mathcal{T}\sqcup \mathcal{V}$, where $\mathcal{V}=\cX\cap \cX'$. The set $\mathcal{T}$ is AP and of spherical type, so there exists the longest element $w_{\mathcal{T}}$ in the Coxeter group $\W_{\mathcal{T}}\subset \W_{\cX}$. This element has the property that $w_{\mathcal{T}}(\mathcal{T})=-\mathcal{T}$ and that $w_{\mathcal{T}}^2=1$. Furthermore, we can write $V_{\cX}=V_{\mathcal{T}}\oplus V_{\mathcal{V}}$. \\
\\
\noindent
For $\cZ$ being one of the subsets $\mathcal{T}$, $-\mathcal{T}$, $\mathcal{V}$, $-\mathcal{V}$ of $\cX$, we define $\cZ^*=\{\beta_{\cZ}^*\,|\,\beta\in \cZ\}$
to be the dual basis to $\cZ$ in the subspace $V_{\cZ}= \langle \cZ \rangle \subset V$, with respect to the restriction
of the scalar product on $V$ to $V_{\cZ}$.\\
\\
\noindent Let $\beta$ be in $\mathcal{T}$ and $\gamma \in \mathcal{V}$. Since $\beta$ and $\gamma$ belong to $\cX$ and $\cX$ is AP, there is $w \in \W[\Gamma]$ such that $\beta=w(\alpha_s)$ and $\gamma=w(\alpha_t)$ for some standard generators $s,t\in S$, and hence $\langle \beta,\gamma \rangle=\langle \alpha_s,\alpha_t\rangle\leq 0$ by the definition of the bilinear form $\langle \cdot,\cdot\rangle$ and properties of the Coxeter matrix. Since $\beta$ and $\gamma$ belong to $\cX$ and $\cX$ is AP, we have that $\langle \beta,\gamma\rangle \leq 0$. Similarly, we get that $\langle -\beta,\gamma\rangle \leq 0$ because $-\beta$ and $\gamma$ belong to $\cX'$ and $\cX'$ is AP. Hence, $\langle \beta,\gamma\rangle = 0$, i.e. the roots in $\mathcal{T}$ are orthogonal to the roots in $\mathcal{V}$.
This directly implies the following properties:
\begin{itemize}
    \item $\beta_{\cX}^*=\beta_{\mathcal{T}}^*$ for all $\beta \in \mathcal{T}$,
    \item $\gamma_{\cX}^*=\gamma_{\mathcal{V}}^*$ for all $\gamma \in \mathcal{V}$,
    \item $t(x)=x$ for all $t\in \W_{\mathcal{T}}$ and $x\in V_{\mathcal{V}}$,
     \item $v(y)=y$ for all $v\in \W_{\mathcal{V}}$ and $y\in V_{\mathcal{T}}$,
     \item $\W_{\cX}=\W_{\mathcal{T}}\times \W_{\mathcal{V}}.$

\end{itemize}
We decompose $o_{\cX}$ as $o_{\cX}=\sum_{\beta\in \mathcal{T}}\beta^*_{\cX}+\sum_{\gamma\in \mathcal{V}}\gamma^*_{\cX}$, and similarly we do for $o_{\cX'}=\sum_{\beta'\in -\mathcal{T}}(\beta')^*_{\cX'}+\sum_{\gamma\in \mathcal{V}}\gamma^*_{\cX'}$. Based on the properties we just established, we get that  $o_{\cX}=\sum_{\beta\in \mathcal{T}}\beta^*_{\mathcal{T}}+\sum_{\gamma\in \mathcal{V}}\gamma^*_{\mathcal{V}}$ and, in the same way,  $o_{\cX'}=\sum_{\beta'\in -\mathcal{T}}(\beta')^*_{-\mathcal{T}}+\sum_{\gamma\in \mathcal{V}}\gamma^*_{\mathcal{V}}$. In particular, we get that $o_{\cX}=o_{\mathcal{T}}+o_{\mathcal{V}}$ and that  $o_{\cX'}=o_{-\mathcal{T}}+o_{\mathcal{V}}$ with $o_{\mathcal{T}},o_{-\mathcal{T}}\in V_{\mathcal{T}}$ and $o_{\mathcal{V}}\in V_{\mathcal{V}}$.  
Observe as well that: 
\[
w_{\mathcal{T}}({\mathcal{T}}^*)=\{-\beta_{\mathcal{T}}^*\,|\,\beta \in \mathcal{T}\}
= \{ (\beta')_{-\mathcal{T}}^* \mid \beta' \in -\mathcal{T} \} ={-\mathcal{T}}^*\,.
\]
This implies that $w_{\mathcal{T}}(o_{\mathcal{T}})=o_{-\mathcal{T}}$. Consequently, we have:
\begin{align*}
    w_{\mathcal{T}}(o_{\cX})=w_{\mathcal{T}}\left(\sum_{\beta \in \mathcal{T}}\beta_{\cX}^*\right)+ w_{\mathcal{T}}\left(\sum_{\gamma \in \mathcal{V}}\gamma_{\cX}^*\right)=w_{\mathcal{T}}(o_{\mathcal{T}})+ w_{\mathcal{T}}\left(o_{\mathcal{V}}\right)=
  o_{-\mathcal{T}}+ o_{\mathcal{V}}=o_{\cX'}.
\end{align*}
Thus, the two sets
\begin{align*}
    \{z(o_{\cX})\;|\;z\in \W_{\cX}\}=\{z'w_{\mathcal{T}}(o_{\cX})\;|\;z'\in \W_{\cX}\}=\{z'(o_{\cX'})\;|\;z'\in \W_{\cX'}\}
\end{align*}
are equal, and so are their convex hulls $C[\cX]=C[\cX']$. To be more precise, let us denote by 
\begin{align*}
    &C[\mathcal{T}]=conv\{t(o_{\mathcal{T}})\,|\, t\in \W_{\mathcal{T}}\}\subset V_{\mathcal{T}}, &C[\mathcal{V}]=conv\{v(o_{\mathcal{V}})\,|\, v\in \W_{\mathcal{V}}\}\subset V_{\mathcal{V}},
\end{align*}
the Coxeter polytopes associated with the AP spherical subgroups $\W_{\mathcal{T}}$ and $\W_{\mathcal{V}}$, respectively.
Since all the $r_{\beta}$'s with $\beta \in \mathcal{T}$ commute with all $r_{\gamma}$'s with $\gamma\in \mathcal{V}$, for all $z\in \W_{\cX}$ we can gather in each of its reduced expressions with respect to $R_{\cX}$ the generators in $R_{\mathcal{T}}$ and in $R_{\mathcal{V}}$. Namely, we can uniquely write $z$ as $z=tv=vt$ with $v\in \W_{\mathcal{V}}$ and $t\in \W_{\mathcal{T}}$. Computing $z(o_{\cX})$ we get \[z(o_{\cX})=tv(o_{\mathcal{T}}+o_{\mathcal{V}})=tv(o_{\mathcal{T}})+vt(o_{\mathcal{V}})=t(o_{\mathcal{T}})+v(o_{\mathcal{V}}).\]
Therefore, we obtain:
\begin{align*}
    C[\cX]=conv\{z(o_{\cX})\,|\,z\in \W_{\cX}\}=conv\{vt(o_{\cX})\,|\,t\in \W_{\mathcal{T}}, v\in \W_{\mathcal{V}}\}=\\conv\{t(o_{\mathcal{T}})+v(o_{\mathcal{V}})\,|\,t\in \W_{\mathcal{T}}, v\in \W_{\mathcal{V}}\}=C[\mathcal{T}]\times C[\mathcal{V}].
\end{align*}
Similarly, we find that $C[\cX']=C[-\mathcal{T}]\times C[\mathcal{V}]$. Note that: \[C[\mathcal{T}]=conv\{t(o_{\mathcal{T}})\,|\, t\in \W_{\mathcal{T}}\}=conv\{t'w_{\mathcal{T}}(o_{\mathcal{T}})\,|\, t'\in \W_{\mathcal{T}}\}=conv\{t'(o_{-\mathcal{T}})\,|\, t'\in \W_{-\mathcal{T}}\}=C[-\mathcal{T}].\]
From this description, we can more clearly conclude that 
\[C[\cX]=C[\mathcal{T}]\times C[\mathcal{V}]=C[-\mathcal{T}]\times C[\mathcal{V}]=C[\cX'].\]
\end{proof}

\noindent
Let us define $\Rf=\{\cX\subset\Phi[\Gamma]\;|\; \cX\mbox{ is AP and of spherical type}\}$. From Proposition \ref{chiapsphtyp}, we know that $\widehatSf=\Rf$, and for each $\cX \in \Rf$, we can construct the Coxeter polytope $C[\cX]$ as described earlier. We introduce the notation $\Rf_k$ for the set $\{\cX\in \Rf \,|\, |\cX|=k\}$. Due to Proposition \ref{diffXeqCoxpol}, different elements $\cX,\cX'$ of $\Rf$ might yield two copies of the same Coxeter polytope $C[\cX]$. However, as Remark \ref{orientedgesparabolic} highlights, even if $C[\cX]$ and $C[\cX']$ coincide as sets, \textbf{they may have different orientations along their edges.}

\begin{ex}
    Let $\cX=\{\beta,\gamma\}$ and $\cX'=\{-\beta,\gamma\}$, where both are AP and of spherical type. In this scenario, $\mathcal{T}=\{\beta\}$ and $\mathcal{V}=\{\gamma\}$. The Coxeter polytope $C[\cX]$ can be described as $C[\mathcal{T}]\times C[\mathcal{V}]$, which is the same as $C[\cX']=C[-\mathcal{T}]\times C[\mathcal{V}]$ when viewed as a set. The point $o_{\cX}$ is $\beta+\gamma=o_{\mathcal{T}}+o_{\mathcal{V}}$, whereas $o_{\cX'}$ is $-\beta+\gamma=o_{\mathcal{-T}}+o_{\mathcal{V}}$. The edge $F(\id,\beta,\cX)$ is oriented from $\beta+\gamma$ towards $r_{\beta}(\beta+\gamma)=-\beta+\gamma$. This edge is identical to $F(\id,-\beta,\cX')$, but in $C[\cX']$, it is oriented from $o_{\cX'}=-\beta+\gamma$ towards $r_{\beta}(-\beta+\gamma)=\beta+\gamma$. 
    \begin{figure}[ht]
        \centering
        \includegraphics[width=16cm]{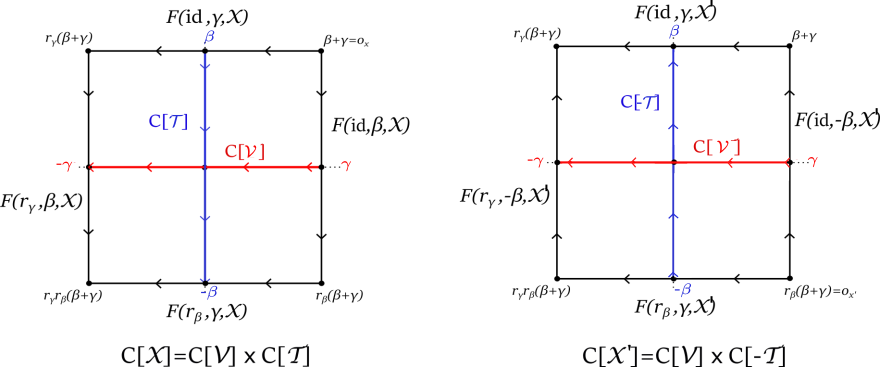}
        \caption{The two polytopes $C[\cX]$ and $C[\cX']$.}
        \label{productofcoxeterpolytopes}
    \end{figure}
    Refer to Figure \ref{productofcoxeterpolytopes}
     to see how the edges corresponding to the roots in $\mathcal{V}$ (in this case, just $\gamma$) preserve their orientation.
\end{ex}

\begin{rmk}\label{orientedgesparabolic}\textbf{Orientation of the edges of $C[\cX]$ and $C[\cX']$.}\\
Suppose $\cX$ and $\cX'$ are as in the statement of
Proposition \ref{diffXeqCoxpol}. The faces of the Coxeter polytope $C[\cX]$ are all of the form $F(u,\cY,\cX)$, with $u\in \W_{\cX}, \; (\emptyset, R_{\cY})$-minimal, as in Corollary \ref{minimal}. In particular, the edges of $C[\cX]$ are of the form $F(u,\beta,\cX)$ with $\beta\in \cX$ and $u$ being $(\emptyset,\{r_{\beta}\})$-minimal.
We use the description of $C[\cX]$ as $C[\mathcal{T}]\times C[\mathcal{V}]$, as described in the proof of Proposition \ref{diffXeqCoxpol}.\medskip\\
\noindent Consider now the edge $F(u,\beta,\cX)=conv\{uw(o_{\cX})\,\mid \,w\in \W_{\beta}=\{\id,r_{\beta}\}\}$, that is the segment $[u(o_{\cX}), ur_{\beta}(o_{\cX})]$, oriented from $u(o_{\cX})$ towards $ur_{\beta}(o_{\cX})$. We have shown that $w_{\mathcal{T}}$, which is the longest element in $\W_{\mathcal{T}}$, sends $o_{\cX}$ to $o_{\cX'}$. Since $w_{\mathcal{T}}$ has order two, we can also write $o_{\cX}=w_{\mathcal{T}}(o_{\cX'})$. Therefore, 
$F(u,\beta,\cX)$ is the segment $[uw_{\mathcal{T}}(o_{\cX'}),ur_{\beta}w_{\mathcal{T}}(o_{\cX'})]$. A classical result in the theory of Coxeter groups is that the longest element $w_{\mathcal{T}}$ satisfies $w_{\mathcal{T}}\mathcal{T}w_{\mathcal{T}}=\mathcal{T}$ (\cite[Lemma 4.6.1]{Davis2008}).\medskip \\
\noindent 
Set $\delta=w_{\mathcal{T}}(\beta)\in -\mathcal{T}$ and $u'=ur_{\beta}w_{\mathcal{T}}$. Since $w_{\mathcal{T}}^2=\id$, we have $r_{\delta}=r_{w_{\mathcal{T}}(\beta)}=w_{\mathcal{T}}r_{\beta}w_{\mathcal{T}}$, and hence $u'r_{\delta}=ur_{\beta}w_{\mathcal{T}}(w_{\mathcal{T}}r_{\beta}w_{\mathcal{T}})=uw_{\mathcal{T}}$. Therefore $F(u,\beta,\cX)=conv\{uw_{\mathcal{T}}(o_{\cX'}), ur_{\beta}w_{\mathcal{T}}(o_{\cX'})\}=conv\{u'r_{\delta}(o_{\cX'}),u'(o_{\cX'})\}=F(u',\delta,\cX')$ as unoriented segments. However, the orientation induced from $C[\cX]$ is $uw_{\mathcal{
T}}(o_{\cX'})\longrightarrow ur_{\beta}w_{\mathcal{T}}(o_{\cX'})$, whereas the orientation induced from $C[\cX']$ is $u'(o_{\cX'})\longrightarrow u'r_{\delta }(o_{\cX'})$, that is, the opposite orientation. Thus the same geometric edge has opposite orientations in $C[\cX]$ and $C[\cX']$.\medskip \\
\noindent On the other hand, if $\gamma \in \mathcal{V}$, then $F(u,\gamma,\cX)=conv\{uw(o_{\cX})\mid w\in \W_{\gamma}=\{\id,r_{\gamma}\}\}$ is the segment $[u(o_{\cX}),ur_{\gamma}(o_{\cX})]$, oriented from $u(o_{\cX})$ towards $ur_{\gamma}(o_{\cX})$. As before, we can replace $o_{\cX}$ by $w_{\mathcal{T}}(o_{\cX'})$. Observe that in the proof of Proposition \ref{diffXeqCoxpol}, we also show that all $\gamma\in \mathcal{V}$ are orthogonal to all $\beta\in \mathcal{T}$. Specifically, the reflection $r_{\gamma}$ commutes with all elements of $\W_{\mathcal{T}}$, and in particular it commutes with $w_{\mathcal{T}}$. Therefore, the edge $F(u,\gamma,\cX)$ is equal to $F(u',\gamma,\cX')=[uw_{\mathcal{T}}(o_{\cX'}),uw_{\mathcal{T}}r_{\gamma}(o_{\cX'})]$, where $u'=uw_{\mathcal{T}}$. This segment is oriented from $u(o_{\cX})=u'(o_{\cX'})$ towards $ur_{\gamma}(o_{\cX})=u'r_{\gamma}(o_{\cX'})$ in both $C[\cX]$ and $C[\cX']$.
\end{rmk}

\begin{rmk}\label{coxpolitopeofaparabolic}
   \textbf{The Coxeter polytope of a parabolic subset of roots.} If $\cX=w(\Pi_X)$ with $w\in \W[\Gamma]$, $X\subset S$, namely if $\cX$ is parabolic, then the Coxeter polytope $C[\cX]$ defined earlier is isometric to the classical polytope $C[X]$ studied in Subsection \ref{coxeterpolytopes}. The groups $\W_{\cX}=w\W_{X}w^{-1}$ and $\W_{X}$ are isomorphic through conjugation. The action of the element $w$ on $V_{X}$ gives $w(V_X)=V_{\cX}$, and since it preserves the scalar product, we get $w(o_{X})=o_{\cX}$. This implies that $w$ sends isometrically $C[X]$ onto $C[\cX]$, as explained below. For all $z\in \W_{X}$, we have: 
    \begin{align*}  w:V_{X}=\bigoplus_{s\in X}\mathbb{R}\cdot \alpha_s&\longrightarrow V_{w(\Pi_X)}=V_{\cX}=\bigoplus_{s\in X}\mathbb{R}\cdot w(\alpha_s),\\
        C[X]&\longrightarrow C[\cX],\\
        z(o_{X})&\longmapsto wzw^{-1}w(o_{X})=wzw^{-1}(o_{\cX}).
    \end{align*}
   \noindent Thus, we conclude that $\widehat{C}[\cX]$ and $C[\cX]$ are the Coxeter polytopes associated with two isomorphic groups.
\end{rmk}
\noindent
 The next step is to describe the linear isometry that maps one polytope to the other. This will be done in Proposition \ref{isomhatnohat}.
 \begin{ex}
An example of Coxeter polytope of a parabolic subset of roots is illustrated in Figure \ref{fig:2coxpolparabolic}, where $\Gamma=A_2$, $S=\{(1\;2), (2\;3)\}$ and $\Phi[A_2]$ is as in Example \ref{exparabolicsofSn}.

\begin{figure}[h!]
    \centering
    \includegraphics[width=12cm]{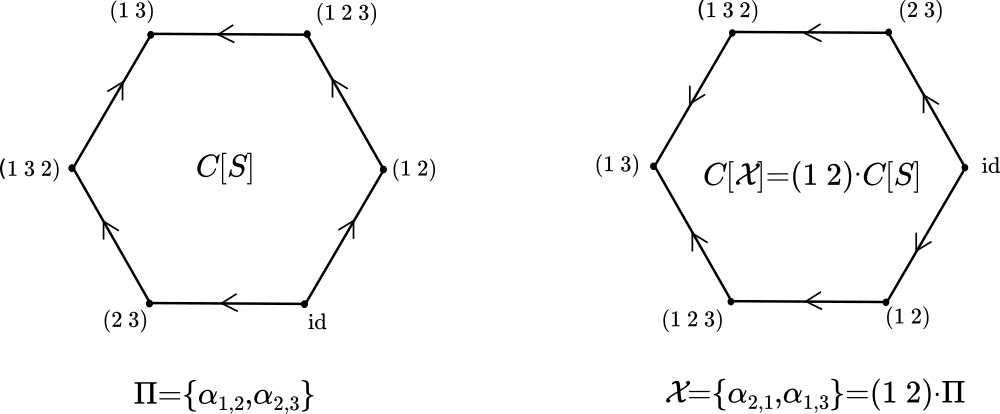}
    \caption{The Coxeter polytope of $\mathfrak{S}_3$ and $C[\cX]$ for $\cX=\{\alpha_{2,1},\alpha_{1,3}\}$.}
    \label{fig:2coxpolparabolic}
\end{figure}
\end{ex}

\begin{rmk}\label{samegroupdiffcX}
Let $\Gamma$ be a Coxeter graph. Observe that it could happen that $\cX,\cX'\subset \Phi[\Gamma]$ are both AP and of spherical type such that $\langle R_{\cX}\rangle=\langle R_{\cX'}\rangle$, but $\cX\neq \cX'$ and $R_{\cX}\neq R_{\cX'}$. For instance, consider $\Gamma=A_2$ as in Example \ref{exparabolicsofSn}. If we set $\cX=(1\;2)\cdot \{\alpha_{1,2},\alpha_{2,3}\}=\{\alpha_{2,1},\alpha_{1,3}\}$, then $\cX$ is parabolic and then clearly almost parabolic. The reflection group $\langle R_{\cX}\rangle$ coincides with the entire symmetric group $\mathfrak{S}_3$, which can also be seen as the reflection group $\langle R_{\cX'}\rangle$, with $\cX'=(1\;3)\cdot\Pi=\{\alpha_{3,2},\alpha_{2,1}\}$.
However, $R_{\cX}=\{(1\;2),(1\;3)\}\neq R_{\cX'}=\{(3\;2),(2\;1)\}$. In this case,\\ $C[\cX]=(1\;2)\cdot C[S]$ is the image of $C[S]$ under the action of the transposition $(1\;2)$, while $C[\cX']=(1\;3)\cdot C[S]$ is the image of $C[S]$ under the action of $(1\;3)$. The two Coxeter polytopes $C[\cX]$ and $C[\cX']$ are illustrated in Figure \ref{fig:2coxpol}.

\begin{figure}[h!]
\centering\includegraphics[width=12cm]{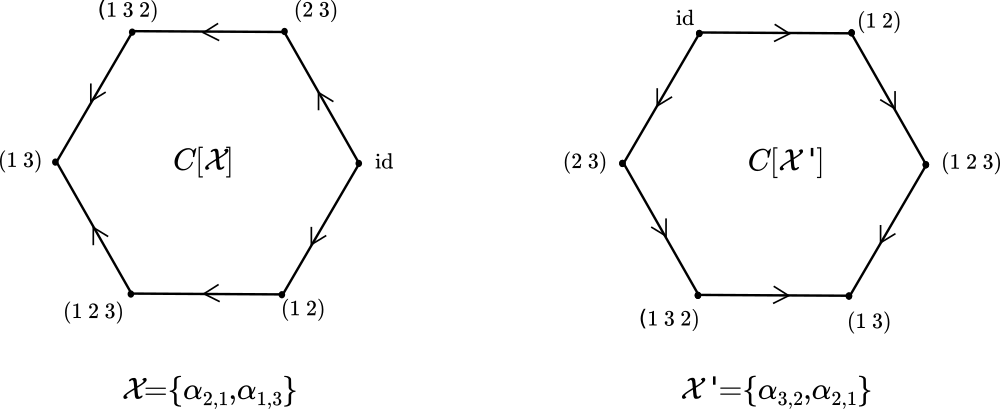}
    \caption{The Coxeter polytopes of $\cX=\{\alpha_{2,1},\alpha_{1,3}\}$ and $\cX'=\{\alpha_{3,2},\alpha_{2,1}\}$.}
    \label{fig:2coxpol}
\end{figure} 
\end{rmk}

\begin{prop}\label{isomhatnohat}
Consider $\cX\subset \Phi[\Gamma]$ AP and of spherical type.
    Then, there is a linear isomorphism $f:\widehat{V}_{\cX}\longrightarrow V_{\cX}$ that sends isometrically $\widehat{C}[\cX]$ to $C[\cX]$. Moreover, $f$ is equivariant under the action of $\widehat{\W}_{\cX}$ in the sense that, for all $\widehat w \in \widehat{\W}_{\mathcal X}$ and all $\widehat x \in \widehat V_{\mathcal X}$, the isomorphism satisfies $f (\widehat w (\widehat x)) = \mu (\widehat w) \big( f (\widehat x) \big)$,  
and it sends $\widehat{o}_{\cX}$ to $o_{\cX}$.
\end{prop}
\begin{proof}
Define $f$ by:
\begin{align*}
    f:\widehat{V}_{\cX}&\longrightarrow V_{\cX}\\
    \widehat{E}_{\beta}&\longmapsto \beta,
\end{align*}
which is clearly an isomorphism since it maps a basis to a basis. Let $\mu(\widehat{z})=z\in \W_{\cX}$, with $\mu$ as in Proposition \ref{reflecrionstandard}. Now, we aim to show that $f(\widehat{C}[\cX])=C[\cX]$, which means that for all $\widehat{z}\in \widehat{\W}_{\cX}$, we need to prove that $f(\widehat{z}(\widehat{o}_{\cX}))=\mu(\widehat{z})(f(\widehat{o}_{\cX}))=z({o}_{\cX})$. This requires two claims.\\
\\
\textit{Claim 1.} The linear map $f$ preserves the scalar product. In particular, $f(\widehat{o}_{\cX})={o}_{\cX}$.\\
\textit{Proof of Claim 1.} Observe that for $\beta,\gamma\in \cX$, we have $\langle f(\widehat{E}_{\beta}),f(\widehat{E}_{\gamma})\rangle=\langle\beta,\gamma\rangle=\llangle \widehat{E}_{\beta},\widehat{E}_{\gamma}\rrangle$. Thus,
for any $\widehat{x},\widehat{y}\in \widehat{V}_{\cX}$, it holds that $\llangle \widehat{x},\widehat{y}\rrangle=\langle f(\widehat{x}),f(\widehat{y})\rangle$. In particular, for all $\beta \in \cX$, $f(\widehat{E}_{\beta,\cX}^{*})=\beta_{\cX}^*$ and this concludes the proof of Claim 1.\\
\\
\textit{Claim 2.} 
    For all $\widehat{z} \in \widehat{\W}_{\cX}$ and for all $ \widehat{x} \in \widehat{V}_{\cX}$, we have $f(\widehat{z}(\widehat{x}))=z(f(\widehat{x}))$. Namely, $f$ is equivariant under the action of $\widehat{\W}_{\cX}$.\\
\textit{Proof of Claim 2.} 
Since we know by Claim 1 that $f$ preserves the scalar product, it suffices to show the second claim for the generators of $\widehat{\W}_{\cX}$. Take $\widehat{z}=\widehat{s}_{\beta}$ for some $\beta\in \cX$. Then, we compute:
    \begin{align*}
    f(\widehat{s}_{\beta}(\widehat{x}))= f\left(\widehat{x}-2\llangle \widehat{x},\widehat{E}_{\beta}\rrangle \widehat{E}_{\beta}\right)=
     f(\widehat{x})-2\llangle \widehat{x},\widehat{E}_{\beta}\rrangle f\left(\widehat{E}_{\beta}\right)=
     r_{\beta}(f(\widehat{x}))=\mu(\widehat{s}_{\beta})(f(\widehat{x})).
\end{align*} 
Thus, the assertion holds true for all $\widehat{z}\in \widehat{\W}_{\cX}$ and $\widehat{x}\in \widehat{V}_{\cX}$, completing the proof of Claim 2.\\
\\
By combining the two claims, we conclude that the linear isometry $f$ sends the polytope $\widehat{C}[\cX]$ to $C[\cX]$ as follows:
\begin{align*}
    f: \widehat{C}[\cX]&\xrightarrow[\quad]{}C[\cX]\\
    \widehat{z}(\widehat{o}_{\cX})&\longmapsto z(o_{\cX}).
\end{align*}
\end{proof}
\begin{rmk}
    The map $f$, in particular, sends the face $\widehat{F}(\widehat{u},\cY,\cX)$ to the face $F(u,\cY,\cX)$, where $u=\mu(\widehat{u})\in \W_{\cX}\;$ is $ (\emptyset,R_{\cY})$-minimal.
\end{rmk}
\noindent
We know that in $\salbargam$ there is a cell $\Delta^k(\cX)$ of dimension $k$ for every $\cX
\subset \Phi[\Gamma]$ such that $|\cX|=k\leq n_{sph}(\widehat{\Gamma})$, and the standard parabolic subgroup $\widehat{\W}_{\cX}$ is of spherical type. As we have seen, this set of roots
$\cX\subset\Phi[\Gamma]$ must be AP and of spherical type. Thus, the interior of the cell $\Delta^k(\cX)$ is homeomorphic to the interior of a copy $\mathbb{D}(\cX)$ of the Coxeter polytope $\widehat{C}[\cX]$ via a characteristic map $\widehat{\phi}_{\cX}^k$. Thanks to Proposition \ref{isomhatnohat} we can take $C[\cX]$ instead of $\widehat{C}[\cX]$.
This allows us to define a CW-complex $\Sigma(\Gamma)=\salbargam$ whose fundamental group is $\KVA[\Gamma]$, without explicitly referring to the graph $\widehat{\Gamma}$. Recall that $\Rf=\{\cX\subset \Phi[\Gamma]\;|\; \cX \mbox{ is AP and of spherical type}\}$ and $\Rf_k=\{\cX\in \Rf \,|\, |\cX|=k\}$.

\begin{defn}\label{sigmadef}
Let $\Gamma$ be a Coxeter graph, and $\Phi[\Gamma]$ be its root system. Consider the sets $\Rf$ and $\Rf_k$ as previously defined. The CW-complex $\Sigma(\Gamma)$ is constructed as follows.\\
\\
    \textbf{Description of the 0-skeleton.} The 0-skeleton $\Sigma^{0}(\Gamma)$ consists of a single point, denoted by $\sigma^0$.\\
    \\
    \textbf{Description of the 1-skeleton.} For every root $\beta\in \Phi[\Gamma]$, we take a copy $\mathbb{A}(\beta)$ of the Coxeter polytope $C[\beta]$, defined as $conv\{z(o_{\beta})\;|\;z\in \W_{\beta}=\{\id,r_{\beta}\}\}$. This polytope is isometric to $C[\beta]$ via the isometry $f_{\beta}:C[\beta]\longrightarrow \mathbb{A}(\beta)$. The basepoint is $o_{\beta}=\beta_{\beta}^*=\beta$, so $C[\beta]$ is simply the segment $[-\beta,\beta]$, naturally oriented from $\beta$ to $-\beta$. This orientation is transmitted to $\mathbb{A}(\beta)$ through the isometry $f_{\beta}$. Even though $C[\beta]=C[-\beta]$, we have $\mathbb{A}(-\beta)\neq \mathbb{A}(\beta)$. Then, for any $\beta\in \Phi[\Gamma]$, we define a map $\varphi^1_{\beta}:\partial \mathbb{A}(\beta)\longrightarrow \Sigma^0(\Gamma)$ by setting $\varphi^1_{\beta}(f_{\beta}(\beta))=\varphi^1_{\beta}(f_{\beta}(-\beta))=\sigma^0$. The 1-skeleton $\Sigma^1(\Gamma)$ is defined as the disjoint union of $\{\sigma^0\}$ and all the cells $\mathbb{A}(\beta)$, under the identifications that attach $x$ to $\varphi^1_{\beta}(x)$ for all $\beta\in \Phi[\Gamma]$ and for all $x\in \partial\mathbb{A}(\beta)$. Each inclusion of $\mathbb{A}(\beta)$ in $\{\sigma^0\}\sqcup (\bigsqcup_{\beta'\in \Phi[\Gamma]}\mathbb{A}(\beta'))$, composed with the quotient projection onto $\Sigma^1(\Gamma)$, gives a characteristic map $\phi^1_{\beta}: \mathbb{A}(\beta)\longrightarrow \Sigma^1(\Gamma)$ whose image is denoted by $\Delta^1(\beta)$. The map $\phi^1_{\beta}$ extends $\varphi^1_{\beta}$ and it is a homeomorphism from the interior of the segment $\mathbb{A}(\beta)$ to $\Delta^1(\beta)\,\backslash\,\{\sigma^0\}$.\\
    \\
    \textbf{Description of the $k$-skeleton.} Now suppose that the $(k-1)$-skeleton of $\Sigma(\Gamma)$ is defined for $k\geq 2$. For each set $\cY\subset \Phi[\Gamma]$ AP and of spherical type of rank $h\leq k-1$, we have a copy $\mathbb{D}(\cY)$ of the Coxeter polytope $C[\cY]$, identified with $C[\cY]$ via an isometry $f_{\cY}:C[\cY]\longrightarrow \mathbb{D}(\cY)$. There is an $h$-cell $\Delta^h(\cY)$ in $\Sigma^{k-1}(\Gamma)$, 
    which is the image of a continuous characteristic map $\phi^h_Y: \mathbb{D}(\cY)\longrightarrow \Delta^h(\cY)$ establishing a homeomorphism between the interior of $\mathbb{D}(\cY)$ and $\Delta^h(\cY)\,\backslash\, \partial\Delta^h(\cY)$. If $|\cY|=1$, then there exists $\beta\in \Phi[\Gamma]$ such that $\cY=\{\beta\}$, and we assume $\mathbb{D}(\cY)=\mathbb{A}(\beta)$. If $\cY=\emptyset$, then we adopt the following conventions: 
    \[
    V_{\emptyset}=\{0\},\qquad \W_{\emptyset}=\{\id\},\qquad o_{\emptyset}=0,\qquad C[\emptyset]=\{0\}.
    \]
    \noindent We also set $\mathbb{D}(\emptyset)=\{0\}$. Take $f_{\emptyset}:C[\emptyset]\longrightarrow \mathbb{D}(\emptyset)$ to be the identity map, and define $\varphi_{\emptyset}^0:\mathbb{D}(\emptyset)\longrightarrow\Sigma^0(\Gamma)$ to be the map sending $0$ to the unique vertex $\sigma^0$. Thus $\Delta^0(\emptyset)=\{\sigma^0\}$.\\
    \\
    
\noindent Now, for each $\cX\in \Rf_k$, take a copy $\mathbb{D}(\cX)$ of the Coxeter polytope $C[\cX]$, identified with $C[\cX]$ via an isometry $f_{\cX}:C[\cX]\longrightarrow \mathbb{D}(\cX)$. For any subset $\cY\subset \cX$ with $|\cY|=h\leq k-1$, and each $u\in \W_{\cX}$ a $(\emptyset,R_{\cY})$-minimal element, define $\xi^k_{u,\cY,\cX}:F(u,\cY,\cX)\longrightarrow \Sigma^{k-1}(\Gamma)$
     as the following composition of maps:
\[\xi^k_{u,\cY,\cX}=\phi^h_{\cY}\circ f_{\cY} \circ \tau(o_{\cY}-o_{\cX})\circ u^{-1}.\]
Recall that if $F(u,\cY,\cX)$ is a face of $C[\cX]$ with $\cY$ and $u$ as before, the faces of $F(u,\cY,\cX)$ are of the form $F(uv,\cZ,\cX)=conv\{uvw(o_{\cX})\,|\,w\in \W_{\cZ}\}$ with $\cZ\subset \cY$ and $v\in \W_{\cY}$, $(\emptyset,R_{\cZ})$-minimal.
   \begin{lem}\label{goodefxi}
    Let $\cX,\cY,\cZ \in \Rf$ be such that $\cZ\subset \cY\subset \cX$, and $|\cX|=k,|\cY|=h,|\cZ|=l$. Take $u\in \W_{\cX}$, $u$ $(\emptyset,R_{\cY})$-minimal and $v\in \W_{\cY}$, $v$ $(\emptyset,R_{\cZ})$-minimal. We have that $F(uv,\cZ,\cX)\subset F(u,\cY,\cX)\subset C[\cX]$. Then, the restriction of $\xi^k_{u,\cY,\cX}$ to $F(uv,\cZ,\cX)$ coincides with $\xi^k_{uv,\cZ,\cX}$. 
\end{lem}
\noindent
The proof of this lemma is identical to that of Lemma \ref{varphiiscontinuous}. Thus, we can define a continuous map $\xi^k_{\cX}:\partial C[\cX]\longrightarrow \Sigma^{k-1}(\Gamma)$ as follows. Let $x\in \partial C[\cX]$. We choose $\cY\subset \cX$ and $u\in \W_{\cX}$ an $(\emptyset, R_{\cY})$-minimal element such that $x\in F(u,\cY,\cX)$ and we set $\xi^k_{\cX}(x)=\xi^k_{u,\cY,\cX}(x)$. Thanks to Lemma \ref{goodefxi}, this definition does not depend on the choice of $F(u,\cY,\cX)$, and the map $\xi^k_{\cX}$ is continuous. Now, define $\varphi^k_{\cX}:\partial \mathbb{D}(\cX)\longrightarrow \Sigma^{k-1}(\Gamma)$ by setting
\[
\varphi^k_{\cX}=\xi^k_{\cX}\circ (f^{-1}_{\cX})|_{\partial \mathbb{D}(\cX)}.
\]

 \noindent   
 The map $\varphi^{k}_{\cX}$ is continuous. Now for all $\cX\in \Rf_k$ and for all $x\in \partial \mathbb{D}(\cX)$, we write $x\sim \varphi^k_{\cX}(x)$. The $k$-skeleton of the cell complex is defined as: 
    \[
    \Sigma^{k}(\Gamma):=\bigslant{\left(\Sigma^{k-1}(\Gamma)\bigsqcup \left( \bigsqcup_{\cX \in \Rf_k}\mathbb{D}(\cX)\right)\right)}{\sim}.
    \]
    For each $\cX\in \Rf_k$, the cell $\mathbb{D}(\cX)$ has a natural characteristic map $\phi^k_{\cX}$ to $\Sigma^k(\Gamma)$, whose image is denoted by $\Delta^k(\cX)$. This $\phi^k_{\cX}$ is the composition of the inclusion of $\mathbb{D}(\cX)$ into the disjoint union and the quotient projection. \\
    \\
    \textbf{Description of the complex.} We set $\Sigma(\Gamma)=\bigcup_{k=0}^{\infty}\Sigma^{k}(\Gamma)$, endowed with the weak topology.
\end{defn}
\noindent
Indeed $\salbargam$ can be defined for any Coxeter graph $\Gamma$ without explicitly referring to the the graph $\widehat{\Gamma}$. In fact, we just used the spherical-type AP reflection subgroups of $\W[\Gamma]$.

\noindent
Since $\Sigma(\Gamma)=\salbargam$, we have the following result.
\begin{prop}
    Let $\Gamma$ be a Coxeter graph. Then the fundamental group $\pi_1(\Sigma(\Gamma))$ is isomorphic to $\KVA[\Gamma]$.
\end{prop}
\noindent
\begin{proof}
We aim to establish the following chain of equalities:
\[\pi_1(\Sigma(\Gamma))=\pi_1(\salbargam)\cong \A[\widehat{\Gamma}]\cong \KVA[\Gamma].\]
\noindent While the first equality holds by definition, the last one follows from Theorem \ref{preskva}, which asserts that $\KVA[\Gamma]\cong \A[\widehat{\Gamma}]$. To show the isomorphism in the center, it suffices to apply Theorem \ref{fundsalvetti} to $\Sigma(\Gamma)=\salbargam$. However, Theorem \ref{fundsalvetti} is classically stated and proven for Coxeter groups of finite rank, which is not the case for $\widehat{\Gamma}$. For this reason, we give a direct analysis of $\salbargam=\Sigma(\Gamma)$. The complex $\Sigma(\Gamma)$ has one 0-cell and one oriented 1-cell $\Delta^1(\beta)$ for every $\beta\in \Phi[\Gamma]$. Thus, the fundamental group of the 1-skeleton $\Sigma^1(\Gamma)$ is the free group on the root system $\Phi[\Gamma]$, where we denote each generator by $\widehat{\delta}_{\beta}$. The 2-cells are indexed by almost parabolic subsets $\cX=\{\beta,\gamma\}$, which are parabolic by Remark \ref{card2parabolic}. The 2-cells of $\Sigma(\Gamma)$ are thus indexed by pairs $\{\beta=w(\alpha_s),\gamma=w(\alpha_t)\}$ with $s,t\in S$, $w\in \W[\Gamma]$, and $\widehat{m}_{\beta,\gamma}=m_{s,t}<\infty$. The attaching map of the corresponding 2-cell is the usual $2m_{s,t}$-gon and imposes precisely the Artin relation $\Prod_R(\widehat{\delta}_\beta, \widehat{\delta}_\gamma;\widehat{m}_{\beta,\gamma})=\Prod_R(\widehat{\delta}_\gamma,\widehat{\delta}_\beta;\widehat{m}_{\beta,\gamma})$. There are no 2-cells for pairs with $\widehat{m}_{\beta,\gamma}=\infty$. Since attaching cells of dimension at least $3$ does not change the fundamental group, the cellular presentation of $\pi_1(\Sigma(\Gamma))=\pi_1(\salbargam)$ is exactly the Artin presentation of $\A[\widehat{\Gamma}]$.

\end{proof}

\noindent
To summarize, the goal of this subsection was to construct the Salvetti complex for the Coxeter graph $\widehat{\Gamma}$, thereby creating a space with the fundamental group isomorphic to $\A[\widehat{\Gamma}]=\KVA[\Gamma]$. This process involved attaching a cell for each spherical-type parabolic subgroup of $\widehat{\W}$. We discovered that these subgroups could be interpreted as specific (spherical-type) reflection subgroups of $\W[\Gamma]$, termed AP spherical-type reflection subgroups. This insight allowed us to re-frame the entire CW-complex construction in terms of AP reflection subgroups of $\W[\Gamma]$, sidestepping the complexities associated with $\widehat{\Gamma}$ and $\widehat{\W}=\W[\widehat{\Gamma}]$. Consequently, we derived a cell complex $\Sigma(\Gamma)$ such that $\pi_1(\Sigma(\Gamma))\cong \KVA[\Gamma]$ for every Coxeter graph $\Gamma$.

\section{The BEER space \texorpdfstring{$\Omega(\Gamma)$}{TEXT}}\label{BEERspace}
We have just seen that given $\Gamma$ any Coxeter graph, we can build a CW-complex $\Sigma(\Gamma)$ (that coincides with the Salvetti complex $\salbargam$) whose fundamental group is the kernel of the homomorphism $\pi_K:\VA[\Gamma]\longrightarrow \W[\Gamma]$ defined in Subsection \ref{virtualartingroups}. The aim now is to define a cell complex whose fundamental group is the kernel $\PVA[\Gamma]$ of the homomorphism $\pi_P:\VA[\Gamma]\longrightarrow \W[\Gamma]$. This space will be a generalization to all Coxeter graphs of the space already described in \cite{BEER} for the case $\Gamma=A_{n-1}$. We will call this space (and our generalization) BEER space, and we will denote it by $\Omega(\Gamma)$. \\
\\
When the graph $\Gamma$ is $A_{n-1}$, the associated Coxeter group $\W$ is the symmetric group $\Sn$, the Artin group $\A[\Gamma]$ is the braid group on $n$ strands, and $\VA[A_{n-1}]$ is $\VB_n$, the virtual braid group on $n$ strands. Hence, the group $\PVA[A_{n-1}]$ is the pure virtual braid group on $n$ strands. The Coxeter polytope $C[S]$ is the permutohedron. The authors in \cite{BEER} define $\Omega(A_{n-1}):=\Omega_n$ to be a quotient of $C[S]\times \Sn$ up to an equivalence relation that involves the opposite faces of the permutohedron. They prove that $\pi_1(\Omega_n)\cong \PVB_n$, and they also claim that the space is locally CAT(0), and then aspherical. Unfortunately, the result of local non-positive curvature is false. We prove in Subsection \ref{originalBEERspace} the following result: 
\begin{reptheorem}{FalseCat0}
The complex $\Omega_3$ is not locally CAT(0).
\end{reptheorem}

\noindent
The tool of local non-positive curvature cannot be used to imply asphericity. The question now is whether or not $\Omega_n$ (and its generalization to all Coxeter graphs $\Gamma$) is aspherical. When $\Gamma$ is of spherical type or of affine type, we prove in Section \ref{commoncovering}, Corollaries \ref{beeraspht} and \ref{beerasphtAFF}, that the answer is yes.

\subsection{The complex \texorpdfstring{$\Omega_n$}{TEXT} and the failure of the CAT(0) property}\label{originalBEERspace}
In this subsection, we present the definition of $\Omega_{n}$ as outlined in \cite[Section 8.3]{BEER}. Since our aim is to extend this construction to encompass all Coxeter graphs, we delve into further details beyond the scope of the original exposition, tailoring the concepts to suit the context of this study. Notably, we furnish a detailed proof of Theorem \ref{fundoriginalbeer} that states $\pi_1(\Omega_n)=\PVB_n$, alongside the proof that this space is not locally CAT(0).\\
\\
\noindent
For $\Gamma=A_{n-1}$ and $S=\{s_1,\ldots,s_{n-1}\}$, the set of simple transpositions $s_i=(i\;i+1)$ for $i\in\{1,\ldots,n-1\}$, the Coxeter group is the familiar symmetric group $\Sn$ and the Coxeter polytope $C[S]$ is the permutohedron. We examine its explicit description following the construction of Subsection \ref{coxeterpolytopes}. 
Let $U=\mathbb{R}^n$ with the standard scalar product and $\left\{e_1, \ldots , e_n\right\}$ as the canonical basis. Denote now by $V$ the hyperplane of $U$ of equation $x_1+ \cdots + x_n=0$ and for $i \in \left\{1, \ldots , n-1\right\}$, set $\alpha_i=\frac{e_i-e_{i+1}}{\sqrt{2}}$. Then, $\Pi=\left\{\alpha_1, \ldots, \alpha_{n-1}\right\}$ is the set of simple roots of $\Sn$ and a basis of $V$.
Take $\Pi^*=\left\{\alpha_1^*, \ldots, \alpha_{n-1}^*\right\}$ the dual basis of $\Pi$ with respect to the standard scalar
product on $U$, restricted to $V$. It is easy to check that for every $k \in \left\{1, \ldots, n-1\right\}$ we can describe 

\[
\alpha_k^*=\frac{\sqrt{2}\left(n-k\right)}{n}\sum_{j=1}^{k}e_j - \frac{\sqrt{2} k}{n}\sum_{j=k+1}^{n} e_j.
\]
In fact, for all $i,j\in \{1,\ldots,n-1\}$, we have
\[ \left\langle \alpha_i^*,\alpha_j\right\rangle=\begin{cases}
\frac{n-i}{n}-\frac{n-i}{n}=0 \mbox{  if $j<i$}\\
\frac{n-i}{n}+\frac{i}{n}=1 \mbox{  if $i=j$}\\
-\frac{i}{n} + \frac{i}{n}=0 \mbox{   if $i<j$.}
\end{cases}\]
Thus, in this setting, we have $o=\sum_{k=1}^{n-1} \alpha_k^*=\sum_{j=1}^{n}\frac{n+1-2j}{\sqrt{2}}e_j$ and $C[S]=conv\{w(o)\,|\,w\in \Sn\}$.
\begin{ex}
For $n=3$, the basis of $V$ is $\alpha_1=\left(\frac{1}{\sqrt{2}},-\frac{1}{\sqrt{2}},0\right), \alpha_2=\left(0,\frac{1}{\sqrt{2}},-\frac{1}{\sqrt{2}}\right)$ and its dual basis is $\alpha_1^*=\left(\frac{2\sqrt{2}}{3}, -\frac{\sqrt{2}}{3},  -\frac{\sqrt{2}}{3}\right), \alpha_2^*=\left( \frac{\sqrt{2}}{3}, \frac{\sqrt{2}}{3},  -\frac{2\sqrt{2}}{3}\right)$, with $o=\left(\sqrt{2},0,-\sqrt{2}\right)$.
We show a picture of the permutohedron for $n=3$ in Figure \ref{perm3}. 
 \begin{figure}[ht]
       \centering
  \includegraphics[width=8cm]{ 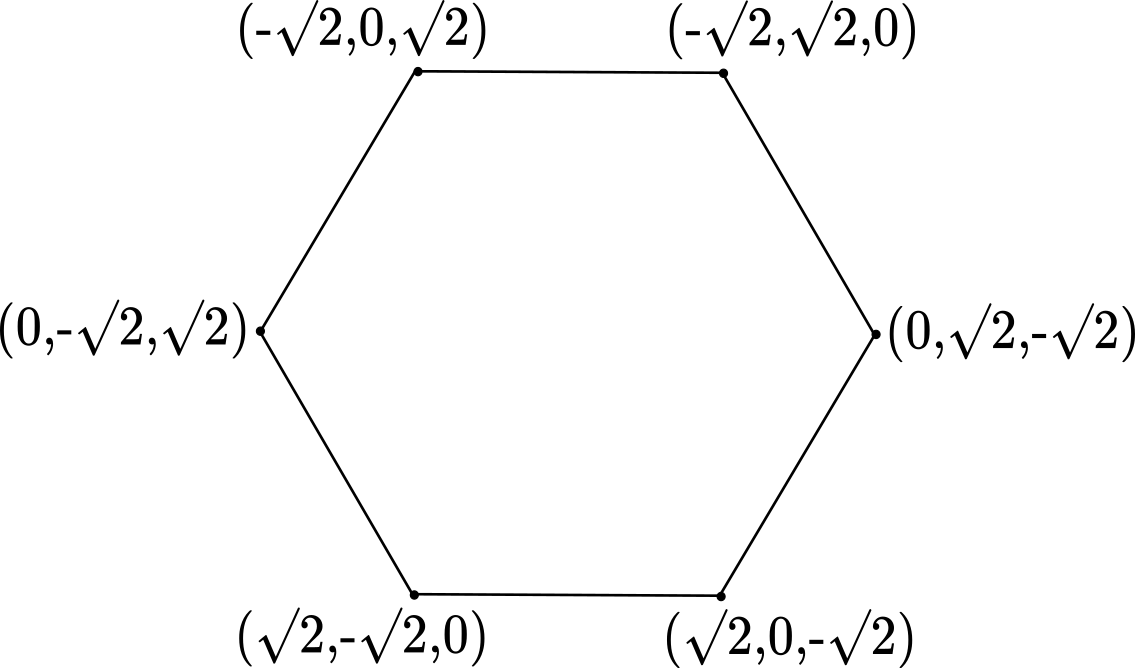}
  \caption{The permutohedron for $n=3$.}
  \label{perm3}
\end{figure}
\end{ex}

\noindent
We follow the construction of \cite[Section 8.1]{BEER} and establish a bijection between the poset $\mathcal{F}(C[S])$ and the poset of ordered partitions of $\{1,\ldots,n\}$ with the relation of refinement. An \textit{ordered partition} of $\left\{1, \ldots , n\right\}$ is a $p$-tuple $P=\left(P_1, \ldots, P_p\right)$ of nonempty subsets of $\left\{1,2,\ldots,n\right\}$ such that $\left\{1,2,\ldots,n\right\}=P_1 \cup P_2 \cup \cdots \cup P_p$ and $P_i\cap P_j=\emptyset$ for $i\neq j$. We will call the subsets $P_i$ the \textit{blocks} of the ordered partition $P$. Permuting the order of the blocks, we obtain a different ordered partition. An ordered partition $P=(P_1,\ldots,P_p)$ is a \textit{refinement} of another ordered partition $Q=\left(Q_1, \ldots , Q_q\right)$ if there exists a choice of ordered indices $0=j_0<j_1<\cdots<j_q=p$ such that $P_{j_{i-1}+1}\cup \cdots \cup P_{j_i}=Q_i$ for every $i=1,\ldots, q$. To be a refinement is an order relation that from now on we will denote by $\leq_{pa}$.\\
\\
\noindent There is a natural action of the symmetric group $\Sn$ on the set of all ordered partitions $\mathcal{P}$ of $\{1,\ldots,n\}$ defined by  $wP=w\left(P_1, \ldots, P_p\right)=\left(w\left(P_1\right),\ldots , w\left(P_p\right)\right)$, for all $P\in \mathcal{P}$ and $w\in \Sn$.\\
\\
\noindent
We already know that the faces $F(u,X,S)$ of the permutohedron $C[S]$ are in bijection with the cosets $u\W_{X}$ of $\W=\Sn$, where $u\in \Sn$ $(\emptyset,X)$-minimal and $X\subset S=\{(12),\ldots ,(n-1\,n)\}$. 
We can associate to an element $u \in \mathfrak{S}_n$ and a subset $X \subset S$ an ordered partition $P(u, X) = (P_1, \dots, P_n) \in \mathcal P$ as follows.
Set $p=n-\left|X\right|$ and write $S \backslash X=\left\{s_{i_1}, \ldots,s_{i_{p-1}}\right\}$ with $i_1< \cdots < i_{p-1}$. Then, set $i_0=0$ and $i_p=n$, and for $k \in \left\{1, \ldots ,p\right\}$ define the $k$-th block of the partition by
\begin{equation}\label{pikappa}
P_k=u\left(\left\{i_{k-1}+1,\ldots, i_k\right\}\right)=\left\{u\left(i_{k-1}+1\right), \ldots, u\left(i_k\right)\right\}.
\end{equation}
Observe that for every block $P \in \mathcal{P}$ there exists a $u \in \W$ and a $X \subset S$ such that $P=P\left(u,X\right)$, with $u$ not necessarily $(\emptyset,X)$-minimal. Note also that for any $u,v\in \W=\Sn$ and $X\subset S$, we have $vP(u,X)=P(vu,X)$.\\
\\
\noindent
For $P=(P_1,\ldots,P_p)$, we define $\Stab(P)=\{w \in \Sn\,|\, wP_i=P_i\; \forall i=1,\ldots,p\}$ the \textit{stabilizer of the ordered partition} $P\in \mathcal{P}$, i.e. the subgroup of all permutations
from $\Sn$ which preserve each block of $P$ as a set.
\begin{ex}
    For $n=3$, the permutohedron is the hexagon shown in Figure \ref{permufig}. To every face $F(u,X,S)$ (or coset $u\W_{X}$) we associate an ordered partition. To every vertex corresponds a partition with three blocks, to every edge corresponds a partition with two blocks, and to the only 2-face corresponds the unique partition with one block.
       \begin{figure}[ht]
       \centering
  \includegraphics[width=9cm]{ 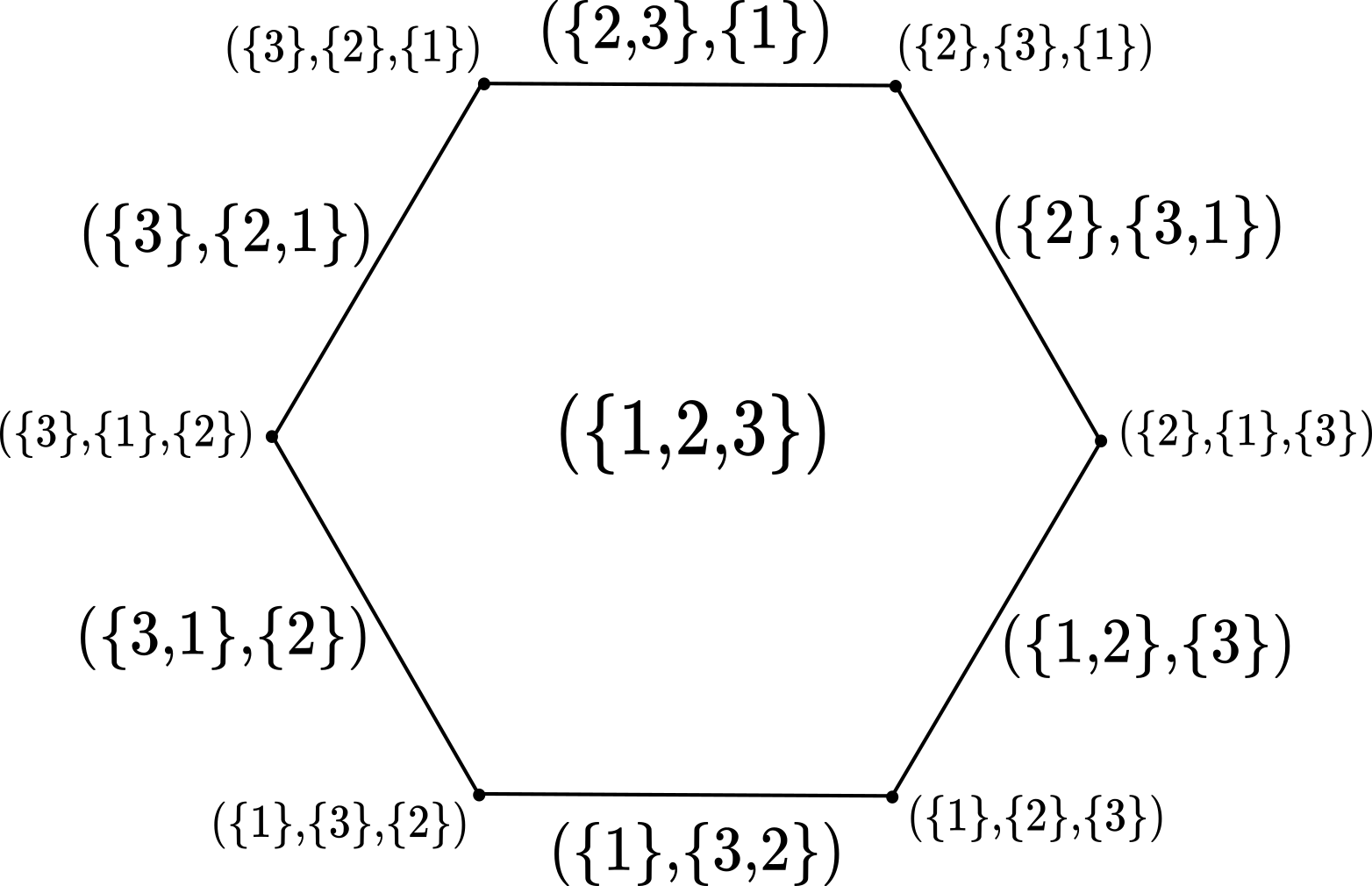}
  \caption{The permutohedron for $n=3$ with the associated ordered partitions.}
  \label{permufig}

\end{figure}
Observe that the vertex $(\{i\},\{j\},\{k\})$ is a face (and then a refinement) of the two edges $(\{i,j\},\{k\})$ and $(\{i\},\{j,k\})$, for all $i\neq j\neq k\neq i$ in $\{1,2,3\}$.
\end{ex}
\noindent
We can state the following result, of which we omit the proof.
\begin{lem}\label{stabandcosets}
$\left.\right.$
\begin{enumerate}
\item[(1)] Let $u\in \W=\Sn$ and $X\subset S$. Then $\Stab(P(u,X))=u\W_Xu^{-1}$.
\item[(2)] Let $u,v \in \W=\Sn$ and $X,Y \subset S$. We have that $P\left(u,X\right)=P\left(v,Y\right)$ if and only if $u\W_{X}=v\W_{Y}$.
\item[(3)] Let $u,v \in \W=\Sn$ and $X,Y \subset S$. We have that $P\left(u,X\right)\leq_{pa}P\left(v,Y\right)$ if and only if  $u\W_{X} \subset v\W_{Y}$.
\end{enumerate}
\end{lem}
\noindent
This assures that there is a poset isomorphism between $\mathcal{F}(C[S])$ and $(\mathcal{P},\leq_{pa})$. Observe also that, as the symmetric group acts on the poset of faces of its Coxeter polytope by $vF(u,X,S)=F(vu,X,S)$, in the same way it acts on the poset of partitions, so $v P(u,X)=P(vu,X)$ is an action that preserves the refinement. The stabilizer of the face $F(u,X,S)$ of the permutohedron is then $u\W_Xu^{-1}$.\\ 
\\
\noindent
Two ordered partitions $P=\left(P_1, \ldots, P_p\right)$ and $Q=\left(Q_1, \ldots, Q_q\right)$ of $\left\{1,2, \ldots, n\right\}$ are said to be \textit{equivalent} if they are equal up to permutation of the blocks, namely if $p=q$ and there exists some permutation $ \pi \in \mathfrak{S}_p$ such that $P_{\pi\left(i\right)}=Q_i$ for all $i \in \left\{1, \ldots, p\right\}$. We easily see that two ordered partitions are equivalent if and only if they have the same stabilizer. By (1) of Lemma \ref{stabandcosets}, we can deduce that for any $u,v\in \Sn$ and $X,Y\subset S$, the ordered partitions $P(u,X)$
and $P(v,Y)$ are equivalent if and only if $u\W_Xu^{-1}=v\W_Yv^{-1}$. This is easily checked to be an equivalence relation and it will be essential to generalize the construction to all Coxeter polytopes.

\begin{defn}
Consider $\mathcal{P}\times \Sn$ and define an equivalence relation $\sim$ on it as follows. Take $P,Q \in \mathcal{P}$ and $w,w' \in \Sn$, with $P=\left(P_1, \ldots, P_p\right)$. We say that $\left(P,w\right)$ is \textit{equivalent} to $\left(Q,w'\right)$, and we write $\left(P,w\right) \sim \left(Q,w'\right)$, if $P$ and $Q$ are equivalent and, for all $k \in \left\{1,2, \ldots, p\right\}$ and all $i,j \in P_k$, $w\left(i\right)<w\left(j\right) \Leftrightarrow w'\left(i\right)<w'\left(j\right)$.
\end{defn}

\begin{rmk}\label{crucialrmk}
Fix one representative $P$ for each equivalence class of ordered partitions under
permutation of the blocks. Then every equivalence class of $P \times \Sn$ with respect to $\sim$ contains a unique pair $(P, h)$ with $h \in \Stab(P)$. Hence $(P\times \Sn )/\sim\;
\cong \bigsqcup_{P} \Stab(P )$,
where $P$ runs through representatives of unordered partitions. More details will be given in Proposition \ref{propdescriptioncellsomegan}.
\end{rmk}

\noindent
It is possible to check that the equivalence relation just introduced on $\mathcal{P}\times \Sn$ behaves well with respect to the refinement. Again, we omit the proof of the following lemma.

\begin{lem} \label{refin}
Let $P,Q \in \mathcal{P}$ and $w,w' \in \Sn$ be such that $\left(P,w\right) \sim \left(Q,w'\right)$. Let $P=\left(P_1 , \ldots , P_p\right)$ and denote by $\pi \in \mathfrak{S}_p$ the unique permutation such that $Q=\left(P_{\pi^{-1}\left(1\right)}, \ldots , P_{\pi^{-1}\left(p\right)}\right)$. Let $P'=\left(P_1', \ldots, P_q'\right)$ be a refinement of $P$ and let $\mu$ be an element of $\mathfrak{S}_q$ such that $Q'=\left(P'_{\mu^{-1}\left(1\right)}, \ldots , P'_{\mu^{-1}\left(q\right)}\right)$ is a refinement of $Q$. Then $\left(P',w\right) \sim \left(Q', w'\right)$.
\end{lem}
\noindent
If two partitions $P(u,X)$ and $P(v,Y)\in \mathcal{P}$ are equivalent, then we say that the two faces $F(u,X,S)$ and $ F(v,Y,S)$ are \textit{equivalent}. Without loss of generality, we can also suppose that for $u,v\in \Sn$ and $X,Y\subset S$, as above, $u$ is $(\emptyset,X)$-minimal and $v$ is $(\emptyset,Y)$-minimal. We then define an identification map 
\[j_{F(u,X,S),F(v,Y,S)}:F(u,X,S)\longrightarrow F(v,Y,S)\] 
given by the translation by vector $v(o)-u(o)$. 
\begin{prop}\label{translfunziona}
If the faces $F(u,X,S)$ and $F(v,Y,S)$ are equivalent, then the translation by vector $v(o)-u(o)$ is a well-defined isometry that sends isometrically $F(u,X,S)$ to $F(v,Y,S)$.
\end{prop}
\noindent
To prove Proposition \ref{translfunziona}, we need some preliminary results. We are working in the case of the symmetric group $\W=\Sn$ generated by the simple transpositions, but Proposition \ref{translfunziona}, as well as the arguments to prove it, works in the general setting of Coxeter graphs of spherical type. Recall that $F(u,X,S)$ is equivalent to $F(v,Y,S)$ if and only if $u\W_Xu^{-1}=v\W_Y v^{-1}$, which we can write as $v^{-1}u\W_Xu^{-1}v=\W_Y$.
We first state and prove the following preliminary lemma.
\begin{lem}\label{lemdoublecoset}
Let $X,Y$ be subsets of $S$ and let $u,v$ be in $\W[\Gamma]$. Suppose that $u\W_X u^{-1}=v\W_Yv^{-1}$ and let $w_0$ be the unique element of minimal length in the double coset $\W_{Y}(v^{-1}u)\W_{X}$. Then $w_0Xw_0^{-1}=Y$.
\end{lem}
\begin{proof}
By hypothesis, we have $(v^{-1}u)\W_{X}(v^{-1}u)^{-1}=\W_{Y}$. Thus, if we set $w=v^{-1}u$, we get $w\W_{X}w^{-1}=\W_{Y}$. Now, write $w=a w_0 b$, where $a \in \W_{Y}$, $b\in \W_{X}$ and $w_0$ is $(Y,X)-$minimal in $\W_{Y}w\W_{X}$. Recall that by Lemma \ref{w0mindoublecoset} we know that the minimal element in the double coset is unique. Substituting the expression for $w$, we obtain
\begin{align*}
w\W_{X}w^{-1}=a w_0 b \W_{X} b^{-1} w_0^{-1}a^{-1}=a w_0 \W_{X}w_0^{-1}a^{-1}=\W_{Y},\\
w_0 \W_{X}w_0^{-1}=a^{-1}\W_{Y}a=\W_{Y}.
\end{align*}
We have shown that $w_0\W_{X}w_0^{-1}=\W_{Y}$. Now, we want to prove the equality for the generators. Take $s\in X$ and consider $w_0 s w_0^{-1}\in \W_{Y}$. There exists some $y\in \W_{Y}$ such that $w_0s=yw_0$. Now, compare the lengths: by the minimality of $w_0$ in $w_0\W_{X}$, we have $l(w_0s)=l(w_0)+1$. But $w_0$ is also of minimal length in the coset $\W_{Y}w_0$, so $l(yw_0)=l(y)+l(w_0)$. This implies that $l(y)=1$ thus $y=t$ for some $t\in Y$ and then $w_0Xw_0^{-1}\subseteq Y$. Applying the same argument to $w_0^{-1}$, which is the minimal element in the double
coset $\W_X(u^{-1}v)\W_Y$, and using $w_0^{-1} \W_Y w_0 = \W_X$, we obtain the other inclusion $w_0^{-1}Y w_0\subseteq  X$. Thus
$w_0Xw_0^{-1} = Y$.
\end{proof}

\noindent
This result implies that there is a well-defined linear map from $V_X$ to $V_Y$ that sends the element of the basis $\alpha_s$ for $s\in X$ to $\alpha_{w_0 s w_0^{-1}}$, that is $\alpha_t$ for some $t\in Y$ since $w_0 s w_0^{-1}\in Y$. We can also remark that such a map coincides with the restriction of the action of the element $w_0$ in the vector space $V$, through the standard linear representation.
\begin{prop}\label{actionofw0}
Let $X,Y$ be as defined in Lemma \ref{lemdoublecoset}, and let $f:V_{X}\longrightarrow V_{Y}$ be the linear map that sends $\alpha_s$ to $\alpha_{w_0sw_0^{-1}}$. Then $f$ agrees with the restriction of the action of $w_0$ on $V_X$ via the canonical representation $\rho$.
\end{prop}
\begin{proof}
Let $s$ be a generator in $X$, and consider the reflection $w_0sw_0^{-1}$. Call the root $w_0(\alpha_s):=\beta$, so that $w_0sw_0^{-1}=r_{\beta}$. By Lemma \ref{lemdoublecoset},  $w_0sw_0^{-1}=t\in Y$. Then, since $\beta$ and $\alpha_t$ are associated to the same reflection, $\beta$ is either $\alpha_t$ or $-\alpha_t$. As recalled in Subsection \ref{CoxandArt}, $w_0(\alpha_s)$ is a positive root if and only if $l(w_0s)=l(w_0)+1$, which holds by the minimality of $w_0$ in the coset $w_0\W_{X}$. In particular, for all $s\in X$, we have $w_0(\alpha_s)=\alpha_{w_0sw_0^{-1}}=\alpha_t=f(\alpha_s)$. 
\end{proof}
\noindent Observe that, in general, $w_0sw_0^{-1}$ is not a generator in $S$, and the map $f$ is well-defined only on the subspace $V_X$. The action of $w_0$, however, is defined in the entire space $V$, and its restriction to $V_X$ coincides with $f$. \\
\\
\noindent Suppose then that $u\W_Xu^{-1}=v\W_Y v^{-1}$ and that $w_0$ is the $(Y,X)$-minimal element in the double coset $\W_Y(v^{-1}u)\W_X$. As a consequence of Proposition \ref{actionofw0}, we obtain that $w_0(V_X)=V_Y$. Since $v^{-1}u$ and $w_0$ by definition belong to the same $(\W_Y,\W_X)$ double coset, we can write $v^{-1}u=aw_0b$ with $a\in \W_Y$ and $b\in \W_X$. Hence, 
\[v^{-1}u(V_X)=aw_0b(V_X)=aw_0(V_X)=a(V_Y)=V_Y.\]
We just obtained that $u(V_X)=v(V_Y)$. By Lemma \ref{lemdoublecoset} we also easily get that $u(\Phi_X)=v(\Phi_Y)$, where $\Phi_X=\{g(\alpha_s)\,|\, g\in \W_X, \; s\in X\}$ and $\Phi_Y=\{h(\alpha_t)\,|\, h\in \W_Y, \; t\in Y\}$.\\
We also assumed that $u$ is $(\emptyset,X)$-minimal, and that $v$ is $(\emptyset,Y)$-minimal. The next lemma gives a stronger description of how $u$ and $v$ act on the root system.
\begin{lem}\label{upixvpiy}
Let $u, v \in \W[\Gamma]$ and $X, Y \subset S$ be such that $u$ is $(\emptyset, X)$-minimal, $v$ is $(\emptyset, Y)$-minimal and $u \W_X u^{-1} = v \W_Y v^{-1}$.
Set $\Pi_X=\{\alpha_s\,|\, s\in X\}$ and $\Pi_Y=\{\alpha_t\,|\, t\in Y\}$.
Then $u (\Pi_X) = v (\Pi_Y)$.
\end{lem} 
\begin{proof}
Recall that $V_X$ denotes the vector subspace of $V$ generated by $\Pi_X$, and that  $V_Y$ denotes the vector subspace of $V$ generated by $\Pi_Y$.
We just saw that from the equality $u \W_X u^{-1} = v \W_Y v^{-1}$ it follows that $u(\Phi_X) = v (\Phi_Y)$ and $u(V_X) = v (V_Y)$. Now take $s \in X$.
Since $u (\alpha_s) \in u (\Phi_X) = v (\Phi_Y) \subset v (V_Y)$, the root $u (\alpha_s)$ can be written in the form $u (\alpha_s) = \sum_{t \in Y} a_{s,t}\, v(\alpha_t)$, where all the real coefficients $a_{s,t}$ are either non-negative, or non-positive.
The element $u$ is $(\emptyset, X)$-minimal, which means that for all $s\in X$, $l(us)=l(u)+1$. As we saw in Subsection \ref{CoxandArt}, this means that $u (\alpha_s)$ is a positive root.
In the same way, for all $t \in Y$, $v (\alpha_t)$ is a positive root.
We deduce that $a_{s,t}$ are all non-negative.
In the same fashion, for all $t \in Y$, $v (\alpha_t)$ can be written as $v (\alpha_t) = \sum_{s \in X} b_{t,s}\, u (\alpha_s)$, where the real coefficients $b_{t,s}$ are non-negative.\medskip\\
\noindent For $s \in X$, set $A_s = \{ t \in Y \mid a_{s,t}>0\}$, and for $t \in Y$, set $B_t = \{ s \in X \mid b_{t,s} > 0\}$.
Fix now an element $s$ in $X$.
Then $\{s\} = \bigcup_{t \in A_s} B_t$.
We now show in particular that, for all $t \in A_s$, $B_t = \{s\}$. By definition of $A_s$ and $B_t$, if we consider $s\in X$ and $t\in Y$, we can rewrite
\begin{equation*}
    u(\alpha_s)=\sum_{t\in A_s}a_{s,t}v(\alpha_t),\qquad \qquad \qquad v(\alpha_t)=\sum_{s'\in B_t}b_{t,s'}u(\alpha_s').
\end{equation*}
\noindent Therefore, we obtain that, for all $s\in X$,
\[u(\alpha_s)=\sum_{t\in A_s}a_{s,t}\sum_{s'\in B_t}b_{t,s'}u(\alpha_s')=\sum_{t\in A_s}\sum_{s'\in B_t}a_{s,t}b_{t,s'}u(\alpha_{s'}).\]
\noindent Suppose now that, for $t\in A_s$, the set $B_t$ contains at least one generator $s'\neq s$. We would have that 
\[u(\alpha_s)=a_{s,t}b_{t,s'}u(\alpha_{s'})+\sum_{t\in A_s}\,\sum_{s''\in B_t\backslash\{s'\}}a_{s,t}b_{t,s''}u(\alpha_{s''}),\]
\noindent which leads to an absurd because $a_{s,t}b_{t,s'}\neq 0$ and the formula above would give a non-trivial linear combination of vectors in the basis $\Pi_X=\{\alpha_s\mid s\in S\}$ that equals the null vector. Hence, $B_t=\{s\}$ for all $t\in A_s$, and $\bigcup_{t\in A_s}B_t=\{s\}$.\medskip\\
\noindent
If $A_s$ had at least two elements, denoted by $t_1$ and $t_2$, then we could not have $B_{t_1} = B_{t_2} = \{s\}$, because in this case $v(\alpha_{t_1})$ and $v (\alpha_{t_2})$ would be linearly dependent, which is not the case.
Therefore, $A_s$ is a singleton. Namely, there exists $t \in Y$ and $a >0$ such that $u(\alpha_s) = a\,v(\alpha_t)$.
But $1 = \langle u (\alpha_s), u (\alpha_s) \rangle = a^2 \langle v (\alpha_t), v (\alpha_t) \rangle = a^2$, thus $a=1$ and $u (\alpha_s) = v (\alpha_t)$.
In the same fashion we show that for all $t \in Y$, there exists an $s \in X$ such that $v (\alpha_t) = u (\alpha_s)$.
We can then conclude that $u (\Pi_X) = v (\Pi_Y)$.
\end{proof}
\noindent
We are now ready to show that the map $j_{F(u,X,S),F(v,Y,S)}=\tau(v(o)-u(o))$ is well-defined.

\begin{proof}[Proof of Proposition \ref{translfunziona}]
To prove that $\tau(v(o)-u(o))$ sends $F(u,X,S)$ to $F(v,Y,S)$, it suffices to show that it commutes with the action of $b$ for all $b\in B=u\W_Xu^{-1}=v\W_Yv^{-1}$. Indeed, take an arbitrary $b\in B$ and suppose that $b$ commutes with $\tau(v(o)-u(o))$. Recall that $F(u,X,S)=conv\{ug(o)\,|\,g\in \W_X\}$ and that $F(v,Y,S)=conv\{vh(o)\,|\,h\in \W_Y\}$ with $o=\sum_{s\in S}\alpha_s^*\in V$. We can write every vertex of the first face as $ug(o)$ for $g\in \W_X$, so as $ugu^{-1}u(o)=bu(o)$ for $b=ugu^{-1}\in B$.\medskip \\
\noindent If we apply the translation to $F(u,X,S)$, we see that it must send the vertex $bu(o)$ to $\tau(v(o)-u(o))(bu(o))=b(\tau(v(o)-u(o))(u(o)))=bv(o)$, that is a vertex $vhv^{-1}v(o)=vh(o)$ of $F(v,Y,S)$, for some $h\in \W_Y$ such that $ugu^{-1}=vhv^{-1}$. So, if the translation commutes with all $b\in B$, then it sends the first face to the second. \medskip  \\
\noindent We now show that these always commute, namely, that each $b\in B$ fixes the translation vector $v(o)-u(o)$. Equivalently, we can show that for every $s_0\in X$, we have $us_0u^{-1}(v(o)-u(o))=v(o)-u(o)$. Note that $us_0u^{-1}$ is a reflection of root vector $u(\alpha_{s_0})$, therefore it fixes the vector as above if and only if $\langle u(\alpha_{s_0}),v(o)-u(o)\rangle=0$. Moreover, by Lemma \ref{upixvpiy}, we know that for any $s_0\in X$, there exists some $t_0\in Y$ such that $u(\alpha_{s_0})=v(\alpha_{t_0})$. \medskip \\
\noindent Now we decompose $v(o)=v(\alpha_{t_0}^*)+\sum_{t\in S\backslash\{t_0\}}v(\alpha_t^*)$ and $u(o)=u(\alpha_{s_0}^*)+\sum_{s\in S\backslash \{s_0\}}u(\alpha_s^*)$. Since $u$ and $v$ are both isometries, the sets $u(\Pi)=\{u(\alpha_s)\,|\, s\in S\}$ and $v(\Pi)=\{v(\alpha_t)\,|\, t\in S\}$ are still two bases of the vector space $V$.
We denote by $u(\Pi)^*=\{u(\alpha_s)^*\,|\, s\in S\}$ the dual basis of $u(\Pi)$ and by $v(\Pi)^*=\{v(\alpha_t)^*\,|\, t\in S\}$ the dual basis of $v(\Pi)$. Remark that we obtain $u(\alpha_s^*)=(u(\alpha_s))^*$ and $v(\alpha_t^*)=(v(\alpha_t))^*$ for all $s,t$ in $S$. We can then rewrite $v(o)=v(\alpha_{t_0})^*+\sum_{t\in S\backslash\{t_0\}}v(\alpha_t)^*$ and $u(o)=u(\alpha_{s_0})^*+\sum_{s\in S\backslash \{s_0\}}u(\alpha_s)^*$. Finally, we compute \begin{equation}
\langle u(\alpha_{s_0}),v(o)-u(o)\rangle=\langle v(\alpha_{t_0}),v(\alpha_{t_0})^*+\sum_{t\in S\backslash\{t_0\}}v(\alpha_t)^*\rangle-\langle u(\alpha_{s_0}),u(\alpha_{s_0})^*+\sum_{s\in S\backslash \{s_0\}}u(\alpha_s)^* \rangle=1-1=0.
\end{equation}
\noindent Therefore, $v(o)-u(o)$ is orthogonal to each $u(\alpha_{s_0})=v(\alpha_{t_0})$, so it commutes with any $b=ugu^{-1}=vhv^{-1}\in B$ for $g\in \W_X$ and $h\in \W_Y$, and it sends $F(u,X,S)$ to $F(v,Y,S)$.
\end{proof}

\begin{ex}
    Consider a face $F(u,X,S)$ of the permutohedron $C[S]$ such that $X=\{s,t\}\subset S$, as illustrated in Figure \ref{perm_in_Beer}.

\begin{figure}[ht]
       \centering
       \begin{minipage}{7cm}
\includegraphics[width=6cm]{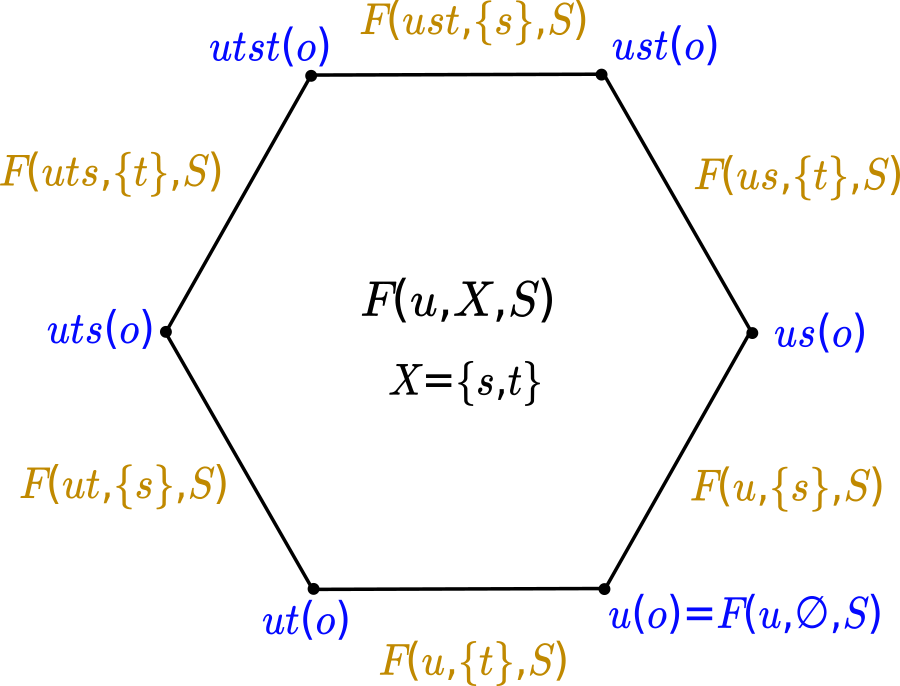}
  \caption{A face of dimension 2 of the permutohedron $C[S]$.}
  \label{perm_in_Beer}
      \end{minipage}
      \qquad 
      \begin{minipage}{7cm}
\includegraphics[width=6cm]{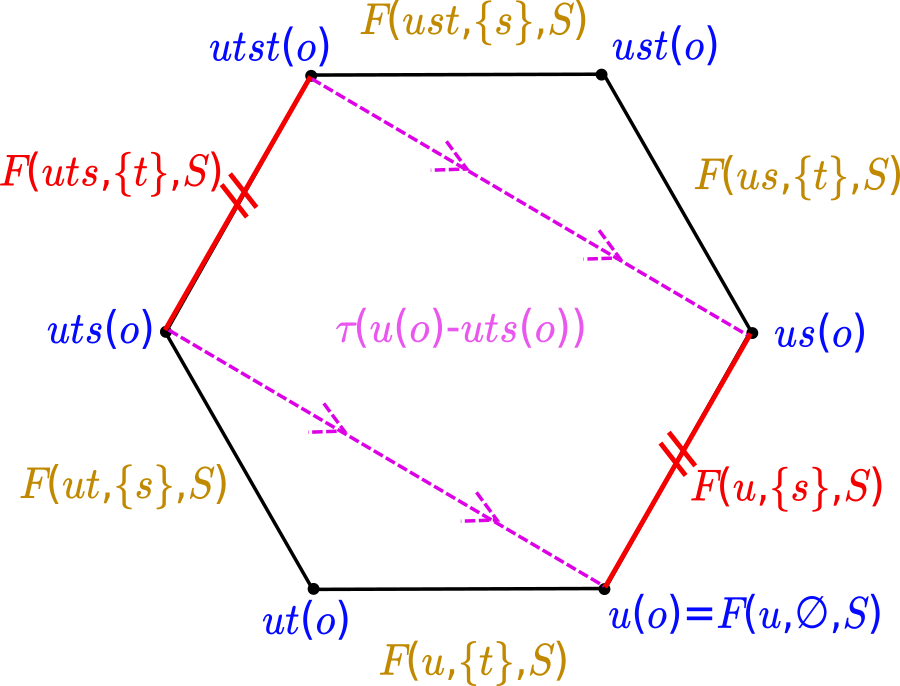}
  \caption{The translation identifying $F(uts,\{t\},S)$ and $F(u,\{s\},S)$.}
\label{perm_in_Beer_trasl}
      \end{minipage}
\end{figure}
\noindent The vertices of such a face are all of the form $ug(o)=F(ug,\emptyset,S)$ with $g\in \W_{X}$. Observe that for all $g_1,g_2\in \W_{X}$, the vertices $F(ug_1,\emptyset,S)$ and $F(ug_2,\emptyset,S)$ are equivalent, in the sense that their associated ordered partitions $P(ug_1,\emptyset)$ and $P(ug_2,\emptyset)$ represent the same unordered partition. Hence, for all $g\in \W_{X}$, $ug(o)$ is identified to $u(o)$ through the translation
\begin{align*}
    \tau(u(o)-ug(o)): ug(o)\longmapsto u(o).
\end{align*}
In Figure \ref{perm_in_Beer}, also observe that the faces $F(uts,\{t\},S)$ and $F(u,\{s\},S)$ are equivalent. Indeed, $uts\W_{\{t\}}stu^{-1}=u\W_{\{s\}}u^{-1}$, which means that $P(uts,\{t\})\sim P(u,\{s\})$. In any copy of $C[S]$ in $C[S]\times \Sn$, we identify these two faces through the translation $\tau(uts(o)-u(o))$:
\begin{align*}
    \tau(u(o)-uts(o))\;:\; F(uts,\{t\},S)&\longrightarrow F(u,\{s\},S),\\
    uts(o) &\longmapsto u(o),\\
    utst(o) &\longmapsto us(o).
\end{align*}
\noindent The translation is illustrated in Figure \ref{perm_in_Beer_trasl}.
\end{ex}

\begin{ex}
Consider now two equivalent faces $F(u,X,S)$ and $F(v,Y,S)$ of $C[S]$ such that $|X|=|Y|=2$. Namely, $u\W_Xu^{-1}=v\W_Y v^{-1}$, and calling $w_0$ the unique $(Y,X)$-minimal element in the double coset $\W_Y (v^{-1}u)\W_X$, thanks to Lemma \ref{lemdoublecoset} we know that $w_0 Xw_0^{-1}=Y$.
\begin{figure}[h!]
    \centering
\includegraphics[width=16cm]{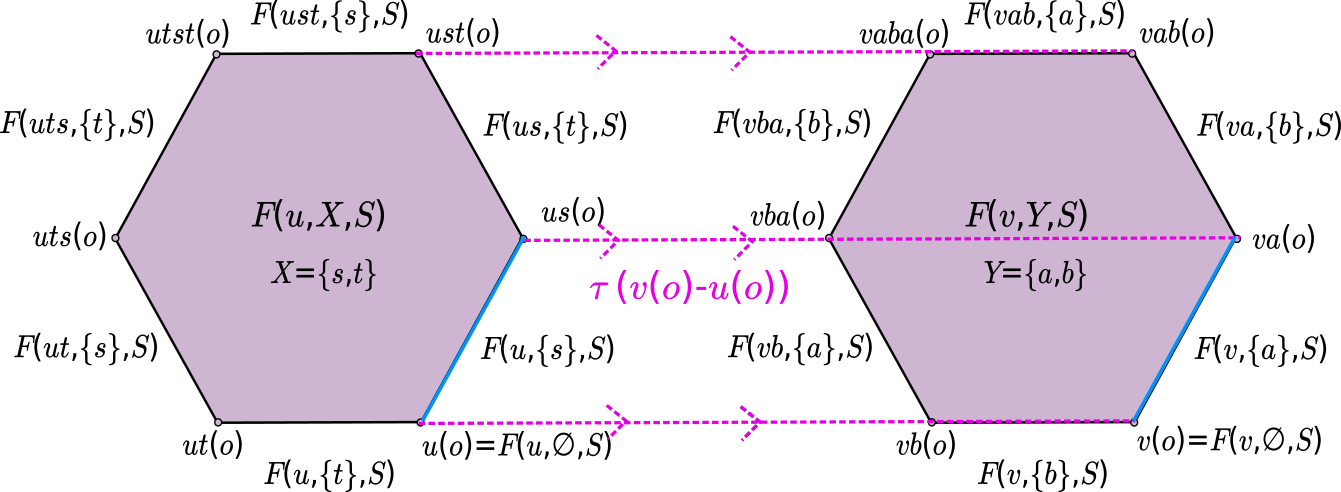}
    \caption{the translation by vector $v(o)-u(o)$.}
\label{perm_in_BEER_trasl2cell}
\end{figure}
    \noindent Suppose that $X=\{s,t\}\subset S$ and $Y=\{a,b\}\subset S$. The translation $\tau(v(o)-u(o))$ sends isometrically $F(u,X,S)$ to $F(v,Y,S)$, as illustrated in Figure \ref{perm_in_BEER_trasl2cell}. \medskip \\
    \noindent Since $w_0Xw_0^{-1}=Y$, either $w_0 s w_0^{-1}=a$ and $w_0 t w_0^{-1}=b$ or  $w_0 s w_0^{-1}=b$ and  $w_0 t w_0^{-1}=a$. Suppose that we are in the first case, and the second will be analogous.
    Therefore, the sub-face $F(u,\{s\},S)$ of $F(u,X,S)$ is identified to $F(v, \{w_0 s w_0^{-1}\},S)\subset F(v,Y,S)$ through the translation $\tau(v(o)-u(o))$.
\end{ex}

\noindent
We have proven that if $F(u,X,S)$ and $F(v,Y,S)$ are equivalent, $j_{F(u,X,S),F(v,Y,S)}=\tau(v(o)-u(o))$ sends isometrically $F(u,X,S)$ to $F(v,Y,S)$.
Consider now two equivalent pairs $F_1=\left(F(u,X,S),w\right) \sim F_2=\left(F(v,Y,S),w'\right)$, that are faces of $C[S]\times \Sn$. Then define an isometry between $k_{F_1,F_2}:(F(u,X,S),w)\longrightarrow (F(v,Y,S),w')$ such that for each $x\in F(u,X,S)$, we set \begin{equation}\label{defattachingmap}
k_{F_1,F_2}(x,w):=(j_{F(u,X,S),F(v,Y,S)}(x),w').
\end{equation}

\noindent
We are now ready to give the definition of the space $\Omega_n$.

\begin{defn}\label{deforiginalbeer}
The \textit{Bartholdi\textendash Enriquez\textendash Etingof\textendash Rains complex}, referred to as the \textit{BEER complex} and denoted by $\Omega_n$, is the quotient of the space $C\left[S\right] \times \Sn$ by the relation $\sim$, which identifies the face $F_1=(F(u,X,S),w)$ with the face $F_2=(F(v,Y,S),w')$ if the associated pairs in $\mathcal{P}\times \Sn$ are such that $\left(P(u,X),w\right) \sim \left(P(v,Y),w'\right)$. The identification for every $x\in F(u,X,S)$ is given by $(x,w)\sim k_{F_1,F_2}(x,w)$.\end{defn}
\noindent
Lemma \ref{refin} shows that this identification is well-defined. Since the identifications are isometries on faces, the Euclidean metric on the permutohedron $C[S]$ induces a well-defined piecewise Euclidean structure on $\Omega_n$ , and we endow $\Omega_n$ with the associated intrinsic path metric.
\begin{nt}
    When we write $[(F,w)]$ with $F\in \mathcal{F}(C[S])$ a face of the permutohedron and $w\in \Sn$, we mean the equivalence class of the face $(F,w)\in C[S]\times \Sn$ in the quotient $\Omega_n$. Thanks to the poset isomorphism between $\mathcal{F}(C[S])$ and the ordered partitions $\mathcal{P}$, for all $P\in \mathcal{P}$, we denote by $F(P)$ the face of the permutohedron associated with $P$. We may also write $[(P,w)]\in \Omega_n$ to mean the class of the pair $(F(P),w)\in C[S]\times\Sn$ under $\sim$.
\end{nt}
\begin{rmk}\label{rmkdecriptionfacesOmegan}
    Lemma \ref{refin} also provides a description of the sub-cells of a given cell $[(F,w)]\in \Omega_n$. Let $F=F(u,X,S)$ with $u\in \Sn$ $(\emptyset,X)$-minimal, and let $X\subset S$. The sub-faces of $F(u,X,S)$ in $C[S]$ are of the form $F(ug,Z,S)$ with $g\in \W_X$, $g$ $(\emptyset,Z)$-minimal and $Z\subset X$. Therefore, the cells contained in $[(F(u,X,S),w)]\in \Omega_n$ are of the form $[(F(ug,Z,S),w)]$.
\end{rmk}
\begin{thm}\cite[Theorem 8.2]{BEER}\label{fundoriginalbeer}
The fundamental group of $\Omega_n$ is the pure virtual braid group on $n$ strands $\PVB_n$.
\end{thm}
\noindent
Recall that a presentation of $\PVB_n$ is given in \cite{Bard04}, with generators $\zeta_{ij}$ for $1\leq i \neq j \leq n$ and relations $\zeta_{ij} \zeta_{kl}= \zeta_{kl}\zeta_{ij}$  for $i,j,k,l$ pairwise distinct, and $\zeta_{ij} \zeta_{ik} \zeta_{jk}= \zeta_{jk}\zeta_{ik}\zeta_{ij}$ for $i,j,k$ pairwise distinct.
\begin{proof}[Proof of Theorem \ref{fundoriginalbeer}.]

We are going to show that:

\begin{enumerate}
\item[(0)] The 0-skeleton of $\Omega_n$ is made of just one point $\omega^0$.
\item [(1)]For $i,j \in \left\{1,2, \ldots, n\right\}$ with $i< j$, we have two different 1-cells $\Dd^1(i, j)$ and $\Dd^1(j,i)$. These 1-cells are pairwise distinct and they form a graph, which is the 1-skeleton of $\Omega_n$.
\item [(2)]\begin{enumerate}
\item For all $i,j,k$ pairwise distinct, there exists a 2-cell $\Dd^2\left(i,j,k\right)$ whose boundary is the loop $\Dd^1(i,j)\Dd^1(i,k)\Dd^1(j,k)\Dd^1(i,j)^{-1}\Dd^1(i,k)^{-1}\Dd^1(j,k)^{-1}$ with the given orientation.
\item For all $i,j,k,l$ pairwise distinct, there exists a 2-cell $\Dd^2\left(i,j,k,l\right)$ whose boundary is the loop $\Dd^1(i,j)\Dd^1(k,l)\Dd^1(i,j)^{-1}\Dd^1(k,l)^{-1}$ with the given orientation.
\item All the 2-cells in $\Omega_n$ are of the form $\Dd^2\left(i,j,k\right)$ or $\Dd^2\left(i,j,k,l\right)$.
\end{enumerate}
\end{enumerate}
These steps prove the result.
To see that the 0-skeleton consists of just one point, observe that the vertices of $C[S]$ are of the form $F(u,\emptyset,S)$ with $u\in \Sn$, so they correspond to the ordered partitions $P(u,\emptyset)=(\{u(1)\},\ldots,\{u(n)\})$ made of $n$ blocks. These ordered partitions all represent the same unordered partition. Moreover, since they consist of singleton blocks, the condition of equivalence on the second components holds clearly. Thus, all the vertices $(F(u,\emptyset,S),w)$ of $C[S]\times \Sn$ represent the same equivalence class, denoted by $\omega^0$.\\
\\
\noindent Recall that for any ordered partition $P\in \mathcal{P}$, we will denote the face of the permutohedron that it represents by $F(P)$. The edges of $C[S]$ are of the form $F(u,s_h,S)$ where $s_h\in S$ and $u\in \Sn$, $(\emptyset,\{s_h\})$-minimal. For each $i< j$ in $\{1,\ldots,n\}$, denote by $P_{ij}$ the partition $P_{ij}=(\{i,j\},\{1\},\ldots,\{i-1\},\{i+1\},\ldots,\{j-1\},\{j+1\},\ldots,\{n\})$. It is not hard to be convinced that each ordered partition $P(u,\{s_h\})$ is equivalent to a unique $P_{ij}$. We call $s_i$ the simple transposition $(i\;i+1)$ for $i\in \{1,\ldots,n-1\}$, and denote by $t_{ij}$ the (not necessarily simple) transposition $(ij)$ with $i, j\in \{1,\ldots,n-1\}$ and $i<j$. Now, the pair $(F(P_{ij}),w)$ of $C[S]\times \Sn$ is equivalent to $(F(P_{ij}),\id)$ if $w(i)<w(j)$, and it is equivalent to $(F(P_{ij}),t_{ij})$ if $w(i)>w(j)$ instead. We call these two classes of edges in $\Omega_n$, respectively, $\Dd^1(i,j)$ and $\Dd^1(j,i)$. Now, for $i,j \in \left\{1, \ldots, n\right\}$ such that $i<j$, let
\begin{align*}
P^0_{ij}&=\left(\left\{i\right\},\left\{j\right\},\left\{1\right\}, \ldots, \left\{i-1\right\},\left\{i+1\right\}, \ldots, \left\{j-1\right\},\left\{j+1\right\}, \ldots,\left\{n\right\}\right),\nonumber \\
P^0_{ji}&=\left(\left\{j\right\},\left\{i\right\},\left\{1\right\}, \ldots, \left\{i-1\right\},\left\{i+1\right\}, \ldots, \left\{j-1\right\},\left\{j+1\right\}, \ldots,\left\{n\right\}\right).\nonumber 
\end{align*}
Thanks to the poset isomorphism between $\mathcal{P}$ and $\mathcal{F}(C[S])$, we know that the endpoints of $F(P_{ij})$ are $F(P^0_{ij})$ and $F(P^0_{ji})$. For $w\in \Sn$, we orient $(F(P_{ij}),w)$ from $\left(F(P^0_{ij}),w\right)$ to $\left(F(P^0_{ji}),w\right)$ if $w\left(i\right)<w\left(j\right)$, and from $\left(F(P^0_{ji}),w\right)$ to $\left(F(P^0_{ij}),w\right)$ if $w\left(i\right)>w\left(j\right)$. In particular, $(F(P_{ij}), \id)$ is oriented from $P_{ij}^0$ to $P_{ji}^0$ and $(F(P_{ij}), t_{ij})$ is oriented from $P_{ji}^0$ to $P_{ij}^0$. These orientations
induce the orientations of $\Dd^1(i,j)$ and $\Dd^1(j,i)$, respectively. \\
\\
\noindent The 2-cells of $C[S]$ correspond to the faces $F(u,X,S)$ with $X\subset S$, $|X|=2$ and $u\in \Sn$ $(\emptyset,X)$-minimal. The associated partitions have $n-2$ blocks. There are two cases, which we now describe. \\
\\
\noindent
Let $i,j,k \in \left\{1, \ldots,n\right\}$ such that $i<j<k$. We write $\left\{1,2, \ldots,n\right\}\backslash \left\{i,j,k\right\}=\left\{r_1,r_2, \ldots, r_{n-3}\right\}$ with $r_1<r_2< \cdots <r_{n-3}$, and set $P_{ijk}=\left(\left\{i,j,k\right\},\left\{r_1\right\},\left\{r_2\right\}, \ldots,\left\{r_{n-3}\right\}\right)$.
Now let $i,j,k,l \in \left\{1, \ldots,n\right\}$ such that $i<j<k<l$. We write $\left\{1,2, \ldots,n\right\}\backslash \left\{i,j,k,l\right\}=\left\{r_1,r_2, \ldots, r_{n-4}\right\}$ with $r_1<r_2< \cdots <r_{n-4}$, and set
$P_{ijkl}=\left(\left\{i,j\right\},\left\{k,l\right\},\left\{r_1\right\},\left\{r_2\right\}, \ldots,\left\{r_{n-4}\right\}\right)$,$P_{ijkl}'=\left(\left\{i,k\right\},\left\{j,l\right\},\left\{r_1\right\},\left\{r_2\right\}, \ldots,\right.$ $\left.\left\{r_{n-4}\right\}\right)$ and $P_{ijkl}''=\left(\left\{i,l\right\},\left\{j,k\right\},\left\{r_1\right\},\left\{r_2\right\}, \ldots,\left\{r_{n-4}\right\}\right)$. The ordered partition $P(u,X)$ associated to $F(u,X,S)$ is equivalent to a unique partition of the form $P_{ijk}$, $P_{ijkl}$,  $P_{ijkl}'$ or  $P_{ijkl}''$.\\
\\
\noindent Suppose first that $P(u,X)$ is equivalent to $P_{ijk}$. Set $I=\{i,j,k\}$ and denote by $\mathfrak{S}(I)$ the set of permutations of $I$. Then the element $(P_{ijk},w)$ of $\mathcal{P}\times \Sn$ is equivalent to exactly one pair $(P_{ijk},\overline{w})$ with $\overline{w}\in \mathfrak{S}(I)$. We denote the corresponding equivalence class in $\Omega_n$ $\Dd^2(\overline{w}(i),\overline{w}(j),\overline{w}(k))$. The edges in the boundary of $(F(P_{ijk}),\overline{w})$ are given by all possible refinements of $P_{ijk}$:
\begin{align}
&P^1_{ijk}=\left(\left\{i,j\right\},\left\{k\right\},\left\{r_1\right\}, \ldots, \left\{r_{n-3}\right\}\right), &P^4_{ijk}=\left(\left\{k\right\},\left\{i,j\right\},\left\{r_1\right\}, \ldots, \left\{r_{n-3}\right\}\right), \\ &P^2_{ijk}=\left(\left\{j\right\},\left\{i,k\right\},\left\{r_1\right\}, \ldots, \left\{r_{n-3}\right\}\right), &P^5_{ijk}=\left(\left\{i,k\right\},\left\{j\right\},\left\{r_1\right\}, \ldots, \left\{r_{n-3}\right\}\right),\nonumber \\
&P^3_{ijk}=\left(\left\{j,k\right\},\left\{i\right\},\left\{r_1\right\}, \ldots, \left\{r_{n-3}\right\}\right), &P^6_{ijk}=\left(\left\{i\right\},\left\{j,k\right\},\left\{r_1\right\}, \ldots, \left\{r_{n-3}\right\}\right) .\nonumber 
\end{align}

\noindent Thus, $F(u,X,S)=F(P_{ijk})$ is a regular hexagon whose boundary consists of $F(P^1_{ijk})\cup\cdots\cup F(P^6_{ijk})$. Moreover, observe that in $\mathcal{P}\times \Sn$, we have $\left(P^m_{ijk},\overline{w}\right) \sim \left(P^{m+3}_{ijk},\overline{w}\right) $ for $m \in \left\{1,2,3\right\}$. When we project the boundary of $(F(P_{ijk}),\overline{w})$ onto $\Dd^2(\overline{w}(i),\overline{w}(j),\overline{w}(k))$ in $\Omega_n$, we obtain a loop made up of six edges that, with a suitable choice of starting points and orientations, can be written as \[\Dd^1(\overline{w}(i),\overline{w}(j))\Dd^1(\overline{w}(i),\overline{w}(k))\Dd^1(\overline{w}(j),\overline{w}(k))\Dd^1(\overline{w}(i),\overline{w}(j))^{-1}\Dd^1(\overline{w}(i),\overline{w}(k))^{-1}\Dd^1(\overline{w}(j),\overline{w}(k))^{-1}.\]

\noindent 
Now consider the case where $P(u, X)$ is equivalent to $P_{ijkl}$ with $i<j<k<l$ (the reasoning for $P_{ijkl}'$
and $P_{ijkl}''$
is completely analogous). A similar argument shows that the partition corresponds to a square, and that the cell $(F(P_{ijkl}),w)$ of $C[S]\times \Sn$ projects onto a 2-cell $\Dd^2(\overline{w}(i),\overline{w}(j),\overline{w}(k),\overline{w}(l))$ of $\Omega_n$, whose boundary with orientation is given by \[\Dd^1(\overline{w}(i),\overline{w}(j))\Dd^1(\overline{w}(k),\overline{w}(l))\Dd^1(\overline{w}(i),\overline{w}(j))^{-1}\Dd^1(\overline{w}(k),\overline{w}(l))^{-1}.\]
\noindent All the 2-cells of $\Omega_n$ are either of the form $\Dd^2(\overline{w}(i),\overline{w}(j),\overline{w}(k),\overline{w}(l))$ or of the form \\$\Dd^2(\overline{w}(i),\overline{w}(j),\overline{w}(k))$, and this concludes the proof.
\end{proof}
\noindent However, in the same theorem, the authors also claim that the space $\Omega_n$ is locally CAT(0), which would imply that $\Omega_n$ is a classifying space for $\PVB_n$. In Theorem \ref{FalseCat0}, we explain why the local CAT(0) property is false.\\
\\
\noindent All the definitions and results about CAT(0) spaces are based on  \cite{BriHaefl}. We summarize in the following theorem the main properties that will be used later.

\begin{thm}(\cite[Theorem  II.5.4]{BriHaefl}) \label{cat0properties} 
Let $X$ be a metric space that is complete and geodesic.
\begin{enumerate}
\item [(1)]If $X$ is CAT(0), then it is contractible.
\item[(2)] If $X$ is locally CAT(0), then the universal covering of $X$ is contractible, implying that $X$ is aspherical.
\item[(3)] Let $X$ be a piecewise Euclidean complex whose cells represent a finite number of isometry classes. Then $X$ is locally CAT(0) if and only if, for every vertex $x \in X$, every non\textendash homotopically trivial closed geodesic in $Lk\left(x,X\right)$ has length $\geq 2\pi$.
\end{enumerate}

\end{thm}

\begin{thm}\label{FalseCat0}
The complex $\Omega_3$ is not locally CAT(0).
\end{thm}

\begin{proof}
Since $\Omega_3$ is a piecewise Euclidean complex with a finite number of isometry classes, we will use Part (3) of Theorem \ref{cat0properties} and show that the geometric link of the only vertex $\omega^0$ in $\Omega_3$ contains a nontrivial loop with length $<2\pi$.\\
\\
\noindent Let $\mathcal{P}$ be the set of ordered partitions of $\{1,2,3\}$, where each $P=P(u,X)$ is associated with a coset $u\W_{X}$ with $X\subset S=\{s_1=(1\;2),s_2=(2\;3)\}$ and $u\in \mathfrak{S}_3$, $(\emptyset,X)$-minimal. The complex $\Omega_3=(C[S]\times \mathfrak{S}_3)/\sim$ is described as follows.\\
\\
\textbf{Description of the 0-skeleton.} There is a unique vertex denoted by $\omega^0$, which corresponds to the only class $\left(P\left(\id, \emptyset\right), \id\right)$ in the quotient.\\
\\
\textbf{Description of the 1-skeleton.} The complex has 6 oriented edges, denoted by\\ $\Dd^1(1,2), \Dd^1(1,3), \Dd^1(2,3), \Dd^1(2,1), \Dd^1(3,1), \Dd^1(3,2)$. These edges were previously described in the proof of Theorem \ref{fundoriginalbeer}.\\
\\
\textbf{Description of the 2-skeleton.} The complex has
 6 classes of 2-cells. In this case $n=3$, and the only parabolic subgroup $\W_{X}$ with $\left|X\right|=2$ is $\W=\mathfrak{S}_3$ itself. The ordered partition associated with the 2-dimensional face of $C[S]$ consists of just one block $\left(\left\{1,2,3\right\}\right)$. In $C[S]\times \mathfrak{S}_3$ we make six copies of such a cell, and each copy $(C[S],w)$, for $w \in \mathfrak{S}_3$, is the only element in its class with respect to $\sim$. Thus, $\Omega_3$ has six 2-cells of type $\left(P_{123},w\right)$ with $w \in \mathfrak{S}_3$ and $P_{123}$ as in the proof of Theorem \ref{fundoriginalbeer}. We denote the class of $(F(P_{123}),w)$ in $\Omega_3$ by $\Dd^2(w)$, for $w\in \mathfrak{S}_3$.\\
 \\
\noindent Each 2-cell is a regular hexagon, and its boundary is described by the relations like in (2)(a) of the proof of Theorem \ref{fundoriginalbeer}.
For example, two of these boundaries are:
\begin{align*}
&\Dd^1(1,2)\;\Dd^1(1,3)\;\Dd^1(2,3)\;\Dd^1(1,2)^{-1}\;\Dd^1(1,3)^{-1}\;\Dd^1(2,3)^{-1}
&\mbox{   for $\Dd^2(\id)$},\nonumber\\
&\Dd^1(1,3)\;\Dd^1(1,2)\;\Dd^1(3,2)\;\Dd^1(1,3)^{-1}\;\Dd^1(1,2)^{-1}\;\Dd^1(3,2)^{-1} &\mbox{   for $\Dd^2((23))$}.\nonumber
\end{align*}
Now, take any two indices $i,j \in \left\{1,2,3\right\}$, with $i\neq j$. The two endpoints of the loop $\Dd^1(i,j)$, which are identified with $\omega^0$, determine two vertices $y^+(i,j)$ and $y^-(i,j)$ of $Lk\left(\omega^0,\Omega_3\right)$. By convention, we fix the orientation from $y^-(i,j)$ to $y^+(i,j)$. All the vertices of $Lk\left(\omega^0,\Omega_3\right)$ are of this form. In particular, $Lk\left(\omega^0,\Omega_3\right)$ has 12 vertices. Since all the 2-cells of $\Omega_3$ are regular hexagons, the edges of $Lk\left(\omega^0,\Omega_3\right)$ all have length $\frac{2 \pi}{3}$. Therefore, $Lk\left(\omega^0,\Omega_3\right)$ has no minimal cycle of length $<2 \pi$ if and only if it is impossible to find two arcs in $Lk\left(\omega^0,\Omega_3\right)$ that share the same endpoints.\\
However, both $\Dd^2(\id)$ and $\Dd^2((23))$ determine arcs connecting $y^+(1,2)$ and $y^-(1,3)$, so these two vertices are linked by two different edges. Thus, there is a non homotopically trivial closed geodesic in $Lk(\omega^0,\Omega_3)$ with length $2\cdot \frac{2}{3}\pi=\frac{4}{3}\pi<2\pi$, meaning that $\Omega_3$ cannot be CAT(0).

\end{proof}

\subsection{The generalized BEER complex}\label{generalBeer}
In this subsection, we give the definition of the generalized BEER space $\Omega(\Gamma)$ and prove that its fundamental group is $\PVA[\Gamma]$. In the construction of Subsection \ref{originalBEERspace}, the faces of the permutohedron are in bijection with the ordered partitions of $\{1,2,\ldots,n\}$. As before, given $P\in \mathcal{P}$ an ordered partition, we denote by $F(P)$ the unique corresponding face of the permutohedron $C[S]$. Two different ordered partitions are equivalent if they represent the same unordered partition. Given this, two different faces $F_1=(F(P),w)$ and $F_2=(F(Q),w')$ in two copies of the permutohedron in $C[S]\times \Sn$ are equivalent if the partitions $P$ and $Q$ are equivalent, and if the permutations $w$ and $w'$ are such that for every block $P_k$ of $P$ (which is also a block of $Q$ by equivalence), for every $i,j\in P_k$, we have $w(i)<w(j)$ if and only if $w'(i)<w'(j)$. The complex $\Omega_n$ is the quotient of $C[S]\times \Sn$ under the defined equivalence relation given by the attaching map $(x,w)\sim k_{F_1,F_2}(x,w)$ for all $x\in F(P)$, where $k_{F_1,F_2}$ is a translation in the first component.\\
\\
\noindent The first intuition for defining the general space $\Omega(\Gamma)$ came from attempting to translate Definition \ref{deforiginalbeer} for $\Gamma$ any Coxeter graph of spherical type, with $(\W[\Gamma],S)$ its Coxeter system. In this case, we can replace the permutohedron by the usual Coxeter polytope $C[S]$, and the product space by $C[S]\times \W[\Gamma]$. Thanks to Lemma \ref{facescoxpol}, the faces $F(u,X,S)$ of $C[S]$ are in bijection with the cosets $u\W_{X}$ of $\W[\Gamma]$ with $X\subset S$ and $u\in \W[\Gamma]$, but as usual we choose $u$ to be the $(\emptyset,X)$-minimal representative of the coset. Now, take two cosets $u\W_{X}$ and $v\W_{Y}$ with $X,Y\subset S$ and $u,v\in \W[\Gamma]$, and consider their associated faces $F(u,X,S)$ and $F(v,Y,S)$ of $C[S]$. In the case $\Gamma=A_{n-1}$, the equivalence of the associated ordered partitions, as deduced from part (1) of Lemma \ref{stabandcosets}, can be rephrased saying that $F(u,X,S)\sim F(v,Y,S)$ if and only if $u\W_{X}u^{-1}=v\W_{Y}v^{-1}$. In general, if this equality of parabolic subgroups holds, we say that the cosets are \textit{equivalent}, and we write $u\W_{X}\sim v\W_{Y}$. \\
\\
\noindent Now, consider two pairs $(F(u,X,S),w)$ and $(F(v,Y,S),w')$ of $C[S]\times \W[\Gamma]$. The condition regarding the action of the permutation in the second component on the indices can be rephrased in the general setting in terms of the action of $\W[\Gamma]$ on its root system $\Phi[\Gamma]$. 
As recalled in Subsection \ref{APrefsubgrp}, the simple root vectors of $\W[A_{n-1}]=\Sn$ can be described as $\alpha_{i,i+1}=\frac{e_i-e_{i+1}}{\sqrt{2}}\in \mathbb{R}^{n+1}$, for all $i\in 1,\ldots,n-1$. The entire root system of the symmetric group can be described as 
\[
\Phi[A_{n-1}]=\{\alpha_{i,j}\mid 1\leq i\neq j\leq n\},
\]
\noindent
with $\alpha_{i,j}=\frac{e_i-e_{j}}{\sqrt{2}}$. In particular, if $i<j$, $\alpha_{i,j}$ is a positive root and $\alpha_{j,i}=-\alpha_{i,j}$ is a negative root. \\
\\
\noindent Recall that the action of $\Sn$ on $\Phi[A_{n-1}]$ is the action on the indices, i.e., for all $w\in \Sn$, $w(\alpha_{i,j})=\alpha_{w(i),w(j)}$. Specifically, if $i<j$ and $w(i)<w(j)$, then $w$ sends the positive root $\alpha_{i,j}$ to a positive root, while if $w(i)>w(j)$, then $w$ sends the positive root $\alpha_{i,j}$ to the negative root $\alpha_{w(i),w(j)}$. \\
\\
\noindent We now come back to the two pairs $(F(u,X,S),w)$ and $(F(u,X,S),w')$ in $C[S]\times \Sn$ that represent the same cell in the original BEER complex $\Omega_n$. Let $P=P(u,X)$ be the ordered partition associated with the face $F(u,X,S)$ of the permutohedron, as defined earlier. In particular $P=(P_1,\ldots,P_p)$, where the number of blocks $p$ is $n-|X|$. Since $(F(u,X,S),w)\sim(F(u,X,S),w')$, then for all $k\in \{1,\ldots,p\}$ and for all $i<j$ with $i,j\in P_k$, we must have $w(i)<w(j)$ if and only if $w'(i)<w'(j)$. In terms of action on the root system, this means that for all $k\in \{1,\ldots,p\}$ and for all $\alpha_{i,j}\in \Phi^+[A_{n-1}]$ with $i,j\in P_k$, we have that $w(\alpha_{i,j})\in \Phi^+[A_{n-1}]$ if and only if $w'(\alpha_{i,j})\in \Phi^+[A_{n-1}]$.\\
\\
\noindent We now aim to generalize the condition $\alpha_{i,j}\in \Phi^+[A_{n-1}]$ with $i,j\in P_k$ to arbitrary Coxeter graphs. Recall that the block $P_k$ of the partition $P=P(u,X)$ is given by:
\[
P_k=u\{i_{{k-1}}+1,\ldots,i_{k}\}=\{u(i_{{k-1}}+1),\ldots,u(i_{k})\},
\]
\noindent where $p=n-|X|$, $S\backslash X=\{s_{i_1},\ldots,s_{i_{p-1}}\}$ with $i_1< \cdots < i_{p-1}$, and $i_0=0$, $i_p=n$. For all $k=1,\ldots,p$, define $X_k=\{s_{i_{{k-1}}+1},\ldots,s_{i_{k}-1}\}$, with the convention that $X_k=\emptyset$ if $i_{k}=i_{k-1}+1$. Observe that $X=X_1\sqcup\cdots\sqcup X_p$ and that the block $P_k$ is stabilized by the group $u\W_{X_k}u^{-1}$, which is a parabolic subgroup of $\Sn$. Therefore, taken a positive root $\alpha_{i,j}$, the condition $i,j\in P_k$ means that 
\[
i,j\in \{u(i_{{k-1}}+1),\ldots,u(i_{k})\},
\]
which is equivalent to saying that
\[
\alpha_{i,j}\in u(\Phi_{X_k})\cap \Phi^+[A_{n-1}],
\]
where $u(\Phi_{X_k})$ is the root system of the parabolic group $u\W_{X_k}u^{-1}$ seen as a Coxeter group. Namely, $u(\Phi_{X_k})=u\{v(\alpha_{s})\mid v\in \W_{X_k}, s\in X_k\}$. We want the above condition to hold for all $k\in \{1,\ldots,p\}$. Observe that \begin{align*}
    u\W_Xu^{-1}=\Stab(P(u,X))=\Stab((P_1,\ldots,P_p))=\\=\Stab(P_1)\times \cdots\times \Stab(P_p)=u\W_{X_1}u^{-1}\times \cdots \times u\W_{X_p}u^{-1}.
\end{align*}
\noindent Hence,
\[u(\Phi_{X})=u(\Phi_{X_1})\sqcup \cdots \sqcup u(\Phi_{X_p}).\]
\noindent
In conclusion, two pairs $(F(u,X,S),w)$ and $(F(v,Y,S),w')$ represent the same face in $\Omega_n$ if and only if $u\W_Xu^{-1}=v\W_Yv^{-1}$ and, for all roots $\alpha_{i,j}\in u(\Phi_{X})\cap \Phi^+[A_{n-1}]$,
\[
w(\alpha_{i,j})\in \Phi^+[A_{n-1}] \;\Longleftrightarrow\; w'(\alpha_{i,j})\in\Phi^+[A_{n-1}].
\]
\noindent We aim now to consider a general Coxeter graph $\Gamma$ of spherical type, with the product space $C[S]\times \W[\Gamma]$.
We say that $(F(u,X,S),w)$ is \textit{equivalent} to $(F(v,Y,S),w')$, and we write $(F(u,X,S),w)\approx (F(v,Y,S),w')$, if $u\W_{X}\sim v\W_{Y}$ and if for any $\beta \in u(\Phi_{X})=v(\Phi_{Y})$, we have that $w(\beta)\in \Phi^+[\Gamma] \Longleftrightarrow w'(\beta)\in \Phi^+[\Gamma]$.
Suppose $F_1 = (F(u,X,S),w)\approx F_2 = (F(v,Y,S),w')$.
Then, as in the case of $\Omega_n$, one can check that the translation by vector $ v(o) - u(o)$ induces an isometry $k_{F_1,F_2} : F_1 \to F_2$.
Define $\Omega' (\Gamma)$ as the quotient of $C [S] \times \W [\Gamma]$ by the relation which, for two equivalent faces $F_1$ and $F_2$ of $C [S] \times \W [\Gamma]$, identifies every $x \in F_1$ with $k_{F_1, F_2} (x) \in F_2$.
Now, it can be shown that the fundamental group of $\Omega' (\Gamma)$ is indeed ${\rm PVA} [\Gamma]$.\\
\\
\noindent However, unlike in the case of $\Gamma = A_{n-1}$, to get a covering of $\Omega' (\Gamma)$ which coincides with a covering of $\Sigma (\Gamma)$, more cells need to be added.
Moreover, such a construction holds only for some Coxeter graphs of spherical type, and we aim for a construction that applies to any Coxeter graph.
Thus, we realized that to consider a cell for each parabolic subgroup of $\W[\Gamma]$ is insufficient to get the desired complex. Therefore, we moved on to the notion of almost parabolic reflection subgroups to include more cells, allowing us to define the appropriate space.\\
\\
\noindent We will give now a description of the generalized BEER complex. As with the Salvetti complex, we will be able to define the BEER complex even for graphs $\Gamma$ which are not necessarily of spherical type or finite, simply by attaching one cell for each AP spherical-type reflection subgroup.\\
\\
\noindent Before defining $\Omega(\Gamma)$, we remark on some properties of the AP spherical-type sets of roots. Recall that $\Rf=\{\cX\subset  \Phi[\Gamma] \; |\;  \cX \mbox{ is AP and of spherical type}\}.$
\begin{rmk}
    Let $\Gamma$ be a Coxeter graph and $\cY\subset\cX\subset\Phi[\Gamma]$, with $\cX\in \Rf$.  We know that $\W_{\cY}$ is a standard parabolic subgroup of $\W_{\cX}$, generated by the reflections $R_{\cY}\subset R_{\cX}$. Let $u\in \W_{\cX}$ be an $(\emptyset,R_{\cY})$-minimal element. Then $u(\cY)$ is an AP and spherical-type set of roots. In fact, $\cY$ is AP and of spherical type, and since $\W_{\cX}$ acts on $V_{\cX}$ by invertible linear transformations, the set of roots $u(\cY)$ is linearly independent whenever $\cY$ is. Clearly, property (AP2) is satisfied, and being contained in $\W_{\cX}$, the reflection subgroup associated to $u(\cY)$ is still of spherical type.
\end{rmk}

\noindent By the presentation for AP reflection subgroups obtained in Theorem \ref{refAP}, Equation \ref{presrefsub}, we have that the AP spherical-type reflection subgroup $\W_{u(\cY)}$ is isomorphic to $u\W_{\cY}u^{-1}$. The latter is, in particular, a parabolic subgroup of $\W_{\cX}$ of rank $|\cY|$. We have seen in Equation \ref{coxpolrefsub} how to describe the Coxeter polytope relative to a spherical-type AP reflection subgroup. Now we investigate the link between $C[\cY]$ and $C[u(\cY)]$.

\begin{prop}\label{uchi}
    Consider $\cY\subset\cX\subset \Phi[\Gamma]$, with $\cX$ being AP and of spherical type. Let $u\in \W_{\cX}$ be an $(\emptyset,R_{\cY})$-minimal element. Then there is a linear isomorphism $V_{\cY}\longrightarrow V_{u(\cY)}$ that sends  $C[\cY]$ isometrically to $C[u(\cY)]$ and is induced by the action of the element $u$ on $V$.
\end{prop}
\begin{proof}
    Recall that $V_{\cY}=\bigoplus_{\beta\in \cY}\mathbb{R}\cdot \beta$ and $V_{u(\cY)}=\bigoplus_{\beta\in \cY}\mathbb{R}\cdot u(\beta)$. The fundamental points are $o_{\cY}=\sum_{\beta\in \cY}\beta_{\cY}^{*}$ and $o_{u(\cY)}=\sum_{\beta\in u(\cY)}\beta_{u(\cY)}^{*}$. Since $u$ preserves the bilinear form, it is easy to see that $u(o_{\cY})=o_{u(\cY)}$. Moreover, $u:V_{\cY}\longrightarrow V_{u(\cY)}$ sends a basis to a basis. For all $z\in \W_{\cY}$, we have then that the vertex $z(o_{\cY})$ of $C[\cY]$ is sent by $u$ to $u(z(o_{\cY}))=uzu^{-1}u(o_{\cY})=uzu^{-1}(o_{u(\cY)})$, which is a vertex of $C[u(\cY)]$ since $uzu^{-1}\in u\W_{\cY}u^{-1}=\W_{u(\cY)}$. Summarizing:
    \begin{align*}
        u:V_{\cY}&\longrightarrow V_{u(\cY)},\\
        C[\cY]&\longrightarrow C[u(\cY)],\\
        z(o_{\cY})&\longmapsto uzu^{-1}(o_{u(\cY)}).
    \end{align*}
    Observe also that if $\cZ\subset \cY$ and $v\in \W_{\cY}$ is $(\emptyset,R_{\cZ})$-minimal, then $uvu^{-1}\in \W_{u(\cY)}$ and it is $(\emptyset,R_{u(\cZ)})$-minimal.
\end{proof}

\begin{defn}\label{defbeer}
Let $\Gamma$ be a Coxeter graph and $\Phi[\Gamma]$ be its root system. The BEER complex, denoted by $\Omega(\Gamma)$, is the CW-complex defined as follows.\\
\\
\textbf{Description of the 0-skeleton.}
The 0-skeleton $\Omega^{0}(\Gamma)$ consists of just one point, $\omega^0$.\\
\\
\textbf{Description of the 1-skeleton.} For every root $\beta\in \Phi[\Gamma]$, take a copy $\mathbb{A}(\beta)$ of the Coxeter polytope $C[\beta]=conv\{z(o_{\beta})\;|\;z\in \W_{\beta}=\{\id,r_{\beta}\}\}$, which we identify with $C[\beta]$ via an isometry $f_{\beta}:C[\beta]\longrightarrow \mathbb{A}(\beta)$. Observe that $o_{\beta}=\beta_{\beta}^*=\beta$, so $C[\beta]$ is the segment $[-\beta,\beta]$, naturally oriented from $\beta$ to $-\beta$. This orientation is transmitted to $\mathbb{A}(\beta)$ via the isometry $f_{\beta}$. It is important to remark that, even if $C[\beta]=C[-\beta]$, we have $\mathbb{A}(-\beta)\neq \mathbb{A}(\beta)$. For any $\beta\in \Phi[\Gamma]$, we define a map $\lambda^1_{\beta}:\partial \mathbb{A}(\beta)\longrightarrow \Omega^0(\Gamma)$ by setting $\lambda^1_{\beta}(f_{\beta}(\beta))=\lambda^1_{\beta}(f_{\beta}(-\beta))=\omega^0$. Define the 1-skeleton $\Omega^1(\Gamma)$ to be the disjoint union of $\{\omega^0\}$ and all the cells $\mathbb{A}(\beta)$, under the identifications that attach $x$ to $\lambda^1_{\beta}(x)$ for all $\beta\in \Phi[\Gamma]$ and for all $x\in \partial\mathbb{A}(\beta)$. The inclusion of $\mathbb{A}(\beta)$ in $\{\omega^0\}\sqcup (\bigsqcup_{\beta'\in \Phi[\Gamma]}\mathbb{A}(\beta'))$, composed with the quotient projection onto $\Omega^1(\Gamma)$, gives a characteristic map $\Lambda^1_{\beta}: \mathbb{A}(\beta)\longrightarrow \Omega^1(\Gamma)$ whose image is denoted by $D^1(\beta)$. The map $\Lambda^1_{\beta}$ extends $\lambda^1_{\beta}$ and is a homeomorphism from the interior of the segment $\mathbb{A}(\beta)$ to $D^1(\beta)\,\backslash\, \{\omega^0\}$.  Thus, as a set, $\Omega^1(\Gamma)=\{\omega^0\}\bigcup\left(\bigcup_{\beta\in \Phi[\Gamma]}D^1(\beta)\right)$, where each $D^1(\beta)$ is a closed segment with both ends attached to the only vertex $\omega^0$.\\
    \\
\textbf{Description of the $k$-skeleton.} Suppose that the $(k-1)$-skeleton of $\Omega(\Gamma)$ is defined for $k\geq 2$. Namely, for each $\cY\subset \Phi[\Gamma]$ AP and of spherical type of rank $h\leq k-1$, we have a copy $\mathbb{D}(\cY)$ of the Coxeter polytope $C[\cY]$, which we identify with $C[\cY]$ via an isometry $f_{\cY}:C[\cY]\longrightarrow \mathbb{D}(\cY)$, and we have an $h$-cell $D^h(\cY)$ in $\Omega^{k-1}(\Gamma)$ 
which is the image of a continuous characteristic map $\Lambda^h_{\cY}\,:\, \mathbb{D}(\cY)\longrightarrow D^h(\cY)$, establishing a homeomorphism between the interior of $\mathbb{D}(\cY)$ and $D^h(\cY)\,\backslash\, \partial D^h(\cY)$. If $|\cY|=1$, then there exists $\beta\in \Phi[\Gamma]$ such that $\cY=\{\beta\}$, and we assume $\mathbb{D}(\cY)=\mathbb{A}(\beta)$. If $\cY=\emptyset$, then we adopt the following conventions: 
    \[
    V_{\emptyset}=\{0\},\qquad \W_{\emptyset}=\{\id\},\qquad o_{\emptyset}=0,\qquad C[\emptyset]=\{0\}.
    \]
    \noindent We also set $\mathbb{D}(\emptyset)=\{0\}$. Take $f_{\emptyset}:C[\emptyset]\longrightarrow \mathbb{D}(\emptyset)$ to be the identity map, and define $\lambda_{\emptyset}^0:\mathbb{D}(\emptyset)\longrightarrow\Omega^0(\Gamma)$ to be the map sending $0$ to the unique vertex $\omega^0$. Thus $D^0(\emptyset)=\{\omega^0\}$.\\
\\
 \noindent Take now, for each $\cX\in \Rf_k$, a copy $\mathbb{D}(\cX)$ of the Coxeter polytope $C[\cX]$, identified with $C[\cX]$ via an isometry $f_{\cX}:C[\cX]\longrightarrow \mathbb{D}(\cX)$. Let $\cY\subset \cX$ with $|\cY|=h\leq k-1$ and $u\in \W_{\cX}$ an $(\emptyset,R_{\cY})$-minimal element. Since $u(\cY)\in \Rf$ and the rank $|u(\cY)|=h\leq k-1$, there is a cell $D^h(u(\cY))$ in $\Omega^{k-1}(\Gamma)$ with its characteristic map $\Lambda^h_{u(\cY)}:\mathbb{D}(u(\cY))\longrightarrow D^h(u(\cY))$.
    Define now $\chi^k_{u,\cY,\cX}:F(u,\cY,\cX)\longrightarrow \Omega^{k-1}(\Gamma)$ to be the following composition of maps:
\[
F(u,\cY,\cX)\xrightarrow[u^{-1}]{} F(\id,\cY,\cX)\xrightarrow[\tau(o_{\cY}-o_{\cX})]{}C[\cY]\xrightarrow[u]{}C[u(\cY)]\xrightarrow[f_{u(\cY)}]{}\mathbb{D}(u(\cY))\xrightarrow[\Lambda^h_{u(\cY)}]{}D^h(u(\cY)),
\]
where $\tau(o_{\cY}-o_{\cX})$ denotes the translation that commutes with the action of $\W_{\cY}$. In other words, the face $F(u,\cY,\cX)$ is first identified with $C[\cY]$, then transported by $u$ to $C[u(\cY)]$, and finally attached to the cell $D^h(u(\cY))$. Observe that $u\circ \tau(o_{\cY}-o_{\cX})\circ u^{-1}=\tau(u(o_{\cY}-o_{\cX}))$, so we can rewrite the expression of $\chi^k_{u,\cY,\cX}$ as    
\[\chi^k_{u,\cY,\cX}=\Lambda^h_{u(\cY)}\circ f_{u(\cY)} \circ\tau(u(o_{\cY}-o_{\cX})).\]
Recall that if $F(u,\cY,\cX)$ is a face of $C[\cX]$ with $\cY$ and $u$ as before, the faces of $F(u,\cY,\cX)$ are of the form $F(uv,\cZ,\cX)=conv\{uvw(o_{\cX})\,|\,w\in \W_{\cZ}\}$ with $\cZ\subset \cY$ and $v\in \W_{\cY}$, $(\emptyset,R_{\cZ})$-minimal.

\begin{lem}\label{lambdawelldefined}
    Let $\cX,\cY,\cZ \in \Rf$ be such that $\cZ\subset \cY\subset \cX$, with $|\cX|=k,|\cY|=h,|\cZ|=l$. Take $u\in \W_{\cX}$, $u$ $(\emptyset,R_{\cY})$-minimal and $v\in \W_{\cY}$, $v$ $(\emptyset,R_{\cZ})$-minimal. We have that $F(uv,\cZ,\cX)\subset F(u,\cY,\cX)\subset C[\cX]$. Then, the restriction of $\chi^k_{u,\cY,\cX}$ to $F(uv,\cZ,\cX)$ coincides with $\chi^k_{uv,\cZ,\cX}$. 
\end{lem}

\noindent We will prove this lemma afterwards. Therefore, we can define a continuous map $\chi^k_{\cX}:\partial C[\cX]\longrightarrow \Omega^{k-1}(\Gamma)$ as follows. Let $x\in \partial C[\cX]$. We choose $\cY\subset \cX$ and $u\in \W_{\cX}$ an $(\emptyset, R_{\cY})$-minimal element such that $x\in F(u,\cY,\cX)$ and we set $\chi^k_{\cX}(x)=\chi^k_{u,\cY,\cX}(x)$. Thanks to the previous lemma, this definition does not depend on the choice of $F(u,\cY,\cX)$, and the map $\chi^k_{\cX}$ is continuous. Then, define $\lambda^k_{\cX}:\partial \mathbb{D}(\cX)\longrightarrow \Omega^{k-1}(\Gamma)$ by setting
\[
\lambda^k_{\cX}=\chi^k_{\cX}\circ (f^{-1}_{\cX})|_{\partial \mathbb{D}(\cX)}.
\]
The map $\lambda^{k}_{\cX}$ is continuous. Now for all $\cX\in \Rf_k$ and for all $x\in \partial \mathbb{D}(\cX)$, we write $x\sim \lambda^k_{\cX}(x)$. The $k$-skeleton of the cell complex will be defined as  
    \[
    \Omega^{k}(\Gamma):=\bigslant{\left(\Omega^{k-1}(\Gamma)\bigsqcup \left( \bigsqcup_{\cX \in \Rf_k}\mathbb{D}(\cX)\right)\right)}{\sim}.
    \]
    For every $\cX\in \Rf_k$, the cell $\mathbb{D}(\cX)$ has a natural characteristic map $\Lambda^k_{\cX}$ to $\Omega^k(\Gamma)$, whose image is denoted by $D^k(\cX)$. This $\Lambda^k_{\cX}$ is given by the composition of the inclusion of $\mathbb{D}(\cX)$ into the disjoint union and the quotient projection.\\
    \\
    \textbf{Description of the complex.} We set $\Omega(\Gamma)=\bigcup_{k=0}^{\infty}\Omega^{k}(\Gamma)$, endowed with the weak topology.
\end{defn}

\noindent The reader may have noticed the similarity between complexes $\Sigma(\Gamma)$ and $\Omega(\Gamma)$.
Their sets of $k$-cells bijectively correspond to each other (and are parametrized by
the sets $\Rf_k$ in both cases), however the attaching maps are different. In $\Sigma(\Gamma)$, a face $F(u,\cY,\cX)$ of the cell indexed by $\cX$ is attached to the cell indexed
by $\cY$. In $\Omega(\Gamma)$, the same face is attached instead to the cell indexed by $u(\cY)$. This
distinction is already visible in dimension 2. Moreover, while each
2-cell of $\Sigma(\Gamma)$ is attached to the union of just two 1-cells, we will see in the proof
of Theorem \ref{pi1beer} that in the complex $\Omega(\Gamma)$ each 2-cell corresponding to $\cX=\{\beta,\gamma\}$
will be attached to the union of $m$ 1-cells, where $m=\widehat{m}_{\beta,\gamma}$ (see Figures \ref{2cellsalfig} and \ref{2cellbeerfig} for
comparison).While each
2-cell of $\Sigma(\Gamma)$ is attached to the union of just two 1-cells, we will see in the proof
of Theorem \ref{pi1beer} that in the complex $\Omega(\Gamma)$ each 2-cell corresponding to $\cX=\{\beta,\gamma\}$
will be attached to the union of $m$ 1-cells, where $m=\widehat{m}_{\beta,\gamma}$ (see Figures \ref{2cellsalfig} and \ref{2cellbeerfig} for
comparison).\\

\begin{proof}[Proof of Lemma \ref{lambdawelldefined}]
The map $\chi^k_{uv,\cZ,\cX}$ is described in detail as follows:
\begin{align*}
    F(uv,\cZ,\cX)\xrightarrow[]{\tau(uv(o_{\cZ}-o_{\cX}))}C[uv(\cZ)]\xrightarrow[]{f_{uv(\cZ)}}\mathbb{D}(uv(\cZ))\xrightarrow[]{\Lambda^l_{uv(\cZ)}} D^l(uv(\cZ)).
\end{align*}
The restriction of $\chi^k_{u,\cY,\cX}$ to $F(uv,\cZ,\cX)$ is described as
\begin{align*}
    F(uv,\cZ,\cX)\xrightarrow[]{u^{-1}}F(v,\cZ,\cX) \xrightarrow[]{\tau(o_{\cY}-o_{\cX})}F(v,\cZ,\cY)\xrightarrow[]{u}F(uv,\cZ,\cY)=\\
    =F(uvu^{-1},u(\cZ),u(\cY))\subset C[u(\cY)]\xrightarrow[]{f_{u(\cY)}}\mathbb{D}(u(\cY))\xrightarrow[]{\Lambda^h_{u(\cY)}} \Omega^h(\Gamma).
\end{align*}
Since $F(uvu^{-1},u(\cZ),u(\cY))\subset \partial C[u(\cY)]$, we have that $f_{u(\cY)}(F(uvu^{-1},u(\cZ),u(\cY)))\subset \partial \mathbb{D}(u(\cY))$. By induction, the restriction of $\Lambda^h_{u(\cY)}$ to the boundary of the cell is given by $\lambda^h_{u(\cY)}=\chi^h_{u(\cY)}\circ (f^{-1}_{u(\cY)})|_{\partial{\mathbb{D}(u(\cY))}}$. Thus, we have \[\Lambda^h_{u(\cY)}\circ f_{u(\cY)} |_{F(uvu^{-1},u(\cZ),u(\cY))}=\chi^h_{uvu^{-1},u(\cZ),u(\cY)},\]
\noindent which is the composition of the maps
\begin{align*}
F(uvu^{-1},u(\cZ),u(\cY))\xrightarrow[]{uv^{-1}u^{-1}}F(\id,u(\cZ),u(\cY))\xrightarrow[]{\tau(o_{u(\cZ)}-o_{u(\cY)})}C[u(\cZ)]\xrightarrow[]{}\\\xrightarrow[]{uvu^{-1}} C[uv(\cZ)]
\xrightarrow[]{f_{uv(\cZ)}}\mathbb{D}(uv(\cZ))\xrightarrow[]{\Lambda^l_{uv(\cZ)}}D^l(uv(\cZ)).
\end{align*}
By combining the previous equations, we obtain that the restriction of  $\chi^k_{u,\cY,\cX}$ to $F(uv,\cZ,\cX)$ is the composition

\begin{equation}\label{eqchi}
    \Lambda^l_{uv(\cZ)}\circ f_{uv(\cZ)}\circ uvu^{-1}\circ \tau(o_{u(\cZ)}-o_{u(\cY)})\circ uv^{-1}u^{-1}\circ u\circ \tau(o_{\cY}-o_{\cX})\circ u^{-1}.
\end{equation}
\noindent As seen in Proposition \ref{uchi}, since $u\in \W_{\cX}$ is $(\emptyset,R_{\cY})$-minimal (and thus, $(\emptyset,R_{\cZ})$-minimal) we get that $u(o_{\cY})=o_{u(\cY)}$ and $u(o_{\cZ})=o_{u(\cZ)}$. Therefore, we get $u^{-1}\circ\tau(o_{u(\cZ)}-o_{u(\cY)})\circ u=\tau(u^{-1}u(o_{\cZ}-o_{\cY}))=\tau(o_{\cZ}-o_{\cY})$. In Equation \ref{eqchi}, we can cancel the sequence $u^{-1}\circ u$ in the middle and rewrite the entire equation as $\Lambda^l_{uv(\cZ)}\circ f_{uv(\cZ)}\circ uv\circ \tau(o_{\cZ}-o_{\cY})\circ v^{-1}\circ\tau(o_{\cY}-o_{\cX})\circ u^{-1}$. Now, exploiting the fact that $v^{-1}$, being an element of $\W_{\cY}$, commutes with the translation $\tau(o_{\cY}-o_{\cX})$, we can rewrite the restriction of $\chi^k_{u,\cY,\cX}$ to $F(uv,\cZ,\cX)$ as $\Lambda^l_{uv(\cZ)}\circ f_{uv(\cZ)}\circ uv\circ \tau(o_{\cZ}-o_{\cY})\circ \tau(o_{\cY}-o_{\cX})\circ v^{-1}u^{-1}$, which simplifies to $\Lambda^l_{uv(\cZ)}\circ f_{uv(\cZ)}\circ \tau(uv(o_{\cZ}-o_{\cX}))$ and coincides with $\chi^k_{uv,\cZ,\cX}$.

\end{proof}

\medskip

\noindent Therefore, the BEER complex $\Omega(\Gamma)$ is well-defined. In Subsection \ref{eqialityofspaces} we show that, for $\Gamma=A_{n-1}$, $\Omega(\Gamma)$ coincides with the space $\Omega_n$ described in Subsection \ref{originalBEERspace}.  \\
\\
\noindent We now proceed to study the fundamental group of the generalized BEER complex $\Omega(\Gamma)$.\\
We aim to show that the presentation of $\pi_1(\Omega(\Gamma))$ coincides with the one of the pure virtual Artin group $\PVA[\Gamma]$ given in  \cite{BellParThiel}, which we recalled in Theorem \ref{prespva}. We adopt here the notation already used in Section \ref{virtualartingroups}, Equation \ref{betakappa}. \\
 \begin{thm}\label{pi1beer}
     Let $\Gamma$ be a Coxeter graph, $\PVA[\Gamma]$ its associated pure virtual Artin group, and $\Omega(\Gamma)$ the BEER space.
     Then
     \[ \pi_1(\Omega(\Gamma))\cong \PVA[\Gamma].\]
 \end{thm}
\begin{proof}
    We can restrict our study to the 2-skeleton $\Omega^2(\Gamma)$. Let us analyze its cellular structure.\\
    \\
    \textbf{Description of the 0-skeleton.} There is just one vertex $\omega^0$.\\
    \\
    \textbf{Description of the 1-skeleton.} For each root $\beta\in \Phi[\Gamma]$, there is an oriented edge $D^1(\beta)$ with both endpoints identified with $\omega^0$. The fundamental group is then generated by the 1-cells $D^1(\beta)$, thus there is one generator for each root in $\Phi[\Gamma]$.\\
    \\
    \noindent Now, let us examine how the 2-cells of $\Omega^2(\Gamma)$ are formed, in order to obtain the relations to add in the presentation of the fundamental group.\\
\\
\textbf{Description of the 2-skeleton.} The faces of $\Omega^2(\Gamma)$ are of the form $D^2(\cX)$, where $\{\beta,\gamma\}=\cX\subset\Phi[\Gamma]$ is an AP set of roots of spherical type. In particular, thanks to Remark \ref{card2parabolic}, we know that $\cX$ is parabolic. Thus, it must be of the form $\cX=w(\Pi_X)$ with $X=\{s,t\}\subset S$ and  $m_{s,t}\neq \infty$, which from now on we will simply denote by $m$. We set $\beta=w(\alpha_s)$, $\gamma=w(\alpha_t)$.\\
\\
By definition, the interior of $D^2(\cX)$ is homeomorphic to the interior of a copy $\mathbb{D}(\cX)$ of the Coxeter polytope $C[\cX]$, which is a regular $2m$-gon, whose vertices are in bijection with the elements in the parabolic subgroup $w\W_{X}w^{-1}=\W_{\cX}$:
    \[\W_{\cX}=\{\id, wsw^{-1}, wtw^{-1}, wstw^{-1},\ldots, w\Prod_{R}(s,t;m)w^{-1}\}=\{\id, r_{\beta}, r_{\gamma}, r_{\beta}r_{\gamma}, \ldots , \Prod_R(r_{\beta},r_{\gamma};m)\}.\]
    \noindent As noted in Remark \ref{coxpolitopeofaparabolic},
 the polytope $C[\cX]$ is the image of $C[X]$ (shown in Figure \ref{coxpolfig}) under the action of $w$.\\
 \\
\noindent As we see in Figure \ref{figcoxpolparabsubgr}, the edges that bound this polygon are represented by the expressions $F(u,\cY,\cX)$, where $u\in \W_{\cX}$ $(\emptyset,\cY)$-minimal and $\cY\subset \cX$, namely $\cY=\{\delta\}$ for $\delta\in \{\beta,\gamma\}$. In Figure
\ref{figcoxpolparabsubgr}, to ease the notation, we omit the point $o_{\cX}$ at the vertices.\\
\\
\noindent The next aim is to show that the 1-cells in the  boundary of $\Lambda^2_{\cX}(\mathbb{D}(\cX))=D^2(\cX)$ are $D^1(\beta_k)$, with $\beta_k$ as in Equation \ref{betakappa} and $k\leq m$.\\
\\
\noindent First, recall that the edge $F(u,\delta,\cX)$ of $C[\cX]$ is sent in the BEER space to the 1-cell $D^1(u(\delta))$ through $\chi^2_{u,\delta,\cX}$.
Consider first the polytope $C[\cX]$ shown in Figure \ref{figcoxpolparabsubgr}, which is isometric to the model cell $\mathbb{D}(\cX)$ through a map $f_{\cX}$.\\
\\
\noindent Starting from the vertex labeled by $\id$ in Figure \ref{figcoxpolparabsubgr}, we move counterclockwise and first find the edge $F(\id,\beta,\cX)$ sent by $\chi^2_{\id,\beta,\cX}$ to the arc $D^1(\beta)$, where $\beta=w(\alpha_s)$. The second edge, $F(r_{\beta},\gamma,\cX)$, is sent to $D^1(r_{\beta}(\gamma))$ where
 $ws(\alpha_t)=wsw^{-1}(\gamma)=r_{\beta}(\gamma)$, and the third is sent to $D^1(r_{\beta}r_{\gamma}(\beta))$ with $wstw^{-1}(\beta)=r_{\beta}r_{\gamma}(\beta)$. Continuing in this fashion, we find that the $k$-th edge for $k\leq m$ is associated with the root $\Prod_R(r_{\beta},r_{\gamma};k-1)(\beta)$ if $k$ is odd, and with $\Prod_R(r_{\gamma},r_{\beta};k-1)(\gamma)$ if $k$ is even. Thus, the $k$-th edge for $k\leq m$ is sent to the cell $D^1(\beta_k)$ with the notation as in Equation \ref{betakappa}.\\
 \\
  \noindent For the second half of the perimeter, we move clockwise from the vertex labeled by $\id$ and find the edges $F(\id,\gamma,\cX)$, $F(r_{\gamma},\beta,\cX)$, $F(r_{\gamma}r_{\beta},\gamma,\cX)$, and so on, which are, respectively, sent to $D^1(\gamma), D^1(r_{\gamma}(\beta))$, $D^1(r_{\gamma}r_{\beta}(\gamma))$ etc. by the maps $\chi^2_{\id,\gamma,\cX},\chi^2_{r_{\gamma},\beta,\cX}$ etc. Thus, the $h$-th edge with $h\leq m$ clockwise is sent to the cell $D^1(\gamma_h)$. By the same reasoning as in the proof of
Lemma 2.5 of \cite{BellParThiel}, we have that $\gamma_{m+1-k}=\beta_k$, so two opposite edges in the polygon are sent to the same 1-cell in $\Omega(\Gamma)$. Consequently, any edge in $\partial D^2(\cX)$ is of the form $D^1(\beta_k)$ for $k\leq m$.
 \begin{figure}[ht]
\centering
\begin{minipage}{8.5cm}
  \centering
  \includegraphics[width=8.5cm]{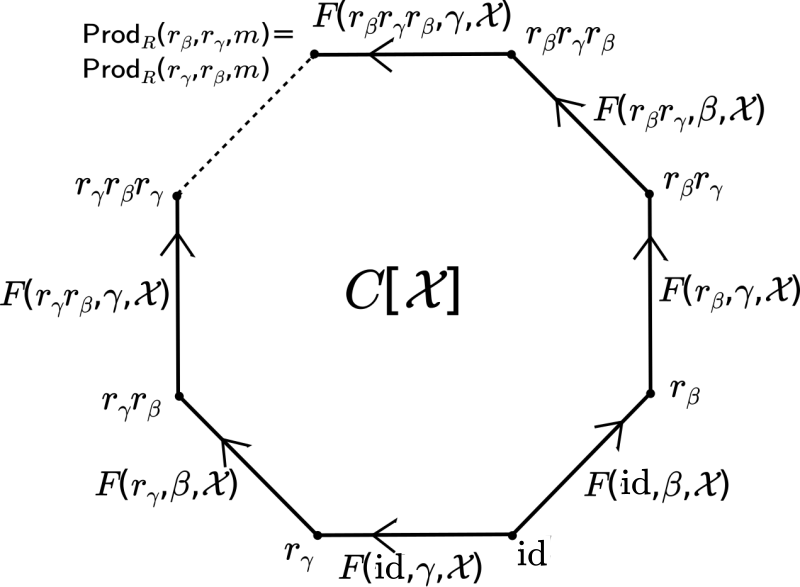}
     \caption{The Coxeter polytope of the parabolic subgroup $\W_{\cX}$ with $\cX=w(\Pi_{\{s,t\}})$ and $m=4$.}
\label{figcoxpolparabsubgr}
\end{minipage}
\qquad
\begin{minipage}{7.5cm}
  \centering
\includegraphics[width=7.5cm]{ 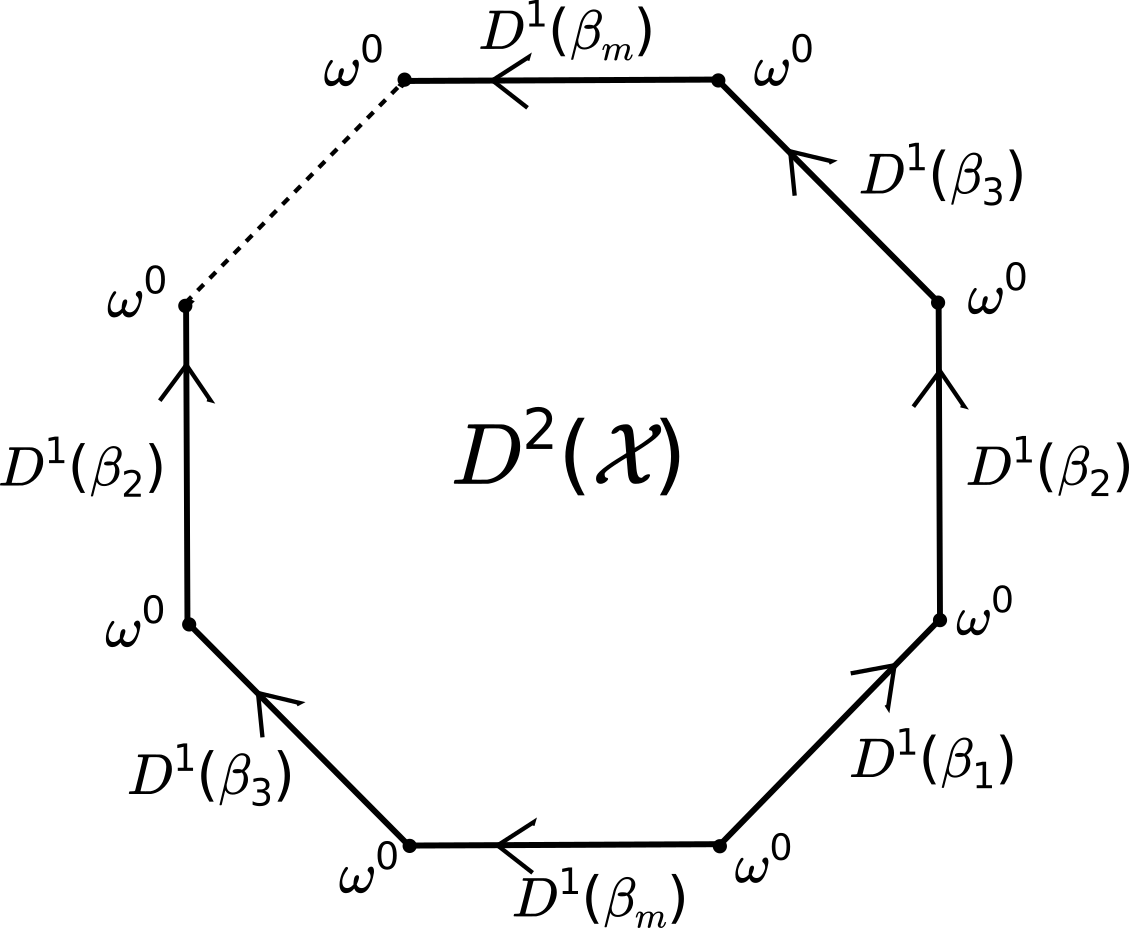}
  \caption{The 2-cell $D^2(\cX)$ of the BEER complex $\Omega(\Gamma)$.}
  \label{2cellbeerfig}
\end{minipage}
\end{figure}
\noindent The orientation of the boundary of $C[\cX]$ induces an orientation on its copy $\mathbb{D}(\cX)$, which is preserved in $D^2(\cX)\subset \Omega^2(\Gamma)$. In particular, thanks to Remark \ref{orientedgesparabolic} we know that the edge $F(u,\delta,\cX)$ is oriented from the vertex $u$ to $ur_{\delta}$. As we can see in Figure \ref{2cellbeerfig}, the boundary of $D^2(\cX)$ is given by the oriented path 
   \begin{equation}
       \partial D^2(\cX)=D^1(\beta_1)D^1(\beta_2)\dots D^1(\beta_m)(D^1(\beta_1))^{-1}(D^1(\beta_2))^{-1}\dots (D^1(\beta_m))^{-1}.
   \end{equation}
For all $\beta,\gamma\in \Phi[\Gamma]$ such that $\beta=w(\alpha_s)$, $\gamma=w(\alpha_t)$ with $w\in \W[\Gamma]$, $s,t\in S$ and $m_{s,t}\neq \infty$, we set $\beta_k,\gamma_k$ as before and $Z(D^1(\gamma),D^1(\beta);m_{s,t})=D^1(\beta_{m_{s,t}})\dots D^1(\beta_2) D^1(\beta_1)$. As in Subsection \ref{virtualartingroups}, we have \[Z(D^1(\beta),D^1(\gamma);m_{s,t})=D^1(\gamma_{m_{s,t}})\dots D^1(\gamma_2) D^1(\gamma_1),\]
\noindent which is the reverse word of $Z(D^1(\gamma),D^1(\beta);m_{s,t})$.\\
\\
\noindent In conclusion, in the presentation of $\pi_1(\Omega(\Gamma))$, for all $\cX\in \Rf$ with $|\cX|=|\{\beta=w(\alpha_s),\gamma=w(\alpha_t)\}|=2$, we must add the relation
\[D^1(\beta_1)D^1(\beta_2)\cdots D^1(\beta_{m_{s,t}})=D^1(\beta_{m_{s,t}})\dots D^1(\beta_2) D^1(\beta_1).\]
\noindent Therefore, the presentation of the fundamental group of the BEER complex is 
\begin{equation}
\begin{split}
    \pi_1(\Omega(\Gamma))=\left\langle D^1(\beta), \beta \in \Phi[\Gamma]\;|\;Z(D^1(\gamma),D^1(\beta);m_{s,t})=Z(D^1(\beta),D^1(\gamma);m_{s,t})\right. \\
    \left. \forall \beta,\gamma\in \Phi[\Gamma]\mbox{ s.t. } ,\beta\neq \gamma,\;\,\beta=w(\alpha_s),\gamma=w(\alpha_t) \mbox{ with } s,t\in S, w\in \W[\Gamma], m_{s,t}\neq \infty \right\rangle.
    \end{split}
\end{equation}
\noindent By Theorem \ref{prespva}, we obtain that $\pi_1(\Omega(\Gamma))\cong \PVA[\Gamma]$.
   
\end{proof}

\subsection{The equality between \texorpdfstring{$\Omega_n$}{TEXT} and \texorpdfstring{$\Omega(A_{n-1})$}{TEXT}}\label{eqialityofspaces}
\noindent
The purpose of this subsection is to show that, when $\Gamma=A_{n-1}$, the generalized BEER complex defined in Definition \ref{defbeer} and studied in Subsection \ref{generalBeer} coincides with the space originally proposed by Bartholdi, Enriquez, Etingof and Rains in \cite{BEER}.\\
\\
We first show that the two complexes $\Omega_n$ and $\Omega(A_{n-1})$ have their cells in bijective correspondence. Then, to prove that they yield the same cellular complex, we verify that their attaching maps, gluing the boundary of a $k$-cell to the $(k-1)$-skeleton, coincide.\\
\\
We know that $\Omega(A_{n-1})$ has a $k$-cell (for $k\leq n-1$) for each subset of roots $\cX\subset \Phi[A_{n-1}]$ that is almost parabolic and of spherical type. Since all reflection subgroups of the finite group $\W[A_{n-1}]=\Sn$ must be of spherical type, our task reduces to understanding the condition of being almost parabolic for a subset $\cX$ of the root system of the symmetric group. This analysis will be carried out in Subsubsection \ref{subsubparabolicSn}.\medskip\\
\noindent In Subsubsection \ref{cellulardescriptionOmegan}, we describe the cellular structure of the original BEER complex $\Omega_n$, reformulating several arguments already used in the proof of Theorem \ref{fundoriginalbeer}, but now expressed in terms of parabolic subgroups and root systems.\medskip\\
\noindent Finally, in Theorem \ref{omegan=omegaxan}, we prove that the generalized BEER complex $\Omega(\Gamma)$ introduced in Section \ref{generalBeer}, for 
$\Gamma=A_{n-1}$ coincides with $\Omega_n$. The proof proceeds by induction on the 
$k$-skeleta of the corresponding CW-complexes.

\subsubsection{Parabolic subgroups of the symmetric group \texorpdfstring{$\Sn$}{TEXT}}\label{subsubparabolicSn}

\noindent An exhaustive description of the root system of the Coxeter group of type $A_{n-1}$ can be found in \cite[Planche I]{Bourbaki}. We provide a summary here below.\\
\\
\noindent For $\Gamma=A_{n-1}$ and $S=\{s_1,\ldots,s_{n-1}\}$, the set of simple transpositions $s_i=(i\;i+1)$ for $i\in\{1,\ldots,n-1\}$, the Coxeter group is the familiar symmetric group $\Sn$. We examine its explicit description following the construction of Subsection \ref{coxeterpolytopes}. 
Let $U=\mathbb{R}^n$ with the standard scalar product and $\left\{e_1, \ldots , e_n\right\}$ as the canonical basis. Denote now by $V$ the hyperplane of $U$ of equation $x_1+ \cdots + x_n=0$ and for $i \in \left\{1, \ldots , n-1\right\}$, set $\alpha_{i,i+1}=\frac{e_i-e_{i+1}}{\sqrt{2}}$. Then, $\Pi=\left\{\alpha_1, \ldots, \alpha_{n-1}\right\}$ is the set of simple roots of $\Sn$ and a basis of $V$. For simplicity, the roots $\alpha_{i,i+1}$ will be simply denoted by $\alpha_i$. The root system of the symmetric group can be described as 
\[
\Phi[A_{n-1}]=\{\alpha_{i,j}\mid 1\leq i\neq j\leq n\},
\]
\noindent and the reflection associated to the root $\alpha_{i,j}$ is the transposition $t_{i,j}=(i\;j)\in \Sn$. \\
\\
\noindent  It can be observed that the action of $\Sn$ on $\Phi[A_{n-1}]$ is the action on the indices, namely, for all $w\in \Sn$, $w(\alpha_{i,j})=\alpha_{w(i),w(j)}$.\\
\\
\noindent For any permutation $w\in \Sn$, we define the \textit{support} of $w$ as the set of indices in $\{1,\ldots ,n\}$ that are not fixed by $w$; that is, the support consists of all $i$ such that $w(i)\neq i$.\medskip \\
\noindent Note that the support of a permutation is always contained in the union of the supports of the simple transpositions that appear in any reduced expression for $w$. \medskip \\
\noindent Now, given a collection of permutations $w_1,\ldots,w_l\in \Sn$, we define the support of the subgroup they generate, $\langle w_1,\ldots,w_l\rangle$, as the set of indices $i\in\{1,\ldots,n\}$ for which some permutation in the subgroup does not fix $i$.\medskip\\
\noindent The following classical result provides a complete characterization of all reflection subgroups in $\Sn$. Its statement can be found for instance in \cite[Theorem 3.1]{DPR13}, but since we do not find a proof in the literature, for the sake of completeness, we write it here.

\begin{thm}\label{refsbgpofSn}
Let $\Gamma$ be the Coxeter graph $A_{n-1}$. Then all reflection subgroups of $\W[\Gamma]=\Sn$ are parabolic. In other words, any subgroup of the symmetric group generated by transpositions, is a conjugate of a subgroup generated by \textit{simple} transpositions. 
\end{thm}
\begin{proof}
    A reflection subgroup of $\Sn$ is a group generated by transpositions $t_{i,j}=(i\;j)$, with $i,j\in \{1,\ldots,n\}$ and $i\neq j$. Consider a set of transpositions $\{t_{i_1,j_1},\ldots,t_{i_k,j_k}\}$. The support of the reflection subgroup $\langle t_{i_1,j_1},\ldots,t_{i_k,j_k}\rangle$ is contained in the set $\{i_1,j_1,i_2,j_2,\ldots ,i_k,j_k\}$.\medskip
    \\
    \noindent Observe that a subgroup of $\Sn$ generated by transpositions is isomorphic to a direct product of symmetric groups $\mathfrak{S}_{l_1}\times\cdots\times \mathfrak{S}_{l_p}$. Indeed, consider a graph whose vertices are the integers from 1 to $n$, and there is an edge joining $i$ and $j$ if $t_{i,j}\in \{t_{i_1,j_1},\ldots,t_{i_k,j_k}\}$. The reflection subgroup $\langle t_{i_1,j_1},\ldots,t_{i_k,j_k}\rangle$ is the direct product of symmetric groups $\mathfrak{S}_{l_1}\times\cdots\times \mathfrak{S}_{l_p}$\ where each factor is the group of permutations of the set of vertices in a connected component of such a graph.
    \medskip\\
    \noindent Let $l_1+\cdots+l_p=:l\leq n$ denote the cardinality of the support of $\langle t_{i_1,j_1},\ldots,t_{i_k,j_k}\rangle$.  Each direct factor $\mathfrak{S}_{l_m}$ for $m\in \{1,\ldots,p\}$ has a support of size $l_m$. The support of $\mathfrak{S}_{l_1}$, for instance, is a translate of $\{1,\ldots,l_1\}$ under some permutation $u\in \Sn$. Consequently, the first factor is of the form $u\W_{X_1}u^{-1}$, where $X_1=\{s_{1},s_2,\ldots,s_{l_1-1}\}$, and $s_i$ denotes the simple transposition $(i\;i+1)$. This holds for each factor. \medskip \\
    \noindent Namely, there exists a permutation (in general not unique) $u\in \Sn$ such that, for each $m=1,2,\ldots,p$, $u$ maps the interval $\{l_1+l_2+\cdots+l_{m-1}+1,\ldots,l_1+\cdots+l_{m-1}+l_m\}$ to the support of the corresponding factor $\mathfrak{S}_{l_m}$ (with the convention that $l_0=0$). 
   For each $m\in \{1,\ldots,p\}$, let $X_m:=\{s_{l_1+\cdots+l_{m-1}+1},\ldots,s_{l_1+\cdots+l_{m-1}+l_m-1}\}\subset S$. Then the reflection subgroup decomposes as 
    \[\langle t_{i_1,j_1},\ldots,t_{i_k,j_k}\rangle=\mathfrak{S}_{l_1}\times\cdots\times \mathfrak{S}_{l_p}=u\W_{X_1}u^{-1}\times\cdots\times u\W_{X_p}u^{-1}.\]
    \noindent Now let $X=\bigsqcup_{m=1}^pX_m=\{s_1,\ldots,s_{l_1-1},s_{l_1+1},\ldots,s_{l_1+l_2-1},\ldots ,s_{l_1+\cdots +l_p-1}\}$, and observe that $X=\{s_1,\ldots,s_l\}\backslash \{s_{l_1},s_{l_1+l_2},\ldots,s_{l_1+\cdots+l_p}\}$. Since $\W_{X}=\W_{X_1\sqcup\cdots\sqcup X_p}=\W_{X_1}\times \cdots\times \W_{X_p}$, we have:
    \[\langle t_{i_1,j_1},\ldots,t_{i_k,j_k}\rangle=u(\W_{X_1}\times 
\ldots \times \W_{X_p})u^{-1}=u\W_{X}u^{-1},
    \]
   \noindent showing that the reflection subgroup is parabolic.

\end{proof}

\noindent This theorem implies in particular that all almost parabolic reflection subgroups of $\Sn$ are parabolic. 

\begin{rmk}\label{parsubsofSnminimal}
    In general, there exist several $X\subset S$ and $u\in \Sn$ such that $\langle t_{i_1,j_1},\ldots,t_{i_k,j_k}\rangle=u\W_{X}u^{-1}$. However, once $X$ is fixed, there is always at least one choice of $u$ such that $u$ is $(\emptyset,X)$-minimal. From now on, we will always make such a choice.
\end{rmk}

\begin{ex}\label{exrefsubn=6}
   Let $n=6$, so $\W[A_5]=\mathfrak{S}_6$. Consider the reflection subgroup generated by the transpositions $t_{2,3},t_{2,4},t_{1,6}$, whose support is $\{1,2,3,4,6\}$. The transpositions $t_{2,3}$ and $t_{2,4}$ generate a subgroup isomorphic to $\mathfrak{S}_3$ acting on the set $\{2,3,4\}$, which corresponds to the conjugate of $\W_{X_1}$ with $X_1=\{s_1,s_2\}$ under some permutation $u$.\medskip \\
   \noindent The subgroup $\langle t_{1,6}\rangle$ is isomorphic to $\mathfrak{S}_2$ on the support $\{1,6\}$, which is disjoint from $\{2,3,4\}$. We can express $\langle t_{1,6}\rangle=u\W_{X_2}u^{-1}$, where $X_2=\{s_4\}$ and $u$ is a permutation that sends $ \{1,2,3\}$ to $\{2,3,4\}$ and $\{4,5\}$ to $\{1,6\}$. Now set $u=(1\;3\;4)(5\;6)$ and $X=X_1\cup X_2=\{s_1,s_2,s_4\}$. We get that
   \[
   \langle t_{2,3},t_{2,4},t_{1,6} \rangle= \langle t_{2,3},t_{2,4}\rangle \times \langle t_{1,6} \rangle=u \W_{X}u^{-1}\cong \mathfrak{S}_3\times \mathfrak{S}_2.
   \]
\end{ex}

\begin{rmk}\label{XapinSn}
    Let $\cX\subset\Phi[A_{n-1}]$ be any subset of roots of the root system of the symmetric group $\Sn$. By Theorem \ref{refsbgpofSn}, we know that $\langle R_{\cX}\rangle$ is a parabolic subgroup. Namely, there exist $u\in 
    \Sn$ and $X\subset S=\{s_1,\ldots,s_{n-1}\}$ such that $\langle R_{\cX}\rangle=u\W_Xu^{-1}$. Set $\cX'=u(\Pi_{X})$. Therefore, even though $\cX$ may not be parabolic, there exists a choice of a parabolic $\cX'\subset \Phi[A_{n-1}]$ such that $\langle R_{\cX}\rangle=\langle R_{\cX'}\rangle$.
\end{rmk}

\begin{ex}\label{exparabolicsofSn}
    Consider $\Gamma=A_2$, so that $\W[\Gamma]$ is the symmetric group $\mathfrak{S}_3$. The root system $\Phi[\Gamma]$ contains six root vectors:
    \[
    \Phi[A_2]=\{\alpha_{1,2},\alpha_{2,3},\alpha_{1,3},\alpha_{2,1},\alpha_{3,2},\alpha_{3,1}\},
    \]
    \noindent whose base is $\Pi=\{\alpha_{1,2},\alpha_{2,3}\}$. The set $\cX=\{\alpha_{1,2},\alpha_{1,3}\}$ cannot be parabolic, otherwise there would exist a $w\in \mathfrak{S}_3$ such that $\cX=w(\Pi)$. But \[\langle\alpha_{1,2},\alpha_{1,3}\rangle=\left\langle \frac{e_1-e_2}{\sqrt{2}},\frac{e_1-e_3}{\sqrt{2}}\right\rangle=\left\langle \left(\frac{1}{\sqrt{2}},-\frac{1}{\sqrt{2}},0\right),\left(\frac{1}{\sqrt{2}},0,-\frac{1}{\sqrt{2}}\right)\right\rangle=\frac{1}{2},\]
    \noindent while 
    \[
    \langle w(\alpha_{1,2}),w(\alpha_{2,3})\rangle=\langle \alpha_{1,2},\alpha_{2,3}\rangle=\left\langle \left(\frac{1}{\sqrt{2}},-\frac{1}{\sqrt{2}},0\right),\left(0,\frac{1}{\sqrt{2}},-\frac{1}{\sqrt{2}}\right)\right\rangle=-\frac{1}{2}.
    \]
    \noindent However, $\langle R_{\cX}\rangle=\mathfrak{S}_3=\langle R_{\Pi}\rangle$ is a parabolic subgroup. 
\end{ex}

\noindent We know that all almost parabolic reflection subgroups of $\Sn$ are parabolic. We want now to obtain the same result for almost parabolic subsets of roots $\cX\subset\Phi[A_{n-1}]$.

\begin{lem}\label{allAPareparabolic}
Let $\Gamma$ be the Coxeter graph $A_{n-1}$, and let $\cX\subset\Phi[A_{n-1}]$ be an almost parabolic set of roots. Then, there exist $X\subset S$, $u\in \Sn$ $(\emptyset,X)$-minimal and $\overline{w}\in u\W_Xu^{-1}$ such that $\cX=\overline{w}u(\Pi_X)$.
\end{lem}
\begin{proof}
Let $\cX\subset\Phi[A_{n-1}]$ be an almost parabolic subset of roots. By Theorem \ref{refAP}, we know that the reflection subgroup associated with $\cX$ admits a Coxeter presentation with Coxeter system $(\W_{\cX},R_{\cX})$.
   According to Theorem \ref{refsbgpofSn}, we know that there exist $u\in \Sn$ and $X\subset S$ such that $\langle R_{\cX}\rangle=\W_{\cX}=u\W_{X}u^{-1}=\W_{u(\Pi_X)}$. As pointed out in Remark \ref{parsubsofSnminimal}, suppose that $u$ is $(\emptyset,X)$-minimal.
   Thus, the reflection subgroup $\W_{\cX}$ is an almost parabolic reflection subgroup with respect to two different AP sets of roots: $\cX$ and $\cX'=u(\Pi_X)$. As in the proof of Theorem \ref{refsbgpofSn}, $\W_{\cX}=u\W_Xu^{-1}$ is isomorphic to a direct product of symmetric groups $\mathfrak{S}_{l_1}\times \cdots\times \mathfrak{S}_{l_p}$. With the same notation as in Theorem \ref{refsbgpofSn}, for all $1\leq m\leq p$ let $X_m$ be the subset of $S$ such that $X=X_1\sqcup\cdots\sqcup X_p$ and $\mathfrak{S}_{l_m}=u\W_{X_m}u^{-1}$.
   \medskip\\
   \noindent By abuse of notation, we write the Coxeter matrix $\M_{\cX}$ (resp. $\M_{\cX'}$) to mean the set of Coxeter relations $(r_{\beta}r_{\gamma})^{m_{\beta,\gamma}}=\id$, $r_{\beta}^2=\id$ for all $\beta,\gamma\in \cX$ (resp. $(r_{\beta'}r_{\gamma'})^{m_{\beta',\gamma'}}=\id$, $r_{\beta'}^2=\id$ for all $\beta',\gamma'\in \cX'$). Observe that, since $\cX'=u(\Pi_X)$, the Coxeter matrix $\M_{\cX'}$ coincides with the sub-matrix $\M_X$ of $\M[\Gamma]$. Then, such a reflection subgroup has two Coxeter presentations:
   \[
   \langle R_{\cX}\mid \M_{\cX}\rangle=\W_{\cX}=\langle R_{\cX'}\mid \M_{\cX'}\rangle=\langle usu^{-1}\mid (usu^{-1})^2=\id,\;(ustu^{-1})^{m_{s,t}}=\id\;\forall s,t\in X\rangle.
   \]
   \noindent If $R_{\cX}\neq R_{\cX'}$, then these are two different presentations of the same spherical-type Coxeter group $\W_{\cX}=\mathfrak{S}_{l_1}\times \cdots \times \mathfrak{S}_{l_p}$ for some $p\in \mathbb{N}$. Since no symmetric group can be decomposed as a direct product of two proper subgroups, both presentations correspond to the same Coxeter graph $A_{l_1-1}\sqcup\cdots \sqcup A_{l_p-1}$. \medskip \\
   \noindent
Changing the generating set in the presentation of a factor $\mathfrak{S}_{l_m}$ for $1\leq m\leq p$, amounts to changing the base of the associated root system $\Phi[A_{l_m-1}]$. Specifically, if $\mathfrak{S}_{l_m}$ has two Coxeter presentations corresponding to bases $\cX_m'=\{ \alpha_{1},\ldots,\alpha_{l_m-1}\}$ and $\cX_m=\{ \beta_{1},\ldots,\beta_{l_m-1}\}$ of $\Phi[A_{l_m-1}]$, there exists a unique $w_m\in \mathfrak{S}_{l_m}$ such that $w_m(\{ \alpha_{1},\ldots,\alpha_{l_m-1}\})=\{ \beta_{1},\ldots,\beta_{l_m-1}\}$, for all $m=1,\ldots, p$. (See \cite[Section 1.4]{Hump}) \medskip \\
   \noindent Let $\cX=\cX_1\,\sqcup\,\cdots\,\sqcup\,\cX_p$ and $\cX'=\cX'_1\,\sqcup\,\cdots\,\sqcup\,\cX'_p=u(\Pi_{X_1})\,\sqcup \,\cdots\,\sqcup\, u(\Pi_{X_p})$, where $\cX_m$ (resp. $\cX'_{m}$) is the subset of roots such that $R_{\cX_m}$ (resp. $R_{\cX'_m}$) generates the factor $\mathfrak{S}_{l_m}$, for all $m=1,\ldots,p$. Let $w_m\in \W_{\cX_m}$ be the unique element such that $w_m(\cX'_{m})=\cX_m$, and let $\overline{w}:=w_1\,\cdots \,w_p$, where $w_1,\ldots,w_p$ commute pairwise  since their supports are disjoint. Also observe that $\overline{w}\in\W_{\cX}=u\W_Xu^{-1}$ and that $\overline{w}(\cX_m)=w_m(\cX_m)=\cX_{m}'$, for all $m=1,\ldots,p$. Then:
   \begin{multline*}
\cX=\cX_1\sqcup\cdots\sqcup\cX_p=w_1(\cX'_1)\sqcup\cdots\sqcup w_p(\cX'_p)=\overline{w}(\cX'_1)\sqcup\cdots\sqcup \overline{w}(\cX'_p)
=\overline{w}u(\Pi_{X_1})\sqcup\cdots\sqcup \overline{w}u(\Pi_{X_p})=\\=\overline{w}u(\Pi_{X_1}\sqcup\cdots \sqcup \Pi_{X_p})=\overline{w}u(\Pi_{X_1\sqcup\cdots \sqcup X_p})= \overline{w}u(\Pi_{X})=\overline{w}(\cX'),
   \end{multline*}
   \noindent where we recall that $X=X_1\sqcup \cdots \sqcup X_p$. Let $w:=\overline{w}u$. We got finally that $\cX=\overline{w}u(\Pi_X)$, with $\overline{w}, u$ and $X$ as in the statement of the lemma.
\end{proof}
\begin{rmk}\label{rmk-wbar-unique}
    In the previous proof, remark that, once $u\in \Sn$ and $X\subset S$ such that $\W_{\cX}=u\W_X u^{-1}$ are fixed, the element $\overline{w}=w_1\cdots w_p$ such that $\cX=\overline{w}u(\Pi_X)$ is unique. Indeed, each $w_m$ is the unique element of $\mathfrak{S}_{l_m}
    $ that conjugates two simple systems of $\mathfrak{S}_{l_m}
    $.
\end{rmk}
\begin{rmk}\label{oioia}
From the proof of Lemma \ref{allAPareparabolic}, we also deduce that, if $\cX$ and $\cX'$ are two almost parabolic subsets of roots of $\Phi[A_{n-1}]$ such that $\W_{\cX}=\W_{\cX'}=\W_{u(\Pi_X)}$, then there exists $\overline{w}\in \W_{\cX}$ such that $\cX=\overline{w}(\cX')$.
\end{rmk}
\begin{cor}\label{corallAPareparabolic}
Let $\Gamma$ be the Coxeter graph $A_{n-1}$, and let $\cX\subset\Phi[A_{n-1}]$ be an almost parabolic set of roots. Then, $\cX$ is parabolic.
\end{cor}
\begin{proof}
By Lemma \ref{allAPareparabolic}, we know that there exist $u$, $\overline{w}$ and $X\subset S$ such that $\cX=\overline{w}u(\Pi_X).$ By setting $\overline{w}u=:w$, we get that $\cX=w(\Pi_X)$, hence $\cX$ is parabolic.
\end{proof}

\noindent 
By the description of $\Omega(\Gamma)$ as in Definition \ref{defbeer}, we know that it has a $k$-dimensional cell $D^k(\cX)$ for each $\cX\subset \Phi[\Gamma]$ subset of roots almost parabolic and of spherical type, such that $|\cX|=k$. Thanks to Lemma \ref{allAPareparabolic} and Corollary \ref{corallAPareparabolic} we know that, when $\Gamma=A_{n-1}$, all AP subsets of roots $\cX$ are indeed parabolic, of the form $\cX=\overline{w}u(\Pi_X)$ with $\overline{w}\in u\W_{X}u^{-1}=\W_{\cX}$, $u\in \Sn$ a $(\emptyset,X)$-minimal element, and $X$ a subset of $S$ of cardinality $k$. However, the choice of $u$ and $X$ is not unique.\medskip \\
\noindent
In the next result, we show that even though $u$ and $X$ giving $\cX=\overline{w}u(\Pi_X)$ are not unique, $\cX$ is uniquely determined by the pair $(u\W_Xu^{-1},\overline{w})$. Therefore, a parabolic subset of $\Phi[A_{n-1}]$, and thus a $k$-cell of $\Omega(A_{n-1})$, is bijectively associated with a parabolic subgroup of $\Sn$ of rank $k$ and an element in such a parabolic subgroup. 
\begin{lem}\label{APsubsetsofSn}
Let $\cX\subset\Phi[A_{n-1}]$ be a parabolic subset of roots of $\Sn$, of the form $\cX=\overline{w}u(\Pi_X)$ as in Lemma \ref{allAPareparabolic}. Then $\cX$ is uniquely determined by the pair $(u\W_Xu^{-1},\overline{w})$, where $u\W_Xu^{-1}$ is a parabolic subgroup of $\Sn$ and $\overline{w}$ is an element in $u\W_Xu^{-1}$.  
\end{lem}
\begin{proof}
Given a parabolic set of roots $\cX\subset \Phi[A_{n-1}]$, let $\W_{\cX}$ be its associated reflection subgroup. According to Theorem \ref{refsbgpofSn}, there exist $u\in \Sn$ and $X\subset S$ such that $\W_{\cX}=u\W_Xu^{-1}$. As in Lemma \ref{allAPareparabolic}, we know that there exists a unique $\overline{w}=\overline{w}_{u,X}\in \W_{\cX}$ such that $\cX=\overline{w}_{u,X}\,u(\Pi_X)$.\medskip\\
\noindent
We want to show that the pair $(u\W_Xu^{-1},\overline{w})$ does not depend on the choice of $u$ and $X\subset S$. Indeed,
as in Theorem \ref{refsbgpofSn}, if there exist $Y\subset S$ and $v\in \Sn$ $(\emptyset,Y)$-minimal such that $\W_{\cX}=v\W_Yv^{-1}$, then obviously $\W_{\cX}=u\W_{X}u^{-1}=v\W_Yv^{-1}$. We want now to show that also $\overline{w}_{u,X}=\overline{w}_{v,Y}$. Since $\W_{\cX}=\W_{u(\Pi_X)}=\W_{v(\Pi_Y)}$, by Remarks \ref{rmk-wbar-unique} and \ref{oioia}, there is a unique element $\overline{w}_{u,X}\in u\W_Xu^{-1}$ such that
\[
\cX=\overline{w}_{u,X}\,u(\Pi_X),
\]
and similarly, there exists a unique element $w_{v,Y}\in v\W_Yv^{-1}$ such that $\cX= \overline{w}_{v,Y}\, v(\Pi_Y)$. By Lemma \ref{upixvpiy}, we know that, for $u$ $(\emptyset,X)$-minimal and $v$ $(\emptyset,Y)$-minimal, $\W_{\cX}=u\W_{X}u^{-1}=v\W_Yv^{-1}$ implies $u(\Pi_X)=v(\Pi_Y)$. Then, replacing $u(\Pi_X)$ by $v(\Pi_Y)$ in the above equality we obtain (again, by Remark \ref{rmk-wbar-unique}) that $\overline{w}_{u,X}=\overline{w}_{v,Y}$.\medskip\\
\noindent
On the other hand, given a pair $(u\W_Xu^{-1},\overline{w})$ where $\overline{w}\in u\W_Xu^{-1}$, $X\subset S$ and  $u\in \Sn$ is an $(\emptyset,X)$-minimal element, we can uniquely associate $(u\W_Xu^{-1},\overline{w})$ with the parabolic subset of roots $\cX=\overline{w}u(\Pi_X)$.
\medskip\\
\noindent
Hence, the parabolic subset of roots $\cX$ is uniquely determined by the pair $(u\W_Xu^{-1},\overline{w})$, and vice-versa.
\end{proof}

\noindent To show that $\Omega(A_{n-1})$ and $\Omega_n$ have their cells in bijective correspondence, we show that also the $k$-cells of $\Omega_n$ can be determined by a parabolic subgroup of $\Sn$ of rank $k$ and an element of such a parabolic subgroup.

\begin{prop}\label{propdescriptioncellsomegan}
    Let $\Omega_n$ be the complex defined in \cite{BEER}, as in Definition \ref{deforiginalbeer}. Then, the cells of $\Omega_n$ are in bijection with the pairs $(u\W_Xu^{-1},\overline{w})$, where $u\W_Xu^{-1}$ is a parabolic subgroup of $\W=\Sn$, and $\overline{w}$ is an element in $u\W_Xu^{-1}$.
    \end{prop}
\begin{proof}
Recall that the BEER complex $\Omega_n$ is defined as the quotient of $C[S]\times \Sn=\mathcal{P}\times \Sn$ under the equivalence relation $\sim$, with attaching maps described in Equation \ref{defattachingmap}.
We have already observed in Remark \ref{crucialrmk} that the equivalence classes in $\mathcal{P}\times \Sn/\sim$ are in bijection with the elements of the set $\bigcup_{P\in\mathcal{P}}([P]\times \Stab(P))$, where $[P]$ is the equivalence class of the ordered partition $P\in \mathcal P$ with respect to the equivalence relation defined in Subsection \ref{originalBEERspace}, and $\Stab(P)<\Sn$ is its stabilizer. Indeed, given two ordered partitions $P:=P(u,X)$ and $Q:=P(v,Y)$, we have that $[P]=[Q]$ if and only if \[\Stab(P(u,X))=u\W_Xu^{-1}=\Stab(P(v,Y))=v\W_Yv^{-1},\]\noindent with $u,v\in \Sn$ and $X,Y\subset S$. Thus, the equivalence class $[P]$ can be represented by the parabolic subgroup $u\W_Xu^{-1}=\Stab(P)$. As usual, without loss of generality, we choose $u$ to be $(\emptyset,X)$-minimal.\\ \\
\noindent Additionally, if a pair $(F(u,X,S),w')$ of $C[S]\times \Sn$ belongs to the same class as $(F(u,X,S),w)$ in $\Omega_n$, it means that $w$ and $w'$ have the same inversions in the blocks of the (un)ordered partition $[P(u,X)]$. Denote by $P_1,\ldots,P_p$ the blocks of the partition $P(u,X)$ associated with the face $F(u,X,S)$ of the permutohedron. As we remarked at the beginning of Subsection \ref{generalBeer}, we can write that $(F(u,X,S),w)\sim(F(u,X,S),w')$ if and only if, for all $k=1,\ldots,p$ and for all $i,j\in P_k$ with $i<j$, $w(i)<w(j)\,\Longleftrightarrow w'(i)<w'(j)$. With the same notations as in Subsection \ref{generalBeer}, for all $k=1,\ldots,p$, to write $i,j\in P_k$ means
\[
i,j\in P_k=\{u(i_{k-1}+1),\ldots,u(i_k)\}.
\]
\noindent Recall that we set $X_k=\{s_{i_{k-1}+1},\ldots,s_{i_k-1}\}\subset S$ for all $k=1,\ldots,p$. Let us denote by $\mathfrak{S}(P_k)$ the group of permutations on the indices in $P_k$. Clearly, $\mathfrak{S}(P_k)$ is the parabolic subgroup $u\W_{X_k}u^{-1}$ of $\Sn$. Then, similarly to what we have observed in the proof of Theorem \ref{fundoriginalbeer}, for all $k=1,\ldots,p$, there exists a unique $w_k\in \mathfrak{S}(P_k)=u\W_{X_k}u^{-1}$ such that
\[
w(i)<w(j)\,\Longleftrightarrow\, w_k(i)<w_k(j).
\]
\noindent This must hold for all blocks $P_1,\ldots,P_p$. Observe that $\Stab(P(u,X))=u\W_Xu^{-1}$ is the direct product of the stabilizers of the blocks $u\W_{X_1}u^{-1}\times \cdots \times u\W_{X_p}u^{-1}$ and that $X=X_1\sqcup\cdots\sqcup X_p$. The elements $w_1,\ldots,w_p$ commute pairwise since their supports are disjoint. Set now $\overline{w}=w_1w_2\cdots w_p$. Therefore, there exists a unique element $\overline{w}\in u\W_{X_1}u^{-1}\times \cdots \times u\W_{X_p}u^{-1}=u\W_Xu^{-1}$ such that, for all $k=1,\ldots,p$ and for all $i,j\in P_k$ with $i<j$, we have
\[
w(i)<w(j)\, \Longleftrightarrow \, \overline{w}(i)<\overline{w}(j).
\]
\noindent Indeed, if $i,j\in P_k$, then $\overline{w}(i)=w_k(i)$ and $\overline{w}(j)=w_k(j)$. Thus, we have checked that there exists a unique $\overline{w}\in \Stab(P(u,X))$ such that the face $(F(u,X,S),w)$ of $C[S]\times \Sn$ is equivalent to $(F(u,X,S),\overline{w})$ in the quotient $\Omega_n$.\medskip\\
\noindent Consequently, the equivalence class of $(F(u,X,S),w)$ in $\Omega_n$ can be represented by $(\Stab(P(u,X),$ $\overline{w})$, and the pair $(u\W_Xu^{-1},\overline{w})$ does not depend on the choice of the representative of\\ $[(F(u,X,S),w)]\in \Omega_n$.\medskip\\
\noindent
On the other hand, given a pair $(u\W_Xu^{-1},\overline{w})$ for some $u\in \Sn$ and $X\subset S$, we can uniquely associate it with the cell $[(F(u,X,S),\overline{w})]$ of $\Omega_n$.
\medskip\\
\noindent
It follows that the cells of $\Omega_n$ are in one-to-one correspondence with the pairs $(u\W_Xu^{-1},\overline{w})$. 
\end{proof}
\noindent
    As usual, given a cell $(F(u,X,S),w)$ of $C[S]\times \Sn$ represented by the pair $(u\W_Xu^{-1},\overline{w}_{u,X})$ in $\Omega_n$, $u$ will be chosen to be $(\emptyset,X)$-minimal.

\subsubsection{Cellular description of \texorpdfstring{$\Omega_n$}{TEXT}}\label{cellulardescriptionOmegan}
\noindent In the proof of Theorem \ref{fundoriginalbeer}, we have given a description of the cells in $\Omega_n$ in terms of ordered partitions and permutations. Here, in light of the results obtained in Proposition \ref{propdescriptioncellsomegan}, we rephrase such a description in terms of roots and parabolic subgroups of $\Sn$. Moreover, we will describe how the boundary of a $k$-cell is glued to the cells in the $(k-1)$-skeleton. These considerations will be a fundamental ingredient to show that $\Omega_n$ and $\Omega(\A_{n-1})$ are indeed the same space. \medskip
\\\noindent
Recall that $\Omega_n$ is defined as a quotient of the product space $C[S]\times \Sn$ under the equivalence relation introduced in Definition \ref{deforiginalbeer}. A cell of dimension $0\leq k\leq n-1$ in $\Omega_n$ is an equivalence class $[(F(u,X,S),w)]$, where $F(u,X,S)$ is a face of dimension $|X|=k$ of the permutohedron $C[S]$, and $w$ is an element in $\Sn$.\medskip\\
\noindent Also recall that $(F(u,X,S),w)$ and $(F(u',X',S),w')$ belong to the same equivalence class in $\Omega_n$ if and only if:
\begin{enumerate}
    \item the parabolic subgroups $u\W_{X}u^{-1}=\Stab(P(u,X))$ and $u'\W_{X'}u'^{-1}=\Stab(P(u',X'))$ coincide;
    \item for all $\gamma\in u(\Phi_X)\cap \Phi^+$ (which coincides with $u'(\Phi_{X'})\cap \Phi^+$), $w(\gamma)\in \Phi^+\Longleftrightarrow w'(\gamma)\in \Phi^+$.
\end{enumerate}
\noindent If we interpret $u(\Phi_X)\cap \Phi^+$ as the positive roots of the Coxeter group $u\W_Xu^{-1}$, we see that there exists a unique element $\overline{w}\in u\W_Xu^{-1}$ such that for all $\gamma \in u(\Phi_X)\cap \Phi^+$, $w(\gamma)\in \Phi^+\Longleftrightarrow \overline{w}(\gamma)\in \Phi^+$. Hence, the class $[(F(u,X,S),w)]$ can be uniquely determined by the pair $(u\W_X u^{-1}, \overline{w})$.
\medskip\\
\noindent We now start analyzing the complex $\Omega_n$ skeleton-wise. Recall that if $X\subset S$ is a singleton $X=\{s\}$, then, to lighten the notations, we write $C[s]$, $F(u,s,S)$ and $\W_s$ instead of $C[\{s\}]$, $F(u,\{s\},S)$ and $\W_{\{s\}}$.\\
\\
\noindent\textbf{Description of the 0-skeleton.}
As we have remarked in the proof of Theorem \ref{fundoriginalbeer}, the 0-skeleton $(\Omega_n)^0$ consists of just one point, $\omega^0$. Indeed, all the vertices in the product space $(F(u,\emptyset,S),w)\in C[S]\times \Sn$ are associated with a partition constituted by $n$ blocks $[P(u,\emptyset)]=[(\{1\},\ldots,\{n\})]$, which is stabilized by the trivial group $u\W_{\emptyset}u^{-1}=\{\id\}$. The only class $[(F(u,\emptyset,S),w)]$ is represented by the pair $(u\W_{\emptyset}u^{-1},\overline{w})=(\{\id\},\id)$. We write the only 0-cell of $\Omega_n$ as $\omega^0=\Dd^0(\emptyset)$.
\\
\\
\textbf{Description of the 1-skeleton.}
The 1-cells of $\Omega_n$ are quotients of the edges in $C[S]\times \Sn$, which are of the form $(F(u,s,S),w)$ with $s\in S$ and $u$ $(\emptyset,s)$-minimal. Since the stabilizer $\Stab(P(u,s))=u\W_{s}u^{-1}=\{\id, usu^{-1}\}$ has two elements, the class $[(F(u,s,S),w)]$ in $\Omega_n$ is either associated with the pair $(u\W_{s}u^{-1},\id)$, or with the pair $(u\W_{s}u^{-1},usu^{-1})$. Indeed, consider $u(\Phi_{s})\cap \Phi^+=\{u(\alpha_s),-u(\alpha_s)\}\cap \Phi^+$. By minimality of $u$, $u(\alpha_s)$ is a positive root. If $w\in \Sn$ sends $\gamma:=u(\alpha_s)$ to a positive root, then $[(F(u,s,S),w)]$ is represented by $(u\W_{s}u^{-1},\id)$, while if $w(\gamma)\in \Phi^{-}$, then $[(F(u,s,S),w)]$ is represented by $(u\W_{s}u^{-1},usu^{-1})$. \medskip\\
\noindent Notice that in the proof of Theorem \ref{fundoriginalbeer}, we have called $\gamma=\alpha_{i,j}$, and the two associated arcs  $[(F(u,s,S),\id)]=\Dd^1(i,j)$; and $[(F(u,s,S),usu^{-1})]=[(F(u,s,S),t_{ij})]=\Dd^1(j,i)$. Let $s\in S$ be $s=s_h$ for some $h\in \{1,\ldots,n-1\}.$ As in the definition of the partition $P(u,s)$ in Equation \ref{pikappa}, we set $S\backslash X=\{s_{i_1},\ldots,s_{i_{n-2}}\}$ with $i_1<\cdots<i_{n-2}$. Thus, for $l\in \{1,\ldots,n-2\}$, 
\[
i_l=\begin{cases}
    l\;&\text{ if }\, l<h,\\l+1\;&\text{ if }\, l\geq h.
\end{cases}
\]
\noindent The only 2-block of the partition $P(u,s_h)=(P_1,\ldots,P_{n-1})$ is $P_h=u\{i_{h-1}+1,i_{h-1}+2\}=\{u(h),u(h+1)\}$.
Therefore, the indices $i,j$ belonging to the only $2$-block of $P(u,s)$ are $i=u(h)$, $j=u(h+1)$. We can then rewrite the cells $\Dd^1(i,j)$ and $\Dd^1(j,i)$ as $\Dd^1(\overline{w}u(h),\overline{w}u(h+1))$, with $\overline{w}\in \{\id, us_hu^{-1}\}$.
\medskip\\
\noindent We now rephrase the considerations above in terms of roots. In general, for all $s\in S$, for all $u\in \Sn$ $(\emptyset,s)$-minimal and all $\overline{w}\in u\W_{s}u^{-1}$, we have a 1-cell  $[(F(u,s,S),w)]\in \Omega_n$ represented by $(u\W_{s}u^{-1},\overline{w})$. Observe that $[(F(u,s,S),\overline{w})]=[(F(\overline{w}u,s,S),\overline{w})]$. If, as before, we set $\gamma:=u(\alpha_s)\in\Phi^+[A_{n-1}]$, then in $(\Omega_n)^1$ we find two cells for each positive root: a cell $\Dd^1(\gamma)=\Dd^1(\id(\gamma))$ and a cell $\Dd^1(r_{\gamma}(\gamma))=\Dd^1(-\gamma)$. Denoting by $\beta$ the root $\overline{w}u(\alpha_s)=\overline{w}(\gamma)$, we can then say that in $\Omega_n$ there is a 1-cell $\Dd^1(\overline{w}u(\alpha_s))=\Dd^1(\beta)$ for each root $\beta\in \Phi[A_{n-1}]$.
\medskip\\
\noindent In $C[S]\times \Sn$, the endpoints of the 1-dimensional face $(F(\overline{w}u,s,S),\overline{w})$ are $(F(\overline{w}u,\emptyset,S),\overline{w})=(\overline{w}u(o),\overline{w})$ and $(F(\overline{w}us,\emptyset,S),\overline{w})=(\overline{w}us(o),\overline{w})$. The orientation of $(F(\overline{w}u,s,S),\overline{w})$ is from $(F(\overline{w}u,\emptyset,S),\overline{w})$ to $(F(\overline{w}us,\emptyset,S),\overline{w})$. Namely, if $\overline{w}=\id$, then $(F(u,\emptyset,S),\overline{w})=(F(u,\emptyset,S),\id)$ is oriented from $(u(o),\id)$ towards $(us(o),\id)$, while if $\overline{w}=usu^{-1}$, the orientation of $(F(\overline{w}u,\emptyset,S),\overline{w})$ $=(F(us,\emptyset,S),\overline{w})$ is from $(us(o),usu^{-1})$ to $(u(o),usu^{-1})$.
In $\Omega_n$, both these vertices are identified with $\omega^0=[(F(\id,\emptyset,S),\id)]$ by the following isometries:
\begin{align*}
(F(u,\emptyset,S),\overline{w})=(u(o),\overline{w})\xrightarrow[]{\tau(o-u(o))\times \overline{w}^{-1}}(o,\id),\\
(F(us,\emptyset,S),\overline{w})=(us(o),\overline{w})\xrightarrow[]{\tau(o-us(o))\times \overline{w}^{-1}}(o,\id).
\end{align*}

\noindent To summarize, the 1-skeleton of $\Omega_n$ is built as follows.\medskip\\
\noindent For each $\beta=\overline{w}u(\alpha_s)\in \Phi[A_{n-1}]$, take a copy $\mathbb{B}(\beta)$ of the face of Coxeter polytope $F(\overline{w}u,s,S)$, identified with $F(\overline{w}u,s,S)$ through an isometry $g_{\beta}:F(\overline{w}u,s,S)\longrightarrow \mathbb{B}(\beta)$. Observe that $F(\overline{w}u,s,S)=conv\{\overline{w}uz(o)\,|\,z\in\{\id,s\}\}$ is the segment $[\overline{w}u(o), \overline{w}us(o)]$, which is oriented from $\overline{w}u(o)$ to $\overline{w}us(o)$. Namely, if $\overline{w}=\id$, which means that $\beta\in \Phi^+$, then $F(\overline{w}u,s,S)=[u(o),us(o)]$ is oriented from $u(o)$ towards $us(o)$, while if $\beta\in \Phi^-$, the segment is oriented from $us(o)$ towards $u(o)$. For each $\beta=\overline{w}u(\alpha_s)\in \Phi[A_{n-1}]$, we define a map $e^1_{\beta}:\partial \mathbb{B}(\overline{w}u(\alpha_s))\longrightarrow (\Omega_n)^0$ by $e^1_{\beta}(g_{\beta}(\overline{w}u(o)))=\omega^0$. Similarly, we also set $e^1_{\beta}(g_{\beta}(\overline{w}us(o)))=\omega^0$. \medskip\\
\noindent The 1-skeleton of $\Omega_n$ is the union $\{\omega^0\}\sqcup(\bigsqcup_{\beta\in \Phi[A_{n-1}]}\mathbb{B}(\beta))$ under the identifications that attach $x$ to $e^1_{\beta}(x)$ for all $\beta\in \Phi[A_{n-1}]$ and all $x\in \partial \mathbb{B}(\beta)$. The image of $\mathbb{B}(\beta)=\mathbb{B}(\overline{w}u(\alpha_s))$ in $(\Omega_n)^1$ under the composition of inclusion and quotient projection, gives the 1-cell $[(F(\overline{w}u,s,S),\overline{w})]=\Dd^1(\overline{w}u(\alpha_s))$. We denote this composition of maps by $E^1_{\beta}:\mathbb{B}(\beta)\longrightarrow \Dd^1(\beta)=\Dd^1(\overline{w}u(\alpha_s))$.\\
\\
\noindent \textbf{\textbf{Description of the $k$-skeleton.}} Before treating the general case for any $2\leq k\leq n-1$, to ease the comprehension, we analyze the particular case $k=2$.\medskip\\
\noindent The 2-cells of $\Omega_n$ of the form $[(F(u,X,S),w)]$, where $X\subset S$ has cardinality 2 and $u\in \Sn$ is $(\emptyset,X)$-minimal, are represented by the pair $(u\W_Xu^{-1},\overline{w})$. The element $\overline{w}\in u\W_{X}u^{-1}$ is the only one such that for all $\gamma\in u(\Phi_X)\cap \Phi^+$, $w(\gamma)$ is a positive root if and only if $\overline{w}(\gamma)$ is a positive root.\medskip \\
\noindent The 2-face $F(u,X,S)$ of the permutohedron, as we have seen in the proof of Theorem \ref{fundoriginalbeer}, is associated with the partition $P(u,X)=(P_1,\ldots,P_{n-2})$ with $n-2$ blocks. Either all the blocks are singletons except for a block with three elements, or all blocks are singletons except two, that are of cardinality 2. In the first case, the associated face $F(u,X,S)$ is a hexagon, denoted by $\Dd^2(i,j,k)$ in Theorem \ref{fundoriginalbeer}, while in the second case it is a square, denoted by $\Dd^2(i,j,k,l)$ in Theorem \ref{fundoriginalbeer}.\medskip\\
\noindent 
In terms of subsets of the root system $\Phi[A_{n-1}]$, we have that there is a 2-cell in $\Omega_n$ for all pairs $(u\W_Xu^{-1},\overline{w})$ such that $X\subset S$ has cardinality 2, $u\in \Sn$ is $(\emptyset,X)$-minimal and $\overline{w}\in u\W_Xu^{-1}$. We denote the cell $[(F(u,X,S),\overline{w})]=[(F(\overline{w}u,X,S),\overline{w})]$ by $\Dd^2(\overline{w}u(\Pi_X))$. Observe that any $\cX$ parabolic subset of roots of $\Phi[A_{n-1}]$ with $|\cX|=2$ can be written in the form $\cX=\overline{w}u(\Pi_X)$ with $X, \overline{w}$ and $u$ as before. We can then state that in $\Omega_n$ we find a 2-cell $\Dd^2(\cX)$ for all parabolic subset of roots $\cX=\overline{w}u(\Pi_X)$ of cardinality 2.
\medskip\\
\noindent We want now to determine the 1-cells in the boundary of $\Dd^2(\overline{w}u(\Pi_X))$. The edges of $\left[\left(F(u,X,S),\right.\right.$ $\left.\left.,w\right)\right]$ $=[(F(\overline{w}u,X,S),\overline{w})]$ are the projections to the quotient space of the edges of $(F(u,X,S),\overline
{w})$ $\in C[S]\times\Sn$, which are of the form $(F(uv,Y,S),\overline{w})$ with $Y=\{s\}\subset X$, $v\in \W_X$ and $v$ $(\emptyset,Y)$-minimal. The face $F(uv,\{s\},S)$ of $C[S]$ corresponds to the partition $P(uv,s)=(Q_1,\ldots,Q_{n-1})$, which is a refinement of $P(u,X)$ for $s\in X$. As before, let $s$ be $s=s_h$ for some $h\in \{1,\ldots,n-1\}$. The only 2-block of the partition $P(uv,s_h)$ is $Q_h=\{uv(h),uv(h+1)\}$. The class $[(F(uv,s_h,S),\overline{w})]$ in $\Omega_n$ is $\Dd^1(\overline{w}uv(\alpha_{s_h}))$, therefore, the edges in the boundary of the 2-cell $[(F(u,X,S),\overline{w})]$ are of the form $\Dd^1(\overline{w}uv(\alpha_s))$ with $Y=\{s\}\subset X$, $v\in \W_{X}$ and $v$ $(\emptyset,Y)$-minimal.\medskip \\
\noindent
In other words, the 2-skeleton of $\Omega_n$ is described as follows. \medskip\\
\noindent
For all parabolic subsets of roots $\cX=\overline{w}u(\Pi_X)\subset \Phi[A_{n-1}]$ such that $|\cX|=2$, we take a copy $\mathbb{C}(\cX)$ of the face $F(\overline{w}u,X,S)$ of the Coxeter polytope $C[S]$,  which we identify with $F(\overline{w}u,X,S)$ through an isometry $g_{\cX}: F(\overline{w}u,X,S)\longrightarrow \mathbb{C}(\overline{w}u(\Pi_X))$. For each $Y\subset X$ and each $v\in \W_X$ $(\emptyset,Y)$-minimal, there is a map $a^2_{\overline{w}u,v,Y,X}: F(\overline{w}uv,Y,S)\longrightarrow\Dd^1(\overline{w}uv(\Pi_Y))$ defined as follows. Let $z:=\overline{w}uvu^{-1}\overline{w}^{-1}\in \overline{w}u\W_Yu^{-1}\overline{w}^{-1}\subset \overline{w}u\W_Xu^{-1}\overline{w}^{-1}$. Then:
\[
F(\overline{w}uv,Y,S)\xrightarrow[g_{\overline{w}uv(\Pi_Y)}]{} \mathbb{B}(\overline{w}uv(\Pi_Y))\xrightarrow[E^1_{\overline{w}uv(\Pi_Y)}]{}\Dd^1(\overline{w}uv(\Pi_Y))=\Dd^1(z\overline{w}u(\Pi_Y)).\]
\noindent Thus, we define $a^2_{\overline{w}u,X}:\partial F(\overline{w}u,X,S)\longrightarrow (\Omega_n)^1$ as follows. Let $x\in \partial F(\overline{w}u,X,S)$. Choose any $Y\subset X$ and $v\in \W_X$ $(\emptyset,Y)$-minimal such that $x\in F(\overline{w}uv,Y,S)$, and set $a^2_{\overline{w}u,X}(x)=a^2_{\overline{w}u,v,Y,X}(x)$. In this way, we have a global attaching map 
\[
    e^2_{\overline{w}u(\Pi_X)}=e^2_{\cX}: \partial \mathbb{C}(\cX)\longrightarrow(\Omega_n)^1
\]
\noindent by composing $(g_{\cX})^{-1}|_{\partial \mathbb{C}(\cX)}$ with $a^2_{\cX}=a^2_{\overline{w}u,X}$. The 2-skeleton of $\Omega_n$ is given by the union of $(\Omega_n)^1$ and $\bigsqcup_{\cX=\overline{w}u(\Pi_X)}\mathbb{C}(\overline{w}u(\Pi_X))$ under the identifications that attach $x$ to $e^2_{\cX}(x)$ for all $\cX$ and all $x\in \partial \mathbb{C}(\cX)$. The composition of inclusion and quotient projection is the characteristic map $E^2_{\cX}:\mathbb{C}(\cX)\longrightarrow \Dd^2(\cX)$ of the cell $\Dd^2(\cX)$.
\medskip \\
\noindent Now we treat the case of a general $2\leq k\leq n-1$, supposing that for all $\cT=\overline{r}q(\Pi_T)$ with $T\subset S$, $|T|=h<k$, $q\in \Sn$ $(\emptyset,T)$-minimal and $\overline{r}\in q\W_Tq^{-1}$, there is a copy $\mathbb{C}(\mathcal{T})$ of $F(\overline{r}q,T,S)$ attached to $(\Omega_n)^h$ through a characteristic map $E^h_{\cT}: \mathbb{C}(\cT)\longrightarrow \Dd^h(\cT)$. \medskip \\
\noindent A $k$-cell of $\Omega_n$ is the equivalence class of a pair $(F(u,X,S),w)$ where $X\subset S$ has cardinality $k$ and $u$ is $(\emptyset,X)$-minimal. Let $\overline{w}$ be the unique element in $u\W_Xu^{-1}$ such that $[(F(u,X,S),w)]=[(F(\overline{w}u,X,S),\overline{w})]$. Then, in $(\Omega_n)^k$ there is a cell for all parabolic set of roots $\cX=\overline{w}u(\Pi_X)$. Take a copy $\mathbb{C}(\cX)$ of $F(\overline{w}u,X,S)$, identified to it through an isometry $g_{\cX}$.
\medskip \\
\noindent As we saw for $k=2$, the cells in the boundary of $[(F(\overline{w}u,X,S),\overline{w})]$ are of the form $\Dd^h(\overline{w}uv(\Pi_Y))$, with $Y\subset X$, $v\in \W_X$ $(\emptyset,Y)$-minimal and $|Y|=h<k$. The boundary of $[(F(\overline{w}u,X,S),\overline{w})]$ is attached to the $(k-1)$-skeleton of $\Omega_n$ as follows.
Let $z:=\overline{w}uvu^{-1}\overline{w}^{-1}\in \overline{w}u\W_Yu^{-1}\overline{w}^{-1}\subset \overline{w}u\W_Xu^{-1}\overline{w}^{-1}$. Then, define $a^k_{\overline{w}u,v,Y,X}$ to be the following composition of maps:
\[
F(\overline{w}uv,Y,S)\xrightarrow[g^h_{\overline{w}uv(\Pi_Y)}]{} \mathbb{C}(\overline{w}uv(\Pi_Y))\xrightarrow[E^h_{\overline{w}uv(\Pi_Y)}]{}\Dd^h(\overline{w}uv(\Pi_Y))=\Dd^h(z\overline{w}u(\Pi_Y)).\]
\noindent Thus, we define $a^k_{\overline{w}u,X}:\partial F(\overline{w}u,X,S)\longrightarrow (\Omega_n)^{k-1}$ as follows. Let $x\in \partial F(\overline{w}u,X,S)$. Choose any $Y\subset X$ and $v\in \W_X$ $(\emptyset,Y)$-minimal such that $x\in F(\overline{w}uv,Y,S)$, and set $a^k_{\overline{w}u,X}(x)=a^k_{\overline{w}u,v,Y,X}(x)$. In this way, we have a global attaching map 
\[
    e^k_{\overline{w}u(\Pi_X)}=e^k_{\cX}: \partial \mathbb{C}(\cX)\longrightarrow(\Omega_n)^{k-1}
\]
\noindent by composing $(g_{\cX})^{-1}|_{\partial \mathbb{C}(\cX)}$ with $a^k_{\overline{w}u,X}$. The $k$-skeleton of $\Omega_n$ is given by the union of $(\Omega_n)^{k-1}$ and $\bigsqcup_{\cX=\overline{w}u(\Pi_X)}\mathbb{C}(\overline{w}u(\Pi_X))$ under the identifications that attach $x$ to $e^k_{\cX}(x)$ for all $\cX$ and all $x\in \partial \mathbb{C}(\cX)$. The composition of inclusion and quotient projection gives a map $E^k_{\cX}:\mathbb{C}(\cX)\longrightarrow \Dd^k(\cX)$, which is the characteristic map of the cell $\Dd^k(\cX)$.
\medskip \\
\noindent Proceeding inductively in this fashion, we obtain a description of all the skeleta of $\Omega_n$. The cells of highest dimension of $\Omega_n$ are associated with the subsets $\cX=\overline{w}u(\Pi_S)$. Since the only $(\emptyset,S)$-minimal element is $u=\id$, we have a cell of dimension $n-1$ for each element  $\overline{w}=w$ of the symmetric group $\Sn$. Thus, in $\Omega_n$ there are $n!$ cells of maximal dimension.

\subsubsection{The equality theorem}\label{sectionomegan=omegaxan}
\noindent In this subsubsection we finally show that $\Omega_n$ and $\Omega(A_{n-1})$ coincide for all $n\in \mathbb{N}$.
\begin{thm}\label{omegan=omegaxan}
The generalized BEER complex $\Omega(\Gamma)$ defined in Subsection \ref{generalBeer}, for $\Gamma=A_{n-1}$, coincides with the space $\Omega_n$ defined in \cite{BEER}.
\end{thm}
\begin{proof}
In this proof, we compare the cellular descriptions of $\Omega_n$ and $\Omega(A_{n-1})$, done respectively in Subsection \ref{generalBeer} and Subsubsection \ref{cellulardescriptionOmegan}. We first observe that these CW-complexes have both their cells in bijective correspondence with the parabolic subsets of roots $\cX\subset \Phi[A_{n-1}]$. Indeed, as remarked in the previous sections, the generalized BEER complex has a $k$-cell $D^k(\cX)$ for each almost parabolic subset of roots $\cX$ such that $|\cX|=k$. Thanks to Lemma \ref{allAPareparabolic}, we know that $\cX$ can be expressed as $\cX=\overline{w}u(\Pi_X)$, where $X\subset S$, $u$ is $(\emptyset,X)$-minimal and $\overline{w}$ belongs to the parabolic subgroup \[u\W_{X}u^{-1}=\overline{w}u\W_{X}u^{-1}\overline{w}^{-1}=\W_{\overline{w}u(\Pi_X)}=\W_{\cX}.\]
\noindent Keep in mind that, in general, $\cX$ can be written in such a form for several $u$ and $X$, but it is uniquely determined by the pair $(u\W_Xu^{-1},\overline{w})$ (see Lemma \ref{APsubsetsofSn}).\medskip \\
\noindent
On the other hand, by Subsubsection \ref{cellulardescriptionOmegan}, we have seen that $\Omega_n$ has a $k$-cell $\Dd^k(\overline{w}u(\Pi_X))=[(F(\overline{w}u,X,S),\overline{w})]$ for each pair $(u\W_Xu^{-1},\overline{w})$ where $X\subset S$, $|X|=k$, $u$ is $(\emptyset,X)$-minimal and $\overline{w}\in \W_{u(\Pi_X)}$. Thus, the two complexes have their cells in bijection.\medskip \\
\noindent
For any parabolic $\cX\subset \Phi[A_{n-1}]$ such that $|\cX|=k$, we denote by $D^k(\cX)$ and $\Dd^k(\cX)$ the $k$-cell associated with $\cX$ in $\Omega(A_{n-1})$ and $\Omega_n$, respectively. Moreover, we check that these cells have the same shape, i.e. that the model cells $\mathbb{D}(\cX)$ and $\mathbb{C}(\cX)$ such that $D^k(\cX)=\Lambda^k_{\cX}(\mathbb{D}(\cX))$ and $\Dd^k(\cX)=E^k_{\cX}(\mathbb{C}(\cX))$, are isometric. 
In the generalized BEER complex $\Omega(A_{n-1})$, a model cell $\mathbb{D}(\cX)$ is isometric to the Coxeter polytope $C[\cX]$ through $f_{\cX}:C[\cX]\longrightarrow \mathbb{D}(\cX)$. In $\Omega_n$, since it is the quotient of the product space $C[S]\times \Sn$, a model cell $\mathbb{C}(\cX)$ is isometric to the face $F(\overline{w}u,X,S)$ of the permutohedron $C[S]$, namely, there is an isometry $g_{\cX}:F(\overline{w}u,X,S)\longrightarrow \mathbb{C
}(\cX)$.\medskip\\
\noindent Since $\cX=\overline{w}u(\Pi_X)$ is parabolic, by Remark \ref{coxpolitopeofaparabolic} we know that $C[\cX]=C[\overline{w}u(\Pi_X)]$ is the isometric image of $C[X]$ under $\overline{w}u$. In particular, the following composition of maps
\[
F(\overline{w}u,X,S)\xrightarrow[(\overline{w}u)^{-1}]{} F(\id,X,S)\xrightarrow[\tau(o_X-o)]{}{}C[X]\xrightarrow[\overline{w}u]{}C[\overline{w}u(\Pi_X)]=C[\cX]\]
is an isometry between $F(\overline{w}u,X,S)$ and $C[\cX]$. Therefore, for all $\cX\subset\Phi[A_{n-1}]$, the model cells $\mathbb{C}(\cX)$ and $\mathbb{D}(\cX)$ are isometric through $im^k_{\cX}:=f_{\cX}\circ\tau(\overline{w}u(o_X)-\overline{w}u(o))\circ(g_{\cX})^{-1}$.\\ \\
\noindent
We observe now that sub-cells of $D^k(\cX)$ and $\Dd^k(\cX)$ are respectively of the form $D^h(z(\cY))$ and $\Dd^h(z(\cY))$, for each $\cY\subset \cX$ of cardinality $|\cY|=h<k$ and each $z\in \W_{\cX}$ that is $(\emptyset,R_{\cY})$-minimal. For $\Omega(A_{n-1})$, the fact that the cells in the boundary of $D^k(\cX)$ are the ones above is a direct consequence of Definition \ref{defbeer}. We now look at the cellular structure of $\Omega_n$ presented in Subsubsection \ref{cellulardescriptionOmegan}. Since the cell $\Dd^k(\cX)=\Dd^k(\overline{w}u(\Pi_X))$ is the image of $F(\overline{w}u,X,S)$ under the characteristic map $E^k_{\cX}$, its sub-cells are the images $\Dd^h(\overline{w}uv(\Pi_Y))$ of $F(\overline{w}uv,Y,S)$ under the attaching maps $a^k_{\overline{w}u,v,Y,X}$, where $Y\subset X$, $|Y|=h<k$, $v\in \W_X$ and $v$ is $(\emptyset,Y)$-minimal. If we set $\cY=\overline{w}u(\Pi_Y)\subset\overline{w}u(\Pi_X)=\cX$ and $z:=\overline{w}uvu^{-1}\overline{w}^{-1}$, we see that $z(\cY)=\overline{w}uv(\Pi_Y)$, and moreover $z\in \W_{\cX}$ and is $(\emptyset,R_{\cY})$-minimal. Hence, the cells in the boundary of $\Dd^k(\cX)$ are $\Dd^h(z(\cY))$ for all $\cY\subset \cX$ with $|\cY|=h<k$, and $z\in \W_{\cX}$ an $(\emptyset,R_{\cY})$-minimal element.\\
\\
\noindent We have just checked that, chosen a parabolic subset of roots $\cX$ of cardinality $k$, the sub-cells of $D^k(\cX)$ are in bijection with the sub-cells of $\Dd^k(\cX)$. We want now to verify that the boundary of the cell associated with $\cX$ is attached to the $(k-1)$-skeleton in the same way in the two complexes. Recall that for $z,\cX$ and $\cY$ as before, the expression of $\lambda^k_{\cX}:\partial \mathbb{D}(\cX)\longrightarrow\Omega^{k-1}(A_{n-1})$ is $\chi^k_{\cX}\circ(f_{\cX})^{-1}|_{\partial \mathbb{D}(\cX)}$, where for $x\in F(z,\cY,\cX)$, the isometry $\chi^k_{\cX}=\chi^k_{z,\cY,\cX}$ is as follows:
\[
F(z,\cY,\cX)\xrightarrow[z^{-1}]{} F(\id,\cY,\cX)\xrightarrow[\tau(o_{\cY}-o_{\cX})]{}C[\cY]\xrightarrow[z]{}C[z(\cY)]\xrightarrow[f_{z(\cY)}]{}\mathbb{D}(z(\cY))\xrightarrow[\Lambda^h_{z(\cY)}]{}D^h(z(\cY)).
\]
\noindent 
On the other hand, in $\Omega_n$, for $z,\cX$ and $\cY$ as before, the expression of $e^k_{\cX}:\partial \mathbb{C}(\cX)\longrightarrow(\Omega_n)^{k-1}$ is $e^k_{\cX}=a^k_{\cX}\circ(g_{\cX}^{-1})|_{\partial \mathbb{C}(\cX)}$, where for $x\in F(\overline{w}uv,Y,S)$, the isometry $a^k_{\cX}=a^k_{\overline{w}u,v,Y,X}$ is as follows:
\[
F(\overline{w}uv,Y,S)\xrightarrow[g^h_{\overline{w}uv(\Pi_Y)}]{} \mathbb{C}(\overline{w}uv(\Pi_Y))\xrightarrow[E^h_{\overline{w}uv(\Pi_Y)}]{}\Dd^h(z\overline{w}u(\Pi_Y))=\Dd^h(z(\cY)).\]
\noindent
To show that $\Omega_n$ and $\Omega(A_{n-1})$ coincide, we verify that, for all parabolic $\cX=\overline{w}u(\Pi_X)\in\Phi[A_{n-1}] $, the cells $\Dd^k(\cX)$ and $D^k(\cX)$ are equal. To do so, we proceed by induction on the dimension $k$ of the cells. \\
\\
\noindent The 0-skeleton is made by just one point $\omega^0$ in the two complexes. Therefore, it is clear that $\Dd^0(\emptyset)=D^0(\emptyset)=\omega^0$.\\
\\
\noindent Take now any root $\beta=\overline{w}u(\alpha_s)$ in $\Phi[A_{n-1}]$, and consider the model cells $\mathbb{A}(\beta)$ and $\mathbb{B}(\beta)$. As explained in Subsubsection \ref{cellulardescriptionOmegan}, the cell $\Dd^1(\beta)$ of $\Omega_n$ is the image of the model cell $\mathbb{B}(\beta)$ (isometric to $F(\overline{w}u, s,S)$) through the characteristic map $E^1_{\beta}$. $E^1_{\beta}$ extends the map $e^1_{\beta}$ defined on the boundary $\partial \mathbb{B}(\beta)$ as $e^1_{\beta}=a^1_{\overline{w}u(\alpha_s)}\circ (g_{\beta}^{-1})|_{\partial \mathbb{B}(\beta)}$, where
\[
a^1_{\beta}:\partial F(\overline{w}u,s,S)=\bigcup_{v\in \W_{\{s\}}} F(\overline{w}uv,\emptyset,S)=\{\overline{w}u(o),\overline{w}us(o)\}\longrightarrow \omega^0,
\]
\noindent In $\Omega(A_{n-1})$, the cell $D^1(\beta)$ is the image of $\mathbb{A}(\beta)$ under the characteristic map $\Lambda^1_{\beta}$, which extends $\lambda^1_{\beta}=\chi^1_{\beta}\circ (f_{\beta}^{-1})|_{\partial \mathbb{A}(\beta)}$, where
\[
\chi^1_{\beta}:\partial C[\beta]=\{-\beta,\beta\}\longrightarrow \omega^0.
\]
\noindent Observe now that $\beta=o_{\beta}\in C[\beta]$ is the image of the base-point $o_s=\alpha_s\in C[s]$ under the action of $\overline{w}u$. By composing $\tau(\overline{w}u(o_s)-\overline{w}u(o))|_{\partial F(\overline{w}u,s,S)}$ with $\chi^1_{\beta}$, we obtain a map $\partial C[\beta]\longrightarrow\omega^0=\Dd^0(\emptyset)=D^0(\emptyset)$ that is, for all $v\in \W_{\{s\}}=\{\id,s\}$
\[
F(\overline{w}uv,\emptyset,S)\xrightarrow[(\overline{w}u)^{-1}]{} F(v,\emptyset,S)\xrightarrow[\tau(o_s-o)]{}{}vC[\emptyset]=\{v(o_s)\}\xrightarrow[\overline{w}u]{}\overline{w}uv(o_s)=\overline{w}uv(\alpha_s)=\pm \beta\xrightarrow[\chi^1_{\beta}]{}\omega^0.
\]
\noindent The map above, written concisely as $\chi^1_{\beta}\circ \tau(\overline{w}u(o_s-o))$, coincides with $a^1_{\beta}$. This means that 
\begin{multline*}
e^1_{\beta}=a^1_{\beta}\circ (g^{-1}_{\beta})|_{\partial \mathbb{B}(\beta)}=\chi^1_{\beta}\circ \tau(\overline{w}u(o_s-o))\circ (g^{-1}_{\beta})|_{\partial \mathbb{B}(\beta)}=\\=\underbrace{\chi^1_{\beta}\circ(f_{\beta}^{-1})|_{\partial \mathbb{A}(\beta)}}_{\lambda^1_{\beta}}\underbrace{\circ(f_{\beta})|_{\partial C[\beta]}\circ \tau(\overline{w}u(o_s-o))\circ (g^{-1}_{\beta})|_{\partial \mathbb{B}(\beta)}}_{(im^1_{\beta})|_{\partial \mathbb{B}(\beta)}}.
\end{multline*}
\noindent We just obtained
\[
e^1_{\beta}=\lambda^1_{\beta}\circ (im^1_{\beta})|_{\partial \mathbb{B}(\beta)}.
\]
\noindent
The 1-skeleton of the original BEER complex $\Omega_n$ is
\[
(\Omega_n)^1=\left(\{\omega^0\}\sqcup\left(\bigsqcup_{\beta=\overline{w}u(\alpha_s)\in \Phi[A_{n-1}]}\mathbb{B}(\beta)\right)\right)\left.\right/\sim,
\]
\noindent where $x\sim e^1_{\beta}(x)$ for all $\beta\in \Phi[A_{n-1}]$ and $x\in \partial \mathbb{B}(\beta)$. In particular, we can write $e^1_{\beta}(x)$ as $\lambda^1_{\beta}\circ (im^1_{\beta})|_{\partial \mathbb{B}(\beta)}(x)$, hence the previous formula becomes
\[
(\Omega_n)^1=\left(\{\omega^0\}\sqcup\left(\bigsqcup_{\beta\in \Phi[A_{n-1}]}\mathbb{A}(\beta)\right)\right)\left.\right/\sim=\Omega^1(A_{n-1}),
\]
\noindent where for all $y=(im^1_{\beta})|_{\partial \mathbb{B}(\beta)}(x)$, we identify $y$ with $\lambda^1_{\beta}(y)$. In other words, the 1-skeleta of the two complexes $\Omega_n$ and $\Omega(A_{n-1})$ coincide. This also implies that the characteristic map $\Lambda^1_{\beta}:\mathbb{A}(\beta)\longrightarrow D^1(\beta)$ and $E^1_{\beta}:\mathbb{B}(\beta)\longrightarrow \Dd^1(\beta)$ satisfy
\[
E^1_{\beta}=\Lambda^1_{\beta}\circ im^1_{\beta}.
\]
\noindent We can conclude that for all roots $\beta\in \Phi[A_{n-1}]$, $D^1(\beta)$ and $\Dd^1(\beta)$ coincide.\\
\\
\noindent The base step of induction is concluded. We now show the induction step. Suppose that for all $1\leq h<k$ the $h$-skeleta of $\Omega_n$ and $\Omega(A_{n-1})$ coincide. In particular suppose that for all parabolic $\cT\subset \Phi[A_{n-1}]$ of rank $h<k$, the cells $\Dd^h(\cT)$ and $D^h(\cT)$ are equal and $E^h_{\cT}=\Lambda^h_{\cT}\circ im^h_{\cT}$.\\
\\
\noindent Take now a parabolic subset $\cX\subset \Phi[A_{n-1}]$ of cardinality $k$. Let us write $\cX$ as $\overline{w}u(\Pi_X)$ for some $X\subset S$, $u$ $(\emptyset,X)$-minimal and $\overline{w}\in u\W_Xu^{-1}$. The associated $k$-cells in $\Omega_
n$ and $\Omega(A_{n-1})$ are respectively $\Dd^k(\cX)=E^k_{\cX}(\mathbb{C(\cX)})$ and $D^k(\cX)=\Lambda
^k_{\cX}(\mathbb{D}(\cX))$.
Consider the face $F(\overline{w}uv,Y,S)$ of $F(\overline{w}u,X,S)$, isomorphic to the model cell $\mathbb{C}(\overline{w}u(\Pi_X))=\mathbb{C}(\cX)$ through $g_{\cX}$. As usual,  $Y\subset X$, $|Y|=h$,$v\in \W_X$ and $v$ is a $(\emptyset,Y)$-minimal element. As before, set $\cX=\overline{w}u(\Pi_X)$, $\cY=\overline{w}u(\Pi_Y)\subset\cX$ and $z=\overline{w}uvu^{-1}\overline{w}^{-1}$.\\
\\
\noindent We observe that the composition $\tau(\overline{w}uv(o_X)-\overline{w}uv(o))$ with $\chi^k_{z,\cY,\cX}$, gives a map that we describe below:
\begin{multline*}
    F(\overline{w}uv,Y,S)\xrightarrow[(\overline{w}uv)^{-1}]{} F(\id,Y,S)\xrightarrow[\tau(o_X-o)]{}{}F(\id,Y,X)\subset \partial C[X]\xrightarrow[\overline{w}uv]{}F(z,\cY,\cX)\xrightarrow[]{}\\ \xrightarrow[z^{-1}]{} F(\id,\cY,\cX)\xrightarrow[\tau(o_{\cY}-o_{\cX})]{}C[\cY]\xrightarrow[z]{}C[z(\cY)]\xrightarrow[f_{z(\cY)}]{}\mathbb{D}(z(\cY))\xrightarrow[\Lambda^h_{z(\cY)}]{}D^h(z(\cY)).
\end{multline*}
\noindent By induction hypothesis, since $|z(\cY)|=h<k$, we have that $D^h(z(\cY))=\Dd^h(z(\cY))=\Dd^h(\overline{w}uv(\Pi_Y))$ and that $\Lambda^h_{z(\cY)}\circ im^h_{z(\cY)}=E^h_{z(\cY)}$.
As for $k=1$, we see that $\chi^k_{z,\cY,\cX}\circ \tau(\overline{w}uv(o_X)-\overline{w}uv(o))$ is the attaching map $a^k_{\overline{w}u,v,Y,X}:F(\overline{w}uv,Y,S)\subset \partial F(\overline{w}u,X,S)\longrightarrow \Dd^h(\overline{w}uv(\Pi_Y))$. In fact, simplifying the writing, we obtain that \begin{equation}\label{eqcomposition}
    \chi^k_{z,\cY,\cX}\circ \tau(\overline{w}uv(o_X)-\overline{w}uv(o))=\Lambda^h_{z(\cY)}\circ f_{z(\cY)}\circ \tau(z(o_{\cY})-z(o_{\cX}))\circ \tau(\overline{
w}uv(o_X)-\overline{
w}uv(o)).
\end{equation}
\noindent As in Remark \ref{coxpolitopeofaparabolic}, since $\cX=\overline{w}u(\Pi_X)$ and $\cY=\overline{w}u(\Pi_Y)$, we have that $z(o_{\cY})=\overline{w}uv(o_Y)$ and $z(o_{\cX})=\overline{w}uv(o_X)$. Thus, the formula in Equation \ref{eqcomposition} becomes
\[
\Lambda^h_{z(\cY)}\circ f_{z(\cY)}\circ \tau(\overline{
w}uv(o_Y)-\overline{
w}uv(o_X))\circ \tau(\overline{
w}uv(o_X)-\overline{
w}uv(o))=\Lambda^h_{z(\cY)}\circ f_{z(\cY)}\circ \tau(\overline{
w}uv(o_Y)-\overline{
w}uv(o)).
\]
\noindent Now recall that $a^k_{\overline{w}u,v,Y,X}=E^h_{z(\cY)}\circ g^h_{z(\cY)}$. By induction hypothesis, we can replace $E^h_{z(\cY)}$ by $\Lambda^h_{z(\cY)}\circ im^h_{z(\cY)}$. Since \[im^h_{z(\cY)}=f_{z(\cY)}\circ \tau(\overline{w}uv(o_{Y}-o))\circ (g_{z(\cY)})^{-1},\]
\noindent we finally obtain 
\begin{multline*}
    a^k_{\overline{w}u,v,Y,X}=E^h_{z(\cY)}\circ g^h_{z(\cY)}=\Lambda^h_{z(\cY)}\circ f_{z(\cY)}\circ \tau(\overline{w}uv(o_{Y}-o))\circ (g_{z(\cY)})^{-1}\circ (g_{z(\cY)})=\\
    =\Lambda^h_{z(\cY)}\circ f_{z(\cY)}\circ \tau(\overline{w}uv(o_{Y}-o))=\chi^k_{z,\cY,\cX}\circ \tau(\overline{w}uv(o_X)-\overline{w}uv(o)).
\end{multline*}
\noindent Then, we got that, for all $x\in \partial\mathbb{C(\cX)}$ such that $g_{\cX}^{-1}(x)\in F(\overline{w}uv,Y,S)$,
\begin{multline*}
e^k_{\cX}(x)=a
^k_{\overline{w}u,v,Y,X}\circ (g_{\cX}^{-1})|_{\partial \mathbb{C}(\cX)}(x)=\chi^k_{z,\cY,\cX}\circ \tau(\overline{w}uv(o_X)-\overline{w}uv(o))\circ(g_{\cX}^{-1})|_{\partial \mathbb{C}(\cX)}(x)=\\
=\underbrace{\chi^k_{z,\cY,\cX}\circ (f_{\cX}^{-1})|_{\partial \mathbb{D}(\cX)}}_{\lambda^k_{\cX}} \circ \underbrace{(f_{\cX})|_{\partial C[\cX]} \circ \tau(\overline{w}uv(o_X)-\overline{w}uv(o))\circ(g_{\cX}^{-1})|_{\partial \mathbb{C}(\cX)}}_{im^k_{\cX}|_{\partial \mathbb{C}(\cX)}}(x)\\
\Longrightarrow e^k_{\cX}(x)=\lambda^k_{\cX} \circ im^k_{\cX}|_{\partial \mathbb{C}(\cX)}(x).
\end{multline*}
\noindent Analogously to what we obtained for $k=1$, a consequence of the equality above is that the $k$-cells $D^k(\cX)$ and $\Dd^k(\cX)$ coincide, the $k$-skeleta of $\Omega_n$ and $\Omega(A_{n-1})$ are equal, and $E^k_{\cX}=\Lambda^k_{\cX}\circ im^k_{\cX}$.\\
\\
\noindent By induction, we conclude that $\Omega_n$ and $\Omega(A_{n-1})$ are the same CW-complex.
\end{proof}

\section{The common covering of \texorpdfstring{$\Sigma(\Gamma)$}{TEXT} and \texorpdfstring{$\Omega(\Gamma)$}{TEXT}}\label{commoncovering}
Consider $\Gamma$ to be a Coxeter graph and $\VA[\Gamma]$ its virtual Artin group, with the two subgroups $\PVA[\Gamma]$ and $\KVA[\Gamma]$. In Subsection \ref{salbargam} we have built the Salvetti complex $\salbargam=\Sigma(\Gamma)$, whose fundamental group is isomorphic to the kernel $\KVA[\Gamma]$. Meanwhile, in Subsection \ref{generalBeer}, we constructed the BEER complex $\Omega(\Gamma)$, with $\pi_1$ isomorphic to $\PVA[\Gamma]$.
For $\Gamma=A_{n-1}$, the complex $\Omega(\Gamma)$ coincides with the space originally defined in \cite{BEER} and presented in Subsection \ref{originalBEERspace}. However, the local CAT(0) approach cannot be used to prove its asphericity, since
Theorem \ref{eqialityofspaces} shows that $\Omega_3$ is not locally CAT(0). To investigate this question, we will compare the (generalized) BEER complex to the Salvetti complex and find that these two spaces, for any $\Gamma$, share the same universal covering. Thanks to some existing results, we will be able to state that this common universal covering is contractible when $\Gamma$ is of spherical or affine type. \\
\\
We recall that for a given finite Coxeter graph $\Gamma$, $\salbar$ is homotopically equivalent to the complexified complement of the reflection arrangement of $\W[\Gamma]$, modulo the action of the Coxeter group (see \cite{Sal87}). Stating that $\salbar$ is aspherical is equivalent to saying that $\A[\Gamma]$ satisfies the $K(\pi,1)$-conjecture. We now wish to study the conjecture for the graph $\widehat{\Gamma}$. Since the set of vertices of $\widehat{\Gamma}$ correspond bijectively to the root system $\Phi[\Gamma]$, the only case in which $\widehat{\Gamma}$ is a finite graph is when $\Gamma$ is of spherical type.\\
\\
\noindent In this work, $\Gamma$ is not necessarily a finite Coxeter graph, and thus $\widehat{\Gamma}$ is not finite either. The spaces we built in Subsections \ref{salbargam} and \ref{generalBeer} are well-defined even when $\Gamma$ has a countable set of vertices, as well as their covering spaces. Moreover, the main result of this work (Theorem \ref{maintheorem}), which establishes the existence of a common covering between the two spaces, holds in the general setting. However, the asphericity of $\Sigma(\Gamma)$ is guaranteed only if $\Gamma$ is of spherical or affine type. In \cite{BellParThiel}, the authors prove the following fundamental result:
\begin{thm}[Theorem 4.1 in \cite{BellParThiel}.]\label{freeofinfBPT}
Let $\Gamma$ be a Coxeter graph of spherical type or of affine type, and let $\cX$ be a free of infinity subset of $\Phi[\Gamma]$. Then $\cX$ is finite, $\widehat{\Gamma}_{\cX}$ is of spherical type or of affine type, and $n_{sph}(\widehat{\Gamma}_{\cX})\leq n_{sph}(\Gamma).$   
\end{thm}
\noindent Recall that $n_{sph}$ denotes the spherical dimension of the Artin group, i.e., the maximal cardinality of a subset of the generators such that the standard parabolic subgroup generated by this set is of spherical type. Theorem \ref{freeofinfBPT} says, in particular, that there does not exist a free of infinity subset of vertices in $\widehat{\Gamma}$ with an infinite number of elements. Recall also that $\cX\subset \Phi[\Gamma]$ is said to be free of infinity if $\widehat{m}_{\beta,\gamma}\neq \infty$ for all $\beta,\gamma$ in $\cX$. Free of infinity subsets of generators are very important for the study of Artin groups. In particular, Ellis and Sk{\"o}ldberg (and also Godelle and Paris, with a different proof) demonstrated the following theorem: 
\begin{thm}[\cite{EllSk10, GodPar12}]\label{freeinfelsk}
    Let $\Gamma$ be a Coxeter graph on a finite set of vertices $S$. If for each $X\subset S$ free of infinity $\Gamma_X$ satisfies the $K(\pi,1)$-conjecture, then $\Gamma$ satisfies the $K(\pi,1)$-conjecture.
\end{thm}
\noindent Combining Theorem \ref{freeinfelsk} and Theorem \ref{freeofinfBPT}, along with the solutions to the conjecture for $\Gamma$ of spherical or affine type, Bellingeri, Paris and Thiel obtained the following theorem:
\begin{thm}[Theorem 6.3 in \cite{BellParThiel}]\label{gammahatkpi1}\nonumber
    Let $\Gamma$ be a Coxeter graph of spherical type, then $\A[\widehat{\Gamma}]$ satisfies the $K(\pi,1)$-conjecture.
\end{thm}
\noindent
From this result, it follows that when $\Gamma$ is of spherical type, $\salbargam=\Sigma(\Gamma)$ is aspherical. However, Theorem \ref{freeofinfBPT} works also for the affine case. The only obstacle to applying the same combination of results to show that $\Sigma(\Gamma)$ is aspherical for $\Gamma$ of affine type is that the Tits cone is not well-defined, and therefore, the classical statement of the $K(\pi,1)$-conjecture does not make sense. But if we extend the conjecture by requiring the Salvetti complex to be aspherical, then there is no need to assume that the set of generators is finite. Similarly, Theorem \ref{freeinfelsk} does not need in the hypotheses a finite Coxeter graph, as long as the free of infinity subsets we consider are finite, which holds in our case thanks to Theorem \ref{freeofinfBPT}. Thus, with the extended conjecture, we can generalize Theorem \ref{gammahatkpi1} to the affine case as well. Specifically, the following result holds:
\begin{thm}\label{gammahatkpiAFF} Let $\Gamma$ be a Coxeter graph of spherical type or of affine type, then $\A[\widehat{\Gamma}]$ satisfies the $K(\pi,1)$-conjecture. Namely, $\Sigma(\Gamma)$ is aspherical.
\end{thm}

\begin{proof}
    The only case to be checked is the affine one. For every finite subset $X\subseteq S $ of vertices, the finite full subgraph $\widehat{\Gamma_X}$ satisfies the hypothesis of Theorem \ref{freeinfelsk}, hence $\overline{Sal}(\widehat{\Gamma_X})$ is aspherical. Moreover,
    \[
    \overline{Sal}(\widehat{\Gamma})=\bigcup_{X\subseteq S,\,|X|<\infty}\overline{Sal}(\widehat{\Gamma_X}).
    \]
    \noindent Since every map from a compact sphere into $\overline{Sal}(\widehat{\Gamma})$ has image contained in some subcomplex, it factors through some $\overline{Sal}(\widehat{\Gamma_X})$ and is null-homotopic there. Hence $\overline{Sal}(\widehat{\Gamma})$ is aspherical.
\end{proof}
\noindent With this approach in mind, we now turn to the study of the common covering of the two spaces.\\
\\
\noindent Since we are working with cell complexes, all our (pointed) spaces $(X,x)$ will be path-connected, locally path-connected and semilocally simply-connected.
By the theory of covering spaces, we know that for every subgroup $H$ of the fundamental group $\pi_1(X,x)$, there corresponds a covering space $p: (\overline{X}_H, \overline{x})\longrightarrow( X,x)$ such that the image under the induced group homomorphism is $p^*(\pi_1(\overline{X}_H,\overline{x}))=H$.\\
\\
\noindent We will consider a common subgroup, called $\LVA[\Gamma]$ of $\PVA[\Gamma]$ and $\KVA[\Gamma]$. For such a subgroup, we can construct two different complexes: a covering of $\Sigma(\Gamma)$ and a covering of $\Omega(\Gamma)$, both with fundamental group isomorphic to $\LVA[\Gamma]$. The main result of this section is Theorem \ref{maintheorem}, where we show that there is an isomorphism between these two spaces. This implies that the Salvetti complex $\Sigma(\Gamma)$ and the BEER complex $\Omega(\Gamma)$ have a common covering, and thus the same universal covering, that is contractible when $\KVA[\Gamma]\cong \A[\widehat{\Gamma}]$ satisfies the $K(\pi,1)$-conjecture.

\subsection{The group \texorpdfstring{$\LVA[\Gamma]$}{TEXT} and the two regular coverings}
As before, let $\Gamma$ be a Coxeter graph, for the moment, not necessarily finite. We recall that we have two short exact sequences of groups:

  \begin{align*}
  \KVA[\Gamma]\underset{i}{\hookrightarrow} \VA[\Gamma]\underset{\pi_K}{\twoheadrightarrow}\W[\Gamma]; \qquad\qquad
  \PVA[\Gamma]\underset{i}{\hookrightarrow} \VA[\Gamma]\underset{\pi_P}{\twoheadrightarrow}\W[\Gamma].
  \end{align*}
Consider the restrictions $\pi_K|_{\PVA[\Gamma]}:\PVA[\Gamma]\longrightarrow \W[\Gamma]$ and $\pi_P|_{\KVA[\Gamma]}:\KVA[\Gamma]\longrightarrow \W[\Gamma]$. Observe that $Ker(\pi_K|_{\PVA[\Gamma]})=\PVA[\Gamma]\cap \KVA[\Gamma]=Ker(\pi_P|_{\KVA[\Gamma]})$. Note also that the restriction of $\pi_K$ to $\PVA[\Gamma]$ and the restriction of $\pi_P$ to $\KVA[\Gamma]$ are both surjective.

\begin{defn}
    The intersection $\PVA[\Gamma]\cap \KVA[\Gamma]$ is denoted by $\LVA[\Gamma]$ and it is a normal subgroup of both $\PVA[\Gamma]$ and $\KVA[\Gamma]$.
\end{defn}
\noindent
 By definition of the kernel $\LVA[\Gamma]$, we have two short exact sequences of groups:
 \begin{align*}
 1\rightarrow \LVA[\Gamma]\xrightarrow[]{i}\KVA[\Gamma]\xrightarrow[]{\pi_P}\W[\Gamma]\rightarrow 1; \qquad\qquad  1\rightarrow \LVA[\Gamma]\xrightarrow[]{i}\PVA[\Gamma]\xrightarrow[]{\pi_K}\W[\Gamma]\rightarrow 1
  ;
  \end{align*}
which imply that $\W[\Gamma]\cong \,\KVA[\Gamma]/{\LVA[\Gamma]}\cong \PVA[\Gamma]/\LVA[\Gamma]$.
Being $\LVA[\Gamma]$ normal in both groups, we have that there exist two regular coverings  $p_{\Sigma}:(\overline{\Sigma}(\Gamma),\overline{\sigma}^0)\longrightarrow (\Sigma(\Gamma),\sigma^0)$ and  $p_{\Omega}:(\overline{\Omega}(\Gamma),\overline{\omega}^0)\longrightarrow (\Omega(\Gamma),\omega^0)$ such that 
\[p_{\Sigma}^*(\pi_1(\overline{\Sigma}(\Gamma),\overline{\sigma}^0))=\LVA[\Gamma]=p_{\Omega}^*(\pi_1(\overline{\Omega}(\Gamma),\overline{\omega}^0)).\]

\noindent By the theory of regular covering spaces, the deck group of the covering $p_{\Sigma}$ is isomorphic to $\nicefrac{\pi_1(\Sigma(\Gamma))}{\pi_1(\overline{\Sigma}(\Gamma))}\cong \W[\Gamma]$. Similarly, the deck group of the covering $p_{\Omega}$ is also isomorphic to $\W[\Gamma]$.
\begin{nt}\textbf{Action of the Coxeter group on the fibers of the covering space}\label{fiberaction}\\
In general, let $p:(\overline{X},\overline{x}^0)\longrightarrow( X,x^0)$ be a covering, bearing in mind that in our case the base space $(X,x^0)$ will be $(\Sigma(\Gamma),\sigma^0)$ or $(\Omega(\Gamma),\omega^0)$. The fiber $p^{-1}(x^0)$ is in bijection with the right cosets of $p^*(\pi_1(\overline{X},\overline{x}^0))=:H$ in $\pi_1(X,x^0)$. This bijection is obtained as follows: for any $g\in \pi_1(X,x^0)$, which represents a homotopy class of a loop $g:[0,1]\longrightarrow X$ such that $g(0)=g(1)=x^0$, there exists a unique lift $\overline{g}$ in $\overline{X}$ such that $\overline{g}(0)=\overline{x}^0$. We associate to the element $g $ the point $\overline{g}(1)$, which depends only on the right coset $Hg$. We can easily show that this defines the required bijective correspondence. In this matter, we compose the elements of the fundamental group from left to right.\\
\\
\noindent If the covering $p$ is regular, then $H$ is normal in $\pi_1(X,x^0)$, so for all $g \in \pi_1(X,x^0) $ we have that $Hg=gH$ and we can adopt the more familiar notation of left cosets. Thus, we establish a correspondence between the quotient $H\backslash\pi_1(X,x^0)=\pi_1(X,x^0)/H=:G$ and the fiber $p^{-1}(x^0)$. This also gives an action of $G$ on the fiber, coming from the action of $G$ on itself by right multiplication. Precisely, for all $g'$ in $G$ and $\overline{x}$ in $p^{-1}(x^0)$, we define $\overline{x}\cdot g'$ as $(\overline{g\cdot g'})(1)$, where $g$ is an element of $G$ such that its lift satisfies $\overline{g}(1)=\overline{x}$. Notice that the composition of loops $g\cdot g'$ means that we first perform $g$ and then $g'$. In particular, if $g=\id$, its lift $\overline{g}$ is the constant path at $\overline{x}^0$, and, for this reason, we write $\overline{x}^0=\overline{x}(\id)$. \\
\\
\noindent When the covering $p$ is regular, we also have a natural isomorphism between the quotient $\pi_1(X,x^0)/p^*(\pi_1(\overline{X},\overline{x}^0))=:G$ and the deck group $\mathrm{Deck}(p)$. Recall that $\mathrm{Deck}(p)$ is the group of automorphisms $d:\overline{X}\longrightarrow \overline{X}$ such that $p\circ d=d$ (for further information, see \cite[Chapter III, Sections 5,6]{bredon1993topology}). Each deck transformation $d\in \mathrm{Deck}(p)$ permutes the points in a given fiber of $p$, and it is completely determined by the image of a point $\overline{x}\in p^{-1}(x^0)$. The isomorphism
    $G\longrightarrow \mathrm{Deck}(p)$ is defined by $g\longmapsto d_g$, where $d_g$ is the unique element in $\mathrm{Deck}(p)$ such that $d_g(\overline{x}^0)=\overline{x}^0\cdot g$. This gives a left action of $G$ on the covering space $(\overline{X},\overline{x}^0)$ by
\begin{align}\label{actionGoncovering}
      G\times \overline{X}&\longrightarrow \overline{X},
\nonumber\\ (g,\overline{x})&\longmapsto d_g(\overline{x}),
\end{align}
\noindent which is given on the fiber $p^{-1}(x^0)$ by the left multiplication of labels. Indeed, let $\overline{x}(h):=\overline{x}^0\cdot h$ denote the point of the fiber labeled by $h\in G$, then
\[
d_g(\overline{x}(h))=d_{g}(\overline{x}^0\cdot h)=d_{g}(\overline{x}^0)\cdot h=(\overline{x}^0\cdot g)\cdot h=\overline{x}^0\cdot (gh)=\overline{x}(gh).
\]
 
\noindent In our case, $G\cong \W[\Gamma]$, so up to choosing a basepoint $\overline{x}^0$ in $\overline{X}$, for all $z\in \W[\Gamma]$, the point $\overline{x}(z)\in p^{-1}(x^0)$ is, by definition, the endpoint of the lift $\overline{z}(1)$. With this notation, the right monodromy action of the Coxeter group on the fiber of the regular covering is given by
\begin{align}\label{actionWonfiber}
     p^{-1}(x^0)\times \W[\Gamma]&\longrightarrow p^{-1}(x^0),
\nonumber\\ \overline{x}(z)\cdot z'&\longmapsto \overline{x}(zz').
\end{align}
Observe that the endpoint of the lift of $zz'$ starting at $\overline{x}^0$ coincides with the endpoint of the lift of $z'$ starting at $\overline{x}(z)$.
\end{nt}
\noindent Having established this notation, we now proceed to study the CW-structure of $\overline{\Sigma}(\Gamma)$ and $\overline{\Omega}(\Gamma)$.

\subsection{The covering space \texorpdfstring{$\overline{\Sigma
}(\Gamma)$}{TEXT}}

The cellular definition of $\Sigma(\Gamma)$ was provided in Definition \ref{sigmadef}. To summarize here, we recall that in $\Sigma(\Gamma)$ there are:
\begin{itemize}
\item one vertex $\sigma^0$;
    \item an oriented 1-cell $\Delta^1(\beta)$ for each $\beta\in \Phi[\Gamma]$, with $\mathfrak{s}(\Delta^1(\beta))=\mathfrak{t}(\Delta^1(\beta))=\sigma^0$;
    \item a $k$-cell $\Delta^k(\cX)$ for each $\cX\in \Rf_k$ with $k\geq 2$. The boundary of such a cell is composed of the $h$-cells $\Delta^{h}(\cY)$ for $\cY\subset\cX$ and $|\cY|=h\leq k-1$.
\end{itemize}
\noindent The orientation on the 1-skeleton is induced by the orientation of the edges of the copies $\mathbb{D}(\cX)$ of the Coxeter polytopes $C[\cX]$. \\
\\
\noindent The group $\LVA[\Gamma]$ is normal in $\KVA[\Gamma]\cong \pi_1(\Sigma(\Gamma))$, so by choosing a suitable basepoint $\overline{\sigma}^0\in \overline{\Sigma}(\Gamma)$, there exists a regular covering $p_{\Sigma}:(\overline{\Sigma}(\Gamma),\overline{\sigma}^0)\longrightarrow (\Sigma(\Gamma),\sigma^0)$ with deck group $\W[\Gamma]$. Recall that, if $\beta=w(\alpha_s)$ with $w\in \W[\Gamma]$ and $s\in S$, then $\KVA[\Gamma]$ is an Artin group generated by the elements $\delta_{\beta}=w\cdot (\sigma_s)=\iota_{\W}\;(w) \sigma_s \;\iota_{\W}(w)^{-1}$ (using the same notation as in Subsection \ref{virtualartingroups}). The homomorphism $\pi_P$ is defined by $\pi_P(\sigma_s)=s$ and $\pi_P(\tau_s)=s$.  
 We now detail the restriction of $\pi_P$:
\begin{align}\label{dovevadeltabeta}
\pi_P|_{\KVA[\Gamma]}:\KVA[\Gamma]&\longrightarrow \W[\Gamma]\nonumber\\
    \iota_{\W}(w) \;\sigma_s \; \iota_{\W}(w)^{-1}=\delta_{\beta}&\longmapsto  r_{\beta}= wsw^{-1},
\end{align}
which is the projection of $\KVA[\Gamma]$ onto its quotient $\KVA[\Gamma]/\LVA[\Gamma]$.
\begin{nt}
    From now on, for any Coxeter polytope $C[\cX]$, associated with a set of roots $\cX \in \Rf$ which is AP and of spherical type, we choose the basepoint as the vertex $\id(o_{\cX})=o_{\cX}$. Observe that $C[\cX]$ is path-connected, locally path-connected and simply connected.
\end{nt}
\noindent It is generally known that the covering $(\overline{X},\overline{x}^0)$ of a CW-complex $(X,x)$ has a CW-complex structure obtained by lifting the characteristic maps to $\overline{X}$. Our goal now is to describe the cellular structure of the covering space of $\Sigma(\Gamma)$. We will first define a CW-complex $\overline{\Sigma}(\Gamma)$ inductively on its skeleta, endowed with a free and properly discontinuous left action of the Coxeter group. Then, we will observe that the quotient under this action indeed yields $\Sigma(\Gamma)$, meaning that the projection $p_{\Sigma}:\overline{\Sigma}(\Gamma)\longrightarrow \Sigma(\Gamma)$ is a covering map.
\begin{defn}
The complex $\overline{\Sigma}(\Gamma)$ has the following cellular structure.\\
\\
\textbf{Description of the 0-skeleton.}
    The 0-skeleton is a discrete set of vertices 
 $\overline{\Sigma}^0(\Gamma)=\{\overline{\sigma}(z)\,|\,z\in \W[\Gamma]\}$, which is in one-to-one correspondence with $\W[\Gamma]$. \\
 \\
\textbf{Description of the 1-skeleton.} 
For all $z\in \W[\Gamma]$ and all $\beta\in \Phi[\Gamma]$, we take a copy $\mathbb{A}(z,\beta)$ of the Coxeter segment $C[\beta]=[\beta,-\beta]$, identified with $C[\beta]$ via an isometry $f_{z,\beta}:C[\beta]\longrightarrow \mathbb{A}(z,\beta)$. We define $\overline{\varphi}^1_{z,\beta}:\partial \mathbb{A}(z,\beta)\longrightarrow \overline{\Sigma}^0(\Gamma)$ by setting $\overline{\varphi}^1_{z,\beta}(f_{z,\beta}(\beta))=\overline{\sigma}(z)$ and $\overline{\varphi}^1_{z,\beta}(f_{z,\beta}(-\beta))=\overline{\sigma}(zr_{\beta})$. Then, we define: 
\[\overline{\Sigma}^1(\Gamma)=\left(\overline{\Sigma}^0(\Gamma)\sqcup\left(\bigsqcup_{z\in \W[\Gamma],\beta\in \Phi[\Gamma]}\mathbb{A}(z,\beta)\right)\right)/\sim,\]
where, for all $z\in \W[\Gamma]$, all $\beta$ in $\Phi[\Gamma]$ and all $x\in \partial \mathbb{A}(z,\beta)$, we have $x\sim \overline{\varphi}^1_{z,\beta}(x)$. The embedding of $\mathbb{A}(z,\beta)$ into $\left(\overline{\Sigma}^0(\Gamma)\sqcup\left(\bigsqcup_{z'\in \W[\Gamma],\beta'\in \Phi[\Gamma]}\mathbb{A}(z',\beta')\right)\right)$ induces a characteristic map $\overline{\phi}^1_{z,\beta}:\mathbb{A}(z,\beta)\longrightarrow \overline{\Sigma}^1(\Gamma)$, whose image is denoted by $\overline{\Delta}^1(z,\beta)$. Furthermore we orient $\mathbb{A}(z,\beta)$ from $f_{z,\beta}(\beta)$ towards $f_{z,\beta}(-\beta)$, which induces an orientation on the 1-cell $\overline{\Delta}^1(z,\beta)$, from $\overline{\sigma}(z)$ towards $\overline{\sigma}(zr_{\beta})$.
\\
\\
\textbf{Description of the $k$-skeleton.}
Let $k\geq 2$. We define the $k$-skeleton $\overline{\Sigma}^k(\Gamma)$ by induction on $k$. Since $\overline{\Sigma}^1(\Gamma)$ is already defined, we can assume that $k\geq 2$ and that $\overline{\Sigma}^h(\Gamma)$ is defined for all $h\leq k-1$.
In particular, for each $z\in \W[\Gamma]$ and each $\cY\subset \Phi[\Gamma]$ which is AP and of spherical type of rank $h\leq k-1$, we have a copy $\mathbb{D}(z,\cY)$ of the Coxeter polytope $C[\cY]$, which we identify with $C[\cY]$ via an isometry $f_{z,\cY}:C[\cY]\longrightarrow \mathbb{D}(z,\cY)$. We also have an $h$-cell $\overline{\Delta}^h(z,\cY)$ in $\overline{\Sigma}^{k-1}(\Gamma)$, whose interior is homeomorphic to the interior of $\mathbb{D}(z,\cY)$ via a continuous characteristic map $\overline{\phi}^h_{z,\cY}: \mathbb{D}(z,\cY)\longrightarrow \overline{\Delta}^h(z,\cY)$. If $h=|\cY|=1$, then there exists $\beta\in \Phi[\Gamma]$ such that $\cY=\{\beta\}$, and we assume $\mathbb{D}(z,\cY)=\mathbb{A}(z,\beta)$. If $\cY=\emptyset$, then as usual we adopt the following conventions: 
    \[
    V_{\emptyset}=\{0\},\qquad \W_{\emptyset}=\{\id\},\qquad o_{\emptyset}=0,\qquad C[\emptyset]=\{0\}.
    \]
    \noindent We also set $\mathbb{D}(z,\emptyset)=\{0\}$. Take $f_{z,\emptyset}:C[\emptyset]\longrightarrow \mathbb{D}(z,\emptyset)$ to be the identity map, and define $\overline{\varphi}_{z,\emptyset}^0:\mathbb{D}(z,\emptyset)\longrightarrow\overline{\Sigma}^0(\Gamma)$ to be the map sending $0$ to the vertex $\overline{\sigma}(z)$. Thus $\overline{\Delta}^0(z,\emptyset)=\overline{\sigma}{(z)}$.
\\
\\
\noindent Now, for each $\cX\in \Rf_k$ and for each $z\in \W[\Gamma]$ take a copy $\mathbb{D}(z,\cX)$ of the Coxeter polytope $C[\cX]$, identified with $C[\cX]$ via an isometry $f_{z,\cX}:C[\cX]\longrightarrow \mathbb{D}(z,\cX)$. Let $\cY\subset \cX$ with $|\cY|=h\leq k-1$ and $u\in \W_{\cX}$ an $(\emptyset,R_{\cY})$-minimal element.
Define now $\overline{\xi}^k_{z,u,\cY,\cX}:F(u,\cY,\cX)\longrightarrow \overline{\Sigma}^{k-1}(\Gamma)$ as the composition:
\[F(u,\cY,\cX)\xrightarrow[]{u^{-1}}F(\id,\cY,\cX)\xrightarrow[]{\tau(o_{\cY}-o_{\cX})}C[\cY]\xrightarrow[]{f_{zu,\cY}}\mathbb{D}(zu,\cY)\xrightarrow[]{\overline{\phi}^h_{zu,\cY}}\overline{\Delta}^h(zu,\cY)\subset \overline{\Sigma}^{k-1}(\Gamma).\]
In other words:
\[\overline{\xi}^k_{z,u,\cY,\cX}=\overline{\phi}^h_{zu,\cY}\circ f_{zu,\cY}\circ \tau(o_{\cY}-o_{\cX})\circ u^{-1}. \]
Recall that if $F(u,\cY,\cX)$ is a face of $C[\cX]$ with $\cY$ and $u$ as before, the faces of $F(u,\cY,\cX)$ are of the form $F(uv,\cZ,\cX)=conv\{uvz'(o_{\cX})\,|\,z'\in \W_{\cZ}\}$ with $\cZ\subset \cY$ and $v\in \W_{\cY}$, $(\emptyset,R_{\cZ})$-minimal.

 \begin{lem}\label{goodefxibar}
    Let $\cX,\cY,\cZ \in \Rf$ be such that $\cZ\subset \cY\subset \cX$, and $|\cX|=k,|\cY|=h,|\cZ|=l$. Take $u\in \W_{\cX}$, $u$ $(\emptyset,R_{\cY})$-minimal and $v\in \W_{\cY}$, $v$ $(\emptyset,R_{\cZ})$-minimal. We have that $F(uv,\cZ,\cX)\subset F(u,\cY,\cX)\subset C[\cX]$. Then, the restriction of $\overline{\xi}^k_{z,u,\cY,\cX}$ to $F(uv,\cZ,\cX)$ coincides with $\overline{\xi}^k_{z,uv,\cZ,\cX}$. 
\end{lem}
\noindent The proof of this lemma follows similar arguments from previous sections.
Thus, we can define a continuous map $\overline{\xi}^k_{z,\cX}:\partial C[\cX]\longrightarrow \overline{\Sigma}^{k-1}(\Gamma)$ as follows. Let $x\in \partial C[\cX]$. We choose $\cY\subset \cX$ and $u\in \W_{\cX}$ an $(\emptyset, R_{\cY})$-minimal element such that $x\in F(u,\cY,\cX)$, and set $\overline{\xi}^k_{z,\cX}(x)=\overline{\xi}^k_{z,u,\cY,\cX}(x)$. Thanks to Lemma \ref{goodefxibar}, this definition does not depend on the choice of $F(u,\cY,\cX)$, and the map $\overline{\xi}^k_{z,\cX}$ is continuous. We then define $\overline{\varphi}^k_{z,\cX}:\partial \mathbb{D}(z,\cX)\longrightarrow \overline{\Sigma}^{k-1}(\Gamma)$ by setting:
\[
\overline{\varphi}^k_{z,\cX}=\overline{\xi}^k_{z,\cX}\circ (f^{-1}_{z,\cX})|_{\partial \mathbb{D}(z,\cX)}.
\]
The map $\overline{\varphi}^{k}_{z,\cX}$ is then continuous. Now for all $\cX\in \Rf_k$, for all $z\in \W[\Gamma]$ and for all $x\in \partial \mathbb{D}(z,\cX)$, we write $x\sim \overline{\varphi}^k_{z,\cX}(x)$. The $k$-skeleton of the cell complex will be defined as  
    \[
    \overline{\Sigma}^{k}(\Gamma):=\bigslant{\left(\overline{\Sigma}^{k-1}(\Gamma)\bigsqcup \left( \bigsqcup_{\cX \in \Rf_k,z\in \W[\Gamma]}\mathbb{D}(z,\cX)\right)\right)}{\sim}.
    \]
    For every $\cX\in \Rf_k$ and every $z\in \W[\Gamma]$, the cell $\mathbb{D}(z,\cX)$ has a natural characteristic map $\overline{\phi}^k_{z,\cX}$ to $\overline{\Sigma}^k(\Gamma)$, whose image is denoted by $\overline{\Delta}^k(z,\cX)$. This $\overline{\phi}^k_{z,\cX}$ is given by the composition of the inclusion of $\mathbb{D}(z,\cX)$ into the disjoint union and the quotient projection. \\
    \\
    \textbf{Description of the complex.} We set $\overline{\Sigma}(\Gamma)=\bigcup_{k=0}^{\infty}\overline{\Sigma}^{k}(\Gamma)$, endowed with the weak topology.
    \end{defn}
    \noindent    The right monodromy action of $\W[\Gamma]$ on the fiber labels is given by $\sigma(z)\cdot z'=\sigma(zz').$ The corresponding deck action of $\W[\Gamma]$ on $\overline{\Sigma}(\Gamma)$ is the left action determined by $z'\cdot \sigma(z)=\sigma(z'z)$. We extend this cellularly by setting $z'\cdot \overline{\Delta}^k(z,\cX)=\overline{\Delta}^k(z'z,\cX)$, with $f_{z'z,\cX}\circ f_{z,\cX}^{-1}:\mathbb{D}(z,\cX)\longrightarrow\mathbb{D}(z'z,\cX)$ inducing the action of $z'$ on the corresponding closed cell. With this convention the attaching maps are equivariant: if a boundary face of $\overline{\Delta}^k(z,\cX)$ is attached to $\overline{\Delta}^h(zu,\cY)$, then applying $z'$ sends it to $\overline{\Delta}^h(z'zu,\cY)$, which is precisely the corresponding boundary cell of $\overline{\Delta}^k(z'z,\cX)$. The following result is a straightforward consequence of the description of $\Sigma(\Gamma)$ and $\overline{\Sigma}(\Gamma)$.

    \begin{lem}
        We have $\overline{\Sigma}(\Gamma)/\W[\Gamma]=\Sigma(\Gamma)$, and the regular covering associated with this quotient is the covering induced by the epimorphism $\pi_P:\KVA[\Gamma]\longrightarrow \W[\Gamma]$.
    \end{lem}
\begin{cor}\label{sigmatilde}
     The regular covering space $\overline{\Sigma}(\Gamma)$ has the following CW-structure: 
    \begin{itemize}
        \item The 0-skeleton consists of a single point $\overline{\sigma}(z)$ for each $z\in \W[\Gamma]$.
        \item The 1-skeleton is composed of an oriented 1-cell $\overline{\Delta}^1(z,\beta)$ for each $z\in \W[\Gamma]$ and $\beta\in \Phi[\Gamma]$, with source $\mathfrak{s}(\overline{\Delta}^1(z,\beta))=\overline{\sigma}(z)$ and target $\mathfrak{t}(\overline{\Delta}^1(z, \beta))=\overline{\sigma}(zr_{\beta})$.
        \item For any $k\geq 2$, there is a $k$-cell $\overline{\Delta}^k(z,\cX)$ for each $z\in \W[\Gamma]$ and $\cX\in \Rf_k$. The boundary of this cell is made up of $h$-cells of the form $\overline{\Delta}^{h}(zu,\cY)$ where $\cY\subset \cX$, $|\cY|=h\leq k-1$, $u\in \W_{\cX}$ and $u$ is $(\emptyset, R_{\cY})$-minimal.
    \end{itemize}
\end{cor}    

\noindent We include a picture in Figure \ref{fig2cellsigmatilde} to represent the 2-cell $\overline{\Delta}^2(z, \{\beta,\gamma\})$ of $\overline{\Sigma}(\Gamma)$ for $\cX=\{\beta,\gamma\}\in \Rf_2$, with the orientations fixed on the edges.
\begin{figure}[ht]
    \centering
    \includegraphics[width=10cm]{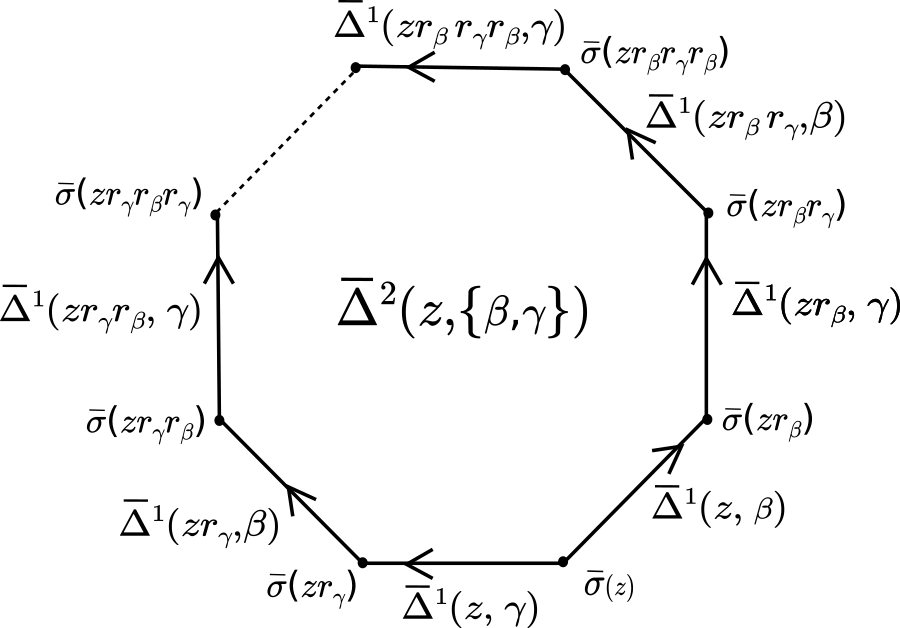}
    \caption{A 2-cell of the covering space $\overline{\Sigma}(\Gamma)$.}
    \label{fig2cellsigmatilde}
\end{figure}

\subsection{The covering space \texorpdfstring{$\overline{\Omega
}(\Gamma)$}{TEXT}}
In the same way as we did for $\overline{\Sigma}(\Gamma)$, we describe the covering space $\overline{\Omega}(\Gamma)$, whose fundamental group is $\LVA[\Gamma]$. Recall that the BEER complex $\Omega(\Gamma)$ is defined skeleton-wise inductively as follows. There are:
\begin{itemize}
    \item one vertex $\omega^0$;
    \item there is an oriented 1-cell $D^1(\beta)$ for each $\beta\in \Phi[\Gamma]$, with $\mathfrak{s}(D^1(\beta))=\mathfrak{t}(D^1(\beta))=\omega^0$;
    \item there is a $k$-cell $D^k(\cX)$ for each $\cX\in \Rf_k$. The boundary of this cell is formed by the $h$-cells  $D^{h}(u(\cY))$ where $\cY\subset \cX$, $|\cY|=h\leq k-1$, $u\in \W_{\cX}$ and $u$ is $(\emptyset, R_{\cY})$-minimal.
\end{itemize}
Again, the orientation on the 1-skeleton is induced by the orientation of the edges of the copies $\mathbb{D}(\cX)$ of the Coxeter polytopes $C[\cX]$.\\
\\
\noindent There exists a choice for $\overline{\omega}^0\in \overline{\Omega}(\Gamma)$ and a regular covering $p_{\Omega}:(\overline{\Omega}(\Gamma), \overline{\omega}^0)\longrightarrow (\Omega(\Gamma),\omega^0)$ with deck group $\W[\Gamma]$. Recall that $\PVA[\Gamma]$ is generated by the elements $\zeta_{\beta}=\iota_{\W}(w)\; \tau_s  \sigma_s \;\iota_{\W}(w)^{-1}$ for $\beta=w(\alpha_s) \in \Phi[\Gamma]$. The homomorphism $\pi_K$ is defined by $\pi_K(\sigma_s)=1$ and $\pi_K(\tau_s)=s$.  
We have the group homomorphism
\begin{align*}
\pi_K|_{\PVA[\Gamma]}:\PVA[\Gamma]&\longrightarrow \W[\Gamma]\\
    \zeta_{\beta}&\longmapsto r_{\beta}.   
\end{align*}

\begin{rmk}\label{rmksimply}
    For any $a\in \W[\Gamma]$ and $\beta,\gamma\in \Phi[\Gamma]$, set $\beta'=a(\beta)$ and $\gamma'=a(\gamma)$. Then $r_{\beta'}=ar_{\beta}a^{-1}$, and $r_{\beta'}(\gamma')=ar_{\beta}a^{-1}a(\gamma)=ar_{\beta}(\gamma)$.
    Moreover, if $\{\beta,\gamma\}\in \Rf$ and $\beta_k,\gamma_k$ are defined for $k\geq 1$ as in Equation \ref{betakappa}, applying recursively the second equality, we easily get $\beta'_k=a(\beta_k)$. 
\end{rmk}
\noindent As we did for $\overline{\Sigma}(\Gamma)$, we now describe a CW-complex $\overline{\Omega}(\Gamma)$, which we will later observe to be the regular covering of $\Omega(\Gamma)$ associated with the epimorphism $\pi_K:\PVA[\Gamma]\longrightarrow \W[\Gamma]$. As a consequence, we will have $\pi_1(\overline{\Omega}(\Gamma))=\LVA[\Gamma]$.

\begin{defn} The complex $\overline{\Omega}(\Gamma)$ has the following cellular structure.\\
\\
\textbf{Description of the 0-skeleton.}
    The 0-skeleton is a discrete set of vertices 
 $\overline{\Omega}^0(\Gamma)=\{\overline{\omega}(z)\,|\,z\in \W[\Gamma]\}$, in one-to-one correspondence with $\W[\Gamma]$. \\
 \\
\textbf{Description of the 1-skeleton.} 
For all $z\in \W[\Gamma]$ and all $\beta\in \Phi[\Gamma]$, we take a copy $\mathbb{A}(z,\beta)$ of the Coxeter segment $C[\beta]=[\beta,-\beta]$, which we identify with $C[\beta]$ via an isometry $f_{z,\beta}:C[\beta]\longrightarrow \mathbb{A}(z,\beta)$. We define now $\overline{\lambda}^1_{z,\beta}:\partial \mathbb{A}(z,\beta)\longrightarrow \overline{\Omega}^0(\Gamma)$ by setting $\overline{\lambda}^1_{z,\beta}(f_{z,\beta}(\beta))=\overline{\omega}(z)$ and $\overline{\lambda}^1_{z,\beta}(f_{z,\beta}(-\beta))=\overline{\omega}(zr_{\beta})$. Then we set 
\[\overline{\Omega}^1(\Gamma)=\left(\overline{\Omega}^0(\Gamma)\sqcup\left(\bigsqcup_{z\in \W[\Gamma],\beta\in \Phi[\Gamma]}\mathbb{A}(z,\beta)\right)\right)/\sim,\]
where, for all $z\in \W[\Gamma]$, all $\beta$ in $\Phi[\Gamma]$ and all $x\in \partial \mathbb{A}(z,\beta)$, we have $x\sim \overline{\lambda}^1_{z,\beta}(x)$. The embedding of $\mathbb{A}(z,\beta)$ into $\left(\overline{\Omega}^0(\Gamma)\sqcup\left(\bigsqcup_{z'\in \W[\Gamma],\beta'\in \Phi[\Gamma]}\mathbb{A}(z',\beta')\right)\right)$ induces a characteristic map $\overline{\Lambda}^1_{z,\beta}:\mathbb{A}(z,\beta)\longrightarrow \overline{\Omega}^1(\Gamma)$ whose image is denoted by $\overline{D}^1(z,\beta)$. Furthermore we orient $\mathbb{A}(z,\beta)$ from $f_{z,\beta}(\beta)$ towards $f_{z,\beta}(-\beta)$, which induces an orientation on the 1-cell $\overline{D}^1(z,\beta)$, from $\overline{\omega}(z)$ towards $\overline{\omega}(zr_{\beta})$.
\\
\\
\textbf{Description of the $k$-skeleton.}
Let $k\geq 2$. We define the $k$-skeleton $\overline{\Omega}^k(\Gamma)$ by induction on $k$. Since $\overline{\Omega}^1(\Gamma)$ is already defined, we can assume that $k\geq 2$ and that $\overline{\Omega}^h(\Gamma)$ is defined for all $h\leq k-1$.
In particular, for each $z\in \W[\Gamma]$ and each $\cY\subset \Phi[\Gamma]$ AP and of spherical type of rank $h\leq k-1$, we have a copy $\mathbb{D}(z,\cY)$ of the Coxeter polytope $C[\cY]$ which we identify with $C[\cY]$ via an isometry $f_{z,\cY}:C[\cY]\longrightarrow \mathbb{D}(z,\cY)$, and we have an $h$-cell $\overline{D}^h(z,\cY)$ in $\overline{\Omega}^{k-1}(\Gamma)$ whose interior is homeomorphic to the interior of $\mathbb{D}(z,\cY)$ through a continuous characteristic map $\overline{\Lambda}^h_{z,\cY}: \mathbb{D}(z,\cY)\longrightarrow \overline{D}^h(z,\cY)$. If $h=|\cY|=1$, then there exists $\beta\in \Phi[\Gamma]$ such that $\cY=\{\beta\}$, and we assume $\mathbb{D}(z,\cY)=\mathbb{A}(z,\beta)$. If $\cY=\emptyset$, then as usual we adopt the following conventions: 
    \[
    V_{\emptyset}=\{0\},\qquad \W_{\emptyset}=\{\id\},\qquad o_{\emptyset}=0,\qquad C[\emptyset]=\{0\}.
    \]
    \noindent We also set $\mathbb{D}(z,\emptyset)=\{0\}$. Take $f_{z,\emptyset}:C[\emptyset]\longrightarrow \mathbb{D}(z,\emptyset)$ to be the identity map, and define $\overline{\lambda}_{z,\emptyset}^0:\mathbb{D}(z,\emptyset)\longrightarrow\overline{\Omega}^0(\Gamma)$ to be the map sending $0$ to the vertex $\overline{\omega}(z)$. Thus $\overline{D}^0(z,\emptyset)=\overline{\omega}{(z)}$.
\\
\\
Take now, for each $\cX\in \Rf_k$ and for each $z\in \W[\Gamma]$, a copy $\mathbb{D}(z,\cX)$ of the Coxeter polytope $C[\cX]$, identified with $C[\cX]$ via an isometry $f_{z,\cX}:C[\cX]\longrightarrow \mathbb{D}(z,\cX)$. Let $\cY\subset \cX$ with $|\cY|=h\leq k-1$ and $u\in \W_{\cX}$ an $(\emptyset,R_{\cY})$-minimal element.
Define now $\overline{\chi}^k_{z,u,\cY,\cX}:F(u,\cY,\cX)\longrightarrow \overline{\Omega}^{k-1}(\Gamma)$ to be the following composition of maps:
\begin{align*}
    F(u,\cY,\cX)\xrightarrow[]{u^{-1}}F(\id,\cY,\cX)\xrightarrow[]{\tau(o_{\cY}-o_{\cX})}C[\cY]\xrightarrow[]{u}C[u(\cY)]\xrightarrow[]{}\\\xrightarrow[]{f_{zu^{-1},u(\cY)}}\mathbb{D}(zu^{-1},u(\cY))\xrightarrow[]{\overline{\Lambda}^h_{zu^{-1},u(\cY)}}\overline{D}^h(zu^{-1},u(\cY))\subset \overline{\Omega}^{k-1}(\Gamma).
\end{align*}
In other words,
\[\overline{\chi}^k_{z,u,\cY,\cX}=\overline{\Lambda}^h_{zu^{-1},u(\cY)}\circ f_{zu^{-1},u(\cY)}\circ \tau(u(o_{\cY}-o_{\cX})). \]
Recall that if $F(u,\cY,\cX)$ is a face of $C[\cX]$ with $\cY$ and $u$ as before, the faces of $F(u,\cY,\cX)$ are of the form $F(uv,\cZ,\cX)=conv\{uvz'(o_{\cX})\,|\,z'\in \W_{\cZ}\}$ with $\cZ\subset \cY$ and $v\in \W_{\cY}$, $(\emptyset,R_{\cZ})$-minimal.

 \begin{lem}\label{gooddeflambdabar}
    Let $\cX,\cY,\cZ \in \Rf$ be such that $\cZ\subset \cY\subset \cX$, and $|\cX|=k,|\cY|=h,|\cZ|=l$. Take $u\in \W_{\cX}$, $u$ $(\emptyset,R_{\cY})$-minimal and $v\in \W_{\cY}$, $v$ $(\emptyset,R_{\cZ})$-minimal. We have that $F(uv,\cZ,\cX)\subset F(u,\cY,\cX)\subset C[\cX]$. Then, the restriction of $\overline{\chi}^k_{z,u,\cY,\cX}$ to $F(uv,\cZ,\cX)$ coincides with $\overline{\chi}^k_{z,uv,\cZ,\cX}$. 
\end{lem}
\noindent The proof of this lemma is analogous to the ones we have seen in previous sections.
Therefore, we can define a continuous map $\overline{\chi}^k_{z,\cX}:\partial C[\cX]\longrightarrow \overline{\Omega}^{k-1}(\Gamma)$ as follows. Let $x\in \partial C[\cX]$. We choose $\cY\subset \cX$ and $u\in \W_{\cX}$ an $(\emptyset, R_{\cY})$-minimal element such that $x\in F(u,\cY,\cX)$, and we set $\overline{\chi}^k_{z,\cX}(x)=\overline{\chi}^k_{z,u,\cY,\cX}(x)$. Thanks to Lemma \ref{gooddeflambdabar}, this definition does not depend on the choice of $F(u,\cY,\cX)$, and the map $\overline{\chi}^k_{z,\cX}$ is continuous. Then we define $\overline{\lambda}^k_{z,\cX}:\partial \mathbb{D}(z,\cX)\longrightarrow \overline{\Omega}^{k-1}(\Gamma)$ by setting
\[
\overline{\lambda}^k_{z,\cX}=\overline{\chi}^k_{z,\cX}\circ (f^{-1}_{z,\cX})|_{\partial \mathbb{D}(z,\cX)}.
\]
\noindent The map $\overline{\lambda}^{k}_{z,\cX}$ is then continuous. Now, for all $\cX\in \Rf_k$, for all $z\in \W[\Gamma]$, and for all $x\in \partial \mathbb{D}(z,\cX)$, we write $x\sim \overline{\lambda}^k_{z,\cX}(x)$. The $k$-skeleton of the cell complex will be defined as  
    \[
    \overline{\Omega}^{k}(\Gamma):=\bigslant{\left(\overline{\Omega}^{k-1}(\Gamma)\bigsqcup \left( \bigsqcup_{\cX \in \Rf_k,z\in \W[\Gamma]}\mathbb{D}(z,\cX)\right)\right)}{\sim}.
    \]
    For every $\cX\in \Rf_k$ and every $z\in \W[\Gamma]$, the cell $\mathbb{D}(z,\cX)$ has a natural characteristic map $\overline{\Lambda}^k_{z,\cX}$ to $\overline{\Omega}^k(\Gamma)$, whose image is denoted by $\overline{D}^k(z,\cX)$. This $\overline{\Lambda}^k_{z,\cX}$ is given by the composition of the inclusion of $\mathbb{D}(z,\cX)$ into the disjoint union and the quotient projection. \\
    \\
    \textbf{Description of the complex.} We set $\overline{\Omega}(\Gamma)=\bigcup_{k=0}^{\infty}\overline{\Omega}^{k}(\Gamma)$, endowed with the weak topology.
    \end{defn}
    \noindent It is easily seen that we have a free and properly discontinuous left action of $\W[\Gamma]$ on $\overline{\Omega}(\Gamma)$. If $z,z'\in \W[\Gamma]$ and $\cX\in \Rf_k$, then $z'\cdot\overline{D}^k(z,\cX):=\overline{D}^k(z'z,\cX)$. More precisely, the isometry $f_{z'z,\cX}\circ f_{z,\cX}^{-1}:\mathbb{D}(z,\cX)\longrightarrow\mathbb{D}(zz',\cX)$ induces the action of $z'$ on $\overline{D}^k(z,\cX)$. The following result is a straightforward consequence of the description of $\Omega(\Gamma)$ and $\overline{\Omega}(\Gamma)$.

    \begin{lem}
        We have $\overline{\Omega}(\Gamma)/\W[\Gamma]=\Omega(\Gamma)$, and the regular covering associated with this quotient is the covering induced by the epimorphism $\pi_K:\PVA[\Gamma]\longrightarrow \W[\Gamma]$.
    \end{lem}
\begin{cor}\label{omegatilde}
    The regular covering space $\overline{\Omega}(\Gamma)$ has the following CW-structure: 
    \begin{itemize}
        \item The 0-skeleton consists of one point $\overline{\omega}(z)$ for each $z\in \W[\Gamma]$.
        \item The 1-skeleton consists of an oriented 1-cell $\overline{D}^1(z,\beta)$ for each $z\in \W[\Gamma]$ and $\beta\in \Phi[\Gamma]$, with source $\mathfrak{s}(\overline{D}^1( z,\beta))=\overline{\omega}(z)$ and target $\mathfrak{t}(\overline{D}^1(z,\beta))=\overline{\omega}(zr_{\beta})$.
        \item For any $k\geq 2$, there is a $k$-cell $\overline{D}^k(z,\cX)$ for each $z\in \W[\Gamma]$ and $\cX\in \Rf_k$. The boundary of this cell is composed of the $h$-cells of the form $\overline{D}^{h}(zu^{-1},u(\cY))$, where $\cY\subset \cX$, $|\cY|=h\leq k-1$, $u\in \W_{\cX}$, and $u$ is $(\emptyset, R_{\cY})$-minimal.
    \end{itemize}    
\end{cor}

\noindent We include Figure \ref{fig2cellomegatilde} to represent the 2-cell $\overline{D}^2(z,\{\beta,\gamma\})$ of $\overline{\Omega}(\Gamma)$ for $z\in \W[\Gamma]$, $\cX=\{\beta,\gamma\}\in \Rf_2$, and the orientations fixed on the edges. Remember that the roots $\beta_k,\gamma_k$ have the same expression as in Equation \ref{betakappa}.
\begin{figure}[ht]
    \centering
    \includegraphics[width=11cm]{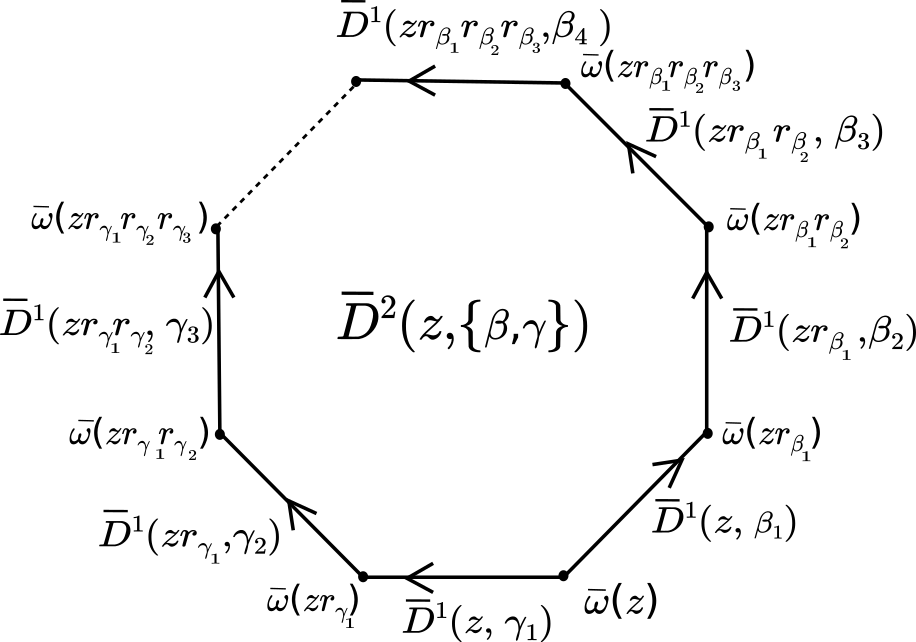}
    \caption{A 2-cell of the covering space $\overline{\Omega}(\Gamma)$.}
    \label{fig2cellomegatilde}
\end{figure}

\subsection{The isomorphism between \texorpdfstring{$\overline{\Sigma
}(\Gamma)$}{TEXT} and \texorpdfstring{$\overline{\Omega
}(\Gamma)$}{TEXT}}
In this last subsection, we define a map $g$ between the two regular coverings $\overline{\Sigma}(\Gamma)$ and $\overline{\Omega}(\Gamma)$ and we prove that this is a cellular isomorphism. Before doing so, we present a preliminary result.

\begin{lem}\label{lemmasimply}
    Let $\beta,\gamma\in \Phi[\Gamma]$ be such that $m:=\widehat{m}_{\beta,\gamma}\neq \infty$, and $\beta_k$ is defined as in Equation \ref{betakappa} for $2\leq k\leq m$. Then \[
    r_{\beta_1}r_{\beta_2}\dots r_{\beta_{k-1}}=\begin{cases}
        \Prod_L(r_{\beta},r_{\gamma};k-1) \mbox{ if $k$ is even,}\\
        \Prod_L(r_{\gamma},r_{\beta};k-1) \mbox{ if $k$ is odd.}
    \end{cases}
    \]
\end{lem}
\begin{proof}
    In general, $r_{\beta_k}=\Prod_R(r_{\gamma},r_{\beta};k-1)r_{\gamma}\Prod_L(r_{\beta},r_{\gamma};k-1)$ if $k$ is even, and $r_{\beta_k}=\Prod_R(r_{\beta},r_{\gamma};k-1)r_{\beta}\Prod_L(r_{\gamma},r_{\beta};k-1)$ if $k$ is odd. Suppose $k$ is even. Then, we compute
    \begin{align*}
        r_{\beta_{k-1}}r_{\beta_k}=\Prod_R(r_{\beta},r_{\gamma};k-2)r_{\beta}\Prod_L(r_{\gamma},r_{\beta};k-2)\Prod_R(r_{\gamma},r_{\beta};k-1)r_{\gamma}\Prod_L(r_{\beta},r_{\gamma};k-1)=\\
        \Prod_R(r_{\beta},r_{\gamma};k-2)r_{\beta}r_{\beta}r_{\gamma}\Prod_L(r_{\beta},r_{\gamma};k-1)=\Prod_R(r_{\gamma},r_{\beta};k-3)\Prod_L(r_{\beta},r_{\gamma};k-1)=r_{\gamma}r_{\beta}.
    \end{align*}
        Thus, if $k$ is odd, $k-1$ is even and we get $r_{\beta_1}r_{\beta_2}\dots r_{\beta_{k-1}}=\Prod_L(r_{\gamma},r_{\beta};k-1)$. If $k$ is even, then $r_{\beta_1}r_{\beta_2}\dots r_{\beta_{k-1}}=\Prod_L(r_{\gamma},r_{\beta};k-2)r_{\beta_{k-1}}=\Prod_L(r_{\beta},r_{\gamma};k-1)$,
which is the desired equality.

\end{proof}
\noindent Thanks to Lemma \ref{lemmasimply}, if $z\in \W[\Gamma]$ and we set $\beta'=z(\beta)$, $\gamma'=z(\gamma)$, we also get the expression \[
    r_{\beta'_1}r_{\beta'_2}\dots r_{\beta'_{k-1}}=zr_{\beta_1}r_{\beta_2}\dots r_{\beta_{k-1}}z^{-1}\] 
    \noindent in terms of $r_{\gamma}$ and $r_{\beta}$.\\
    \\
\noindent The main result of this work is the following theorem.
\begin{thm}\label{isomorphismovcoverings}
 There is a cellular isomorphism $g$ between the two CW-complexes $\overline{\Sigma}(\Gamma)$ and $\overline{\Omega}(\Gamma)$.
\end{thm}

\noindent This immediately implies the next theorem.

\begin{thm}\label{maintheorem}
The Salvetti complex $\Sigma(\Gamma)$ and the BEER complex $\Omega(\Gamma)$ have a common covering space, and thus, they have the same universal covering.    
\end{thm}

\begin{proof}[Proof of Theorem \ref{isomorphismovcoverings}]
Define the map $g:\overline{\Sigma}(\Gamma)\longrightarrow \overline{\Omega}(\Gamma)$ by: 
    \begin{itemize}
        \item $g^0:\overline{\Sigma}^0(\Gamma)\longrightarrow \overline{\Omega}^0(\Gamma)$ is  $g^0(\overline{\sigma}(z))= \overline{\omega}(z^{-1}) $, for all $z\in \W[\Gamma]$;
        \item $g^1:\overline{\Sigma}^1(\Gamma)\longrightarrow \overline{\Omega}^1(\Gamma)$, is $g^1(\overline{\Delta}^1(z,\beta))= \overline{D}^1(z^{-1},z(\beta))$, for all $z\in \W[\Gamma]$ and $\beta\in \Phi[\Gamma]$;
        \item For all $k\geq 2$, $g^k:\overline{\Sigma}^k(\Gamma)\longrightarrow \overline{\Omega}^k(\Gamma)$ is
$g^k(\overline{\Delta}^k(z,\cX))=\overline{D}^k(z^{-1},z(\cX))$, for all $z\in \W[\Gamma]$ and $\cX\in \Rf_k$.
    \end{itemize}
\noindent Our goal is to show that $g$ is a well-defined cellular isomorphism. To do so, we must verify that it behaves well with respect to the boundary. Specifically, we proceed as follows.\\
\\
\textbf{Step 1.} Show that the source $\mathfrak{s}( g^1(\overline{\Delta}^1(z,\beta)))$ coincides with $g^0(\mathfrak{s}(\overline{\Delta}^1(z,\beta)))$ and that the target $\mathfrak{t}( g^1(\overline{\Delta}^1(z,\beta)))$ coincides with $g^0(\mathfrak{t}(\overline{\Delta}^1(z,\beta)))$, for all $z$ in $\W[\Gamma]$ and $\beta$ in $\Phi[\Gamma]$.\\
\\
\noindent Indeed, $\mathfrak{s}((\overline{\Delta}^1(z,\beta)))=\overline{\sigma}(z)$ and $\mathfrak{t}((\overline{\Delta}^1(z,\beta)))=\overline{\sigma}(zr_{\beta})$. Computing $g^0$ on these, we obtain $\overline{\omega}(z^{-1})$ and $\overline{\omega}(r_{\beta}z^{-1})$, respectively. 
On the other hand, $\mathfrak{s}(g^1(\overline{\Delta}^1(z,\beta)))=\mathfrak{s}(\overline{D}^1(z^{-1},z(\beta)))=\overline{\omega}(z^{-1})$, and $\mathfrak{t}(g^1(\overline{\Delta}^1(z,\beta)))=\overline{\omega}(z^{-1}r_{z(\beta)})=\overline{\omega}(r_{\beta}z^{-1})$, where we used that $z^{-1}r_{z(\beta)}=z^{-1}zr_{\beta}z^{-1}=r_{\beta}z^{-1}$. Thus, we obtain the desired equality.\\
\\
\textbf{Step 2.} We show that the oriented boundary $\partial g^2(\overline{\Delta}^2(z,\cX))$ coincides with $g^1(\partial \overline{\Delta}^2(z,\cX))$, for every $z\in \W[\Gamma]$ and $\cX\in \Rf_2$, with $ \cX=\{\beta,\gamma\}$.\\
\\
\noindent We illustrated the cell $\overline{\Delta}^2(z,\{\beta,\gamma\})$ in Figure \ref{fig2cellsigmatilde}. Suppose $\beta=w(\alpha_s)$ and $\gamma=w(\alpha_t)$ for some $w\in \W[\Gamma]$ and $s,t\in S$, and set $m:=m_{s,t}$. In this case, we assume $m$ is an even number, though the result is analogous if $m$ is odd.  We now examine the oriented boundary, starting from $\overline{\sigma}(z)$ and moving counterclockwise:
\[
\overline{\Delta}^{1}(z,\beta)\overline{\Delta}^{1}(zr_{\beta},\gamma)\cdots \overline{\Delta}^{1}(z\Prod_L(r_{\beta},r_{\gamma};m-1),\gamma)(\overline{\Delta}^{1}(z\Prod_L(r_{\gamma},r_{\beta};m-1)),\beta)^{-1}\dots
(\overline{\Delta}^{1}(z,\gamma))^{-1}.
\]
Applying $g^1$ to this path, we obtain:
\begin{align*}
&\overline{D}^{1}(z^{-1},z(\beta))\,\overline{D}^{1}(r_{\beta}z^{-1},zr_{\beta}(\gamma))\,\cdots\, \overline{D}^{1}(\Prod_L(r_{\beta},r_{\gamma};m-1)z^{-1},z\Prod_L(r_{\beta},r_{\gamma};m-1)(\gamma))\\
&(\overline{D}^{1}(\Prod_L(r_{\gamma},r_{\beta};m-1)z^{-1},z\Prod_L(r_{\gamma},r_{\beta};m-1)(\beta)))^{-1}\,\cdots\,
  (\overline{D}^{1}(r_{\gamma}z^{-1},zr_{\gamma}(\beta)))^{-1}\,(\overline{D}^{1}(z^{-1},z(\gamma)))^{-1},
\end{align*}
which we can rewrite as: 
\begin{align}
\label{f1ofbound}
&\overline{D}^{1}(z^{-1},z(\beta_1))\,\overline{D}^{1}(r_{\beta}z^{-1},z(\beta_2))\,\cdots\, \overline{D}^{1}(\Prod_L(r_{\beta},r_{\gamma};m-1)z^{-1},z(\beta_m))\nonumber\\
&(\overline{D}^{1}(\Prod_L(r_{\gamma},r_{\beta};m-1)z^{-1},z(\gamma_m)))^{-1}\,\cdots\,
  (\overline{D}^{1}(r_{\gamma}z^{-1},z(\gamma_2)))^{-1}\,(\overline{D}^{1}(z^{-1},z(\gamma_1)))^{-1}.
\end{align}
Next, we compute $\partial g^2(\overline{\Delta}^2(z,\cX))=\partial\overline{D}^2(z^{-1},z(\cX))$, and set $\beta'=z(\beta), \gamma'=z(\gamma)$. Using Remark \ref{rmksimply}, we find that the oriented boundary is:

\begin{align}\label{boudoff2}
\overline{D}^{1}(z^{-1},\beta'_1)\;\overline{D}^{1}(z^{-1}r_{\beta'_1},\beta'_2)\cdots \overline{D}^{1}(z^{-1}r_{\beta'_1}\dots r_{\beta'_{m-1}},\beta'_m)\nonumber\\
(\overline{D}^{1}(z^{-1}r_{\gamma'_1}\dots r_{\gamma'_{m-1}},\gamma'_m))^{-1} \dots(\overline{D}^{1}(z^{-1}r_{\gamma'_1},\gamma'_2))^{-1}(\overline{D}^{1}(z^{-1},\gamma'_1))^{-1}.
\end{align}
Now we know that for all $1\leq j \leq m$,  $z(\beta_{j})=\beta'_j$ for $\beta'=z(\beta)$ and  $\gamma'=z(\gamma)$. Moreover, if $j\geq2$: \[r_{\beta'_1}\dots r_{\beta'_{j-1}}=zr_{\beta_1}\dots r_{\beta_{j-1}}z^{-1},\] and hence $z^{-1}r_{\beta'_1}\dots r_{\beta'_{j-1}}=\Prod_L(r_{\beta},r_{\gamma};j-1)z^{-1}$ if $j$ is even, and $z^{-1}r_{\beta'_1}\dots r_{\beta'_{j-1}}=\Prod_L(r_{\gamma},r_{\beta};j-1)z^{-1}$ if $j$ is odd.
This proves that the expressions in Equations \ref{f1ofbound} and \ref{boudoff2} coincide.\\
\\
\textbf{Step 3.} We show that for all $k\geq 3$, the boundary $\partial g^k(\overline{\Delta}^k(z,\cX))$ coincides with $g^{k-1}(\partial \overline{\Delta}^k(z,\cX))$, for every $z\in \W[\Gamma]$ and $\cX\in \Rf$, where $|\cX|=k$.\\
\\
\noindent Consider a $k$-cell $\overline{\Delta}^k(z,\cX)$ in $\overline{\Sigma}(\Gamma)$, where $z\in \W[\Gamma]$ and $\cX\in \Rf_k$. By Corollary \ref{sigmatilde}, $\partial\overline{\Delta}^k(z,\cX)$ is given by the union $\bigcup\overline{\Delta}^{h}(zu,\cY)$ where $\cY\subset\cX$, $|\cY|=h\leq k-1$, and $u\in \W_{\cX}$ is $(\emptyset,R_{\cY})$-minimal. Computing $g^{k-1}\partial(\overline{\Delta}^k(z,\cX))$, we get: 
\[\bigcup\overline{D}^{h}(u^{-1}z^{-1},zu(\cY)),\]
for $\cY$, $h$ and $u$ as before.\\
\\
\noindent On the other hand, $g^k(\overline{\Delta}^k(z,\cX))=\overline{D}^k(z^{-1},z(\cX))$. 
According to Corollary \ref{omegatilde}, the faces of the $k$-cells of $\overline{\Omega}(\Gamma)$ are described by: \[\partial\overline{D}^k(z^{-1},z(\cX))=\bigcup\overline{D}^{h}(z^{-1}u'^{-1},u'(\cY')),\] where $\cY'\subset z(\cX)$, $|\cY'|=h\leq k-1$, and $u'\in \W_{z(\cX)}$ is $(\emptyset,R_{\cY'})$-minimal. \\
\\
\noindent We claim that $\partial g^k(\overline{\Delta}^k(z,\cX))$ and $g^{k-1}(\partial \overline{\Delta}^k(z,\cX))$ represent the same union of $h$-cells. For each $\cY \subset \cX$, $|\cY|=h\leq k-1$ and $u\in \W_{\cX}$ $(\emptyset,R_{\cY})$-minimal, choose $u':=zuz^{-1}\in z\W_{\cX}z^{-1}=\W_{z(\cX)}$ and $\cY':=z(\cY)\subset z(\cX)$, which has cardinality $h$. The set of roots satisfies $u'(\cY')=zuz^{-1}z(\cY)=zu(\cY)$, and $z^{-1}u'^{-1}=z^{-1}zu^{-1}z^{-1} = u^{-1}z^{-1}$. The only remaining task is
to show that $zuz^{-1}=u'$ is $(\emptyset,R_{\cY'})$-minimal. Conjugation by $z\in \W[\Gamma]$ is an isomorphism between the AP spherical-type reflection subgroups $\W_{\cX}$ and $\W_{z(\cX)}$, preserving  length, as it sends the generator $r_{\beta}$ to $r_{z(\beta)}$ for all $\beta \in \cX$. Thus, for $\beta\in \cY$, \[l_{z(\cX)}(u'r_{z(\beta)})=l_{\cX}(ur_{\beta})=l_{\cX}(u)+1=l_{z(\cX)}(u')+1,\] where $l_{\cX}$ (resp $l_{z(\cX)}$) denotes the word length in $\W_{\cX}$ (resp in $\W_{z(\cX)}$) with respect to $R_{\cX}$ (resp $R_{z(\cX)}$). Since $u$ is $(\emptyset,R_{\cY})$-minimal, it follows that $u'\in \W_{z(\cX)}$ is $(\emptyset,\cY')$-minimal.\\
\\
\noindent Hence, for each pair $(u,\cY)$ such that $\cY\subset\cX$, $|\cY|=h$, and $u\in \W_{\cX}$ is $(\emptyset,R_{\cY})$-minimal, there exists a unique pair $(u',\cY')$ with $\cY'\subset z(\cX)$, $|\cY'|=h$, $u'\in \W_{z(\cX)}$ $(\emptyset,R_{\cY'})$-minimal, satisfying: 
\[\overline{D}^{h}(u^{-1}z^{-1},zu(\cY))=\overline{D}^{h}(z^{-1}u'^{-1},u'(\cY')).\]
\noindent Similarly, for each pair $(u',\cY')$ such that $\cY'\subset z(\cX)$, $|\cY'|=h$, and $u'\in \W_{z(\cX)}$ is $(\emptyset,R_{\cY'})$-minimal, there exists a unique pair $(u,\cY)$ with $\cY\subset \cX$, $|\cY|=h$, and $u\in \W_{\cX}$ $(\emptyset,R_{\cY})$-minimal such that 
$\overline{D}^{h}(u^{-1}z^{-1},zu(\cY))=\overline{D}^{h}(z^{-1}u'^{-1},u'(\cY'))$. This concludes the proof of the third claim and the overall theorem.\\
\end{proof}

\noindent This theorem, as stated in \ref{maintheorem}, indicates that the two CW-complexes $\overline{\Sigma}(\Gamma)$ and $\overline{\Omega}(\Gamma)$ are, up to relabeling the faces, the same CW-complex.
 \begin{defn}\label{defteta}
Let $\Gamma$ be a Coxeter graph. The \textit{SB complex}, denoted by $\Theta(\Gamma)$, is the common covering space of both $\Omega(\Gamma)$ and $\Sigma(\Gamma)$, which is cellularly isomorphic to $\overline{\Sigma}(\Gamma)$ and $\overline{\Omega}(\Gamma)$.
 \end{defn}

\noindent Corollaries \ref{sigmatilde} and \ref{omegatilde} provide a detailed description of the skeleton structure of the complex $\Theta(\Gamma)$. The fundamental group of the SB complex is $\LVA[\Gamma]$, which does not have a specific name in the literature. Additionally, for $\Gamma$ of spherical type or of affine type, the SB complex is aspherical, thereby serving as a classifying space for $\LVA[\Gamma]$. This relationship arises because $\Theta(\Gamma)$ shares its universal covering space with $\Sigma(\Gamma)$, which is known to be contractible in the cases of spherical and affine types, as established in Theorem \ref{gammahatkpiAFF}.
\noindent We now present the two corollaries that affirm the asphericity of the BEER complex $\Omega(\Gamma)$ for $\Gamma$ of spherical or affine type, which was the original objective of this work.
\begin{cor}\label{beeraspht}
 If $\Gamma$ is a Coxeter graph of spherical type, then  the BEER complex $\Omega(\Gamma)$ is aspherical. Consequently, it is a classifying space for the pure virtual Artin group $\PVA[\Gamma]$.  
\end{cor}

\noindent In particular, the result holds for $\Gamma=A_{n-1}$. We obtain then the same result of asphericity for $\Omega_n$, even if it is not locally CAT(0). Combining Theorems \ref{omegan=omegaxan} and \ref{maintheorem}, we obtain the following.

\begin{cor}\label{originalbeeraspht}
    The Bartholdi\textendash Enriquez\textendash Etingof\textendash Rains complex $\Omega_n$ is aspherical.
\end{cor}
\noindent Lastly, we extend Corollary \ref{beeraspht} to the affine case, drawing upon the arguments discussed in the introduction to this subsection.

\begin{cor}\label{beerasphtAFF} If $\Gamma$ is a Coxeter graph of affine type, then  the BEER complex $\Omega(\Gamma)$ is aspherical. Consequently, it is a classifying space for the pure virtual Artin group $\PVA[\Gamma]$.   
\end{cor}

\begin{rmk}
When $\Gamma$ is not of spherical or of affine type, it remains unclear whether $\Omega(\Gamma)$ is aspherical. From Theorem \ref{maintheorem} we can deduce that this question is equivalent to the $K(\pi,1)$-conjecture for the Artin group $\A[\widehat{\Gamma}]$, indicating that we do not anticipate an easy resolution.
\end{rmk}

\medskip

\medskip
\noindent
\textit{\textbf{Federica Gavazzi} \\  Université Bourgogne Europe, CNRS, IMB UMR 5584, 21000 Dijon, France.} \par
 \noindent \textit{E-mail address:} \texttt{\href{mailto:Federica.Gavazzi@u-bourgogne.fr}{Federica.Gavazzi@u-bourgogne.fr}}

\medskip

\bibliographystyle{alpha}
\bibliography{Bibliography}

\bigskip\bigskip{\footnotesize%

\end{document}